\documentclass[10pt,reqno]{amsart}
\usepackage{a4wide}
\numberwithin{equation}{section}
\usepackage{floatflt}
\usepackage{tikz}
\usepackage{pgfplots}
\usepackage{subfig}
\usepackage{mathrsfs}
\usepackage{amsfonts}
\usepackage{amsmath}
\usepackage{stmaryrd}
\usepackage{amssymb}
\usepackage{cancel}
\usepackage{amsthm}
\usepackage{mathrsfs}
\usepackage{url}
\usepackage{amsfonts}
\usepackage{wrapfig}
\usepackage{amscd}
\usepackage{indentfirst}
\usepackage{enumerate}
\usepackage{amsmath,amsfonts,amssymb,amsthm}
\usepackage{amsmath,amssymb,amsthm,amscd}
\usepackage{graphicx,mathrsfs}
\usepackage{appendix}
\usepackage[T1]{fontenc}
\usepackage[utf8]{inputenc}
\usepackage{wrapfig}
\usepackage[numbers,sort&compress]{natbib}
\usetikzlibrary{calc}
\usepackage{amsmath,graphicx}
\usepackage{xcolor}

\usepackage{empheq}

\usepackage{color}
\usepackage[colorlinks, linkcolor=black, citecolor=blue]{hyperref}

\newtheorem{teo}{Theorem}[section]
\newtheorem{lem}[teo]{Lemma}
\newtheorem{prop}[teo]{Proposition}

\newtheorem{rem}[teo]{Remark}
\newtheorem{defi}[teo]{Definition}

\newcommand\R{\mathbb R}

\newcommand\mbb\mathbb
\newcommand\mbf\mathbf
\newcommand\mcal\mathcal
\newcommand\mfrak\mathfrak
\newcommand\mrm\mathrm
\newcommand\msf\mathsf
\renewcommand\a\alpha
\renewcommand\b\beta
\newcommand\g\gamma
\newcommand\G\Gamma
\renewcommand\d\delta
\newcommand\D\Delta
\newcommand\e\varepsilon
\newcommand\z\zeta
\renewcommand\t\theta
\newcommand\Th\Theta
\newcommand\la\lambda
\newcommand\La\Lambda
\newcommand\s\sigma
\newcommand\si\varsigma
\newcommand\Si\Sigma
\newcommand\ups\upsilon
\newcommand\U\Upsilon
\newcommand\ph\varphi
\renewcommand\o\omega
\renewcommand\O\Omega
\newcommand\wt\widetilde
\newcommand\wh\widehat
\newcommand\ol\overline
\newcommand\ul\underline
\newcommand\mr\mathring
\newcommand\ub\underbrace
\newcommand\pa\partial
\newcommand\n\nabla
\newcommand\fa\forall
\newcommand\ex\exists
\newcommand\es\emptyset
\newcommand\wk\rightharpoonup
\newcommand\inc\hookrightarrow
\newcommand\linf\varliminf
\newcommand\lsup\varlimsup
\newcommand\os\overset
\newcommand\us\underset
\newcommand\sr\stackrel
\newcommand\Ot\Leftarrow
\newcommand\To\Rightarrow
\newcommand\map\mapsto
\newcommand\ot\leftarrow
\newcommand\lot\longleftarrow
\newcommand\lto\longrightarrow
\newcommand\tot\leftrightarrow
\newcommand\ltot\longleftrightarrow
\newcommand\sm\backslash
\renewcommand\Cup\bigcup
\renewcommand\Cap\bigcap
\newcommand\sub\subset
\newcommand\Sub\Subset
\newcommand\sne\subsetneq
\newcommand\bus\supset
\newcommand\Bus\Supset

\newcommand\eq\equiv
\newcommand\ox\otimes
\newcommand\Ox\bigotimes
\newcommand\pl\oplus
\newcommand\Pl\bigoplus
\newcommand\x\times
\renewcommand\c\circ
\newcommand\q\quad
\renewcommand\l\left
\renewcommand\r\right
\newcommand\fr\frac

\allowdisplaybreaks

\definecolor{darkgreen}{rgb}{0.0, 0.4, 0.0}

\def\sideremark#1{\ifvmode\leavevmode\fi\vadjust{\vbox to0pt{\vss
 \hbox to 0pt{\hskip\hsize\hskip1em
 \vbox{\hsize2.1cm\tiny\raggedright\pretolerance10000
  \noindent #1\hfill}\hss}\vbox to15pt{\vfil}\vss}}}%

\begin{document}

\title[Qualitative analysis on the critical points of the Kirchhoff-Routh function]
{Qualitative analysis on the critical points \\ of the Kirchhoff-Routh function}
\author[F. Gladiali, M. Grossi, P. Luo and S. Yan]{Francesca Gladiali, Massimo Grossi, Peng Luo, Shusen Yan}

\address[Francesca Gladiali]{Department of Chemical, Physical, Mathematical and Natural Sciences,
Via Vienna 2,
07100 Sassari,
Italy, }
 \email{fgladiali@uniss.it}

\address[Massimo Grossi]{Dipartimento di Scienze di Base Applicate per l'Ingegneria , Universit$\grave{a}$ Sapienza, P.le Aldo Moro 5, 00185 Roma, Italy}
 \email{massimo.grossi@uniroma1.it}

 \address[Peng Luo]{School of Mathematics and Statistics, Key Laboratory of Nonlinear Analysis and Applications
(Ministry of Education), and Hubei Key Laboratory of Mathematical
Sciences,
Central China Normal University, Wuhan 430079, China}
 \email{pluo@ccnu.edu.cn}

\address[Shusen Yan]{School of Mathematics and Statistics, Key Laboratory of Nonlinear Analysis and Applications
(Ministry of Education), Central China Normal University, Wuhan 430079, China}
\email{syan@ccnu.edu.cn}

\begin{abstract}
In this paper, we study the number of critical points of the  Kirchhoff-Routh function
 \begin{equation*}
\mathcal{KR}_D(x,y)=\La_1^2\mathcal{R}_D(x)+\La_2^2\mathcal{R}_D(y)-2\La_1\La_2G_D(x,y),
\end{equation*}
where $D$ is a bounded domain in $\mathbb{R}^2$, $x,y\in D$, $\Lambda_1,\Lambda_2>0$, $\mathcal{R}_D$ is the Robin function, and $G_D$ is
the Green function
of the operator $-\Delta$ with $0$ Dirichlet boundary condition on $D$.
This function arises from concentration phenomena in nonlinear elliptic problems
and from the de-singularization problem for the steady Euler equation. For domains with a small hole, we
establish not only the exact number and the location of the critical points of  $\mathcal{KR}_D$, but also their nondegeneracy. We show that the location of the hole plays a crucial role. Finally in the context of elliptic problems, we establish the existence of multiple two-peak solutions.
\end{abstract}
\date{\today}
\maketitle

 \tableofcontents

\section{Introduction and main results}
Let $D\subset\R^N$, $N\ge2$, be a smooth bounded domain. For $(x,y)\in D\times D$, $x\ne y$, we denote by $G_D(x,y)$ the Green function of $D$, which satisfies
\begin{equation*}
\begin{cases}
-\D_x G_D(x,y)=\delta_x(y),&\hbox{in }D,\\[2mm]
G_D(x,y)=0,&\hbox{on }\partial D,
\end{cases}
\end{equation*}
in the sense of distributions. We have the classical representation formula
\begin{equation*}
G_D(x,y)=S(x,y)-H_D(x,y),
\end{equation*}
where $H_D(x,y)$ is the {\em regular part of the Green function}, which is harmonic in both variables $x$ and $y$, and
$S(x,y)$ is the {\em fundamental solution} given by
\begin{equation}\label{sec1.01}
S(x,y)=
\begin{cases}
-\frac{1}{2\pi}\ln |x-y|,&~\mbox{if}~N=2,\\[2mm]
\frac{C_N}{|x-y|^{N-2}},&~\mbox{if}~N\geq 3,
\end{cases}
\end{equation}
where $C_N:=\frac{1}{N(N-2)\omega_N}$, with $\omega_N$ being the volume of the unit ball in $\R^N$.
We denote by $\mathcal{R}_D$ the \emph{Robin} function of $D$, namely
\begin{equation}\label{sec1.02}
\mathcal{R}_D(x):=H_D(x,x).
\end{equation}Let us recall the definition of the \emph{Kirchhoff-Routh} function. For $D\subset\R^N$, $k\ge 1$, and $(\La_1,\cdots,\La_k)\in \R^k$, with $\La_i\neq 0$ for $i=1,\cdots,k$,
set $\mathcal{KR}_{k,D}(x_1,\cdots,x_k):\underbrace{D\times\cdots\times D}_{:=D^k}\to \R$ defined as
\begin{equation}\label{sec1.03}
\mathcal{KR}_{k,D}(x_1,\cdots,x_k)=\sum_{i=1}^k\La_i^2\mathcal{R}_D(x_i)-\sum_{i\ne j\ i,j=1}^k\La_i\La_jG_D(x_i,x_j).
\end{equation}
The case $k=1$ corresponds to $\mathcal{KR}_{1,D}(x)=\La_1^2\mathcal{R}_D(x)$.
\vskip0.05cm
The \emph{Kirchhoff-Routh} function for the case $N=2$ was introduced by Kirchhoff and Routh in the 19th century (see \cite{Kirchhoff1876,Routh1881}). They derived the formal dynamical law for the evolution of vortex trajectories in the study of the two-dimensional {\em Euler flow} for an incompressible fluid confined to a smooth domain. In the case of point vortex solutions, for which the vorticity is given by $\sum_{j=1}^k \Lambda_j\delta_{x_j}$,
the vortices can be located only at a critical point of the $\mathcal{KR}_{k,D}$-function (see \cite{LinPNAS1941}).
\vskip0.05cm

The computation of the number of critical points of $\mathcal{KR}_{k,D}$ has some important applications in various PDE problems. Some of them are the Gel'fand problem
 \begin{equation}\label{sec1.04}
\begin{cases}
-\D u =\lambda e^u,~u>0,~&\hbox{in }D,\\[1mm]
u=0,&\hbox{on }\partial D,
\end{cases}
\end{equation}
and the Lane-Emden problem
 \begin{equation}\label{sec1.05}
\begin{cases}
-\D u =u^p,~u>0,~&\hbox{in }D,\\[1mm]
u=0,&\hbox{on }\partial D,
\end{cases}
\end{equation}
 where $D$ is a bounded and smooth domain of $\R^2$,
 $\lambda>0$ is a small parameter in \eqref{sec1.04} while $p>1$ is large in \eqref{sec1.05}.

\vskip0.05cm
In both problems, as the parameter $\la\to0$ and $p\to+\infty$, concentration phenomena occur. More precisely, regarding problem \eqref{sec1.04}, if we denote by $x_\la$ the maximum point of the solution $u_\la(x)$, then $u_\la(x_\la)\to+\infty$ as $\la\to0$ (an analogous phenomenon occurs for \eqref{sec1.05} as $p\to+\infty$). Of course, investigating the limiting position of the points $x_\la$ is a problem of great interest. In this context various papers (see for example \cite{bjly2,gg1,ggos} for \eqref{sec1.04} and \cite{ag,emp,GILY} for \eqref{sec1.05})  proved the following results.
\vskip 0.1cm

\noindent\textbf{Theorem A}
{\em Assume that $\La_i=1$ for any $i=1,\cdots,k$ in \eqref{sec1.03}, then
\begin{itemize}
\item [(i)] If $u_\lambda$ is a solution of problem \eqref{sec1.04} (or $u_p$ for \eqref{sec1.05})  which concentrates at $(x_1,\cdots,x_k)\in D^k$, we have
$$\nabla\mathcal{KR}_{k,D}(x_1,\cdots,x_k)=0.$$
\item [(ii)] Furthermore, if $(\bar x_1,\cdots,\bar x_k)$
is a nondegenerate critical point of $\mathcal{KR}_{k,D}(x_1,\cdots,x_k)$, then there exists a family of solutions $u_\lambda$ to \eqref{sec1.04} (or $u_p$ for \eqref{sec1.05}), which concentrate at $\bar x_1,\cdots,\bar x_k$,
as $\lambda\to  0$ (or $p\to+\infty$).  \vskip 0.1cm
\item [(iii)] The solution $u_\lambda$ (or $u_p$) is locally unique provided that
$(\bar x_1,\cdots,\bar x_k)$
is a nondegenerate critical point of $\mathcal{KR}_{k,D}(x_1,\cdots,x_k)$.
Here “local uniqueness” means that if two solutions concentrate at
$(\bar x_1,\cdots,\bar x_k)$, then they coincide.\vskip 0.1cm
\item [(iv)]If $(x_1,\cdots,x_k)\in D_1\times\cdots\times D_k$
is a nondegenerate critical point of $\mathcal{KR}_{k,D}(x_1,\cdots,x_k)$, the Morse index of above concentrated solutions is
$k+m\big(\mathcal{KR}_{k,D}(x_1,\cdots,x_k)\big)$. Here $m\big(\mathcal{KR}_{k,D}(x_1,\cdots,x_k)\big)$ is the number of negative eigenvalues of the Hessian matrix of $\mathcal{KR}_{k,D}(x_1,\cdots,x_k)$.
\end{itemize}}

\vskip 0.05cm
On the other hand, the following de-singularization problem has been studied
extensively: 
\begin{equation}\label{sec1.06}
\begin{cases}
-\Delta u=\alpha\displaystyle\sum^k_{j=1}1_{B(x_j,\delta)}\Big(u-\frac{\Lambda_j\ln \alpha}{2\pi}\Big)_+^p, &\mbox{in}~~D,\\
u=0, & \mbox{on}~~ \partial D,
\end{cases}
\end{equation}where $\Lambda_j>0$, $\alpha>0$, $p\ge 0$, $D$ is a bounded domain in $\R^2$, $x_j\in D$ satisfying $x_i\ne x_j$, and $
\delta>0$ is small such that $B(x_i,\delta)\cap
B(x_j,\delta)=\emptyset$ for $i\ne j$, and $1_S=1$ in $S$ and $1_S=0$ elsewhere  (see for example
 \cite{CGPY2019,CPY2010,CPY2015}).
We want to find a solution $u_\alpha$ for \eqref{sec1.06} satisfying the following
property:
\vskip 0.1cm

\noindent$(V)$  As $\alpha\to
+\infty$, the support of $\bigl(u_\alpha-\frac{\Lambda_j\ln \alpha}{2\pi}\bigr)_+$ in $ B(x_j,\delta)$ shrinks to  $x_j$.

\vskip 0.1cm

Such a solution $u _\alpha$ satisfies that, as $\alpha\to +\infty$,
\[
\alpha \cdot 1_{B(x_j,\delta)}\Big(u_\alpha-\frac{\Lambda_j\ln \alpha}{2\pi}\Big)_+\rightharpoonup
\Lambda_j \delta_{x_j}.
\]
For the de-singularization problem \eqref{sec1.06}, we have similar existence and uniqueness results as for \eqref{sec1.04} and \eqref{sec1.05}.
\vskip0.05cm

From the previous results we get that the number of solutions for \eqref{sec1.04}, \eqref{sec1.05}
and  \eqref{sec1.06} is closely linked to the existence of critical points of $\mathcal{KR}_{k,D}$ and their non-degeneracy.
For these reasons, in the last decades, there has been great interest in computing and locating the critical points of $\mathcal{KR}_{k,D}$.

\vskip0.05cm
Let us start by recalling a result from \cite{GrossiTakahashi}.
\vskip0.1cm

\noindent\textbf{Theorem B} (c.f. \cite{GrossiTakahashi})
{\em If $D\subseteq \R^N$ ($N\geq 2$) is convex and $k\geq 2$, then for any $\La_1,\cdots,\La_k>0$, there are no critical points of $\mathcal{KR}_{k,D}$.}
\vskip0.2cm
Various existence results for critical points of $\mathcal{KR}_{k,D}$ in non-convex domains $D$ can be found in \cite{barpis,bpw,dkm,egp,emp}. For example, in \cite{egp} it is shown that in a {\em dumbbell}-type domain with $m$ handles, the function $\mathcal{KR}_{k,D}$ (as well as its $C^1$ perturbation)
admits at least one  critical point for every $k\leq m+1$.
In  \cite{CPY2010,dkm}, it was proved that at least one  critical point of $\mathcal{KR}_{k,D}$ (as well as its $C^1$ perturbation) exists for any $k\geq 2$,
if $\La_i>0$ and $D$ is a domain with holes. In this paper, we improve this result  for $k=2$, assuming that the size of the hole is {\em small}, and we prove more precise multiplicity results and the nondegeneracy for the critical points.
\vskip0.05cm

This paper continues the project started in \cite{ggly1}, where an analysis of the critical points of the Robin function \eqref{sec1.02} in a domain with a small hole was carried out. We briefly summarize some of the main results from \cite{ggly1}.
\vskip 0.2cm

\noindent\textbf{Theorem C} (c.f. \cite{ggly1})
\emph{
Suppose $\Omega$ is a bounded smooth domain in $\R^N (N\ge2)$ such that all the critical points of $\mathcal R_{\O}(x)$ in $\O$ are nondegenerate. Let $P\in\Omega$ and set  $\Omega_\e=\O\setminus B(P,\e).$ For $\varepsilon$ small enough, we have the following results. \vskip 0.1cm
\begin{itemize}
\item If $\nabla \mathcal{R}_\O(P)\neq 0$, then
$$\sharp\Big\{\mbox{critical points of  $\mathcal{R}_{\O_\e}$ in $\O_\e$}\Big\}=1+\sharp\Big\{\mbox{critical points of  $\mathcal{R}_{\O}$ in $\O$}\Big\}.$$
Moreover, the additional critical point $x_\e\in \O_\e$ of $\mathcal{R}_{\O_\e}$ is nondegenerate and $x_\e\to P$ as $\e\to 0$.\vskip 0.1cm
\item If $\nabla \mathcal{R}_\O(P)=0$ and
the Hessian matrix $\nabla^2\big(\mathcal{R}_\O(P)\big)$ has $N$ simple positive
eigenvalues, then
 $$\sharp\Big\{\mbox{critical points of  $\mathcal{R}_{\O_\e}$ in $\O_\e$}\Big\}=2N-1+\sharp\Big\{\mbox{critical points of  $\mathcal{R}_{\O}$ in $\O$}\Big\}.$$
\end{itemize}}
The previous theorem shows that the location of the hole $B(P,\e)$ is important. Indeed, the number of critical points of $\mathcal{R}_{\O_\e}$ changes depending on whether $P$ is a critical point of $\mathcal{R}_\O$ or not. Note that for any bounded $\O\subset\R^N$, a minimum of $\mathcal{R}_\O$ always exists.
\vskip0.05cm

If we consider the \emph{Kirchhoff-Routh} function with $k>1$, some similarities with Theorem C are expected. For example, the number of critical points of $\mathcal{KR}_{k,\O\setminus B(P,\e)}$ for small $\e$ will be influenced by the corresponding number for the ``unperturbed'' $\mathcal{KR}_{k,\O}$-function.

On the other hand, there are important differences that make the problem very interesting. The main one is that a critical point for $\mathcal{R}_{\O}$ always exists for any bounded $\O\subset\R^N$, whereas this is not true for $\mathcal{KR}_{k,\O}$ in convex domains, by Theorem B. Secondly, we will see that even if a critical point of $\mathcal{KR}_{k,\O}$ exists, the role of the location of the hole is much more involved.
\vskip0.05cm
The study of the critical points of $\mathcal{KR}_{k,\O}$ is more complex than it may seem and cannot simply be reduced to a straightforward extension of the case of the Robin function.
For this reason, and to keep the paper within a reasonable length, we consider only the case $N=2$, $k=2$, and $\La_1,\La_2>0$. In fact, even this simpler case involves several, often delicate, estimates.  Still keeping in mind the parallelism with semilinear elliptic problems, we must note that the role of the Kirchhoff-Routh function involves the parameters $\Lambda_i$ in a much more intricate way. Indeed, in many semilinear problems, the Kirchhoff-Routh function is typically replaced by
\[ \mathcal{KR}_{k,\O}(x,y)+f(\Lambda_1,\Lambda_2)
\]
for some suitable function $f$. However, we believe that the techniques introduced in this paper will make it possible to deal with this case as well. All these will be the subject of future work.

It would also be interesting to study the case $k=2$, where $\La_1$ and $\La_2$ have opposite signs, since we expect different results from our case. However, this study is beyond the scope of the present work.

From now on, we take $N=2$, $k=2$, $\La_1,\La_2>0$ and set $\mathcal{KR}_{2,D}=\mathcal{KR}_D$ with
\begin{equation}\label{form}
\mathcal{KR}_D(x,y)=\La_1^2\mathcal{R}_D(x)+\La_2^2\mathcal{R}_D(y)-2\La_1\La_2G_D(x,y).
\end{equation}
We are interested in studying the critical points of \eqref{form} where $D$ is a domain with a small hole. Therefore, we take a smooth bounded domain $\O$ such that $P\in\O$, and set
\begin{equation*}
\O_\e=\O\setminus B(P,\e)
\end{equation*}
and look for critical points of the function $ \mathcal{KR}_{\O_\e}(x,y)$.
\vskip 0.1cm
We observe that $\mathcal{KR}_{\Lambda_1,\Lambda_2,D}(x,y)=\mathcal{KR}_{\Lambda_2,\Lambda_1,D}(y,x)$.  In particular when $\Lambda_1=\Lambda_2$ if $(x,y)$ is a critical point for $\mathcal{KR}_{D}$, then $(y,x)$ is also a critical point.

 \begin{defi}\label{def1}
 If $\Lambda_1=\Lambda_2$,
 we say that two critical points $(x_1,y_1)$ and $(x_2, y_2)$ for $\mathcal{KR}_{D}$ are {\em nontrivially different}, if $(x_2,y_2)\ne (y_1, x_1)$. \end{defi}

 Nontrivially different critical points for $\mathcal{KR}_{D}$ produce nonequivalent solutions for problems \eqref{sec1.04}, \eqref{sec1.05} and \eqref{sec1.06}. \vskip 0.1cm
Let us state a first property satisfied by the critical points of $ \mathcal{KR}_{\O_\e}(x,y)$.
\begin{prop}\label{sec1-prop.02}
Let $(x_\e,y_\e)$ be a critical point of $\mathcal{KR}_{\Omega_\e}(x,y)$
with $(x_\e,y_\e)\to (x_0,y_0)\in \overline \O\times \overline \O$ as $\e\to 0$. Then
\vskip 0.1cm
\noindent \textup{(1)} there exists a positive constant $\delta$ such that
\begin{equation}\label{sec1-07}
\min\Big\{dist\{x_0,\partial \Omega\},dist\{y_0,\partial \Omega\}\Big\}\geq \delta.
\end{equation}
\vskip 0.1cm
\noindent \textup{(2)} if $x_0=y_0$, then it holds $x_0=y_0=P$.
\end{prop}
This proposition is proved in Section \ref{sec-3}. Next, we provide a classification of the critical points of $\mathcal{KR}_{\O_\e }(x,y)$.

\vskip 0.2cm

 \begin{defi}\label{def2}
\emph{Let  $(x_\e,y_\e)$ be a critical point of $\mathcal{KR}_{\O_\e }(x,y)$ with
$(x_\e,y_\e)\to (x_0,y_0)\in\O\times \O$ as $\e\to 0$. We define}

\vskip 0.2cm

\emph{\noindent\textup{(1)} $(x_\e,y_\e)$ is of \emph{type I} if $x_0\neq P$ and $y_0\neq P$.}

\vskip 0.2cm

\emph{\noindent\textup{(2)} $(x_\e,y_\e)$ is of \emph{type II} if $x_0=P$ and $y_0\neq P$ (or $x_0\neq P$ and $y_0=P$).}

\vskip 0.2cm

\emph{\noindent\textup{(3)} $(x_\e,y_\e)$ is of \emph{type III} if $x_0=y_0=P$.}
\end{defi}
\vskip 0.2cm

Different types of critical points lead to different situations, which we analyze separately.

\subsection{Critical points of type I}\

\vskip 0.1cm

\begin{figure}
\begin{tikzpicture}
    \begin{scope}[rotate=-45, scale=2]
        \draw (0,0) to [out=140,in=90] (-1,-1)
        to [out=-90,in=240] (0.8,-0.6)
        to [out=60,in=-60] (1.2,1.2)
        to [out=120,in=90] (0.3,0.7)
        to [out=-90,in=20] (0.3,0)
        to [out=200,in=-40] (0,0);

        \node[anchor=north east] at (-0.5,-0.3) {$x_\e\to x_0$};
        \node[anchor=south west] at (0.5,0.5) {$y_\e\to y_0$};
        \node at (0.3,-0.5) {$P$};

        \draw (0.5,-0.5) circle [radius=0.1];
        \fill  (-0.43,-0.35) circle[radius=0.7pt];
        \fill  (1,1.1) circle[radius=0.7pt];
        \fill  (0.5,-0.5) circle[radius=0.6pt];
    \end{scope}
\end{tikzpicture}
\caption{The case where $(x_\e,y_\e)\to (x_0,y_0)$ with $x_0,y_0\ne P$}
\end{figure}
These critical points appear as perturbations of those of $\mathcal{KR}_{\O}(x,y)$. Thus from Theorem B, they occur in specific non-convex settings, as shown in Figure 1.
Since we are removing a small ball $B(P,\e)$ {\em far away} from both points $x_0$ and $y_0$, the problem is not too complicated and can be approached using the classical critical point theory.
 \begin{teo}\label{SEC1-TEO.03}
 Let $(x_\e,y_\e)$ be a type I critical point of $\mathcal{KR}_{\O_\e }(x,y)$ such that $(x_\e,y_\e)\to (x_0,y_0)$ as $\e\to 0$. Then
 $(x_0, y_0)$ must be a critical point of $\mathcal{KR}_{\O}(x,y)$.

\vskip 0.1cm

 Conversely, if $\mathcal{KR}_{\O}(x,y)$ has a nondegenerate critical point
  $(x_0, y_0)$, then $\mathcal{KR}_{\O_\e }(x,y)$ has exactly
{\bf{one}} critical point of type I in  $B(x_0,d)\times B(y_0,d)$ for small fixed $d>0$.
 Furthermore, this critical point is nondegenerate and satisfies $(x_\e,y_\e)\to(x_0, y_0)$ as $\e\to 0$.
\end{teo}
\begin{rem}\label{sec1-rem.04}
Theorem~\ref{SEC1-TEO.03}, together with Theorem B (see \cite{GrossiTakahashi}), implies that
$\mathcal{KR}_{\O_\e }(x,y)$ has no critical points of type I if $\Omega$ is convex.
\end{rem}
These results are proved in Section \ref{sec-4}.
\vskip 0.2cm
\subsection{Critical points of type II}\

\vskip 0.1cm
In this case, several unexpected and interesting phenomena appear.
First of all we have the following necessary condition.
 \begin{prop}\label{sec1-prop.05} Let $(x_\e,y_\e)$ be a type II critical point of $\mathcal{KR}_{\O_\e }(x,y)$. If
$x_\e\to P$ and $y_\e\to y_0\in\O\backslash\{P\}$ as $\e\to 0$, then
\begin{equation}\label{sec1-08}
\frac{\partial \mathcal{KR}_{\O}(P,y_0)}{\partial y_j}=0,~~~\mbox{for}~~~j=1,2.
\end{equation}
Similarly if $x_\e\to x_0\in\O\backslash\{P\}$ and $y_\e\to P$ as $\e\to 0$, then
\begin{equation}\label{sec1-09}
\frac{\partial \mathcal{KR}_{\O}(x_0,P)}{\partial x_j}=0,~~~\mbox{for}~~~j=1,2.
\end{equation}
\end{prop}
\begin{figure}
\centering
\subfloat[$\nabla \mathcal{KR}_{\Omega}(P,y_0)\ne0$]
{\begin{tikzpicture}[scale=0.95]
    \begin{scope}[rotate=-45, scale=2] 
         \draw (0,0) circle [radius=1] ;
         \node[anchor=north east] at (0,-0.1) {$x_\e\to P$};
        \node[anchor=south west] at (-0.5,-0.2) {$y_\e\to y_0$};[anchor=north east] at

        \draw (0,0) circle [radius=0.1] ; 

        \fill  (-0.06,0.27) circle[radius=0.7pt]; 
        \fill   (0,0) circle[radius=0.6pt]; 
    \end{scope}
\end{tikzpicture}}\qquad\qquad\qquad
\subfloat[$\nabla \mathcal{KR}_{\O}(P,y_0)=0$]
{\begin{tikzpicture}[scale=0.9]
    \begin{scope}[rotate=-45, scale=2] 
        \draw (0,0) to [out=140,in=90] (-1,-1)
        to [out=-90,in=240] (0.8,-0.6)
        to [out=60,in=-60] (1.2,1.2)
        to [out=120,in=90] (0.3,0.7)
        to [out=-90,in=20] (0.3,0)
        to [out=200,in=-40] (0,0);

        \node[anchor=north east] at (-0.5,-0.33) {$x_\e\to P$};
        \node[anchor=south west] at (0.5,0.5) {$y_\e\to y_0$};[anchor=north east] at

        \draw (-0.5,-0.3) circle [radius=0.1] ; 

        \fill  (-0.5,-0.3) circle[radius=0.7pt]; 
        \fill  (1,1) circle[radius=0.7pt]; 
    \end{scope}
\end{tikzpicture}}
\caption{}
\end{figure}
\vskip 0.1cm

Let us focus on the case $x_\e\to P$ and $y_\e\to y_0\neq P$, noting that the same results hold in the other case as well.
We have to consider the following alternative, see Figure 2,
\vskip 0.2cm

\begin{center}
 (i)  $\nabla \mathcal{KR}_{\O}(P,y_0)\neq 0$,\,\,\,\,\,\,\,\,
 (ii)  $\nabla \mathcal{KR}_{\O}(P,y_0)=0$.
\end{center}

\vskip 0.1cm
For simplicity, we just study the case {\rm (i)} in a convex domain, where only {\rm (i)} occurs.
\vskip0.1cm
\noindent{\bf Case {\rm (i)}: Suppose that
$\Omega$ is convex.
}
\vskip0.1cm
Our starting point is to study the solutions of either \eqref{sec1-08} or \eqref{sec1-09}. If $\Omega$ is convex, we know that  $\mathcal{KR}_{\O}(x,y)$ has no critical points, hence if \eqref{sec1-08} holds, then
$\nabla_x \mathcal{KR}_{\O}(P,y_0)\neq 0$. Alternatively if \eqref{sec1-09} holds, then
$\nabla_y \mathcal{KR}_{\O}(x_0,P)\neq 0$.

\vskip 0.05cm
We start considering the case where $\Omega=B(0,r)$ is a ball.
\begin{teo}\label{SEC1-TEO06}
Assume that $\O=B(0,r)$ with $P\in \O$ and $\O_\e=B(0,r)\setminus B(P,\e)$. Then denoting by $d=dist\{P,\partial B(0,r)\}$, we have that there exist $d_1,d_2 \in (0, 1)$ such that if
\begin{itemize}
\item[a)]
$d>\max\{d_1,d_2\}$, then $\mathcal{KR}_{\O_\e}(x,y)$ has {\bf{no}} type II critical points.\vskip 0.1cm
\item[b)]
$d<\min\{d_1,d_2\}$, then
$\mathcal{KR}_{\O_\e}(x,y)$ has exactly  {\bf{four}}  type II critical points such that
\begin{align}
(x_{1,\e},y_{1,\e})\to\big(P,y_1(P)\big)~~~\mbox{and}~~~
(x_{2,\e},y_{2,\e})\to\big(P,y_2(P)\big),\label{sec1-10}\\
(x_{3,\e},y_{3,\e})\to\big(x_1(P),P\big)~~~\mbox{and}~~~
(x_{4,\e},y_{4,\e})\to\big(x_2(P),P\big),\label{sec1-11}
\end{align}
where $y_i(P)$ are the solutions to \eqref{sec1-08} for $d<d_1$ and $x_i(P)$ are the solutions to \eqref{sec1-09} for $d<d_2$.
Moreover these critical points are nondegenerate
and satisfy
 \begin{equation}\label{sec1-12}
\begin{split}
index \big(\nabla_y \mathcal{KR}_{\O_\e}(x_{1,\e},\cdot),y_{1,\e}\big)=1\hbox{ and }index\big(\nabla_y  \mathcal{KR}_{\O_\e}(x_{2,\e},\cdot),y_{2,\e}\big)=-1,\\
index \big(\nabla_x \mathcal{KR}_{\O_\e}(\cdot,y_{3,\e}),x_{3,\e}\big)=1\hbox{ and }index\big(\nabla_x \mathcal{KR}_{\O_\e}(\cdot,y_{4,\e}),x_{4,\e}\big)=-1.
\end{split}
\end{equation}
Finally, if the hole $P$ approaches the boundary of $B(0,r)$ (so that $d\to 0$) we have
\begin{align*}
\lim_{d\to 0}|y_1(P)-P|=0,~~~\lim_{d\to 0}y_2(P)=0~~~
\mbox{and}~~~\lim_{d\to 0}|x_1(P)-P|=0,~~~\lim_{d\to 0}x_2(P)=0.
\end{align*}\vskip 0.1cm
\item[c)] $\min\{d_1,d_2\}<d<\max\{d_1,d_2\}$, then
$\mathcal{KR}_{\O_\e}(x,y)$ has exactly {\bf{two}} nondegenerate type II critical points that verify
one among \eqref{sec1-10} and \eqref{sec1-11} and the corresponding properties in \eqref{sec1-12}.
\end{itemize}
\end{teo}

\vskip 0.2cm 
\begin{center}
 \begin{minipage}{0.3\textwidth}
        \centering
\begin{tikzpicture}
   \draw (0,0) circle (1);
       \node at (0,-0.2){\tiny$O$};
       \fill (0,0) circle (0.8pt);
\fill[gray, opacity=0.3]  (0.6,0.6) circle (0.1);
\draw  (0.6,0.6) circle (0.1);
 \node at (0.7,0.4) {\tiny$P$};
\fill (0.6,0.6) circle (0.8pt);
\node at (0,-1.3) {$d<\min\{d_1,d_2\}$};
\node at (0,-1.8) {Four type II critical points};
\end{tikzpicture}
\end{minipage}
 \hfill
 \begin{minipage}{0.3\textwidth}
        \centering
\begin{tikzpicture}
   \draw (0,0) circle (1);
       \node at (0,-0.2) {\tiny$O$};
\fill (0,0) circle (0.8pt);
\fill[gray, opacity=0.3]  (0.4,0.4) circle (0.1);
\draw (0.4,0.4) circle (0.1);
 \node at (0.5,0.2) {\tiny$P$};
\fill (0.4,0.4) circle (0.8pt);
\node at (0,-1.3) {$\min\{d_1,d_2\}<d<\max\{d_1,d_2\}$};
\node at (0,-1.8) {Two type II critical points};
\end{tikzpicture}
\end{minipage}
\hfill
 \begin{minipage}{0.3\textwidth}
\begin{tikzpicture}
   \draw (0,0) circle (1);
    \node at (0,-0.2) {\tiny$O$};
\fill (0,0) circle (0.8pt);
\fill[gray, opacity=0.3] (0.1,0.1) circle (0.07);
\draw (0.1,0.1) circle (0.07);
 \node at (0.35,0.1) {\tiny$P$};
\fill (0.1,0.1) circle (0.8pt);
\node at (0,-1.3) {$d>\max\{d_1,d_2\}$};
\node at (0,-1.8) {No type II critical points};
\end{tikzpicture}
\end{minipage}
\end{center}
\vskip0.2cm
\begin{rem}\label{sec1-rem07}
When $\La_1=\La_2$, then $d_1=d_2$ and assertion $c)$ does not appear and in case $b)$ we have four type II critical points but only two nontrivially different, see Remark \ref{sec5-rem5.4} below.
\end{rem}

Next result concerns more general convex domains, where the role of the centre of the ball $B(0,r)$ is replaced by a critical point of Robin function $\mathcal{R}_{\Omega}(x)$.
\begin{teo}\label{sec1-teo08}
Assume that $\O\subset\R^2$ is a smooth bounded convex domain with $P\in \O$ and $\O_\e=\O\setminus B(P,\e)$.
\begin{itemize}
\item
 Denoting by $d=dist\{P,\partial \O\}$, if $d$ is small enough then
$\mathcal{KR}_{\Omega_\e}(x,y)$ has exactly {\bf{four}} type II critical points
that satisfy
\begin{align*}
(x_{1,\e},y_{1,\e})\to\big(P,y_1(P)\big)~~~\mbox{and}~~~
(x_{2,\e},y_{2,\e})\to\big(P,y_2(P)\big),\\
(x_{3,\e},y_{3,\e})\to\big(x_1(P),P\big)~~~\mbox{and}~~~
(x_{4,\e},y_{4,\e})\to\big(x_2(P),P\big),
\end{align*}
where $y_i(P)$ are solutions to \eqref{sec1-08} and $x_i(P)$ are solutions to \eqref{sec1-09} for $i=1,2$.
We have also, that
\begin{align*}
\lim_{d\to 0}|y_1(P)-P|=0,~~~\lim_{d\to 0}y_2(P)=Q~~~
\mbox{and}~~~\lim_{d\to 0}|x_1(P)-P|=0,~~~\lim_{d\to 0}x_2(P)=Q,
\end{align*}
where $Q$ is the unique critical  point of Robin function $\mathcal R_\O(x)$ in  $\O$.
Furthermore,
\begin{equation*}
\begin{split}
index \big(\nabla_y \mathcal{KR}_{\O_\e}(x_{1,\e},\cdot),y_{1,\e}\big)=1\hbox{ and }index\big(\nabla_y  \mathcal{KR}_{\O_\e}(x_{2,\e},\cdot),y_{2,\e}\big)=-1,\\
index \big(\nabla_x \mathcal{KR}_{\O_\e}(\cdot,y_{3,\e}),x_{3,\e}\big)=1\hbox{ and }index\big(\nabla_x \mathcal{KR}_{\O_\e}(\cdot,y_{4,\e}),x_{4,\e}\big)=-1.
\end{split}
\end{equation*}
Moreover, they are all nondegenerate.

\vskip 0.1cm

\item If  $|P-Q|$ is small,
 $\mathcal{KR}_{\Omega_\e}(x,y)$ has {\bf{no}} type II critical points.
\end{itemize}

\begin{center}
 \begin{minipage}{0.49\textwidth}
\begin{tikzpicture}
   \draw[decorate, decoration={curveto}] plot[smooth cycle] coordinates {(-0.6,0) (2,-0.3) (1.7,0.5) (0,0.9)};
 \fill[gray, opacity=0.3](-0.4,0.3) circle (0.1cm);
 \draw (-0.4,0.3) circle (0.1cm);
     \fill (-0.4,0.3) circle (0.8pt);
     \node at (-0.1,0.2) {\text{$P$}};
\node at (1.4,0) {\large\text{$\Omega$}};
\node at (1,-1) {$d$ small: four type II critical points};
  \end{tikzpicture}
  \end{minipage}
 \begin{minipage}{0.49\textwidth}
\begin{tikzpicture}
   \draw[decorate, decoration={curveto}] plot[smooth cycle] coordinates {(-0.6,0) (2,-0.3) (1.7,0.5) (0,0.9)};
     \draw (0.8,0.3) circle (0.1cm);
\node at (1.4,0) {\large\text{$\Omega$}};
 \fill (0.8,0.3) circle (0.8pt);
 \fill[gray, opacity=0.3] (0.8,0.3) circle (0.1cm);
     \node at (1.1,0.4) {\text{$P$}};
\fill (0.6,0.3) circle (0.8pt);
\node at (0.4,0.4) {\text{$Q$}};
\node at (1,-1) {$|P-Q|$ small:  no type II critical points};
  \end{tikzpicture}
\end{minipage}
\end{center}
\end{teo}
\vskip0.1cm
\noindent{\bf Case {\rm (ii)}: $\nabla \mathcal{KR}_{\O}(P,y_0)=0$.}
\vskip0.1cm
Due to Theorem B, this case cannot appear when $\O$ is convex. Note the similarity of the next result with Theorem 1.8 in \cite{ggly1}.  
\begin{teo}\label{SEC1-TEO09}
Suppose that $(P,y_0)$ is a nondegenerate critical  point of $\mathcal{KR}_{\O}(x,y)$. Assume that the matrix $
 \left(\frac{\partial^2 \mathcal{KR}_\Omega(P,y_0)}{\partial y_i\partial y_j}   \right)_{1\leq i,j\leq 2}$
is invertible and set
\begin{small}\begin{equation*}
\textbf{M}_0 =
   \left(\frac{\partial^2 \mathcal{KR}_\Omega(P,y_0)}{\partial x_i\partial x_j}   \right)_{1\leq i,j\leq 2}-\left(\frac{\partial^2 \mathcal{KR}_\Omega(P,y_0)}{\partial x_i\partial y_j}   \right)_{1\leq i,j\leq 2}
   \left( \left(\frac{\partial^2 \mathcal{KR}_\Omega(P,y_0)}{\partial y_i\partial y_j}   \right)_{1\leq i,j\leq 2}\right)^{-1}
   \left(\frac{\partial^2 \mathcal{KR}_\Omega(P,y_0)}{\partial y_i\partial x_j}   \right)_{1\leq i,j\leq 2}.
\end{equation*}
\end{small}Then any simple, positive eigenvalue ${\lambda_i}$ of  $\textbf{M}_0$  generates exactly {\bf two} type $II$
critical points $(x^{(i),\pm}_{\e}, y^{(i),\pm}_{\e})$ of $\mathcal{KR}_{\O_\e}(x,y)$ which
are nondegenerate, and satisfy, as $\e\to 0$, the following as\hbox{ym}ptotic expansion,
\begin{equation}\label{sec1-13}
\frac{ x^{(i),\pm}_{\e} -P}{ |x^{(i),\pm}_{\e}-P| }\to \pm \eta^{(i)} ~~~~\mbox{and}~~~~
|x^{(i),\pm}_{\e}-P|=r_\e^i,
\end{equation}
where $r_\e^i$ is the unique solution to $\frac{\ln r}{r^2 \ln \e}= \frac{\lambda_i\pi}{ \La_1^2}$,  $\eta^{(i)}$ is a unit eigenvector of $\textbf{M}_{0}$ related to $\lambda_i$, and
\begin{equation}\label{sec1-14}
\begin{split}
y^{(i),\pm}_{\e}-y_0 = &-\left(\Big(\frac{\partial^2 \mathcal{KR}_\Omega(P,y_0)}{\partial y_l\partial y_j}\Big)_{1\leq l,j\leq 2}   \right)^{-1}\!\!\!
   \left(\frac{\partial^2 \mathcal{KR}_\Omega(P,y_0)}{\partial y_l\partial x_j}   \right)_{1\leq l,j\leq 2}\Big(
x^{(i),\pm}_{\e}-P\Big) \big(1+o(1)\big).
 \end{split}
\end{equation}
Moreover it holds
\begin{equation}\label{sec1-15}
index \big(\nabla \mathcal{KR}_{\O_\e}, (x_\e^{(i),\pm},  y_\e^{(i),\pm})\big)
=sign \left[\det\left(\frac{\partial^2 \mathcal{KR}_\Omega(P,y_0)}{\partial y_k\partial y_j} \right)_{1\leq k,j\leq 2} \left(\lambda_l-\lambda_i\right)\right],
\end{equation}
where $l\in\{1,2\}$ with $l\neq i$, and $\lambda_j$ for $j=1,2$ are all the eigenvalues of $\textbf{M}_{0}$.

\vskip 0.1cm
 Furthermore, if  all the eigenvalues of $\textbf{M}_{0}$ are simple and positive (that is $0<\lambda_1<\lambda_2$), we have exactly {\bf four} type $II$
critical points $(x^{(i),\pm}_{\e}, y^{(i),\pm}_{\e})$ for $i=1,2$, which are nondegenerate, and  satisfy
\eqref{sec1-13}, \eqref{sec1-14} and \eqref{sec1-15}.

\end{teo}
\begin{rem}\label{SEC1-REM10}
  Let $\Omega$ be a disk with a small punctured hole near the boundary.  Then by Theorem \ref{sec1-teo08}, $\mathcal{KR}_{\O}(x,y)$ has exactly four type II critical points, which are
  all nondegenerate.  Let $(x_0,y_0)$ be  a type II critical point, with
   $x_0$ close to $0$. Removing a small hole centered at $x_0$, in
   Appendix \ref{app-B}, we will check all the conditions in Theorem~\ref{SEC1-TEO09} hold
   (see (1) of Proposition \ref{lem-B-1}  in Appendix \ref{app-B}).

\end{rem}

\begin{rem}
Unlike the critical points of type I, which are perturbation of the critical points of $\mathcal{KR}_{\O}(x,y)$, Theorem \ref{SEC1-TEO06} and Theorem \ref{sec1-teo08} show that critical points of type II appear only for suitable locations of the hole.
This is a rather surprising phenomenon.
We also stress the (quite unexpected) role of Robin function in Theorem \ref{sec1-teo08}.
\end{rem}
\vskip 0.1cm

\subsection{Critical points of type III}\

\vskip 0.1cm
Let us now turn our discussion to the critical points of type III. Since
$\O_\e$ has a hole, it is known that $\O_\e$ admits at least one critical point for any $\e>0$ (see \cite{CPY2010}).
On the other hand, the previous discussion shows that if $\Omega_\e=\Omega\setminus B(P,\e)$, where $\Omega$ is convex and
$P$ is close to the harmonic center of $\Omega$, then $\mathcal{KR}_{\Omega_\e}(x,y)$ has neither type I nor type II critical points for small $\e>0$.
Thus, in this case, $\mathcal{KR}_{\Omega_\e}(x,y)$ can only possess type III critical points.
This strongly suggests that $\mathcal{KR}_{\Omega_\e}(x,y)$ should always have type III critical points, as stated in the next result.
\begin{teo}\label{sec1-teo12}
Let $\O\subset \R^2$ be a bounded smooth domain such that $P\in \Omega$ and $\Omega_\e=\Omega\setminus B(P,\e)$. We have the following results.
\vskip 0.1cm
\begin{itemize}
\item[(1)]\textup{[Necessary conditions]} Let $(x_\e,y_\e)$ be a type III critical point of $\mathcal{KR}_{\Omega_\e}(x,y)$. Then
\begin{equation}\label{sec1-16}
|x_\e-P|=C_\tau\big(1+o(1)\big)\e^{\beta}, \ |y_\e-P|= \frac{C_\tau}\tau\big(1+o(1)\big) \e^{\beta},
\end{equation}
where $\beta:=\frac{\tau}{(\tau+1)^2}$, $\tau:=\frac{\La_1}{\La_2}$, $C_\tau:=\tau^{\frac{1}{\tau+1}}
e^{-\frac{2\pi \mathcal{R}_{\Omega}(0)(\tau^2+ \tau+1 )}{(\tau+1)^2}}$.
\item[(2)]\textup{[Existence]}
$\mathcal{KR}_{\Omega_\e}(x,y)$, as well as its $C^1$ perturbation,  admits at least two  critical points, and one of them is
a local minimum point. Moreover, one of the following alternatives holds:
\begin{itemize}
\item the critical points are  isolated, in this case there exists at least one additional critical point with negative index;
\item the critical points are not isolated, and therefore there exist infinitely many critical points.
\end{itemize}
\end{itemize}
\end{teo}
\begin{rem}\label{sec1-rem13}
The previous result is not completely satisfactory since it does not provide the full
asymptotic behavior of $x_\e$ and $y_\e$, but only their distance from $P$. Moreover no information about the nondegeneracy or the exact number of critical points is provided. However, if $P=0$ and $\Omega_\e=B(0,1)\setminus B(0,\e)$, there actually exist infinitely many type III critical points and this shows that without certain restriction on $\Omega_\e$, it is impossible to
determine $\frac{ x_\e-P }{|x_\e-P|}$  and $\frac{ y_\e-P }{|y_\e-P|}$.
However, this is an exceptional situation caused by the symmetry of $\Omega_\e$.
In the following, we shall obtain much more precise results under additional assumptions.
\end{rem}
\vskip 0.1cm

\begin{rem}
When studying the existence of type III critical points,
the leading term in the expansion of $\nabla\mathcal{KR}_{\O_\e}(x,y)$, after rescaling, is given by $\nabla\mathcal{KR}_{(B(0,1))^c}(x,y)$.
However, $\mathcal{KR}_{(B(0,1))^c}(x,y)$ has no critical points (see Section \ref{sec-6-1}). This makes the problem complicated because further expansion for $\nabla\mathcal{KR}_{\O_\e}(x,y)$ is needed in order to solve $\nabla\mathcal{KR}_{\O_\e}(x,y)=0$.
\end{rem}

\vskip 0.05cm
Theorem~\ref{sec1-teo12} shows that $\mathcal{KR}_{\O_\e}(x,y)$ always possesses type III critical points. Next, we address the exact multiplicity and nondegeneracy of such critical points. As suggested in Remark \ref{sec1-rem13}, in order to determine the precise number of type III critical points,
one option is to break the
symmetry of $\Omega_\e$.
As in the case of the Robin function studied in \cite{ggly1}, the appropriate way to ensure
nondegeneracy is to choose the position of the hole so that $\nabla \mathcal{R}_{\Omega}(P)\ne 0$. This appears to be the correct condition for every domain $\Omega$.
\begin{teo}\label{sec1-teo15}
Let $\O\subset \R^2$ be a bounded domain with $P\in \Omega$ and  $\O_\e=\O\setminus B(P,\e)$.  If
\begin{equation*}
\La_1\neq \La_2~~\mbox{and}~~\nabla\mathcal{R}_{\Omega}(P)\neq 0,
\end{equation*}
then
$\mathcal{KR}_{\Omega_\e}(x,y)$ has exactly {\bf{two}}  type III critical points $(x_\e^{(m)},y_\e^{(m)})$, $m=1,2$, which are nondegenerate, and satisfy
\begin{small}\begin{equation*}
\begin{cases}
\big(x_\e^{(1)},y_\e^{(1)}\big)=\left(P+C_\tau\e^{\beta}
\frac{\nabla \mathcal{R}_\Omega(P)}{|\nabla \mathcal{R}_\Omega(P)|}
+O\big(\e^{2\beta}
\big), P-\frac{C_\tau\e^{\beta}}{\tau}
\frac{\nabla \mathcal{R}_\Omega(P)}{|\nabla \mathcal{R}_\Omega(P)|}
+O\big(\e^{2\beta}
\big)\right),\\[3mm]
\big(x_\e^{(2)},y_\e^{(2)}\big)=\left(P-C_\tau\e^{\beta}
\frac{\nabla \mathcal{R}_\Omega(P)}{|\nabla \mathcal{R}_\Omega(P)|}
+O\big(\e^{2\beta}
\big), P+\frac{C_\tau\e^{\beta}}{\tau}
\frac{\nabla \mathcal{R}_\Omega(P)}{|\nabla \mathcal{R}_\Omega(P)|}
+O\big(\e^{2\beta}
\big)\right),\end{cases}\end{equation*}
\end{small}where $\beta:=\frac{\tau}{(\tau+1)^2}$, $\tau:=\frac{\La_1}{\La_2}$, $C_\tau:=  \tau^{\frac{1}{1+\tau}}
e^{-\frac{2\pi \mathcal{R}_{\Omega}(P)(\tau^2+ \tau+1 )}{(1+\tau)^2}}$.
\end{teo}
Next, we consider the case when $\Lambda_1=\Lambda_2$.
As mentioned earlier,  if $(x,y)$ is a critical point for $\mathcal{KR}_{\Omega_\e}$, then $(y,x)$ is also a critical point.  Our interest lies in critical points that are nontrivially distinct, see Definition \ref{def1}.
\vskip 0.1cm
\begin{teo}\label{sec1-teo16}
Let $\O\subset \R^2$ be a bounded domain with $P\in \Omega$ and  $\O_\e=\O\setminus B(P,\e)$. Suppose that
$$\La_1=\La_2\hbox{ and }\nabla\mathcal{R}_{\Omega}(P)\neq 0.$$
If the matrix ${\widetilde{\bf{M}}}:=\left( \frac{\partial^2H_{\Omega}(P,P)}{\partial x_i\partial x_j}-3 \pi \frac{\partial\mathcal{R}_\Omega(P)}{\partial x_i}\frac{\partial\mathcal{R}_\Omega(P)}{\partial x_j}  \right)_{1\leq i,j\leq 2}$ has two different eigenvalues $\lambda_m$,  $m=1,2$, whose unit eigenvectors
are $\nu^{(m)}$ respectively,
then $\mathcal{KR}_{\Omega_\e}(x,y)$ has exactly {\bf two}
nontrivially different type III critical points $\big(x_{\e}^{(m) },y_{\e}^{(m) }\big)$,  $m=1,2$, which
 are nondegenerate and satisfy
\begin{equation*}
\begin{cases}
 |x^{(m)}_\e-P|=
e^{-\frac{3\pi \mathcal{R}_{\Omega}(P)}{2}} \e^{\frac{1}{4}} +O\big(\e^{\frac{1}{2}}\big), \\[2mm]
 \frac{x^{(m)}_\e-P}{|x^{(m)}_\e-P|}= \nu^{(m)}+o\big(1\big),
\end{cases} ~~~~~~~~~~~~~~~~~
\begin{cases}
 |y^{(m)}_\e-P|=
e^{-\frac{3\pi \mathcal{R}_{\Omega}(P)}{2}}  \e^{\frac{1}{4}} +O\big(\e^{\frac{1}{2}}\big), \\[2mm]
 \frac{y^{(m)}_\e-P}{|y^{(m)}_\e-P|}= - \nu^{(m)}+o\big(1\big).
\end{cases}
\end{equation*}

\end{teo}

\begin{rem}
Let us point out that if $\Omega= B(Q,1)$ and $P\in\O$ with $P\ne Q$, then $\nabla \mathcal{R}_{\Omega}(P)\ne 0$ and
the matrix ${\widetilde{\bf{M}}}$ defined in Theorem~\ref{sec1-teo16} has two different eigenvalues. On the other hand, for any bounded domain $\Omega$, we can also show that if $d:=dist\{P,\partial \Omega\}$ is small enough, then $\nabla \mathcal{R}_{\Omega}(P)\neq 0$
and ${\widetilde{\bf{M}}}$ has two different eigenvalues.
See Remark~\ref{sec7-rem7.20} and Proposition \ref{lem-B-2}.

\end{rem}
\vskip 0.05cm

Now we study the case $\nabla\mathcal{R}_{\O}(P)=0$. Here the shape of $\O$ plays a crucial role.
\begin{teo}\label{sec1-teo17}
Let $\O\subset \R^2$ be a bounded domain with $P\in \Omega$ and  $\O_\e=\O\setminus B(P,\e)$. Suppose that
$$\nabla\mathcal{R}_{\Omega}(P)= 0.$$
If the matrix
\begin{small}\begin{equation}\label{sec1-17}
\overline{\bf{M}}:=\left[ (\tau^4+\tau^2+ 1)
 \frac{\partial^2H_{\Omega}(P,P)}{\partial x_i\partial x_j}
+(\tau^2-1)^2 \frac{\partial^2H_{\Omega}(P,P)}{\partial y_i\partial x_j} \right]_{1\leq i,j\leq 2},\end{equation}
\end{small}with $\tau=\frac{\La_1}{\La_2}$
has two
different eigenvalues $\mu_m$, $m=1,2$,  whose unit eigenvectors are
$\nu^{(m)}$ respectively,
then $\mathcal{KR}_{\Omega_\e}(x,y)$ has exactly {\bf four} type III critical points $\big(x_{\e}^{(m),\pm },y_{\e}^{(m),\pm }\big)$, $m=1,2$, which
 are nondegenerate and satisfy
\begin{equation*}
\begin{cases}
 |x^{(m),\pm}_\e-P|=C_\tau \e^{\beta} +O\big(\e^{2\beta}\big), \\[2mm]
 \frac{x^{(m),\pm}_\e-P}{|x^{(m),\pm}_\e-P|}=\pm \nu^{(m)}+o\big(1\big),
\end{cases} ~~~~~~~~~~~~~~~~~
\begin{cases}
 |y^{(m),\pm}_\e-P|=\frac{C_\tau \e^{\beta}}{\tau} +O\big(\e^{2\beta}\big), \\[2mm]
 \frac{y^{(m),\pm}_\e-P}{|y^{(m),\pm}_\e-P|}= \mp \nu^{(m)}+o\big(1\big),
\end{cases}
\end{equation*}
where $C_\tau:=  \tau^{\frac{1}{1+\tau}}
e^{-\frac{2\pi \mathcal{R}_{\Omega}(0)(\tau^2+ \tau+1 )}{(1+\tau)^2}}$ and $\beta=\frac{\tau}{(1+\tau)^2}$.
Moreover, if $\Lambda_1=\Lambda_2$, {\bf only two} of them are nontrivially different.
\end{teo}
\begin{rem}\label{Sec1-rem18}
Let us point out that if $P=0$ and
\begin{equation*}
\begin{split}\Omega_\delta=\Big\{(x_1,x_2)\in \mathbb{R}^2, x_1^2\big(1+\alpha_1\delta\big)^2+x_2^2\big(1+\alpha_2\delta\big)^2 <1,~~\delta>0,~~\alpha_1,\alpha_2\geq 0\Big\},
\end{split}
\end{equation*}with $\alpha_1\neq \alpha_2$
and $\delta>0$ small, then the matrix $\overline{\bf{M}}$   defined in Theorem~\ref{sec1-teo17} has two different eigenvalues, see Proposition \ref{app-teo-B.2} in Appendix \ref{app-B}. On the other hand, it is immediate to verify that if $\Omega=B(0,1)$ then the corresponding matrix $\overline{\bf{M}}$ has two equal eigenvalues.
\end{rem}

Let us point out that the conditions on Theorem \ref{sec1-teo15}, Theorem \ref{sec1-teo16}
and Theorem \ref{sec1-teo17} make it possible to determine both
$\frac{ x_\e-P }{|x_\e-P|}$  and $\frac{ y_\e-P }{|y_\e-P|}$ and so the full asymptotic of $x_\e$ and $y_\e$.
On the other hand, if $\Omega_\e=B(0,1)\setminus B(0,\e)$,
we can use the   symmetry to  fix $\frac{ x_\e }{|x_\e|}=(1,0)$,
or $\frac{ y_\e }{|y_\e|}=(1,0)$
by suitable rotations. We then have following result.

\begin{teo}\label{sec1-teo19}Let $\O_\e=B(0,1)\backslash B(0,\e)$. Then
up to a rotation,
the number of type III critical points for
$\mathcal{KR}_{\Omega_\e}(x,y)$ is exactly
 {\bf two} if $\La_1\ne \La_2$, while it is {\bf one}
 if $\La_1= \La_2$.
 Furthermore, they
are nondegenerate in the radial direction.
\end{teo}

\subsection{Summary and examples}\label{sec1.4}\

\vskip 0.1cm

In this subsection, we summarize the previous results considering some classes of domains.
\vskip 0.2cm
\noindent{\bf (a) $\O=B(0,1)$.}
\vskip 0.05cm

In this case we can give a complete description of the number of critical points.  By Theorem B we do not have critical points of type I.
\vskip 0.05cm

Let us  denote by $d=dist\big\{P,\partial B(0,1)\big\}$. Collecting the previous results we get following results.

\vskip 0.1cm

\noindent{\bf (a-1) $\La_1\neq \La_2$.}
\vskip 0.05cm
\begin{center}
 \begin{minipage}{0.3\textwidth}
        \centering
\begin{tikzpicture}
   \draw (0,0) circle (1);
       \node at (0,-0.2){\tiny$O$};
       \fill (0,0) circle (0.8pt);
\fill[gray, opacity=0.3]  (0.6,0.6) circle (0.1);
\draw  (0.6,0.6) circle (0.1);
 \node at (0.7,0.4) {\tiny$P$};
\fill (0.6,0.6) circle (0.8pt);
\node at (0,-1.3) {$d<\min\{d_1,d_2\}$};
\node at (0,-1.8) {$6$ nondegenerate critical points};
\end{tikzpicture}
\end{minipage}
 \hfill
 \begin{minipage}{0.3\textwidth}
        \centering
\begin{tikzpicture}
   \draw (0,0) circle (1);
       \node at (0,-0.2) {\tiny$O$};
\fill (0,0) circle (0.8pt);
\fill[gray, opacity=0.3]  (0.4,0.4) circle (0.1);
\draw (0.4,0.4) circle (0.1);
 \node at (0.5,0.2) {\tiny$P$};
\fill (0.4,0.4) circle (0.8pt);
\node at (0,-1.3) {$\min\{d_1,d_2\}<d<\max\{d_1,d_2\}$};
\node at (0,-1.8) {$4$ nondegenerate critical points};
\end{tikzpicture}
\end{minipage}
\hfill
 \begin{minipage}{0.3\textwidth}
\begin{tikzpicture}
   \draw (0,0) circle (1);
    \node at (0,-0.2) {\tiny$O$};
\fill (0,0) circle (0.8pt);
\fill[gray, opacity=0.3] (0.1,0.1) circle (0.07);
\draw (0.1,0.1) circle (0.07);
 \node at (0.35,0.1) {\tiny$P$};
\fill (0.1,0.1) circle (0.8pt);
\node at (0,-1.3) {$\max\{d_1,d_2\}<d<1$};
\node at (0,-1.8) {$2$ nondegenerate critical points};
\end{tikzpicture}
\end{minipage}
\end{center}
\vskip 0.1cm

\noindent{\bf (a-2) $\La_1=\La_2$.}
\vskip 0.05cm

In this case, $d_1=d_2$.
\begin{center}
 \begin{minipage}{0.3\textwidth}
        \centering
\begin{tikzpicture}
   \draw (0,0) circle (1);
       \node at (0,-0.2){\tiny$O$};
       \fill (0,0) circle (0.8pt);
\fill[gray, opacity=0.3]  (0.6,0.6) circle (0.1);
\draw  (0.6,0.6) circle (0.1);
 \node at (0.7,0.4) {\tiny$P$};
\fill (0.6,0.6) circle (0.8pt);
\node at (0,-1.3) {$d<d_1$};
\node at (1,-1.8) {$8$ nondegenerate critical points};
\node at (1,-2.3) {$4$ of them are nontrivially different};
\end{tikzpicture}
\end{minipage}
 \hfill
 \begin{minipage}{0.3\textwidth}
\begin{tikzpicture}
   \draw (0,0) circle (1);
    \node at (0,-0.2) {\tiny$O$};
\fill (0,0) circle (0.8pt);
\fill[gray, opacity=0.3] (0.1,0.1) circle (0.07);
\draw (0.1,0.1) circle (0.07);
 \node at (0.35,0.1) {\tiny$P$};
\fill (0.1,0.1) circle (0.8pt);
\node at (0,-1.3) {$d_1<d<1$};
\node at (1,-1.8) {$4$ nondegenerate critical points};
\node at (1,-2.3) {$2$ of them are nontrivially different};
\end{tikzpicture}
\end{minipage}
\end{center}
For $d=1$, that is $P=0$ and then Theorem \ref{sec1-teo19} holds.
This ends the discussion if $\O$ is a disk.
\vskip 0.2cm
\noindent{\bf (b) $\O$ is a convex domain.}
\vskip 0.05cm

Again by Theorem B here we do not have critical points of type I and the Robin function $\mathcal{R}_\O$ has a unique critical point that we denote by $Q$.
Then for $P\neq Q$, we have following results.

\vskip 0.1cm

\noindent{\bf (b-1) $\La_1\neq \La_2$.}
\vskip 0.05cm
\begin{center}
 \begin{minipage}{0.49\textwidth}
\begin{tikzpicture}
   \draw[decorate, decoration={curveto}] plot[smooth cycle] coordinates {(-0.6,0) (2,-0.3) (1.7,0.5) (0,0.9)};
 \fill[gray, opacity=0.3](-0.4,0.3) circle (0.1cm);
 \draw (-0.4,0.3) circle (0.1cm);
     \fill (-0.4,0.3) circle (0.8pt);
     \node at (-0.1,0.2) {\text{$P$}};
\node at (1.4,0) {\large\text{$\Omega$}};
\node at (1,-0.8) {$dist\{P,\partial\O\}$ small};
\node at (1,-1.3) {$6$ nondegenerate critical points};
  \end{tikzpicture}
  \end{minipage}
 \begin{minipage}{0.49\textwidth}
\begin{tikzpicture}
   \draw[decorate, decoration={curveto}] plot[smooth cycle] coordinates {(-0.6,0) (2,-0.3) (1.7,0.5) (0,0.9)};
     \draw (0.8,0.3) circle (0.1cm);
\node at (1.4,0) {\large\text{$\Omega$}};
 \fill (0.8,0.3) circle (0.8pt);
 \fill[gray, opacity=0.3] (0.8,0.3) circle (0.1cm);
     \node at (1.1,0.4) {\text{$P$}};
\fill (0.6,0.3) circle (0.8pt);
\node at (0.4,0.4) {\text{$Q$}};
\node at (1,-0.8) {$|P-Q|$ small};
\node at (1,-1.3) {$2$ nondegenerate critical points};
  \end{tikzpicture}
\end{minipage}
\end{center}

\vskip 0.1cm

\noindent{\bf (b-2) $\La_1=\La_2$.}
\vskip 0.05cm

If $\La_1=\La_2$ and $dist\{P,\partial\O\}$ is small, then
the matrix \begin{small}$${\widetilde{\bf{M}}}:=\left( \frac{\partial^2H_{\Omega}(P,P)}{\partial x_i\partial x_j}-3 \pi \frac{\partial\mathcal{R}_\Omega(P)}{\partial x_i}\frac{\partial\mathcal{R}_\Omega(P)}{\partial x_j}  \right)_{1\leq i,j\leq 2}$$\end{small}has two different eigenvalues (see Proposition \ref{lem-B-2} in Appendix \ref{app-B}). Hence in this case, we have $4$ nontrivially different critical points instead of $6$. The same conclusion as above still holds when
$|P-Q|$ is small and ${\widetilde{\bf{M}}}$ has two different eigenvalues. Here we point out that if $\Omega$ is a disk or an ellipse which is close to a disk, then ${\widetilde{\bf{M}}}$ has two different eigenvalues, see Remark \ref{sec7-rem7.20}.

\begin{center}
 \begin{minipage}{0.49\textwidth}
\begin{tikzpicture}
   \draw[decorate, decoration={curveto}] plot[smooth cycle] coordinates {(-0.6,0) (2,-0.3) (1.7,0.5) (0,0.9)};
 \fill[gray, opacity=0.3](-0.4,0.3) circle (0.1cm);
 \draw (-0.4,0.3) circle (0.1cm);
     \fill (-0.4,0.3) circle (0.8pt);
     \node at (-0.1,0.2) {\text{$P$}};
\node at (1.4,0) {\large\text{$\Omega$}};
\node at (1,-0.8) {$dist\{P,\partial\O\}$ small};
\node at (1,-1.3) {$8$ nondegenerate critical points};
\node at (1,-1.8) {$4$ of them are nontrivially different};
  \end{tikzpicture}
  \end{minipage}
 \begin{minipage}{0.49\textwidth}
\begin{tikzpicture}
   \draw[decorate, decoration={curveto}] plot[smooth cycle] coordinates {(-0.6,0) (2,-0.3) (1.7,0.5) (0,0.9)};
     \draw (0.8,0.3) circle (0.1cm);
\node at (1.4,0) {\large\text{$\Omega$}};
 \fill (0.8,0.3) circle (0.8pt);
 \fill[gray, opacity=0.3] (0.8,0.3) circle (0.1cm);
     \node at (1.1,0.4) {\text{$P$}};
\fill (0.6,0.3) circle (0.8pt);
\node at (0.4,0.4) {\text{$Q$}};
\node at (1,-0.8) {$|P-Q|$ small and ${\widetilde{\bf{M}}}$ has two different eigenvalues};
\node at (1,-1.3) {$4$ nondegenerate critical points};
\node at (1,-1.8) {$2$ of them are nontrivially different};
  \end{tikzpicture}
\end{minipage}
\end{center}
\vskip 0.2cm
\noindent{\bf (c) A disk with a punctured hole.}

\vskip 0.05cm
Let $\Omega=B(0,1)\backslash B(y_0,\delta)$ with $y_0\in B(0,1)$, $\delta>0$ is small, $|y_0|$ is close to $1$.  Then  $\mathcal{KR}_{\O}(x,y)$ has both type II and type III critical points, which are
  all  nondegenerate.  Let $(x_\delta,y_\delta)$ be  a type II critical point of $\mathcal{KR}_{\O}(x,y) $, with
   $(x_\delta, y_\delta)\to (x_0,y_0)$ as $\delta\to 0$. Then
$\frac{\partial \mathcal{KR}_{\O}(x_0,y_0)}{\partial x_j}=0$ for $j=1,2$ and $x_0\neq 0$. Theorem
\ref{SEC1-TEO06} gives
  \begin{align*}
\lim_{dist\{y_0,\partial B(0,1)\}\to 0}|x_0|=0~~~\mbox{or}~~~
\lim_{dist\{y_0,\partial B(0,1)\}\to 0}|x_0-y_0|=0.\end{align*}Here we choose $x_0$, which closes to $0$, and we remove a small hole centered at $x_\delta$, see Figure 3.
  \begin{center}
 \begin{minipage}{0.3\textwidth}
        \centering
\begin{tikzpicture}
   \draw (0,0) circle (2);
       \node at (0,-0.2){\tiny$O$};
       \fill (0,0) circle (0.8pt);
\fill[gray, opacity=0.3]  (1.2,1.2) circle (0.2);
\draw  (1.2,1.2) circle (0.2);
 \node at (0.95,1.45) {\tiny$y_0$};
\fill (1.2,1.2) circle (0.4pt);
\fill[gray, opacity=0.3]  (0.1,0.1) circle (0.05);
\draw  (0.1,0.1) circle (0.05);
 \node at (0.1,0.25) {\tiny$x_0$};
\fill (0.1,0.1) circle (0.4pt);
\node at (0,-2.3) {$\Omega=B(0,1)\backslash B(y_0,\delta)$};
\node at (0,-2.9) {$\Omega_\e=\big(B(0,1)\backslash B(y_0,\delta)\big)\backslash  B(x_\delta,\e)$};
\node at (0,-3.6) {Figure 3.};
\end{tikzpicture}
\end{minipage}
\end{center}
This last case is interesting because critical points of type I arise. Moreover, choosing $\delta$ small in Figure 3, we have that the matrices $\bf{M_0}$ and  ${\widetilde{\bf{M}}}$
 in Theorem \ref{SEC1-TEO09} and Theorem
\ref{sec1-teo16}  have simple eigenvalues (see Proposition \ref{lem-B-1}  in Appendix \ref{app-B}). Hence fix $\delta>0$ small such that these properties hold and then choose $\e$ small in order to apply the previous theorems.
Using {\bf (a-1)} and {\bf (a-2)}, we have the following results.

\begin{itemize}
\item[(1)] Case $\Lambda_1\neq \Lambda_2$.\vskip 0.05cm
\begin{itemize}
\item[(1-i)] $\mathcal{KR}_{\Omega_\e}$ has exactly {\bf five}  type I critical points. All of them are nondegenerate and nontrivially different.

\item[(1-ii)] $\mathcal{KR}_{\Omega_\e}$ has exactly {\bf four}  type II critical points. All of them are nondegenerate and nontrivially different.\vskip 0.05cm

    \item[(1-iii)] $\mathcal{KR}_{\Omega_\e}$ has exactly {\bf two}   type III critical points. All of them are nondegenerate and nontrivially different.\vskip 0.05cm
\end{itemize}

\item[(2)] Case $\Lambda_1=\Lambda_2$.\vskip 0.05cm
\begin{itemize}
\item[(2-i)] $\mathcal{KR}_{\Omega_\e}$ has exactly {\bf six}   type I critical points. All of them are nondegenerate and
    {\bf three} of them are nontrivially different.

\item[(2-ii)] $\mathcal{KR}_{\Omega_\e}$ has exactly {\bf eight}   type II critical points. All of them are nondegenerate and  {\bf four} of them are nontrivially different.\vskip 0.05cm

    \item[(2-iii)] $\mathcal{KR}_{\Omega_\e}$ has exactly {\bf four}   type III critical points. All of them are nondegenerate and  {\bf two} of them are nontrivially different.\vskip 0.05cm
\end{itemize}
 \end{itemize}

\subsection{Applications to nonlinear elliptic problems}\

\vskip 0.1cm

The previous results can now be employed to establish the existence of two--peak solutions for the elliptic problems \eqref{sec1.04}, \eqref{sec1.05}, and \eqref{sec1.06}.
These problems involve parameters $\la$, $p$, or $\alpha$ that must be chosen appropriately in order to ensure the existence of solutions. Therefore, due to the presence of the additional parameter $\varepsilon$, the analysis naturally involves a two--parameter dependence. Owing to the delicate nature of this setting, we now outline the strategy we will follow.

\vskip 0.05cm

Given a domain $\Omega \subset \mathbb{R}^2$, fix a point $P \in \Omega$. Then there exists $\varepsilon_0>0$, depending on $\Omega$ and $P$, such that the existence and nondegeneracy results for type~I, type~II, and type~III critical points of $\mathcal{KR}_{\Omega_\varepsilon}$ hold for every $\varepsilon\in (0,\varepsilon_0)$.  For problem \eqref{sec1.06}, $(x,y)$
is a critical point of
$\mathcal{KR}_{\Omega_\varepsilon}$, and $
\delta>0$ is small such that $B(x,\delta)\cap
B(y,\delta)=\emptyset$.

\vskip 0.05cm

Suppose that $(x_\varepsilon, y_\varepsilon)$ is a nondegenerate critical point of $\mathcal{KR}_{\Omega_\varepsilon}$. Then it generates, for each $\varepsilon\in (0,\varepsilon_0)$, families of two–peak solutions
$u_{\varepsilon,\lambda}$, $u_{\varepsilon,p}$ and $u_{\varepsilon,\alpha}$ to problems \eqref{sec1.04} (for $\lambda>0$ small), \eqref{sec1.05} (for
$p>0$ large), and \eqref{sec1.06} (for $\alpha>0$ large), respectively, which concentrate at $x_\varepsilon$ and $y_\varepsilon$ as $\lambda \to 0$, $p \to \infty$, and $\alpha \to +\infty$.

\vskip 0.05cm

Recalling the classification of the critical points of $\mathcal{KR}_{\Omega_\varepsilon}$ in Definition \ref{def2}, we may state that

\vskip 0.1cm
\begin{itemize}
\item[(1)] $u_{\varepsilon,\lambda_\varepsilon}$ (or $u_{\varepsilon,p_\varepsilon}$ and $u_{\varepsilon,\alpha_\varepsilon}$) is a \emph{type~I} two-peak solution whenever it concentrates at $(x_\varepsilon,y_\varepsilon)$, which is a type~I critical point of $\mathcal{KR}_{\Omega_\varepsilon}$.

\item[(2)] $u_{\varepsilon,\lambda_\varepsilon}$ (or $u_{\varepsilon,p_\varepsilon}$ and $u_{\varepsilon,\alpha_\varepsilon}$) is a \emph{type~II} two-peak solution whenever it concentrates at $(x_\varepsilon,y_\varepsilon)$, which is a type~II critical point of $\mathcal{KR}_{\Omega_\varepsilon}$.

\item[(3)] $u_{\varepsilon,\lambda_\varepsilon}$ (or $u_{\varepsilon,p_\varepsilon}$ and $u_{\varepsilon,\alpha_\varepsilon}$) is a \emph{type~III} two-peak solution whenever it concentrates at $(x_\varepsilon,y_\varepsilon)$, which is a type~III critical point of $\mathcal{KR}_{\Omega_\varepsilon}$.
\end{itemize}We now consider several classes of domains in order to determine the precise multiplicity of two peak solutions.

\vskip 0.2cm

\noindent{\bf (a)} $\Omega = B(0,1)$. Let $\Omega_\e=\Omega\backslash B(P,\e)$ with $P\in \Omega$, we obtain the following results.

\begin{teo}\label{sec1-teo20}
For every $0<\varepsilon<\varepsilon_0$ we have:
\begin{itemize}
\item[(1)] Problems \eqref{sec1.04}, \eqref{sec1.05}, and \eqref{sec1.06} admit {\bf no} {type~I} two–peak solutions in $\Omega_\varepsilon$ as $\lambda\to 0$, as $p\to \infty$, or as $\alpha\to +\infty$, respectively.\vskip 0.05cm

\item[(2)]
\begin{itemize}
\item[(2-i)] Case $\Lambda_1=\Lambda_2$.  There exists a constant $r\in(0,1)$ such that, if $r<|P|<1$, problems \eqref{sec1.04}, \eqref{sec1.05}, and \eqref{sec1.06} have exactly {\bf two}  type~II  two–peak solutions in $\Omega_\varepsilon$ as $\lambda\to 0$, as $p\to \infty$, or as $\alpha\to +\infty$, respectively.
If $|P|<r$, problems \eqref{sec1.04}, \eqref{sec1.05}, and \eqref{sec1.06}  have {\bf no} {type~II} two–peak solutions in $\Omega_\varepsilon$.\vskip 0.05cm

\item[(2-ii)] Case $\Lambda_1\ne\Lambda_2$. There exist constants $r_1,r_2\in(0,1)$ with $r_1<r_2$ such that problem \eqref{sec1.06} has, as $\alpha\to +\infty$, exactly {\bf four} type~II two–peak solutions in $\Omega_\varepsilon$ if $r_2<|P|<1$;
exactly {\bf two} such solutions in $\Omega_\varepsilon$ if $r_1<|P|<r_2$;
and {\bf none} if $|P|<r_1$.\vskip 0.05cm
\end{itemize}

\item[(3)]
\begin{itemize}
\item[(3-i)] Case $P\neq 0$. Problems \eqref{sec1.04}, \eqref{sec1.05}, and \eqref{sec1.06} have exactly {\bf two} {type~III} two–peak solutions in $\Omega_\varepsilon$ as $\lambda\to 0$, as $p\to \infty$, or as $\alpha\to +\infty$, respectively. \vskip 0.05cm
\item[(3-ii)] Case $P=0$. Up
to a rotation, problems \eqref{sec1.04}, \eqref{sec1.05}, and \eqref{sec1.06} has exactly
     {\bf one}  {type~III} two–peak solution in $\Omega_\varepsilon$, as $\lambda\to 0$, as $p\to \infty$, or as $\alpha\to +\infty$, respectively. \vskip 0.05cm
\end{itemize}

\end{itemize}
\end{teo}

\begin{proof}
These conclusions follow from Theorem~\textbf{B} (see Remark~\ref{sec1-rem.04}), Theorem~\ref{SEC1-TEO06}, Remark~\ref{sec1-rem07}, Theorem~\ref{sec1-teo15}, Theorem~\ref{sec1-teo16} and Theorem~\ref{sec1-teo19}, together with Remark~\ref{sec7-rem7.20}  and  the existence and uniqueness results established in \cite{CPY2015,egp,dkm,bjly2,GILY,CGPY2019}.
\end{proof}

\noindent{\bf(b)} $\Omega$ convex.
\vskip 0.05cm

\begin{teo}
 Let $\Omega$ be a convex bounded domain in $\mathbb R^2$ and $\Omega_\e=\Omega\backslash B(P,\e)$ with $P\in \Omega$. Then there is an
 $\varepsilon_0>0$, such that
 for every $\varepsilon\in (0,\varepsilon_0)$, assertion  (1) of Theorem~\ref{sec1-teo20} holds. Moreover:

\begin{itemize}

\item[(2-i)]  If  $|P-Q|$ is small, problems  \eqref{sec1.04}, \eqref{sec1.05}, and \eqref{sec1.06}  have {\bf no} type~II  two–peak solutions in $\Omega_\e$, where $Q$ is the unique minimum point of  Robin function $\mathcal R_\Omega$.\vskip 0.05cm

\item[(2-ii)] Case $\Lambda_1=\Lambda_2$.   If $dist\{P,\partial\Omega\}$ is small, problems \eqref{sec1.04}, \eqref{sec1.05}, and \eqref{sec1.06} have exactly {\bf two}   {type~II} two–peak solutions in $\Omega_\e$ as $\lambda\to 0$, as $p\to \infty$, or as $\alpha\to +\infty$, respectively.

\item[(2-ii)] Case $\Lambda_1\ne\Lambda_2$. If $dist\{P,\partial\Omega\}$ is small, then problem \eqref{sec1.06} has, as $\alpha\to +\infty$, exactly {\bf four} type~II  two–peak solutions in $\Omega_\e$.
\end{itemize} \vskip 0.05cm

\end{teo}

\begin{proof}
These results follow from Theorem~\textbf{B}, Theorem~\ref{sec1-teo08}, together with the existence and uniqueness  results as before.
\end{proof}
\vskip 0.05cm

Let us consider the case of the ellipse: $\Omega=\big\{(x_1,x_2),\sum_{i=1}^2 x_i^2(1+\alpha_i\delta)^2<1\big\}$ with  $\alpha_i\ge 0$ for $i=1,2$
and $\alpha_1\ne \alpha_2$. Denote $\Omega_\e=\Omega\backslash B(P,\e)$ with $P\in \Omega$, we obtain the following results.

\begin{teo}
There is an
 $\varepsilon_0>0$, such that
 for every $\varepsilon\in (0,\varepsilon_0)$,
the following results hold.

\begin{itemize}
\item[(3-i)] Case $\Lambda_1=\Lambda_2$.  If $P\neq 0$ and
$\delta>0$ small, then problems \eqref{sec1.04}, \eqref{sec1.05} and \eqref{sec1.06} have exactly {\bf two} type~III  two–peak solutions in $\Omega_\e$ as $\lambda\to 0$, as $p\to\infty$, or as $\alpha\to +\infty$, respectively.\vskip 0.05cm

\item[(3-ii)] Case $\Lambda_1\neq \Lambda_2$.  If $P\neq 0$  and $\delta>0$ small, then problem \eqref{sec1.06} has, for any $\alpha_i$ and $\delta>0$, exactly {\bf two} type~III two–peak solutions in $\Omega_\e$ as $\alpha\to +\infty$.\vskip 0.05cm

\item[(3-iii)] Case $\Lambda_1=\Lambda_2$.  If $P=0$ and $\delta>0$ small, then problems \eqref{sec1.04}, \eqref{sec1.05}, and \eqref{sec1.06} have exactly {\bf two} type~III two–peak solutions in $\Omega_\e$ as $\lambda\to 0$, as $p\to\infty$, or as $\alpha\to +\infty$, respectively.\vskip 0.05cm

\item[(3-iv)] Case $\Lambda_1\neq \Lambda_2$.  If $P=0$ and $\delta>0$ small, then problem \eqref{sec1.06} has exactly {\bf four}  {type~III} two–peak solutions in $\Omega_\e$ as $\alpha\to +\infty$.

\end{itemize}
\end{teo}

\begin{proof}
These results follow from Theorem~\ref{sec1-teo15}, Theorem~\ref{sec1-teo16}, Theorem~\ref{sec1-teo17}, Remark \ref{sec7-rem7.20} and Proposition~\ref{app-teo-B.2}, together with the existence and uniqueness  results as before.
\end{proof}

\vskip 0.1cm
\noindent{\bf (c)} $\Omega$ is a disk with a punctured hole.

\vskip 0.1cm

Let $\Omega=B(0,1)\backslash B(y_0,\delta)$ and $\Omega_\e=\Omega\backslash B(x_\delta,\e)$ as stated in {\bf (c)} of subsection \ref{sec1.4}, then we obtain the following results.

\begin{teo}\label{sec1-teo22}Suppose that $\La_1\neq \La_2$,
for every $\varepsilon<\varepsilon_0$, we have following results.
\begin{itemize}
\item[(1)]  Problem \eqref{sec1.06} has exactly {\bf five} {type~I} two–peak solutions in $\Omega_\e$ as $\alpha\to +\infty$.\vskip 0.05cm

\item[(2)] Problem \eqref{sec1.06} has exactly {\bf four} {type~II} two–peak solutions in $\Omega_\e$ as $\alpha\to +\infty$.\vskip 0.05cm

\item[(3)]  Problem \eqref{sec1.06} has exactly {\bf two} {type~III} two–peak solutions in $\Omega_\e$ as $\alpha\to +\infty$.
\end{itemize}
\end{teo}

\begin{teo}\label{sec1-teo24}Suppose that $\La_1=\La_2$,
for every $\varepsilon<\varepsilon_0$, we have following results.
\begin{itemize}
\item[(1)]  Problems \eqref{sec1.04}, \eqref{sec1.05} and \eqref{sec1.06} have exactly {\bf three} {type~I} two–peak solution in $\Omega_\e$ as $\lambda\to 0$, as $p\to\infty$, or as $\alpha\to +\infty$, respectively.\vskip 0.05cm

\item[(2)] Problems \eqref{sec1.04}, \eqref{sec1.05} and \eqref{sec1.06} have exactly {\bf four} {type~II} two–peak solutions in $\Omega_\e$ as $\lambda\to 0$, as $p\to\infty$, or as $\alpha\to +\infty$, respectively.\vskip 0.05cm

\item[(3)]  Problems \eqref{sec1.04}, \eqref{sec1.05} and  \eqref{sec1.06} have exactly {\bf two} type~III two–peak solutions in $\Omega_\e$ as $\lambda\to 0$, as $p\to\infty$, or as $\alpha\to +\infty$, respectively.\vskip 0.05cm
\end{itemize}
\end{teo}

\begin{proof}
The results in Theorem \ref{sec1-teo22} and Theorem \ref{sec1-teo24} follow from Theorem~\ref{SEC1-TEO.03}, Theorem~\ref{SEC1-TEO09}, together with Remark~\ref{SEC1-REM10}, Theorem~\ref{sec1-teo17}, and Remark~\ref{Sec1-rem18}, and the existence and uniqueness results as before.
\end{proof}

The paper is organized as follows: in Section \ref{sec-6.1}, we give an outline of the proof of the main results. In Section \ref{sec-3}, we  prove that the critical points of $\mathcal{KR}_{\Omega_\e}(x,y)$ stay far away from the boundary of $\O$ and we give a first expansion of $\nabla \mathcal{KR}_{\Omega_\e}(x,y)$ and $\nabla^2 \mathcal{KR}_{\Omega_\e}(x,y)$ which is useful to handle critical points of $\mathcal{KR}_{\Omega_\e}(x,y)$.
In Section \ref{sec-4} and Section \ref{sec-5}, we study the critical points of type I and II respectively.
We consider the existence of critical points of type III  in  Section \ref{sec-6}. The exact multiplicity and non-degeneracy of type III critical points are given in Section \ref{sec7}. Finally in the Appendix there are the main expansions concerning $\mathcal{KR}_{\Omega_\e}(x,y)$ and its derivatives.
\vskip0.2cm

\section{Outlines of the proofs of the main results}\label{sec-6.1}
In this section, we aim to provide the main ideas on how to find the critical points of $\mathcal{KR}_{\Omega_\e}(x,y)$. Recall the definition of the Kirchhoff-Routh function
\begin{equation}\label{sec2-01}
\begin{split}
\mathcal{KR}_{\Omega_\e}(x,y)=&\La_1^2 \mathcal{R}_{\Omega_\e} (x)+\La_2^2 \mathcal{R}_{\Omega_\e} (y)-2\La_1\La_2 G_{\Omega_\e}(x,y)\\
=&\La_1^2 \mathcal{R}_{\Omega_\e} (x)+\La_2^2 \mathcal{R}_{\Omega_\e} (y)-2\La_1\La_2 S(x,y) +2\La_1\La_2H_{\Omega_\e}(x,y),
\end{split}
\end{equation}
where $S(x,y)$ is as in \eqref{sec1.01} with $N=2$. We have
the following expansion for $\mathcal{KR}_{\Omega_\e}(x,y)$.
\begin{prop}\label{p1-7-8}
 For $x,y\in \Omega_\e$, it holds
\begin{small}\begin{equation}\label{sec2-02}
\begin{split}
\mathcal{KR}_{\Omega_\e}(x,y)=&  \mathcal{KR}_{(B(P,\e))^c}(x,y)+
\mathcal{KR}_{\Omega}(x,y) -\frac{\La_1\La_2}{\pi}\ln|x-y|-
\frac{\big(\La_1\ln\frac{|x-P|}{\e} +\La_2 \ln\frac{|y-P|}{\e}\big)^2}{
2\pi \big(\ln \e+2\pi \mathcal{R}_\O(0)\big)}\\
&
-\frac{2(\La_1\ln |x-P| +\La_2 \ln {|y-P|})}{
\ln \e }  \Big[\La_1 H_{\Omega}(x,P)+\La_2 H_{\Omega}(P,y)-
\big(\La_1+\La_2\big) \mathcal{R}_{\Omega}(P) \Big]\\
&
+ O\left(\frac{1}{|\ln \e|}\right),
\end{split}
 \end{equation}
\end{small}where $(B(P,\e))^c=\R^2\setminus B(P,\e)$.

\end{prop}
Proposition~\ref{p1-7-8} will be proved in Appendix \ref{sec-app1}. 
It is useful to recall the explicit expression of $\mathcal{KR}_{(B(P,\e))^c}(x,y)$,
\begin{small}\begin{equation}\label{sec2-03}
\begin{split}
\mathcal{KR}_{(B(P,\e))^c}(x,y)=&
\frac1{2\pi}\left[\Lambda_1^2\ln\frac\e{|x-P|^2-\e^2}+\Lambda_2^2\ln\frac\e{|y-P|^2-\e^2}
\right]\\&
+\frac{\Lambda_1 \Lambda_2}{\pi}\left[\ln|x-y|-\ln \frac{\sqrt{|x-P|^2|y-P|^2-2(x-P)\cdot (y-P)\e^2+\e^4}}{\e}\right].
 \end{split}\end{equation}
\end{small}

\vskip0.2cm
 \noindent\textbf{Type~I critical points.}
In this case, if $C$ is a compact set $C\subset \O\setminus\{P\}$, then for any $x,y\in C$, \eqref{sec2-03} gives, as $\e\to 0$,
\begin{small}
\[
 \mathcal{KR}_{(B(P,\e))^c}(x,y)=
 \frac{(\La_1+\La_2)^2}{2\pi}\ln \e-\frac{(\La_1+\La_2)}{\pi}\Big( \La_1\ln|x-P|+\La_2 \ln|y-P|\Big)+\frac{\La_1\La_2}{\pi}\ln|x-y|+O\big(\e^{2}\big).
 \]
\end{small}Thus from \eqref{sec2-02}, we obtain, for any $x,y\in C$, as $\e\to 0$,
 \[
\mathcal{KR}_{\Omega_\e}(x,y)=
\mathcal{KR}_{\Omega}(x,y) +O\left(\frac{1}{|\ln\e|}\right).
 \]
This shows that, necessarily, type I critical points of $\mathcal{KR}_{\Omega_\e}$ converge to critical points of $\mathcal{KR}_{\Omega}$, and conversely, under suitable non-degeneracy assumptions, critical points of $\mathcal{KR}_{\Omega}$ give rise to type I critical points of $\mathcal{KR}_{\Omega_\e}$ (see Theorem \ref{SEC1-TEO.03}). Naturally, this situation occurs for domains $\O$ with “rich” geometries. Indeed, if $\O$ is convex, $\mathcal{KR}_{\Omega}$ admits no critical points.
\vskip 0.2cm

\noindent\textbf{Type~II critical points.}
The situation here becomes more involved because $x_\e \to P$ (while $y_\e \to y_0 \neq P$), and the expansion in \eqref{sec2-02} becomes more delicate to handle. In this case, the term $\mathcal{KR}_{(B(P,\e))^c}$ plays a crucial role, leading to new and sometimes unexpected phenomena of significant interest.
\vskip 0.1cm

Specifically, if $C$ is a compact set $C\subset \O\setminus\{P\}$, then for any $y\in C$, and $x\in \O_\e$,  \eqref{sec2-03}  gives
\begin{small}\begin{equation}\label{sec2-04}
\begin{split}
  \mathcal{KR}_{(B(P,\e))^c}(x,y)=&
 \frac{(\La_1+\La_2)^2}{2\pi}\ln \e-\frac{(\La_1+\La_2)}{\pi}\Big( \La_1\ln|x-P|+\La_2 \ln|y-P|\Big)\\&+\frac{\La_1\La_2}{\pi}\ln|x-y|+O\left(\frac{\e^{2}}{|x-P|^2}\right).
\end{split}
 \end{equation}
\end{small}Combining \eqref{sec2-02} and  \eqref{sec2-04}, for any $y\in C$, and $x\in \O_\e$,  we obtain, as $\e\to 0$,
\begin{small}\begin{equation*}
\begin{split}
\mathcal{KR}_{\Omega_\e}(x,y)=&
\mathcal{KR}_{\Omega}(x,y) +\frac{\La_1^2(\ln|x-P|)^2}{2\pi\ln \e} +O\left(\frac{\e^2} {|x-P|^2}\right)+O\left(\frac{\ln|x-P|}{|\ln \e|}\right).
\end{split}
 \end{equation*}
\end{small}Furthermore, from \eqref{sec5-02} and \eqref{sec5-04b} below, we have
\begin{small}\begin{equation}\label{sec5-03a}
\begin{cases}
\nabla_x\mathcal{KR}_{\Omega_\e}(x,y)=
\nabla_x \mathcal{KR}_{\Omega}(x,y) +\Big(\frac{\La_1^2 \ln|x-P| }{ \pi\ln \e} \Big)\frac{x-P}{|x-P|^2} +O\left(\frac{\e^2} {|x-P|^3}\right)+O\left(\frac{1}{|x-P|\cdot|\ln \e|}\right),\\[2mm]
 \nabla_y { \mathcal{KR}_{\Omega_\e}(x,y)} = \nabla_y { \mathcal{KR}_{\Omega}(x,y)} +
  O\left(\Big|\frac{\ln |x-P|}{  \ln \e } \Big|  \right)+o\big(1\big).
\end{cases}\end{equation}
\end{small}Starting from these, it can be shown that if $(x_\e,y_\e)$ is a type II critical point of $\mathcal{KR}_{\Omega_\e}$, then   $\frac{\ln |x_\e-P|}{  \ln \e }\to 0$, see
 \eqref{sec5-06} below (This shows $|x_\e-P|\geq \sqrt{\e}$). Hence, to consider type II critical points,  \eqref{sec5-03a} can be further simplified to
\begin{small}\begin{equation}\label{sec2-06}
 \begin{cases}
\nabla_x\mathcal{KR}_{\Omega_\e}(x,y)=
\nabla_x \mathcal{KR}_{\Omega}(x,y) +\Big(\frac{\La_1^2 \ln|x-P| }{ \pi\ln \e} \Big)\frac{x-P}{|x-P|^2} +O\left(\frac{1}{|x-P|\cdot|\ln \e|}\right),\\[2mm]
 \nabla_y { \mathcal{KR}_{\Omega_\e}(x,y)} = \nabla_y { \mathcal{KR}_{\Omega}(x,y)} +o\big(1\big).
\end{cases}
\end{equation}
\end{small}The term $\Big(\frac{\La_1^2 \ln|x-P| }{ \pi\ln \e} \Big)\frac{x-P}{|x-P|^2}$ in \eqref{sec2-06} plays a crucial role in the analysis of Type II critical points. Also from the second identity of \eqref{sec2-06}, the necessary condition satisfied by a type II critical point is given by formula \eqref{sec1-08}, that is, $\nabla_y \mathcal{KR}_{\Omega}(P,y_0)=0$, see Proposition \ref{sec1-prop.05}. The same condition \eqref{sec1-08}, together with  non-degeneracy assumptions, is also sufficient to guarantee the existence of critical points of this type, see Theorem \ref{SEC1-TEO06}, Theorem \ref{sec1-teo08}, and Theorem \ref{SEC1-TEO09}. See more details on our strategies in the study of type II crtical points at the beginning of Section \ref{sec-5}.
\vskip 0.2cm

\noindent\textbf{Type~III critical points.}
In this case, for the simplicity of the notations, we assume  that $P=0$ and hence $(x_\e,y_\e)\to (0, 0)$ as $\e\to 0$. Then from \eqref{sec6-09} below, it holds
\begin{small}\begin{equation*}
\begin{split}
\mathcal{KR}_{\Omega_\e}(x,y)=&\frac{\La_2^2}{\pi \e^{\beta}}\left(
F_\e(w,z)\Big|_{(w,z)=\big(\e^{-\beta}x,\e^{-\beta}y\big)}+o(1)\right)~~~\hbox{with}
~~~\beta=\frac{\tau}{(\tau+1)^2}~~~\hbox{and}
~~~\tau=\frac{\La_1}{\La_2},
\end{split}
 \end{equation*}
\end{small}where
\begin{small}\begin{equation*}
\begin{split}
F_\e(w,z)=&-\frac{\tau}{\tau+1}\Big( \tau \ln |w|+ \ln |z|\Big)+\tau \ln |w-z|
- \frac{\left(\tau\ln|w|+ \ln|z|+2\pi\frac{\tau^2+\tau+1}{\tau+1} \mathcal R _\Omega(0)\right)^2}{2\big(\ln\e+2\pi \mathcal R_\Omega (0)\big)}.
\end{split}
\end{equation*}
\end{small}First we know that the existence of critical points for $F_\e(w,z)$ will be essential for the existence of type~III critical points (see Section \ref{sec-6} below). However, for any rotation $T \in O(2)$, it holds that $F_\e(w,z) = F_\e(Tw,Tz)$. This shows that the critical points of $F_\e(w,z)$ are not isolated.

\vskip 0.05cm

To compute the critical points of $F_\e(w,z)$, we define
\begin{small}\[
\widetilde F_\e(\tilde w, \tilde z) = F_\e(w,z)\Big|_{(w,z)=\bigl( (\tilde w, 0), (\tilde z,0)\bigr)}, ~~~~~ \text{for}~~~~~ (\tilde w, \tilde z) \in \mathbb R^2,  |\tilde w|^2 + |\tilde z|^2 > 1.
\]
\end{small}We will show that $\widetilde F_\e(\tilde w, \tilde z)$ has a unique nondegenerate minimum at $(\tilde w_0, \tilde z_0)$. Next, we define a torus-type domain as follows,
\begin{small}\begin{equation*}
 \begin{split}
 B^*_{\delta(\e)}=&\Big\{(x,y)=(\e^{\beta}w,\e^{\beta}z); w,z\in\R^2, \,\exists \,~\mbox{a rotation}~\, T,
 s.t. \,\,T(w,z)=\big((\widetilde{w},0),(\widetilde{z},0)\big),\,\, (\widetilde w,\widetilde z)\in B\big((\tilde w_0,\tilde z_0),\delta(\e)\big)
\Big\}
 \end{split}
 \end{equation*}
\end{small}and we will prove that  the minimum of  $\mathcal{KR}_{\Omega_\e}(x,y)$ in $\bar B^*_{\delta(\e)}$ is achieved
in the interior of $B^*_{\delta(\e)}$.
So $\mathcal{KR}_{\Omega_\e}(x,y)$  has at least one minimum point in $B^*_{\delta(\e)}$.

\vskip 0.1cm

Now we turn to the discussion of the multiplicity of type~III critical points. From the properties of
$F_\e(w,z)$, if the critical point of $ \mathcal{KR}_{\Omega_\e}(x,y)$ is isolated (otherwise, there exists infinitely many critical points), then from Poincar\'e-Hopf theorem, we can derive the following result,
 \begin{equation}\label{sec2-07}
\deg\big(\nabla
  \mathcal{KR}_{\Omega_\e}(x,y), \bar B^*_{\delta(\e)},0\big)=\deg\Big(\nabla
F_\e(w,z)\Big|_{(w,z)=(\e^{-\beta}x,\e^{-\beta}y)}, \bar B^*_{\delta(\e)},0\Big)= \chi(\mathbb S^{1})=0,
 \end{equation}where $\chi$ is the Euler characteristic number.
From \eqref{sec2-07}, we see that $\mathcal{KR}_{\Omega_\e}(x,y)$ can not just have one (isolated) minimum point in $B^*_{\delta(\e)}$ and so it has at least two critical points.

\vskip 0.05cm

Finally, we  discuss the exact number of type III critical points. As stated in Theorem \ref{sec1-teo12} (see Section \ref{sec-6}), we only compute $|x_\e-P|$ and $|y_\e-P|$.
To determine the direction of $x_\e$ and $y_\e$, we need to expand $\mathcal{KR}_{\O_\e}(x,y)$ more precisely, up to the point where the leading term no longer exhibits rotational invariance. To do this, we introduce the following transform
\begin{equation*}\mbox{$(w,\gamma):=(\frac{x}{\e^\beta},\frac{x+\tau y}{\e^{2\beta}})$ with $\beta=\frac{\tau}{(1+\tau)^2}$ and $\tau=\frac{\La_1}{\La_2}$}.
\end{equation*}
And then we have (see Proposition \ref{sec7-prop7.5})
 \begin{small}\begin{equation}\label{sec2-08}
 \begin{cases}
  \frac{\partial \mathcal{KR}_{\Omega_\e}(x,y)}{\partial x_j}\Big|_{(x,y)=
\big(\e^{\beta}w,
\frac{-\e^{\beta} w+
\e^{ 2\beta } \gamma }{\tau}\big)}
 \\[4mm]=-\frac{\La_1\La_2}{\pi} \left\{\left[
\frac{k\big( |w|,\tau\big)}{|w|^2\e^{\beta}(\ln \e+2\pi \mathcal{R}_{\Omega}(0))}-
 2\beta  \frac{(w\cdot \gamma)}{|w|^4}
\right]w_j
-\frac{\pi(1+\tau+\tau^2)}{1+\tau}
\frac{\partial \mathcal{R}_\Omega(0)}{\partial x_j}+\frac{\beta}{|w|^2}\gamma_j  \right\}
+O\left(\frac{1}{|\ln \e|}
\right),\\[4mm]
 \frac{\partial \mathcal{KR}_{\Omega_\e}(x,y)}{\partial y_j}\Big|_
 {(x,y)=
\big(\e^{\beta }w,
\frac{-\e^{\beta } w+
\e^{2\beta } \gamma }{\tau}\big)}
 \\[4mm]=-\frac{\La_2^2}{\pi}
  \left\{\left[-
\frac{\tau k\big( |w|,\tau\big)}{|w|^2\e^{\beta}(\ln \e+2\pi \mathcal{R}_{\Omega}(0))}-
 2\beta \tau^2 \frac{(w\cdot \gamma)}{|w|^4}
\right]w_j
-\frac{\pi(1+\tau+\tau^2)}{1+\tau}
\frac{\partial \mathcal{R}_\Omega(0)}{\partial x_j}+\frac{\tau^2\beta}{|w|^2}\gamma_j  \right\}
+O\left(\frac{1}{|\ln \e|}
\right), \end{cases}
\end{equation}\end{small}where $
 k(r,\tau):=(1+\tau)\big(\ln r+2(1-\beta)\pi \mathcal{R}_{\Omega}(0)\big)-\ln \tau$, $j=1,2$. Then we try to solve
  \begin{small}\begin{equation}\label{sec2-09}
\begin{cases}
\left[
\frac{k\big( |w|,\tau\big)}{|w|^2\e^{\beta}(\ln \e+2\pi \mathcal{R}_{\Omega}(0))}-
 2\beta  \frac{(w\cdot \gamma)}{|w|^4}
\right]w_j
-\frac{\pi(1+\tau+\tau^2)}{1+\tau}
\frac{\partial \mathcal{R}_\Omega(0)}{\partial x_j}+\frac{\beta}{|w|^2}\gamma_j=0,\\[3mm]
 \left[-
\frac{\tau k\big( |w|,\tau\big)}{|w|^2\e^{\beta}(\ln \e+2\pi \mathcal{R}_{\Omega}(0))}-
 2\beta \tau^2 \frac{(w\cdot \gamma)}{|w|^4}
\right]w_j
-\frac{\pi(1+\tau+\tau^2)}{1+\tau}
\frac{\partial \mathcal{R}_\Omega(0)}{\partial x_j}+\frac{\tau^2\beta}{|w|^2}\gamma_j =0,
\end{cases}
\end{equation}\end{small}which is the main term of system \eqref{sec2-08}.
A crucial finding is that if $\La_1 \neq \La_2$ and $\nabla \mathcal{R}_\Omega(0) \neq 0$, then system \eqref{sec2-09} has exactly two solutions.

\vskip 0.1cm

If $\La_1 = \La_2$ or $\nabla \mathcal{R}_\Omega(0) = 0$, the expansion in \eqref{sec2-08} is insufficient. The main idea is to expand $\mathcal{KR}_{\Omega_\e}$ further until the effects of the hole's location and the geometry of $\Omega$ become apparent (see Proposition \ref{sec7-prop7.12}). For instance, if $\La_1 = \La_2$, then, instead of \eqref{sec2-09}, we need to study the following system:
\begin{small}
\begin{equation*}
\begin{cases}
\left[
\frac{2 k(|w| )}{ |w|^2 \e^{\frac{1}{2}}(\ln \e +2\pi \mathcal{R}_{\Omega}(0))} -
\frac{ |\gamma|^2     }{4 |w|^4}\right] w_j - \frac{3\pi  (w\cdot \gamma)}{|w|^2} \frac{\partial \mathcal{R}_\Omega(0)}{\partial x_j}
-6\pi \displaystyle \sum^2_{i=1} \frac{\partial^2H_{\Omega}(0,0)}{\partial x_i\partial x_j}w_i=0,\\[4mm]
\frac{ (w\cdot \gamma)w_j }{|w|^4} - \frac{\gamma_j }{ 2 |w|^2}  +3\pi \frac{\partial \mathcal{R}_\Omega(0)}{\partial x_j}=0,
\end{cases}
\end{equation*}
\end{small}with $j=1,2$ and $k(r)=2\ln r-3\pi \mathcal{R}_\Omega(0)$. We will prove that it has exactly four solutions if the matrix defined in Theorem~\ref{sec1-teo16} has two distinct eigenvalues.

\vskip 0.05cm

When $\nabla \mathcal{R}_\Omega(0) = 0$, it becomes crucial to study:
\begin{small}
\begin{equation*}
\begin{cases}
\left[
\frac{k(|w|,\tau)}{\e^{2\beta}|w|^2(\ln \e +2\pi \mathcal{R}_{\Omega}(0))}
\right]
w_j   -\frac{
2\pi}{\tau^2(\tau+1)}  \Big(\overline{\bf{M}}w \Big)_j=0,\\[2mm]
\frac{ (w\cdot \gamma)w_j }{\e^\beta|w|^4} - \frac{\gamma_j }{ 2\e^\beta |w|^2}
-\frac{ \pi(\tau^2-1)}{\tau^3}\Big({\bf{M}}_1 w \Big)_j=0,
\end{cases}
\end{equation*}
\end{small}where $j=1,2$, $\overline{\bf{M}}$ is the matrix in \eqref{sec1-17}  and
${\bf{M}}_1 := \left[ (\tau^2+\tau+1) \frac{\partial^2H_{\Omega}(0,0)}{\partial x_i\partial x_j} + (\tau+1)^2 \frac{\partial^2H_{\Omega}(0,0)}{\partial y_i\partial x_j} \right]_{1\leq i,j\leq 2}$. We will prove that this system has exactly four solutions if $\overline{\bf{M}}$ has two distinct eigenvalues.

\vskip 0.05cm

We point out that estimating the determinant of the Hessian of $\mathcal{KR}_{\Omega_\e}$ is highly nontrivial. Fortunately, it can be computed at each type~III critical point of $\mathcal{KR}_{\Omega_\e}$, which establishes the non-degeneracy of all type~III critical points. More importantly, this allows us to compute the degree of each type~III critical point of $\mathcal{KR}_{\Omega_\e}$. Then by computing the total degree, a considerably easier task, we can determine the exact number of type~III critical points of $\mathcal{KR}_{\Omega_\e}$. More details on the strategy used to find type III critical points can be found at the beginning of Section \ref{sec7}.
\section{A first necessary condition of critical points }\label{sec-3}

In this section, we will prove
that any critical point of $\mathcal{KR}_{\Omega_\e}$ must be away from $\partial\Omega$ (Proposition \ref{sec1-prop.02}).  Our first tool is an expansion of $\mathcal{KR}_{\O_\e}(x,y)$ which will play an important role in the rest of the paper.
Passing to the gradient of \eqref{sec2-01}, we have
\begin{equation}\label{sec3-01}
\begin{cases}
\frac{\partial \mathcal{KR}_{\Omega_\e}(x,y)}{\partial x_j}=\La_1^2\frac{\partial \mathcal{R}_{\Omega_\e}(x)}{\partial x_j}+\frac{\La_1\La_2}{\pi}\frac{x_j-y_j}{|x-y|^2}+2\La_1\La_2\frac{\partial H_{\Omega_\e}(x,y)}{\partial x_j},
\\[4mm]
\frac{\partial \mathcal{KR}_{\Omega_\e}(x,y)}{\partial y_j}=\La_2^2\frac{\partial \mathcal{R}_{\Omega_\e}(y)}{\partial y_j}-\frac{ \La_1\La_2}{\pi}\frac{x_j-y_j}{|x-y|^2}+2\La_1\La_2\frac{\partial H_{\Omega_\e}(x,y)}{\partial y_j}.
\end{cases}
\end{equation}
Another important tool which will be used in all the paper is the explicit expression of $\nabla\mathcal{KR}_{(B(P,\e))^c}(x,y)$ and $\nabla^2\mathcal{KR}_{(B(P,\e))^c}(x,y)$, which are direct by \eqref{sec2-03},
\begin{small}
\begin{equation}\label{sec3-02}
\begin{cases}
\frac{\partial \mathcal{KR}_{(B(P,\e))^c}(x,y)}{\partial x_j}=-\frac{\La_1}{\pi}\left[\Big(
\La_1\frac{x_{j}-P_j}{|x-P|^2-\e^2}
+\La_2\frac{|y-P|^2(x_{j}-P_j)-\e^2(y_{j}-P_j) }{|x-P|^2|y-P|^2-2 \e^2(x-P)\cdot (y-P)+\e^4}\Big) -\La_2\frac{x_{j}-y_{j} }{|x-y|^{2}}\right],\\[4mm]
\frac{\partial \mathcal{KR}_{(B(P,\e))^c}(x,y)}{\partial y_j}=-\frac{\La_2}{\pi}\left[
\Big( \La_2\frac{y_{j}-P_j}{|y-P|^2-\e^2}
+\La_1\frac{|x-P|^2(y_{j}-P_j)-\e^2(x_{j}-P_j)}{|x-P|^2|y-P|^2-2 \e^2(x-P)\cdot (y-P)+\e^4}\Big)-\La_1\frac{y_{j}-x_{j} }{|x-y|^{2}}\right],
\end{cases}
 \end{equation}
 \end{small}and
 \begin{small}\begin{equation}\label{sec3-03}
 \begin{split} \begin{cases}
\frac{\partial^2 \mathcal{KR}_{(B(P,\e))^c}(x,y)}{\partial x_i\partial x_j}=- \frac{\La_1}{\pi}\left[
\La_1 \Big(\frac{\delta_{ij}} {|x-P|^2-\e^2}-
\frac{2(x_i-P_i)(x_j-P_j)} {(|x-P|^2-\e^2 )^{2} } \Big)
+\La_2 \Big(\frac{|y-P|^2\delta_{ij}} {|x-P|^2|y-P|^2-2\e^2(x-P)\cdot(y-P) +\e^4}\right.\\[3mm]
\left. \,\,\,\,\,\,\,\,\,\,\,\,\,\,\,\,\,\,\,\,\,\,\,\,\,\,\,\,\,\,\,\,\,\,\,\,\,\,\,\,\,\,\,\,\,\,\,\,\,\,\,\,\,-
\frac{2\big(|y-P|^2(x_i-P_i)-\e^2(y_i-P_i)\big)\big(|y-P|^2(x_j-P_j)-\e^2(y_j-P_j)\big)} {(|x-P|^2|y-P|^2-2\e^2 (x-P)\cdot(y-P)+\e^4)^{2}} \Big) \right]+\frac{\La_1\La_2}{\pi|x-y|^2}\Big(\delta_{ij}-\frac{2(x_i-y_i)(x_j-y_j)}{|x-y|^2}\Big)
,\\[4mm]
\frac{\partial^2 \mathcal{KR}_{(B(P,\e))^c}(x,y)}{\partial x_i\partial y_j}=-\frac{ \La_1\La_2 }{\pi} \left[\frac{2(y_j-P_j)(x_i-P_i)-\e^2\delta_{ij}} {|x-P|^2|y-P|^2-2\e^2 (x-P)\cdot(y-P) +\e^4}-
\frac{2(|y-P|^2(x_i-P_i)-\e^2(y_i-P_i))(|x-P|^2(y_j-P_j)-\e^2(x_j-P_j))} {(|x-P|^2|y-P|^2-2\e^2 (x-P)\cdot (y-P) +\e^4)^{2}} \right]\\[3mm]
\,\,\,\,\,\,\,\,\,\,\,\,\,\,\,\,\,\,\,\,\,\,\,\,\,\,\,\,\,\,\,\,\,\,\,\,\,\,\,\,\,\,\,\,\,\,\,\,
\,\,\,\,\,-\frac{\La_1\La_2}{\pi|x-y|^2}\Big(\delta_{ij}-\frac{2(x_i-y_i)(x_j-y_j)}{|x-y|^2}\Big)
,\\[4mm]
\frac{\partial^2 \mathcal{KR}_{(B(P,\e))^c}(x,y)}{\partial y_i\partial y_j}=- \frac{\La_2}{\pi}\left[
\La_2\Big(\frac{\delta_{ij}} {|y-P|^2-\e^2}-
\frac{2(y_i-P_i)(y_j-P_j)} {(|y-P|^2-\e^2 )^{2} } \Big)
+\La_1 \Big(\frac{|x-P|^2\delta_{ij}} {|x-P|^2|y-P|^2-2\e^2 (x-P)\cdot (y -P) +\e^4}\right.\\[3mm]
\left.\,\,\,\,\,\,\,\,\,\,\,\,\,\,\,\,\,\,\,\,\,\,\,\,\,\,\,\,\,\,\,\,\,\,\,\,\,\,\,\,\,\,\,\,\,\,\,\,\,\,\,\,\,-
\frac{2(|x-P|^2(y_i-P_i)-\e^2(x_i-P_i))(|x-P|^2(y_j-P_j)-\e^2(x_j-P_j))} {(|x-P|^2|y-P|^2-2\e^2 (x-P)\cdot(y-P) +\e^4)^{2}} \Big) \right]+\frac{\La_1\La_2}{\pi |x-y|^2}\Big(\delta_{ij}-\frac{2(x_i-y_i)(x_j-y_j)}{|x-y|^2}\Big).
\end{cases}\end{split}\end{equation}
\end{small}

Now we recall an interesting identity involving the Green function $G_D(x,y)$.
\begin{lem}\label{sec3-lem3.1}
Let $D\subset \R^2$, be a smooth bounded domain. For any $a_0\in \R^2$ and $a,b\in D$, $a\neq b$, there holds
\begin{small}\begin{equation}\label{sec3-04}
\begin{split}
\int_{\partial D}& \big(x-a_0\big)\cdot \nu(x)\left(\frac{\partial G_D(x,a)}{\partial \nu_x} \right) \left(\frac{\partial G_D(x,b)}{\partial \nu_x}\right)ds_x =
 \big(a_0-a\big)\cdot\nabla_xG_{D}(a,b)+ \big(a_0-b\big)\cdot\nabla_xG_{D}(b,a),
\end{split}
\end{equation}
\end{small}where $\nu(x)$ is the unit outer normal at $x\in \partial D$.
\end{lem}
\begin{proof}
See Lemma 3.1 in \cite{GrossiTakahashi}.
\end{proof}

Here we give an expansion of $\nabla \mathcal{R}_{\O}(x)$ near the boundary of $\Omega$.
\begin{lem}\label{sec3-lem3.2}
Let $d_x=dist\{x,\partial \Omega\}$ for $x\in \Omega$, then as $d_x\to 0$,
\begin{small}\begin{equation}\label{sec3-05}
\begin{split}
\nabla \mathcal{R}_{\O}(x)= \frac{1}{2\pi d_x}\frac{x'-x}{d_x}+ O\Big(1\Big),
\end{split}
\end{equation}
\end{small}where $x'\in \partial \Omega$ is the unique point satisfying $dist\{x,\partial \Omega\}=|x-x'|$.
\end{lem}
\begin{proof}
The main idea is similar to Proposition 6.7.1 in  \cite{CPY2020} for $N\geq 3$. And we would like to put the details of proofs at the end of Appendix \ref{sec-app1-1}.

\end{proof}

\begin{proof}[\bf{Proof of Proposition \ref{sec1-prop.02}}] We divide the proof into two parts. First, we  show assertion (1).
\vskip 0.1cm
\noindent\textup{(1)} If \eqref{sec1-07} is not true, then we have following two cases.
\vskip 0.1cm
\noindent \textbf{Case 1:} $x_0\in \partial \Omega$, $y_0\neq x_0$ (or
 $y_0\in \partial \Omega$, $y_0\neq x_0$).

\vskip 0.1cm

\noindent \textbf{Case 2:} $x_0,y_0 \in \partial \Omega$ with $x_0=y_0$.

\vskip 0.1cm

Now we prove that the above two cases will not occur.
If Case 1 holds, then we have
\begin{small}
\begin{equation*}
\begin{split}
0=\frac{\partial \mathcal{KR}_{\Omega_\e}(x_\e,y_\e)}{\partial x_j}=\La_1^2
\frac{\partial \mathcal{R}_{\Omega_\e}(x_\e)}{\partial x_j}+
\underbrace{\frac{\La_1\La_2}{\pi}\frac{x_{\e,j}-y_{\e,j}}{|x_\e-y_\e|^2}}_{=
\frac{\La_1\La_2}{\pi}\frac{x_{0,j}-y_{0,j}}{|x_0-y_0|^2}+O(1) =O(1)
}+\underbrace{2\La_1\La_2\frac{\partial H_{\Omega_\e}(x_\e,y_\e)}{\partial x_j}}_{=
O(1)},
\end{split}
\end{equation*}
\end{small}which gives us that
$\big| \nabla  \mathcal{R}_{\Omega_\e}(x_\e) \big|=
O\big(1\big)$.
This is a contradiction with \eqref{sec3-05}.
\vskip 0.2cm
If Case $2$ occurs,  by \eqref{sec3-05}, the smoothness of $\partial\O$ and denoting by  $x'$ as the unique point of $\partial \O$ such that $dist\{x,\partial \Omega\}=|x-x'|$,
\begin{equation*}
\lim_{x\to x_0}\frac{\nabla \mathcal{R}_{\Omega}(x)}{|\nabla \mathcal{R}_{\Omega}(x)|}=\lim_{x\to x_0} \frac{x-x'}{|x-x'|}=\nu(x_0).
\end{equation*}
Choosing $Q\in \Omega$ such that $(x_0-Q)\cdot\nu(x_0)>0$, we get
\vskip 0.1cm
\noindent (i) $(x-Q)\cdot \nabla \mathcal{R}_{\Omega}(x)=|\nabla \mathcal{R}_{\Omega}(x)|\big[(x_0-Q)\cdot\nu(x_0)+o(1)\big]\to+\infty$ for any $x\in \Omega$ and close to $x_0$. Hence it holds
$$(x-Q)\cdot \nabla \mathcal{R}_{\Omega}(x)>0,~~~\mbox{for any $x\in \Omega$ and close to $x_0$}.$$
\noindent (ii) $(x-Q)\cdot \nu(x)>0$, for any $x\in \partial \Omega$ closing to $x_0$, where $\nu(x)$ is the unit outward normal of $\partial \Omega$.

\vskip 0.1cm

If $(x_\e,y_\e)$ is a critical point of $\mathcal{KR}_{\Omega_\e}(x,y)$, then by \eqref{sec3-01},
\begin{small}\begin{equation}\label{sec3-06}
\nabla \mathcal{R}_{\Omega_\e}(x_\e) \La^2_1  - \nabla_x G_{\Omega_\e}(x_\e,y_\e)\La_1\La_2=0,~~
\nabla \mathcal{R}_{\Omega_\e}(y_\e) \La^2_2  - \nabla_y G_{\Omega_\e}(x_\e,y_\e)\La_1\La_2=0.
\end{equation}
\end{small}Multiplying $Q-x_\e$ and $Q-y_\e$ to the first and second equation of \eqref{sec3-06} and summing up, we have
\begin{small}\begin{equation}\label{sec3-07}
\begin{split}
&(Q-x_\e)\cdot \nabla \mathcal{R}_{\Omega_\e}(x_\e) \La^2_1 +  (Q-y_\e)\cdot \nabla \mathcal{R}_{\Omega_\e}(y_\e) \La^2_2 \\=&
\Big[ (Q-x_\e)\cdot  \nabla_x G_{\Omega_\e}(x_\e,y_\e) +  (Q-y_\e)\cdot \nabla_y G_{\Omega_\e}(x_\e,y_\e) \Big]\La_1\La_2.
\end{split}
\end{equation}
\end{small}Using Lemma \ref{sec3-lem3.1} with $D=\Omega_\e$, $a_0=Q$, $a=x_\e$ and $b=y_\e$, we get
\begin{small}\begin{equation}\label{sec3-08}
\begin{split}
\int_{\partial \Omega}& \big(x-Q\big)\cdot \nu(x)\left(\frac{\partial G_{\Omega_\e}(x,x_\e)}{\partial \nu_x} \right) \left(\frac{\partial G_{\Omega_\e}(x,y_\e)}{\partial \nu_x}\right)ds_x
 \\=& \int_{\partial B(P,\e)} \big(x-Q\big)\cdot \hat{\nu}(x)\underbrace{\left(\frac{\partial G_{\Omega_\e}(x,x_\e)}{\partial \nu_x} \right) }_{=O(1)}\underbrace{\left(\frac{\partial G_{\Omega_\e}(x,y_\e)}{\partial \nu_x}\right)}_{=O(1)}ds_x\\& +\underbrace{\big(Q-x_\e\big)\cdot\nabla_xG_{\Omega_\e}(x_\e,y_\e) + \big(Q-y_\e\big)\cdot\nabla_xG_{\Omega_\e}(y_\e,x_\e)}_{= \frac{1}{\La_1\La_2}\Big[ (Q-x_\e)\cdot \nabla \mathcal{R}_{\Omega_\e}(x_\e) \La^2_1 +  (Q-y_\e)\cdot \nabla \mathcal{R}_{\Omega_\e}(y_\e) \La^2_2 \Big] ~~~\mbox{by \eqref{sec3-07}}},
\end{split}
\end{equation}
\end{small}where $\nu(x)$ is the unit outer normal at $x\in \partial \Omega$ and
$\hat{\nu}(x)$ is the unit outer normal at $x\in \partial B(P,\e)$.

\vskip 0.1cm

On the other hand by Lemma \ref{app-lem-A.2}, we have
\begin{small}\begin{equation*}
\begin{split}
\int_{\partial \Omega}& \big(x-Q\big)\cdot \nu(x)\left(\frac{\partial G_{\Omega_\e}(x,x_\e)}{\partial \nu_x} \right) \left(\frac{\partial G_{\Omega_\e}(x,y_\e)}{\partial \nu_x}\right)ds_x\\=&
\int_{\partial \Omega}  \big(x-Q\big)\cdot \nu(x)\left(\frac{\partial G_{\Omega }(x,x_\e)}{\partial \nu_x} +
\frac{\partial\big(H_{\Omega_\e}(x,x_\e)-H_{\Omega }(x,x_\e)\big) }{\partial \nu_x}
\right)
\\&\,\,\,\,\,\,\times
\left(\frac{\partial G_{\Omega }(x,y_\e)}{\partial \nu_x} +
\frac{\partial\big(H_{\Omega_\e}(x,y_\e)-H_{\Omega }(x,y_\e)\big) }{\partial \nu_x}
\right)
ds_x\\=&
\int_{\partial \Omega}  \big(x-Q\big)\cdot \nu(x) \left(\frac{\partial G_{\Omega }(x,x_\e)}{\partial \nu_x} \right) \left(\frac{\partial G_{\Omega }(x,y_\e)}{\partial \nu_x}\right)ds_x+o(1).
\end{split}
\end{equation*}
\end{small}By the previous choice of $Q$ there exists a small fixed constant $d_0>0$ such that
$\big(x-Q\big)\cdot \nu(x)>0$ for any $x\in \partial \Omega \cap B(x_0,d_0)$.
Also it holds $\frac{\partial G_\Omega(x,x_\e)}{\partial \nu_x}<0$ and $\frac{\partial G_\Omega(x,y_\e)}{\partial \nu_x}<0$ for any $x\in \partial \Omega$. Then
\begin{small}\begin{equation*}
\begin{split}
\int_{\partial \Omega}& \big(x-Q\big)\cdot \nu(x)\left(\frac{\partial G_{\Omega}(x,x_\e)}{\partial \nu_x} \right) \left(\frac{\partial G_{\Omega}(x,y_\e)}{\partial \nu_x}\right)ds_x\\=&
\underbrace{\int_{\partial \Omega \cap B(x_0,d_0)}  \big(x-Q\big)\cdot \nu(x)\left(\frac{\partial G_{\Omega}(x,x_\e)}{\partial \nu_x} \right) \left(\frac{\partial G_{\Omega}(x,y_\e)}{\partial \nu_x}\right)ds_x}_{\geq 0}\\& +\underbrace{\int_{\partial \Omega \backslash B(x_0,d_0)}  \big(x-Q\big)\cdot \nu(x)\left(\frac{\partial G_{\Omega}(x,x_\e)}{\partial \nu_x} \right) \left(\frac{\partial G_{\Omega}(x,y_\e)}{\partial \nu_x}\right)ds_x}_{=O(1)}.
\end{split}
\end{equation*}
\end{small}Hence there exists a positive constant $C_0$ such that
\begin{equation*}
\mbox{LHS of \eqref{sec3-08}}  \geq-C_0.
\end{equation*}
Next aim is to show that RHS of \eqref{sec3-08} goes to $-\infty$ and this will give a contradiction.
\vskip 0.2cm

From \eqref{sec3-08}, we get
\begin{equation*}
 (Q-x_\e)\cdot \nabla \mathcal{R}_{\Omega_\e}(x_\e) \La^2_1 +  (Q-y_\e)\cdot \nabla \mathcal{R}_{\Omega_\e}(y_\e) \La^2_2 \geq \Big(C_0+o(1)\Big)\La_1\La_2.
\end{equation*}
And then using \eqref{App-A.4}, we have
\begin{equation}\label{sec3-09}
 (Q-x_\e)\cdot \nabla \mathcal{R}_{\Omega }(x_\e) \La^2_1 +  (Q-y_\e)\cdot \nabla \mathcal{R}_{\Omega }(y_\e) \La^2_2 \geq  \widetilde{C}_0,~~\mbox{for some constant}~~\widetilde{C}_0.
\end{equation}
Hence we find
\begin{small}\begin{equation}\label{sec3-10}
\begin{split}
 &(Q-x_\e)\cdot \nabla \mathcal{R}_{\Omega_\e}(x_\e) \La^2_1 +  (Q-y_\e)\cdot \nabla \mathcal{R}_{\Omega_\e}(y_\e) \La^2_2\\=
&\left[(Q-x_\e)\cdot\frac{\nabla \mathcal{R}_{\Omega}(x_\e)}{|\nabla \mathcal{R}_{\Omega}(x_\e)|}|\nabla \mathcal{R}_{\Omega}(x_\e)|+o(1)\right]\La^2_1 +\left[(Q-y_\e)\cdot \frac{\nabla \mathcal{R}_{\Omega}(y_\e)}{|\nabla \mathcal{R}_{\Omega}(y_\e)|}|\nabla \mathcal{R}_{\Omega}(y_\e)|+o(1)\right] \La^2_2\\=
&\Big[\underbrace{(Q-x_0)\cdot\nu(x_0)}_{<0}\underbrace{|\nabla \mathcal{R}_{\Omega}(x_\e)|}_{\to+\infty}+o(1)\Big]\La^2_1+\Big[\underbrace{(Q-x_0)\cdot\nu(x_0)}_{<0}\underbrace{|\nabla \mathcal{R}_{\Omega}(y_\e)|}_{\to+\infty}+o(1)\Big]\La^2_2\to-\infty.
 \end{split}
\end{equation}
\end{small}Finally, from \eqref{sec3-09} and \eqref{sec3-10}, we get a contraction that proves assertion (1).

\vskip 0.1cm

\noindent\textup{(2)}.  Now we prove assertion (2). In the proof above, we have showed that $x_0=y_0\in \partial \Omega$ is impossible. Now we prove that $x_0=y_0\in \Omega\backslash \{P\}$ is also impossible. In fact, if $x_0=y_0\in \Omega\backslash \{P\}$, then
 from $\frac{\partial \mathcal{KR}_{\Omega_\e}(x_\e,y_\e)}{\partial x_j}=0$, we know that
\begin{equation*}
\La_1^2  \underbrace{\frac{\partial\mathcal{R}_{\O}(x_0)}{\partial x_j}}_{=O(1)}
+\frac{\La_1\La_2}{\pi}\underbrace{\frac{x_{0,j}-y_{0,j} }{|x_0-y_0|^2}}_{=\infty}+  2\La_1\La_2
\underbrace{\frac{\partial H_\O(x_0,y_0)}{\partial x_j}}_{=O(1)}=0,
\end{equation*}
which is impossible. Hence our result is completed.
\end{proof}

Proposition \ref{sec1-prop.02} gives us that the critical points of $\mathcal{KR}_{\Omega_\e}(x,y)$ will belong to $\mathcal{D}_\e\times \mathcal{D}_\e$ with $\mathcal{D}_\e:=\Big\{x_\e \in \Omega_\e, dist\{x_{\e},\partial \Omega\}\geq  \delta\Big\}$. Now
we end this section stating some basic estimate of $\nabla \mathcal{KR}_{\Omega_\e}(x,y)$ and $\nabla^2 \mathcal{KR}_{\Omega_\e}(x,y)$ on $\mathcal{D}_\e\times \mathcal{D}_\e$, which will be used in all the paper.

\vskip 0.1cm

\begin{prop}\label{sec3-prop3.3}
For $x,y\in \mathcal{D}_\e$ and $i,j=1,2$, it holds
\begin{small}\begin{equation}\label{sec3-11}
\begin{cases}
 \frac{\partial \mathcal{KR}_{\Omega_\e}(x,y)}{\partial x_j}=
 \frac{\partial \mathcal{KR}_{\Omega}(x,y)}{\partial x_j}+ \frac{\partial \mathcal{KR}_{(B(P,\e))^c}(x,y)}{\partial x_j}+2\La_1\La_2\frac{\partial S(x,y)}{\partial x_j}-\frac{\La_1(x_j-P_j)}{\pi  |x-P|^2}
\frac{\La_1\ln \frac{|x-P|}{\e}+\La_2\ln \frac{|y-P|}{\e}  }{\ln \e+2\pi \mathcal{R}_{\Omega}(P)}
\\[2mm] \,\,\,\,\,\,\,\,\,\,\,\,\,\,\,\,\,\,\,\,\,\,\,\,\,\,\,\,\,\,\,\,\,\,\,
  +O\left(\frac{1}{|x-P|\cdot|\ln \e|}+\big| \frac{\ln |y-P|}{\ln \e} \big| +\frac{\e^2}{|x-P|^2} \right),\\[3mm]
 \frac{\partial \mathcal{KR}_{\Omega_\e}(x,y)}{\partial y_j}=
 \frac{\partial \mathcal{KR}_{\Omega}(x,y)}{\partial y_j}+ \frac{\partial \mathcal{KR}_{(B(P,\e))^c}(x,y)}{\partial y_j}+2\La_1\La_2 \frac{\partial S(x,y)}{\partial y_j}
-\frac{\La_2(y_j-P_j)}{\pi  |y-P|^2}
\frac{\La_1\ln \frac{|x-P|}{\e}+\La_2\ln \frac{|y-P|}{\e}  }{\ln \e+2\pi \mathcal{R}_{\Omega}(P)}
\\[2mm] \,\,\,\,\,\,\,\,\,\,\,\,\,\,\,\,\,\,\,\,\,\,\,\,\,\,\,\,\,\,\,\,\,\,\,
  +O\left(\frac{1}{|y-P|\cdot|\ln \e|}+\big| \frac{\ln |x-P|}{\ln \e} \big| +\frac{\e^2}{|y-P|^2} \right),
\end{cases}
\end{equation}
\end{small}and
\begin{small}\begin{equation}\label{sec3-12}
\begin{cases}
\frac{\partial^2 \mathcal{KR}_{\O_\e}(x,y) }{\partial   x_i\partial   x_j} =
 \frac{\partial^2 \mathcal{KR}_{\Omega}(x,y)}{\partial   x_i\partial x_j}+ \frac{\partial^2 \mathcal{KR}_{(B(P,\e))^c}(x,y)}{\partial   x_i\partial x_j}+2\La_1\La_2\frac{\partial^2 S(x,y)}{\partial   x_i \partial x_j}-\frac{\La_1}{\pi  |x-P|^2}
\frac{\La_1\ln \frac{|x-P|}{\e}+\La_2\ln \frac{|y-P|}{\e}  }{\ln \e+2\pi \mathcal{R}_{\Omega}(P)}
\\[2mm] \,\,\,\,\,\,\,\,\,\,\,\,\,\,\,\,\,\,\,\,\,\,\,\,\,\,\,\,\,\,\,\,\,\,\,\,\,\,\times
\Big(\delta_{ij}-\frac{2(x_i-P_i)(x_j-P_j)}{|x-P|^2}\Big)
  +O\left(\frac{1}{|\ln \e|\cdot |x-P|^2} + \frac{|\ln |y-P||}{|\ln \e|\cdot |x-P|}\right),\\[3mm]
 \frac{\partial^2 \mathcal{KR}_{\Omega_\e}(x,y)}{\partial x_i\partial y_j}=
 \frac{\partial^2 \mathcal{KR}_{\Omega}(x,y)}{\partial x_i \partial y_j}+ \frac{\partial^2 \mathcal{KR}_{(B(P,\e))^c}(x,y)}{\partial x_i \partial y_j}+2\La_1\La_2 \frac{\partial^2 S(x,y)}{\partial x_i \partial y_j}
\\[2mm] \,\,\,\,\,\,\,\,\,\,\,\,\,\,\,\,\,\,\,\,\,\,\,\,\,\,\,\,\,\,\,\,\,\,\,\,\,\,
+O\left(\frac{1}{|\ln \e|\cdot|x-P|\cdot|y-P|} + \frac{\e}{dist\big\{x,\partial B(P,\e)\big\}} \Big(\frac{\e}{|\ln \e|\cdot |y-P|} +\frac{\e^2}{|y-P|^2}+ \e \Big)\right),\\[4mm]
\frac{\partial^2 \mathcal{KR}_{\O_\e}(x,y) }{\partial   y_i\partial   y_j} =
 \frac{\partial^2 \mathcal{KR}_{\Omega}(x,y)}{\partial  y_i\partial y_j}+ \frac{\partial^2 \mathcal{KR}_{(B(P,\e))^c}(x,y)}{\partial  y_i\partial y_j}+2\La_1\La_2\frac{\partial^2 S(x,y)}{\partial   y_i \partial y_j}-\frac{\La_2}{\pi  |y-P|^2}
\frac{\La_1\ln \frac{|x-P|}{\e}+\La_2\ln \frac{|y-P|}{\e}  }{\ln \e+2\pi \mathcal{R}_{\Omega}(P)}
\\[2mm] \,\,\,\,\,\,\,\,\,\,\,\,\,\,\,\,\,\,\,\,\,\,\,\,\,\,\,\,\,\,\,\,\,\,\,\,\,\,\times
\Big(\delta_{ij}-\frac{2(y_i-P_i)(y_j-P_j)}{|y-P|^2}\Big)
  +O\left(\frac{1}{|\ln \e|\cdot |y-P|^2} + \frac{|\ln|x-P||}{|\ln \e|\cdot |y-P|}\right),
\end{cases}
\end{equation}
\end{small}where $\delta_{ij}$ is the Kronecker symbol. Moreover if $P=0$, $|x|,|y|\sim \e ^\beta$ with $\beta=\frac{\La_1\La_2}{(\La_1+\La_2)^2}$, then \eqref{sec3-11} can be simplified and improved into
\begin{equation}\label{sec3-13}
\begin{cases}
 \frac{\partial \mathcal{KR}_{\Omega_\e}(x,y)}{\partial x_j}=
  \frac{\partial \mathcal{KR}_{(B(P,\e))^c}(x,y)}{\partial x_j}-\frac{\La_1x_j}{\pi  |x|^2}
\frac{\La_1\ln \frac{|x|}{\e}+\La_2\ln \frac{|y|}{\e}  }{\ln \e+2\pi \mathcal{R}_{\Omega}(0)}
  +O\big(1\big),\\[3mm]
 \frac{\partial \mathcal{KR}_{\Omega_\e}(x,y)}{\partial y_j}=
  \frac{\partial \mathcal{KR}_{(B(P,\e))^c}(x,y)}{\partial y_j}
-\frac{\La_2y_j}{\pi  |y|^2}
\frac{\La_1\ln \frac{|x|}{\e}+\La_2\ln \frac{|y|}{\e}  }{\ln \e+2\pi \mathcal{R}_{\Omega}(0)}
  +O\big(1\big).
\end{cases}
\end{equation}
\end{prop}
\begin{rem}
 The proof of Proposition \ref{sec3-prop3.3} is a bit technical.  Hence, we have put it in Appendix~\ref{sec-app1}. Estimate \eqref{sec3-13} will be necessary to deal with the case of type III critical points.
\end{rem}
\begin{rem}
We believe it is useful to make a comment on the quantity $
\frac{\La_1\ln \frac{|x-P|}{\e}+\La_2\ln \frac{|y-P|}{\e}  }{\ln \e+2\pi \mathcal{R}_{\Omega}(P)}$ that appears in \eqref{sec3-11}. Since the rates of $x$ and $y$ will depend on $\e$, we cannot write the expansion more explicitly. Moreover, when dealing with type III critical points, in some cases we will need to consider second-order expansions, and therefore the quantity  $2\pi\mathcal{R}_{\Omega}(P)$ will become relevant (otherwise, it will obviously be neglected).
\end{rem}
\section{The critical points of  type I}\label{sec-4}
\vskip0.2cm

First, we recall following lemma, which is useful to analyze the properties of critical points.

\begin{lem}\label{sec4-lem4.1}
If a smooth vector field $V:B(x_0,1)\subset\R^2\to\R^2$ verifies $$V(x_0)=0~~~\mbox{and}~~~\det Jac\big(V(x_0)\big)\ne0,$$ then any approximating vector field $V_\e:B(x_0,1)\subset\R^2\to\R^2$ such that $V_\e\to V$ in $C^1\big(B(x_0,1)\big)$ admits a unique zero $x_\e$ such that $x_\e\to x_0$ and
$\det Jac\big(V_\e(x_\e)\big) \to  \det Jac\big(V(x_0)\big)\ne0 $.
\end{lem}
\begin{proof}
See Remark 6.2 in \cite{gl}.
\end{proof}

Now we are ready to prove Theorem \ref{SEC1-TEO.03}.

\vskip0.2cm
\begin{proof}[\bf{Proof of Theorem \ref{SEC1-TEO.03}}]
Assume that $(x_\e,y_\e)\in\O_\e\times\O_\e$ verifies
\begin{equation*}
\left(\frac{\partial \mathcal{KR}_{\Omega_\e}(x_\e,y_\e)}{\partial x_j},\frac{\partial \mathcal{KR}_{\Omega_\e}(x_\e,y_\e)}{\partial y_j}\right)=(0,0),~~~\mbox{for}~~~j=1,2,
\end{equation*}
with $(x_\e,y_\e)\to(x_0,y_0)\in\O\backslash\{P\}\times\O\backslash\{P\}$.
Then   we have
\begin{small}\begin{equation*}
\begin{split}
&\frac{\partial \mathcal{KR}_{(B(P,\e))^c}(x_\e,y_\e)}{\partial x_j}
-\frac{\La_1 \La_2 }{\pi } \frac{x_{\e,j}-y_{\e,j}}{|x_\e-y_\e|^2}+\frac{(\La_1^2+\La_1 \La_2) (x_{\e,j}-P_j)}{\pi  |x_\e-P|^2}\\=&
\La_1^2\frac{\partial \mathcal{R}_{(B(P,\e))^c}(x_\e,y_\e)}{\partial x_j}
+2\La_1\La_2\frac{\partial H_{(B(P,\e))^c}(x_\e,y_\e)}{\partial x_j}+\frac{(\La_1^2+\La_1 \La_2) (x_{\e,j}-P_j) }{\pi  |x_\e-P|^2}
=O\Big(\frac{1}{|\ln \e|}\Big).
\end{split}
\end{equation*}
\end{small}Hence from $\frac{\partial \mathcal{KR}_{\Omega_\e}(x_\e,y_\e)}{\partial x_j}=0$ and   \eqref{sec3-11}, we know that
$\frac{\partial \mathcal{KR}_{\Omega }(x_\e,y_\e)}{\partial x_j}=o\big(\frac{1}{|\ln \e|}\big)$,
which implies $\frac{\partial\mathcal{KR}_{\O}(x_0,y_0)}{\partial x_j}=0$.
In the same way,  we have $\frac{\partial\mathcal{KR}_{\O}(x_0,y_0)}{\partial y_j}=0$. This proves that $(x_0,y_0)$ is a critical point of $\mathcal{KR}_{\Omega}(x,y)$.

\vskip 0.1cm

Finally, if $x_0=y_0$  we obtain a contradiction since the term $\frac{x_{\e,j}-y_{\e,j}}{|x_\e-y_\e|^2}$  goes to $+\infty$  while all the others are bounded. This means that $x_0\neq y_0$ and gives the first part of Theorem \ref{SEC1-TEO.03}.
The second part follows by Lemma \ref{sec4-lem4.1} and the convergence of the second derivatives of $\mathcal{KR}_{\Omega_\e}$ to $\mathcal{KR}_{\Omega}$.
\end{proof}

\section{The critical points of type II}\label{sec-5}~

 First, let us outline the  strategies used in this section.

\begin{itemize}

\item
  We derive the necessary condition for $(P, y_0)$: $\frac{\partial\mathcal{KR}_{\O}(P,y_0)}{\partial y_j}=0$ with $j=1,2$. \vskip 0.05cm

\item We will study the  existence of solutions $y_0$ for $\frac{\partial\mathcal{KR}_{\O}(P,y_0)}{\partial y_j}=0$ with $j=1,2$ in the unit disk, and
in a general convex domain. It turns out  that the existence of solutions depends on the location
of $P$. We also prove the non-degeneracy of the solutions.

\item  We expand $\nabla_x \mathcal{KR}_{\Omega_\e}(x,y)$  and $\nabla_y \mathcal{KR}_{\Omega_\e}(x,y)$ near $(P,y_0)$  and then compute
the degree of this vector field to prove the existence of type II critical points.\vskip 0.05cm

\item We prove the non-degeneracy of all the type II critical points and then count the exact multiplicity.

\end{itemize}

Now we start this section with a necessary condition satisfied by the critical points of type II.
\begin{prop}\label{sec5-prop5.1}
Assume $(x_\e,y_\e)\in\O_\e\times \O_\e$ is a type II critical point of $\mathcal{KR}_{\Omega_\e}(x,y)$ such that $(x_\e,y_\e)\to (P,y_0)$. Then
\begin{small}\begin{equation}\label{sec5-01}
\frac{\partial\mathcal{KR}_{\O}(P,y_0)}{\partial y_j}=0,~~~\mbox{for}~~~j=1,2.
\end{equation}
\end{small}\end{prop}

\begin{proof}
Firstly, for $x_\e\to P(|x_\e-P|\geq \e)$ and $y_\e\to y_0\neq P$, by \eqref{sec3-02} and \eqref{sec3-11},
we have
\begin{small}\begin{equation}\label{sec5-02}
\begin{split}
\frac{\partial \mathcal{KR}_{\Omega_\e}(x_\e,y_\e)}{\partial y_j}=&
 \frac{\partial \mathcal{KR}_{\Omega}(x_\e,y_\e)}{\partial y_j}
+\frac{ \La_1\La_2 }{\pi }\frac{|x_\e-P|^2(y_{\e,j}-P_j)-\e^2(x_{\e,j}-P_j) }{|x_\e-P|^2|y_\e-P|^2-2 \e^2(x_\e-P)\cdot(y_\e-P)+\e^4}\\& + \frac{ \La^2_2}{\pi }\frac{(y_{\e,j}-P_j)}{ |y_\e-P|^2-\e^2 }- \frac{(\La_2^2+\La_1 \La_2) (y_{\e,j}-P_j) }{\pi  |y_\e-P|^2}
  +O\left(\frac{1}{ |\ln \e|}+ \big|\frac{\ln |x_\e-P|}{  \ln \e } \big|    \right) \\=&
   \frac{\partial \mathcal{KR}_{\Omega}(x_\e,y_\e)}{\partial y_j}
+ O\left(\Big|\frac{\ln |x_\e-P|}{  \ln \e } \Big|  \right)+o\big(1\big).
\end{split}\end{equation}
\end{small}To prove \eqref{sec5-01}, we need to estimate the term $\frac{\ln |x_\e-P|}{  \ln \e }$.  First, we show that
\begin{equation}\label{sec5-03}
\frac\e{|x_\e-P|}\to0.
\end{equation}
Indeed, suppose that  $\frac\e{|x_\e-P|}\to A\in(0,1]$
  by contradiction. From
  \eqref{sec3-11}, we know that
\begin{small}\begin{equation}\label{sec5-04a}
\begin{split}
\frac{\partial \mathcal{KR}_{\Omega_\e}(x_\e,y_\e)}{\partial x_j}=&\frac{\partial \mathcal{KR}_{\Omega }(x_\e,y_\e)}{\partial x_j}+\frac{\partial \mathcal{KR}_{(B(P,\e))^c}(x_\e,y_\e)}{\partial x_j}+\frac{(\La_1^2+\La_1 \La_2)( x_{\e,j}-P_j) }{\pi  |x_\e-P|^2}\\&
-\frac{\La_1^2(x_{\e,j}-P_j ) \ln |x_\e-P|  }{\pi  |x_\e-P|^2\ln \e}
 +O\left(\frac{1}{|x_\e-P|\cdot|\ln \e|}  + \frac{\e^2}{|x_\e-P|^2} \right).
\end{split}\end{equation}
\end{small}Since $\frac{\partial \mathcal{KR}_{\Omega_\e}(x_\e,y_\e)}{\partial x_j}=0$ and $\frac{\partial \mathcal{KR}_{\Omega }(x_\e,y_\e)}{\partial x_j}=O(1)$, then \eqref{sec5-04a} gives
\begin{small}\begin{equation}\label{sec5-04}
\begin{split} \frac{\partial \mathcal{KR}_{(B(P,\e))^c}(x_\e,y_\e)}{\partial x_j}+\frac{(\La_1^2+\La_1 \La_2)( x_{\e,j}-P_j) }{\pi  |x_\e-P|^2}
-\frac{\La_1^2(x_{\e,j}-P_j ) \ln |x_\e-P|  }{\pi  |x_\e-P|^2\ln \e}
 =O\left(\frac{1}{|x_\e-P|\cdot|\ln \e|}+ 1\right).
\end{split}\end{equation}
\end{small}On the other hand, by \eqref{sec3-02}, we can compute
\begin{small}
\begin{equation}\label{sec5-05}
\begin{split}
\frac{\partial \mathcal{KR}_{(B(P,\e))^c}(x_\e,y_\e)}{\partial x_j}=&
-\frac{ \La^2_1}{\pi }\frac{(x_{\e,j}-P_j)}{ |x_\e-P|^2-\e^2 }
-\frac{ \La_1\La_2 }{\pi }\frac{|y_\e-P|^2(x_{\e,j}-P_j)-\e^2(y_{\e,j}-P_j)}{|x_\e-P|^2|y_\e-P|^2-2 \e^2(x_\e-P)\cdot (y_\e-P)+\e^4}\\=&
-\frac{ \La_1 (x_{\e,j}-P_j)}{\pi \e^2}\left( \frac{\La_1}{\big|\frac{1}{A^2}-1+o(1)\big| } +
\frac{\Lambda_2+o(1) }{\frac{1}{A^2}+o(1)} \right)+ O\big(1\big).
\end{split}\end{equation}
\end{small}So from \eqref{sec5-04} and \eqref{sec5-05}, we get
\begin{small}
\begin{equation*}
\frac{x_{\e,j}-P_j}{ \e^2}  \frac{\La_1}{\big|\frac{1}{A^2}-1+o(1)\big| }= O\big(1\big)+ o\Big(\frac{1}{|x_\e-P|}\Big),\end{equation*}
\end{small}which is not possible,  and this proves \eqref{sec5-03}.

\vskip 0.1cm

Now putting  \eqref{sec5-03} into \eqref{sec5-04a} and using \eqref{sec3-02}, we have  \begin{small}\begin{equation}\label{sec5-04b}
\begin{split}
\frac{\partial \mathcal{KR}_{\Omega_\e}(x_\e,y_\e)}{\partial x_j}=&\frac{\partial \mathcal{KR}_{\Omega }(x_\e,y_\e)}{\partial x_j}
-\frac{\La_1^2(x_{\e,j}-P_j ) \ln |x_\e-P|  }{\pi  |x_\e-P|^2\ln \e}
 +O\left(\frac{1}{|x_\e-P|\cdot|\ln \e|}+ \frac{\e^2}{|x_\e-P|^3} \right).
\end{split}\end{equation}
\end{small}Then using \eqref{sec5-04b} and $\frac{\partial \mathcal{KR}_{\Omega_\e}(x_\e,y_\e)}{\partial x_j}=0$ and $\frac{\partial \mathcal{KR}_{\Omega }(x_\e,y_\e)}{\partial x_j}=O(1)$, we have
\begin{small}\begin{equation*}
\begin{split}
 \frac{ (x_{\e,j}-P_j) \ln |x_\e-P|   }{ |x_\e-P|^2\ln \e}
  =O\left(\frac{1}{|x_\e-P|\cdot|\ln \e|}+ \frac{\e^2}{|x_\e-P|^3}  \right)+O\Big( 1 \Big),
\end{split}
\end{equation*}
\end{small}which, together with \eqref{sec5-03},  gives
\begin{small}\begin{equation}\label{sec5-06}
\begin{split}
 \frac{  \ln |x_\e-P|   }{ \ln \e} =o\big(1\big).
\end{split}
\end{equation}
\end{small}Now $\frac{\partial \mathcal{KR}_{\Omega_\e}(x_\e,y_\e)}{\partial y_j}=0$,
\eqref{sec5-02} and \eqref{sec5-06} imply  $\frac{\partial \mathcal{KR}_{\Omega}(x_\e,y_\e)}{\partial y_j}=o\big(1\big)$, and hence \eqref{sec5-01} follows.
\end{proof}

\begin{proof}[\bf{Proof of Proposition \ref{sec1-prop.05}}]
The case $x_\e\to P$ and $y_\e\to y_0\neq P$ is proved in Proposition \ref{sec5-prop5.1}. The other case follows switching the role of $x$ and $y$.
\end{proof}

\begin{prop}\label{sec5-prop5.2}
Assume $(x_\e,y_\e)\in\O_\e\times \O_\e$ is a type II critical point of $\mathcal{KR}_{\Omega_\e}(x,y)$ such that $(x_\e,y_\e)\to (P,y_0)$. If
the matrix
\begin{equation}\label{sec5-07}
 \left(\frac{\partial^2 \mathcal{KR}_\Omega(P,y_0)}{\partial y_i\partial y_j}   \right)_{1\leq i,j\leq 2} \ \text{ is invertible},
\end{equation}
we have, for $\e\to 0$,
\begin{equation}\label{sec5-08}
y_{\e}-y_0 = -\left(\Big(\frac{\partial^2 \mathcal{KR}_\Omega(P,y_0)}{\partial y_l\partial y_j}   \Big)_{1\leq l,j\leq 2}\right)^{-1}
   \left(\frac{\partial^2 \mathcal{KR}_\Omega(P,y_0)}{\partial y_l\partial x_j}   \right)_{1\leq l,j\leq 2} \Big(x_{\e}-P\Big)\big(1+o(1)\big).
\end{equation}
Moreover,  if $\nabla_x \mathcal{KR}_\Omega(P,y_0)\neq 0$, then, for $\e\to 0$,
\begin{equation}\label{sec5-09}
x_\e=
 \frac{  s_{\e} (1+o(1) )  }{ \left| \nabla_x\mathcal{KR}_{\O}(P,y_0)\right| } \nabla_x\mathcal{KR}_{\O}(P,y_0)+P,
\end{equation}
where $s_{\e}\in \Big(\frac{1}{|\ln \e|},\frac{1}{\sqrt{|\ln \e|}}\Big)$ is the unique solution of equation
\begin{equation}\label{sec5-10}
h_{\e}(r):=\frac{\ln r}{r}-\frac \pi{\La_1^2}
\Big|\nabla_x \mathcal{KR}_{\O}\big(P,y_0\big)\Big|\ln\e=0.
\end{equation}
If, instead, $\nabla_x \mathcal{KR}_\Omega(P,y_0)=0$, and \eqref{sec5-07} holds, we have, for $\e\to 0$,
\begin{equation}\label{sec5-11}
\frac{ x_{\e}-P }{|x_{\e}-P| }\to \eta_0
 ~~~~\mbox{and}~~~~
|x_{\e}-P|=r_\e\big(1+o(1)\big),
\end{equation}
where $\lambda_0$ is a positive eigenvalue of the matrix $\textbf{M}_{0}$ (defined in Theorem \ref{SEC1-TEO09}),
$\eta_0$ is a corresponding unit eigenvector and $r_\e$ is the unique positive solution to $\frac{ \ln r   }{r^2\ln \e }= \frac{\lambda_0\pi}{ \La_1^2}$.
\end{prop}
\begin{proof}
Repeating the same computation as before we get
\begin{small}\begin{equation}\label{sec5-12}
\begin{split}
&\left(
  \begin{array}{cc}
    \left(\frac{\partial^2 \mathcal{KR}_\Omega(P,y_0)}{\partial x_i\partial x_j}   \right)_{1\leq i,j\leq 2} & \left(\frac{\partial^2 \mathcal{KR}_\Omega(P,y_0)}{\partial x_i\partial y_j}   \right)_{1\leq i,j\leq 2} \\[4mm]
   \left(\frac{\partial^2 \mathcal{KR}_\Omega(P,y_0)}{\partial y_i\partial x_j}   \right)_{1\leq i,j\leq 2} & \left(\frac{\partial^2 \mathcal{KR}_\Omega(P,y_0)}{\partial y_i\partial y_j}   \right)_{1\leq i,j\leq 2}\\
  \end{array}
\right) \left(
          \begin{array}{c}
            x_\e-P \\[4mm]
            y_\e-y_0 \\
          \end{array}
        \right)+\nabla \mathcal{KR}_\Omega(P,y_0)
       \\[3mm] &\,\,\,\,\,=
    \left(
          \begin{array}{c}
           \frac{\La_1^2  \ln |x_\e-P|   }{ \pi |x_\e-P|^2\ln \e}
 \Big(x_{\e}-P\Big) \\[4mm]
            0 \\
          \end{array}
        \right)
        +  \left(
          \begin{array}{c}
O\left(|x_\e-P|^2+|y_\e-y_0|^2\right)+o \Big(\big|\frac{\ln |x_\e-P|}{ |x_\e-P| \ln \e } \big| \Big)  \\[4mm]
 O\Big(|y_\e-y_0|^2+|x_\e-P|^2+  \big|\frac{\ln |x_\e-P|}{  \ln \e } \big| \Big) \\
          \end{array}
        \right).
\end{split}
\end{equation}
\end{small}If \eqref{sec5-07} holds,  from the second line of \eqref{sec5-12}, we have
\begin{small}\begin{equation}\label{sec5-13}
\begin{split}
 y_\e-y_0 = & -\left(\Big(\frac{\partial^2 \mathcal{KR}_\Omega(P,y_0)}{\partial y_i\partial y_j}\Big)_{1\leq i,j\leq 2}   \right)^{-1}
   \left(\frac{\partial^2 \mathcal{KR}_\Omega(P,y_0)}{\partial y_i\partial x_j}   \right)_{1\leq i,j\leq 2} \Big( x_\e-P\Big) \big(1+o(1)\big) +
 O\left(\left|\frac{\ln |x_\e-P|}{  \ln \e } \right|\right).
\end{split}\end{equation}
\end{small}If $\nabla_x\mathcal{KR}_\Omega(P,y_0)\neq 0$, since $x_\e\to P$ and $y_\e \to y_0$, from the first line of \eqref{sec5-12}, we immediately get \eqref{sec5-09}. Moreover  we claim that the function $h_\e$ in \eqref{sec5-10} has a unique zero $s_\e$. Indeed
since $\frac{d h_{\e}(r)}{dr}=\frac{1-\ln r}{r^2}$, then $\frac{d h_{\e}(r)}{dr}>0$ for $r\in (0,e)$, $\frac{d h_{\e}(r)}{dr}<0$ for $r\in (e,\infty)$ and $\displaystyle\lim_{r\to \infty}h_{\e}(r)>0$. Moreover, it holds
\begin{equation*}
\begin{cases}
h_{\e}\big(\frac{1}{|\ln \e|}\big)= |\ln \e|\left( -\ln |\ln\e|+\frac{\pi}{\La_1^2} \Big|\nabla_x \mathcal{KR}_{\O}\big(P,y_0\big)\Big|\right)<0,\\[4mm]
h_{\e}\big(\frac{1}{\sqrt{|\ln \e|}}\big)= \sqrt{|\ln \e|}\left( -\frac{\ln |\ln\e|}{2}+ \frac{\pi}{\La_1^2}\sqrt{|\ln \e|}\Big|\nabla_x \mathcal{KR}_{\O}\big(P,y_0\big)\Big|\right)>0,
\end{cases}
\end{equation*}
which concludes the claim. Inserting the rate of $|x_\e-P|$ in \eqref{sec5-09} into \eqref{sec5-13} we get \eqref{sec5-08} for $\nabla_x\mathcal{KR}_\Omega(P,y_0)\neq 0$.
\vskip 0.1cm
Next, we consider the case when  $\nabla_x\mathcal{KR}_\Omega(P,y_0)= 0$.
Putting \eqref{sec5-13} into the first line of \eqref{sec5-12}, we get
\begin{small}\begin{equation}\label{sec5-14}
\textbf{M}_0 (x_\e-P)\big(1+o(1)\big)=   \frac{\La_1^2  \ln |x_\e-P|   }{ \pi |x_\e-P|^2\ln \e}
(x_{\e}-P)+ o\left(\left|\frac{\ln |x_\e-P|}{ |x_\e-P| \ln \e } \right| \right).
\end{equation}
\end{small}Dividing by $|x_\e-P|$ we get that $\frac{x_{\e}-P}{|x_{\e}-P| }\to \pm \eta$ which is a unit vector and \eqref{sec5-14} becomes $\textbf{M}_0\eta=\lambda  \eta$ where $\lambda=\displaystyle\lim_{\e \to 0}  \frac{\La_1^2  \ln |x_\e-P|   }{ \pi |x_\e-P|^2\ln \e}$. So $\lambda $ is a nonnegative eigenvalue of $\textbf{M}_0$ and $\eta$ a corresponding unit eigenvector and the rate of $|x_\e-P|$ is given by $r_\e$.  This proves \eqref{sec5-11}, together with \eqref{sec5-13} concludes the proof of \eqref{sec5-08}.
\end{proof}

Now we focus on the existence of type II critical points such that $\nabla \mathcal{KR}_{\O}(P,y_0)\neq 0$.
We recall that for convex domain $\Omega$, $\mathcal{KR}_{\O}(x,y)$ has no critical points.
From the necessary condition \eqref{sec5-01}, to have type II critical points, \eqref{sec5-01} must have solutions.
Now we discuss the validity of equation \eqref{sec5-01}
when $\O$ is a ball firstly.
\begin{prop}\label{sec5-prop5.3}
Assume that $\O=B(0,r)$ is a ball centered at $0$ and radius $r$ such that $P\in B(0,r)$.
Then denoting by $d=dist\big\{P,\partial B(0,r)\big\}$, we have that
 there exists $d_0>0$ such that if \vskip 0.1cm
\begin{itemize}
\item
$d>d_0$, then there is no solution to \eqref{sec5-01}.\vskip 0.1cm
\item
$d=d_0$, then there is one degenerate solution $y_0(P)$ to \eqref{sec5-01}.\vskip 0.1cm
\item
$d<d_0$, then there are two nondegenerate solutions $y_1(P)$ and $y_2(P)$ to \eqref{sec5-01} such that
\begin{equation*}
\lim_{d\to0}|y_1(P) -P|=0~~~\mbox{and}~~~\lim_{d\to0}y_2(P)=0,
\end{equation*}
where by "nondegenerate" we mean that $\det
\left(\frac{\partial^2\mathcal{KR}_{\O}(P,y)}{\partial y_j\partial y_k}\right)_{1\leq j,k\leq 2}\ne0$. Moreover, it holds
\begin{equation}\label{sec5-15}
index \big(\nabla_y \mathcal{KR}_{B(0,r)}(P,\cdot),y_1(P)\big)=-1\,\,\hbox{ and }\,\, index\big(\nabla_y  \mathcal{KR}_{B(0,r)}(P,\cdot),y_2(P)\big)=1.
\end{equation}
\end{itemize}
\end{prop}

\begin{proof}
In order to simplify the notations assume that $\O$ is the unit ball $B(0,1)$. Then  we have
\begin{small}\begin{equation}\label{sec5-16a}
\mathcal{KR}_{B(0,1)}(x, y)=
-\frac{\La_1^2}{2\pi}\ln (1-|x|^2)-\frac{\La_2^2}{2\pi}\ln (1-|y|^2)+
\frac{\La_1\La_2}{\pi} \ln \frac{|x-y|}{\sqrt{|y|^2|x|^2-2x\cdot y+1}}.
\end{equation}\end{small}Then \eqref{sec5-01} becomes
\begin{small}\[\frac{\partial \mathcal{KR}_{B(0,1)}(P,y)}{\partial y_j}=0,~~~\mbox{for}~~~j=1,2,\]
\end{small}where, up to a rotation, we can assume that $P=(s,0)$, with $s\in [0,1)$.
Observe that
\begin{small}\begin{equation}\label{quasi-punto-critico}
\frac{\partial \mathcal{KR}_{B(0,1)}(P,y)}{\partial y_j}=
\frac{\La_2}{\pi} \left( \frac {\La_2y_j}{ 1-|y|^2 }+\La_1\frac{y_j-P_{j}}{|y-P|^2}+\La_1\frac {P_{j}-|P|^2y_j}{ |y|^2|P|^2-2P\cdot y+1 }\right)~~~\mbox{for}~~~j=1,2.
\end{equation}
\end{small}Let us recall $P=(P_1,P_2)$ and consider first the case $P_{2}=0$. We need to solve
\begin{small}\begin{equation}\label{equ}
\frac {\La_2y_2}{1-|y|^2}+\frac{\La_1 y_2}{|P-y|^2}-\frac {\La_1|P|^2y_2}{|y|^2|P|^2-2 P\cdot y+1}=0.
\end{equation}
\end{small}Obviously $y_2=0$ is a solution.  Note that $|y|^2=y_1^2+y_2^2$ (we already have that $|P|^2=s^2$).  We claim that
\begin{small}\begin{equation}\label{eq}
\frac {\La_2}{1-y_1^2-y_2^2}+\frac{\La_1}{s^2+y_1^2+y_2^2-2sy_1}-\frac {\La_1s^2}{(y_1^2+y_2^2)s^2-2 sy_1+1}>0.
\end{equation}
\end{small}
In fact, since
$(y_1^2+y_2^2-1)(s^2-1)\geq0$, we have
$s^2+y_1^2+y_2^2 \leq (y_1^2+y_2^2)s^2+1$ and
\begin{small}\begin{equation*}
\frac{1}{s^2+y_1^2+y_2^2-2sy_1}\geq \frac {1}{(y_1^2+y_2^2)s^2-2 sy_1+1}\geq \frac {s^2}{(y_1^2+y_2^2)s^2-2 sy_1+1}.
\end{equation*}
\end{small}
 So we see that \eqref{eq} holds and \eqref{equ} has only the solution $y_2=0$.

\vskip 0.2cm

Since $y_2=0$, we look for solutions to \eqref{quasi-punto-critico} with $y=(t,0)$ for $t\in (-1,1)$. Then equation \eqref{quasi-punto-critico} becomes
\begin{equation}\label{sec5-19}
\frac {\La_2 t}{1-t^2}+\frac{\La_1(t-s)}{|t-s|^2}+\frac {\La_1s}{1-st}=0,
~~~~\hbox{ for $(s,t)\in [0,1)\times(-1,1)$.}
\end{equation}
Obviously, when $s=0$, \eqref{sec5-19} has no solutions. Then we consider the case $s>0$. And
in this setting the claim becomes:
\begin{itemize}
\item
If $s<\bar s$ then there is no solution to \eqref{sec5-01}.\vskip 0.1cm
\item
If $s=\bar s$ then there is one degenerate solution $t(s)$ to \eqref{sec5-01}.\vskip 0.1cm
\item
If $s>\bar s$ then there are two nondegenerate solutions $t_1(s)$ and $t_2(s)$ to \eqref{sec5-01} such that $t_1(s)\to1$ and $t_2(s)\to0$ as $s\to1$.
\end{itemize}
\begin{figure}
\includegraphics[scale=0.55]{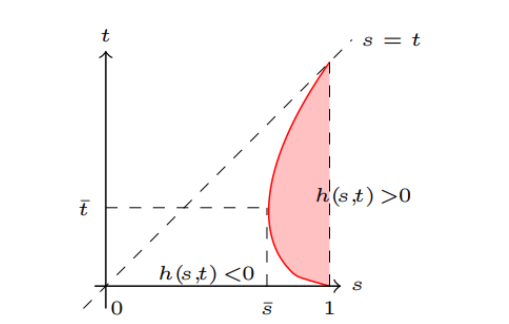}
\begin{center}
 Figure 4. The set of $h(s,t)=0$
\end{center}
\end{figure}
\noindent Now let us introduce the function (See Figure 4),
\begin{small}\[
h(s,t)= \frac {\La_2 t}{1-t^2}+ \frac{\La_1(t-s)}{|t-s|^2}+\frac {\La_1s}{1-st}.
\]
\end{small}Our proof needs several steps.

\vskip0.2cm

\noindent{\bf Step 1: All solutions to $h(s,t)=0$ verify $0<t<s$.}
\vskip0.2cm
If $s\le t<1$, we have that
$h(s,t)>0$.
And then \eqref{sec5-01} does not have solutions.
Next assume that $t<s$. Recalling that $t>-1$, we get $1- ts>s-t$ and then it holds
$\frac1{s-t} -\frac {s}{1-ts}>0$.
This shows that
\begin{small}\[
h(s,t)= \frac {\La_2t}{1-t^2}-\frac{\La_1}{s-t}+\frac {\La_1s}{1-st}<
\frac {\La_2t}{1-t^2}.
\]
\end{small}And then $h(s,t)<0$ if $t\le 0$. This proves the claim of Step 1.
\vskip0.2cm

\noindent{\bf Step 2: There exists $\bar s\in(0,1)$ such that $h(s,t)=0$ has
$
\begin{cases}
\hbox{at least $2$ solutions if }&s>\bar s,\\
1\hbox{ solution if }& s=\bar s,\\
\hbox{no solution if }&s<\bar s.
\end{cases}
$
}

It is easy to see that for $s\to0$ and $t\in (0,s)$, we have $h(s,t)\to-\infty$ and so $h(t, s)<0$ for $t\in [0,s)$ and $s$ small.
Define
\[
\bar s= \sup \big\{ s\in(0,1]:   h(s,t)<0, \; \forall\; t\in(0, s)\big\}>0.
\]
Observe that since $h(1,t)= \frac {\La_2t}{1- t^2}>0$  we have that $\bar s<1$ and then there is $\bar t\in (0,\bar s]$, such that $h(\bar s,\bar t)=0$. On the other hand,
\[
\frac{\partial h(s,t)}{\partial s}=\La_1\left(\frac{1}{(s-t)^2} +\frac{1+ts}{1- ts}\right)>0,\quad \forall\; t\in (0, s),
\]
which gives $h(s,\bar t)>0$ if $s>\bar s$. Then by the intermediate value theorem for continuous functions the claim follows since $h(s,0)=\La_1\left(-\frac {1}{s}+s\right)<0$, and $h(s,t)\to -\infty$ if $t\to s-0$. Observe that if $s>\bar s$ one zero lies
 in $(0,\bar t)$ and the other is in $(\bar t, s)$.



\vskip0.2cm
In next steps, we give additional properties of the zeros of $h(s,t)$.
\vskip0.2cm
\noindent{\bf Step 3: For $s=\bar s$, we have that $\bar t$ is a singular zero for $h(\bar s,t)$.}
\vskip0.2cm
Since $\frac{\partial h(s,t)}{\partial s}>0$ in the region $t<s$, by the implicit function theorem we get that the set of zeros of $h(s,t)$ is a graph $s=\phi(t)$ as in the Figure 4.  So $\phi$ verifies
\begin{equation*}\label{funz-comp}
h\big(\phi(t),t\big)=0.
\end{equation*}
Observe that $\frac{\partial h}{\partial t}(\bar s,\bar t)=0$, since by definition of $\bar s$, the function $h(\bar s,t)$ achieves its maximum at $t=\bar t$. This shows that $\bar t$ is a singular zero for $h(\bar s,t)$.
\vskip0.2cm
\noindent{\bf Step 4: For $s>\bar s$ there exactly two non-singular zeros $t_1(s)$ and $t_2(s)$. Moreover for  $s\to1$, $t_1(s)\to 1$ and $t_2(s)\to 0$.}
\vskip0.2cm
By the definition of $h(s,t)$ we have that
\begin{small}\begin{equation}\label{nlla}
\frac{\partial h(s,t)}{\partial t}=\frac {\La_2(1+t^2)}{(1-t^2)^2}-\frac{\La_1}{(s-t)^2}+\frac{\La_1s^2}{(1-st)^2}.
\end{equation}
\end{small}Since
\begin{small}\[
\frac{\partial}{\partial s}\left(\frac{\partial h}{\partial t}
\right)=2\La_1\left(\frac{1}{(s-t)^{3}}+\frac{s}{(1-st)^2}
+\frac{ts^2}{(1-st)^{3}}\right)>0,
\]
\end{small}and $\frac{\partial h}{\partial t}(\bar s,\bar t)=0$, by the implicit function theorem there exists a function $\psi(t)$, such that $\psi(\bar t) = \bar s$, and
\begin{equation}\label{psi}
\frac{\partial h}{\partial t}\big(\psi(t),t\big)=0.
\end{equation}
Moreover, it holds
\begin{small}\begin{equation}\label{mon}
\frac{\partial h}{\partial t}(s,t)<\frac{\partial h}{\partial t}\big(\psi(t),t\big)=0\; \text{if}\; s<\psi(t),\quad \frac{\partial h}{\partial t}(s,t)>0\; \text{if}\; s>\psi(t).
\end{equation}
\end{small}The next claim will be crucial.

\vskip0.2cm
\noindent{\em CLAIM: the curves $\psi=\psi(t)$ and $\phi=\phi(t)$ intersect only at $t=\bar t$  where $\psi(\bar t)=\phi(\bar t)=\bar s$, and
$\psi(t)>\phi(t)$\, \text{if}\, $t>\bar t$, $\psi(t)<\phi(t)$\, \text{if}\, $t<\bar t$.}

\vskip0.1cm

Once we prove the above claim, we see from \eqref{mon} that $\frac{\partial h}{\partial t}(s,t)\bigr|_{s=\phi(t)}<0$  if $t>\bar t$, while
$\frac{\partial h}{\partial t}(s,t)\bigr|_{s=\phi(t)}>0$  if $t<\bar t$. This
gives

\[
\phi'(t)= -\frac{ \frac{\partial h(\phi(t),t)}{\partial t} }{ \frac{\partial h(\phi(t),t)}{\partial s} }>0,
\]
if $t> \bar t$,  and $\phi'(t)<0$ if $t<\bar t$. Hence, for $s>\bar s$,
$h(s, t)=0$ has exactly two solutions.
\vskip0.2cm

Now we  prove the claim.
Let us show that
\begin{equation*}
\psi'(t)>0.
\end{equation*}
By definition of $\psi$ we have
\begin{small}\begin{equation*}
\begin{split}
\psi'(t) = &-\frac{\frac{\partial^2h}{\partial t^2}(\psi(t),t)}{\frac{\partial^2h}{\partial s\partial t}(\psi(t),t)}=
\frac{\underbrace{
-\frac {\La_2 t\big(t^2+3\big)}{(1-t^2)^{3}}+\frac{\La_1 }{(\psi(t)-t)^{3}}-
\frac{\La_1 \psi^3(t)}{(1-t\psi(t))^{3}}}_{=A(t)}}{\underbrace{\La_1
\left(\frac{1}{(\psi(t)-t)^{3}}+\frac{\psi(t)}{(1-t\psi(t))^2}
+\frac{t \psi^2(t)}{(1-t\psi(t))^{3}}\right)}_{>0}}.
\end{split}
\end{equation*}
\end{small}Let us show that $A(t)>0$. By \eqref{nlla} and \eqref{psi} we get
\begin{small}\begin{equation}\label{lle}
 \frac {\La_2\big(1+t^2\big) }{(1-t^2)^2(\psi(t)-t)}+
\frac{\La_1 \psi^2(t)}{(1-\psi(t)t)^2(\psi(t)-t)}= \frac{\La_1 }{(\psi(t)-t)^{3}}.
\end{equation}
\end{small}Putting \eqref{lle} into $A(t)$ we have
\begin{small}\begin{equation*}
\begin{split}
A(t)= & \La_2 \underbrace{\left(
-\frac {t\big(t^2+3\big)}{(1-t^2)^{3}}+\frac {2\big(1+ t^2\big) }{(1-t^2)^2(\psi(t)-t)}\right)
}_ {=B(t)}
+\La_1\left(-\frac{\psi^3(t)}{(1-t\psi(t))^{3}}+\frac{\psi^2(t)}{(1-t\psi(t))^2(\psi(t)-t)}\right).
\end{split}\end{equation*}
\end{small}It is easy to check that
\begin{small}\begin{equation*}
\begin{split}
-\frac{\psi^3(t)}{(1-t\psi(t))^{3}}+\frac{ \psi^2(t)}{(1-t\psi(t))^2
(\psi(t)-t)}=\frac{ \psi^2(t)}{(1-t\psi(t))^{3}(\psi(t)-t)}
\big(1-\psi^2(t)\big)>0.
\end{split}\end{equation*}
\end{small}Next, we can compute that
\begin{small}\begin{equation*}
\begin{split}
B(t)(1-t^2)^{3}(\psi(t)-t) =&
(-t^4 +3t^2+2)-(t^3+3t) \psi(t) \\> &
-t^4 -t^3+3t^2-3t+2~~\big(\mbox{since $\psi(t)<1$}\big)\\=&
 (1-t)(t^3+2t^2-t+2) \geq 0~~\big(\mbox{since}~~t\in (0,1)\big).
\end{split}\end{equation*}
\end{small}Hence we get $B(t)>0$ and then $\psi'(t)>0$.
\vskip0.2cm
Now we are in position to show that the curves $\psi=\psi(t)$ and $\phi=\phi(t)$ intersect only at $t=\bar t$. Since $\phi(\bar t)= \psi(\bar t)$, $\phi'(\bar t)= 0$  and $\psi'(\bar t)>0$, we deduce $
\psi(t)>\phi(t)$
if  $t-\bar t>0$ is small. Let us assume that there exists $t_1>\bar t$ such that $\phi(t_1)= \psi(t_1)$ and $\psi(t)>\phi(t),~~ t\in (\bar t, t_1).$
This gives
\[
\phi'(t_1)\ge \psi'(t_1)>0.
\]
On the other hand,
by \eqref{psi},
\[\phi'(t_1)=-\frac{ \frac{\partial h}{\partial t}(\phi(t_1),t_1)}{\frac{\partial h}{\partial s}(\phi(t_1),t_1)}=-\frac{ \frac{\partial h}{\partial t}(\psi(t_1),t_1)}{\frac{\partial h}{\partial s}(\phi(t_1),t_1)}=0,
\]
which is a contradiction. Hence, we have $\psi(t)>\phi(t)$ if $t>\bar t$. Similarly,  we can prove that $\psi(t)<\phi(t)$ if $t<\bar t$.

\vskip 0.1cm
We have proved that for each fixed $s>\bar s$, $h(s,t)=0$ has exactly one solution $(s, t_1(s))$ with $t_1(s)\in(\bar t,1)$, and
$h(s,t)=0$ has exactly one solution $(s, t_2(s))$ with $t_2(s)\in(0, \bar t)$.
Moreover, they are both
 non-singular,
 since
 \begin{equation}\label{1-11-10}
 \frac{\partial h}{\partial t}\big(s,t_2(s)\big)<0,\quad
\frac{\partial h}{\partial t}\big(s,t_1(s)\big)>0.
\end{equation}
Using \eqref{quasi-punto-critico} and \eqref{eq}, we find that
$
\frac{\partial^2 \mathcal{KR}_{B(0,1)}\big(P, t_i(|P|) \big)  }{\partial y^2_2}>0.$
It is easy to see that \begin{small}$$
\frac{\partial^2 \mathcal{KR}_{B(0,1)}\big(P, t_i(|P|) \big)  }{\partial y_1\partial y_2}=0.$$
\end{small}Using these relations and \eqref{1-11-10}, we conclude that
$\nabla_y^2 \mathcal{KR}_{B(0,1)}(P, y)$ is non-singular at $ t_i(|P|)$, and
$$index \big(\nabla_y \mathcal{KR}_{B(0,1)}(P, t_2(|P|)\big)=-1\hbox{ and }index\big(\nabla_y  \mathcal{KR}_{B(0,1)}(P, t_1(|P|)\big)=1.$$

\vskip 0.1cm
We end the proof by showing the behavior of $t_1(s),t_2(s)$ as $s\to 1$. Recall that $t_1(s),t_2(s)<s$ and by the definition of $h(s,t)$, we have
\begin{equation}\label{ul}
\frac {\La_2t_i(s)}{1-t_i(s)^2}-\frac{\La_1}{s-t_i(s)}+\frac {\La_1s}{1-st_i(s)}=0,~~~\mbox{for}~~~ i=1,2.
\end{equation}
Up to subsequence we can assume that
$$t_i(s_n)\to t_i\in[0,1],~~~\mbox{for}~~~ i=1,2.$$
Then by \eqref{ul} we get for $i=1,2$,
\begin{equation*}
\La_2  t_i(s)\big(s-t_i(s)\big)\big(1-st_i(s)\big)-\La_1\big(1-t_i(s)^2\big)\big(1-st_i(s)\big) +\La_1s\big(1-t_i(s)^2\big)\big(s-t_i(s)\big)=0.
\end{equation*}
Passing to the limit as $s\to1$ we get
\begin{equation*}
t_i(1-t_i)=0,~~~\mbox{for}~~~ i=1,2.
\end{equation*}
Since $t_1(s)>\bar t$ we have $t_1(s)\to1$ and by $t_2(s)<\bar t$ we get $t_2(s)\to0$.

\end{proof}

\begin{proof}[\bf{Proof of Theorem \ref{SEC1-TEO06}}]
By Proposition \ref{sec5-prop5.3}, there exists $d_1>0$ such that for $d>d_1$ the necessary condition \eqref{sec1-08} does not hold and this implies that there are no type II critical points that verify $x_\e\to P$ and $y_\e\to y_0$. Switching the role of $x$ and $y$ in Proposition \ref{sec5-prop5.3} we get the existence of $d_2>0$ such that there are no type II critical points that verify $x_\e\to x_0$ and $y_\e\to P$.
This proves $a)$.

\vskip 0.1cm

To prove $b)$ and $c)$ we consider the case $d<d_1$ and we will prove the existence of two critical points $(x_{i,\e},y_{i,\e})$, for $i=1,2$, such that $x_{i,\e}\to P$ and $y_{i,\e}\to y_i(P)$ as $\e\to 0$, where $y_i(P)$ for $i=1,2$ are the unique solutions to $\frac{\partial\mathcal{KR}_{\O}(P,y)}{\partial y_j}=0$ given by Proposition \ref{sec5-prop5.3}.  When $d<d_2$ reasoning in the same way we can show the existence of other two critical points $(x_{i,\e},y_{i,\e})$, for $i=3,4$, such that $x_{i,\e}\to x_i(P)$ and $y_{i,\e}\to P$ as $\e\to 0$, where $x_i(P)$ for $i=1,2$ are the unique solutions to $\frac{\partial\mathcal{KR}_{\O}(x,P)}{\partial x_j}=0$ given by the analogous of Proposition \ref{sec5-prop5.3}.
\vskip 0.1cm

Let us define the vector field
\begin{equation*}
\bar{L}_\e(x,y)=
 \left( \nabla_x \mathcal{KR}_{\O}(P,y)- \frac{\La_1^2 \ln |x -P|}{\pi |x-P|^{2} \ln \e}
 (x-P),\nabla_y \mathcal{KR}_{\O}(P,y)\right),
\end{equation*}
and the points $(\tilde x_\e^{(i)},y_i(P))$ for $i=1,2$, where $\tilde x_\e^{(i)}$ is given by
\begin{equation*}
\tilde x_\e^{(i)}=
\left( \frac{\pi  s_{\e,i}^2\ln \e }{\La^2_1\ln s_{\e,i}}\right)
 \nabla_x\mathcal{KR}_{\O}\big(P,y_i(P)\big)  +P,
\end{equation*}
where $s_{\e,i}\in \Big(\frac{1}{|\ln \e|},\frac{1}{\sqrt{|\ln \e|}}\Big)$ is the unique solution of equation
$$h_{\e,i}(r):=\frac{\ln r}{r}-
\Big|\nabla_x \mathcal{KR}_{\O}\big(P,y_i(P)\big)\Big|\ln\e=0.$$
For $i=1,2$ we consider the set $B_\e^i:=B( \tilde x_\e^{(i)},\delta_\e)\times B(y_i(P),\delta)$ where $\delta_\e<<\frac{1}{\sqrt{|\ln \e|}}$  and
$\delta$ is so small that they satisfy
\begin{equation}\label{sec5-26}
\overline{B( \tilde x_\e^{(1)},\delta_\e)\times B(y_1(P),\delta)}\cap \overline{B( \tilde x_\e^{(2)},\delta_\e)\times B (y_2(P),\delta)}=\emptyset.\end{equation}
We want to show that, for $i=1,2$,
\begin{equation}\label{sec5-27}
deg \Big(\nabla\mathcal{KR}_{\O_\e}(x,y),B_\e^i,0\Big)=deg \Big(\bar{L}_\e(x,y)
,B_\e^i,0\Big).
\end{equation}
It is easy to see that the point $(\tilde x_\e^{(i)},y_i(P))$ satisfies $\bar{L}_\e(\tilde x_\e^{(i)},y_i(P))=0$ and it is the unique zero of $\bar{L}_\e(x,y)$ in   $\overline B_\e^i$ by the choice of $\delta$ and $\delta_\e$ in \eqref{sec5-26}. This implies that
 \begin{equation}\label{sec5-28}
\bar{L}_\e(x,y)\neq 0 \ \hbox{ for }(x,y)\in \partial B_\e^i.
\end{equation}
Recalling \eqref{sec3-02}, for every $x\in B( \tilde x_\e^{(i)},\delta_\e)$ and for every $y\in B\big(y_i(P),\delta\big)$, we have
the following expansion, as $\e\to 0$,
\begin{equation*}
\begin{cases}
\frac{\partial \mathcal{KR}_{\Omega_\e}(x,y)}{\partial x_j}=\frac{\partial\mathcal{KR}_{\O}(P,y)}{\partial x_j}- \La_1^2\left(\frac 1{\pi}+o(1)\right)\frac{\ln |x -P|}{\ln \e} \frac{(x_{j}-P_j)}{|x-P|^{2}},
\\[2mm]
\frac{\partial \mathcal{KR}_{\Omega_\e}(x,y)}{\partial y_j}=\frac{\partial\mathcal{KR}_{\O}(P,y)}{\partial y_j}+o(1).
\end{cases}
\end{equation*}
Hence $\nabla\mathcal{KR}_{\O_\e}$ turns to be a small perturbation of $\bar{L}_\e$ and so \eqref{sec5-28} implies,  for $\e$ small enough,
\begin{equation*}
\nabla\mathcal{KR}_{\O_\e}(x,y)\neq 0 \ \hbox{ for }(x,y)\in \partial B_\e^i,
\end{equation*}
and then we prove \eqref{sec5-27} by the homotopy invariance of the degree.

\vskip 0.1cm

It lasts to prove that, for $i=1,2$,
\begin{equation}\label{sec5-29}
deg \Big(\bar{L}_\e(x,y)
,B_\e^i,0\Big)\neq 0,
\end{equation}
which, by \eqref{sec5-27} proves the existence of at least one critical point $(x_{i,\e},y_{i,\e})$ for $\mathcal{KR}_{\O_\e}(x,y)$ in $B_\e^i$. To do this we compute the Jacobian of the vector field $\bar L_\e(x,y)$ at the points $(\tilde x_\e^{(i)},y_i(P))$,
\begin{equation*}
J_\e(x,y)=\left(
  \begin{array}{cc}
\Big(A_{\e,j,k}(x)\Big)_{1\leq j,k\leq 2}
  & \Big(\frac{\partial^2 \mathcal{KR}_\Omega(P,y)}{\partial x_k\partial y_j}\Big)_{1\leq j,k\leq 2}  \\[4mm]
   0  & \Big(\frac{\partial^2 \mathcal{KR}_\Omega(P,y)}{\partial y_k\partial y_j}\Big)_{1\leq j,k\leq 2}\\
  \end{array}
\right),
\end{equation*}
where
\[A_{\e,j,k}(x)=
\frac{\La_1^2}\pi \frac\partial{\partial x_j}\left(\frac{(x_k-P_k)\ln |x-P|}{|x-P|^{2}}\right).
\]
By Proposition \ref{sec5-prop5.3} we know that the submatrix $\frac{\partial^2 \mathcal{KR}_\Omega(P,y)}{\partial y_k\partial y_j}$ is invertible in $y_i(P)$. Moreover
\[
\det
\Big(A_{\e,j,k}(\tilde x_\e^{(i)})\Big)_{1\leq j,k\leq 2} =
\ln|\tilde x_\e^{(i)}|\big(1-\ln |\tilde x_\e^{(i)}|\big)<0.
\]
This shows that
\[sign \left(\det J_\e(\tilde x_\e^{(i)}, y_i(P))\right)=-sign ~\det \left(\Big(\frac{\partial^2 \mathcal{KR}_\Omega(P,y_i(P))}{\partial y_k\partial y_j}\Big)_{1\leq j,k\leq 2}
\right).\]
Then \eqref{sec5-29} holds and \eqref{sec5-27} gives
\[deg \Big(\nabla\mathcal{KR}_{\O_\e}(x,y),B_\e^i,0\Big)\neq 0,
\]
which shows the existence of at least one critical point $(x_{i,\e},y_{i,\e})$ for $\mathcal{KR}_{\O_\e}(x,y)$ in $B_\e^i$. Let
$$\mathcal{E}'_\e:=\Big\{(x,y)\in \Omega_\e\times \Omega_\e, |x-P|=O\big( s_\e\big),~~|y-y_0|=O\big(s_\e\big)\Big\},$$
where $s_\e$ is the unique solution of \eqref{sec5-10}(See Proposition \ref{sec5-prop5.2}), then from \eqref{sec5-08} and \eqref{sec5-09},  we know that all the critical points of $\mathcal{KR}_{\O_\e}(x,y)$ satisfying $(x_\e,y_\e)\to (P,y_0)$ belong to $\mathcal{E}'_\e$.
Moreover for any $(x,y)\in \mathcal{E}'_\e$, we have following estimate
\begin{equation}\label{5-30}
\begin{cases}
\frac{\partial^2 \mathcal{KR}_{\O_\e}(x,y) }{\partial   x_i\partial   x_j} =
-\frac{\La^2_1}{\pi}
\left[\frac{\delta_{ij}} {|x-P|^{2}}-
\frac{2(x_i-P_i)(x_j-P_j)} {|x-P|^{4} } \right]+
\frac{\partial^2 \mathcal{KR}_{\O}(P,y_0) }{\partial   x_i\partial   x_j} +o\left(1 \right),\\[2mm]
\frac{\partial^2 \mathcal{KR}_{\O_\e}(x,y) }{\partial   x_i\partial   y_j} =
\frac{\partial^2 \mathcal{KR}_{\O}(P,y_0) }{\partial   x_i\partial y_j} +o\left(\frac{1}{|x-P|} \right),\\[2mm]
\frac{\partial^2 \mathcal{KR}_{\O_\e}(x,y) }{\partial   y_i\partial   y_j}=
\frac{\partial^2 \mathcal{KR}_{\O}(P,y_0) }{\partial   y_i\partial y_j} +o\left(1\right).
\end{cases}\end{equation}
Hence by the definition of $\bar L_\e(x,y)$, we deduce that
\begin{equation*}
\nabla^2\mathcal{KR}_{\O_\e}(x,y)=\nabla \bar L_\e(x,y)\Big(1+o\left(1\right)\Big),
\end{equation*}
which implies
\begin{equation*}
\begin{split}
&\det\big(\nabla^2\mathcal{KR}_{\O_\e}(x_{i,\e},y_{i,\e})\big)\\&=
\det
\Big(A_{\e,j,k}(x_{i,\e})\Big)_{1\leq j,k\leq 2} \det \left(\frac{\partial^2 \mathcal{KR}_\Omega(P,y_{i,\e})}{\partial y_k\partial y_j}
\right)_{1\leq j,k\leq 2}\Big(1+o\left(1\right)\Big)\\[1mm]&
=-
\frac {\La_1^4}{\pi^2|x_{i,\e}|^{4}} \det \left(\frac{\partial^2 \mathcal{KR}_\Omega (P,y_{i,\e})}{\partial y_k\partial y_j}
\right)_{1\leq j,k\leq 2} \left(1+o\big(1
\big) \right) \neq 0.
\end{split}
\end{equation*}
This gives  that when $d<d_1$ there exist two type II critical points $(x_{1,\e},y_{1,\e})$ and $(x_{2,\e},y_{2,\e})$ that verify \eqref{sec1-10}.
This proves that the critical point $(x_{i,\e},y_{i,\e})$ for $\mathcal{KR}_{\O_\e}(x,y)$ in $B_\e^i$ is nondegenerate and also  unique  in $B_\e^i$.  Since, all the type II critical points are contained in $B_\e^i$ by Proposition \ref{sec5-prop5.2} this gives also the exact multiplicity of the type II critical points. Then there exist exactly two type II critical points $(x_{1,\e},y_{1,\e})$ and $(x_{2,\e},y_{2,\e})$ that verify \eqref{sec1-10}.
\vskip 0.1cm
In the same manner when $d<d_2$ two type II critical points
$(x_{3,\e},y_{3,\e})$ and $(x_{4,\e},y_{4,\e})$ that verify \eqref{sec1-11} can be obtained. This proves $b)$ and $c)$.
Finally \eqref{sec1-12} follows by \eqref{sec5-15} since
\[
index \big(\nabla \mathcal{KR}_{\O_\e}, (x_{i,\e},y_{i,\e})\big)=
index \big(\bar L_\e, (\tilde x^{(i)}_{\e},y_i(P))\big)=-index \big(\nabla_y \mathcal{KR}_{\Omega}(P,\cdot),y_i(P)\big).
\]
\end{proof}

\begin{rem}\label{sec5-rem5.4}
Since for $\La_1=\La_2$, $\mathcal{KR}_{D}(x,y)=\mathcal{KR}_{D}(y,x)$ for any domain $D\subset \R^2$ then when $\La_1=\La_2$, we have that $d_1=d_2$.
\end{rem}

In a general convex domain $\O$, it seems very difficult to get a complete result as in Proposition \ref{sec5-prop5.3}.  Some properties will be deduced  in the next proposition.

\begin{prop}\label{sec5-prop5.5}
Assume $\O\subset\subset\R^2$ is a bounded convex domain and
  $P\in \O$. Then, denoting by $d=dist\{P,\partial \Omega\}$ we have that the equation
\begin{equation}\label{sec5-30}
\frac{\partial \mathcal{KR}_{\O}(P,y)}{\partial y_j}=0,~~\mbox{for}~~j=1,2
\end{equation}
admits exactly two solutions $y_1(P)$, $y_2(P)$ if $d=dist\{P,\partial \Omega\}$ is small enough. Moreover, they are all nondegenerate.
 Furthermore,
we have that  $$|y_1(P)- P|\to 0~~~\mbox{and}~~~y_2(P)\to Q~~~\mbox{as}~~~d\to 0,$$
where $Q$ is the unique critical point of $\mathcal R_\Omega(x)$.
Finally  we have that
\begin{equation}\label{sec5-31}
index~\big(\nabla_y \mathcal{KR}_{\O}(P,\cdot),y_1(P)\big)=-1 ~~~\mbox{and index}~~\big(\nabla_y  \mathcal{KR}_{\O}(P,\cdot),y_2(P)\big)=1.\end{equation}
\end{prop}

To prove Proposition \ref{sec5-prop5.5}, we need an
asymptotic expansion of $G_{\Omega}(y,P)$ and $\mathcal{R}_\O(y)$ as $d\to 0$ and $|y-P|\to 0$.
$G_\Omega(y,P)=\frac{1}{2\pi}\ln \frac{1}{|y-P|}-H_\Omega(y,P)$, then it holds
\begin{equation*}
\begin{cases}
\Delta_yH_\Omega(y,P)=0,&\mbox{in}~~\Omega,\\
H_\Omega(y,P) =\frac{1}{2\pi}\ln \frac{1}{|y-P|},&\mbox{on}~~\partial \Omega.
\end{cases}
\end{equation*}
Assume $P=(0,d)$, near $P$, $\partial \Omega$ is given by $y_2=a_1y^2_1+O(|y_1|^3)$ with $a_1>0$ since $\Omega$ is convex.
We define $f(y)=\La_2\mathcal{R}_\O(y)-2\La_1G_\O(P,y)$. Let
\begin{equation*}
\widetilde{f}_d(z):=f(d z)
=\La_2\mathcal{R}_\O(d z)-2\La_1G_\O(P,d z)~~~\mbox{with}~~~d:=dist\{P,\partial \Omega\},
\end{equation*}where $z\in \Omega_d:=\big\{z: dz\in \Omega\big\}$.
Then we have following result.
\begin{lem}\label{sec5-lem5.6}For any fixed large $R>0$, it holds
\begin{equation}\label{sec5-32} G_\O(P,d z)=\frac{1}{2\pi}\ln \frac{|z+e_2|}{|z-e_2|}+\frac{\ln d }{2\pi}+o(1),~~~\mbox{in}~~~\Omega_{d }\cap B(0,R).\end{equation}
\end{lem}
\begin{proof}
Denote $u_d(z):=H_\Omega(P,dz)$. Then $u_d$ is the solution of the following
problem
\begin{equation*}
\begin{cases}
\Delta u_d=0,~~~\mbox{in}~~~\Omega_{d}:=\big\{z, dz \in \Omega\big\}, \\[2mm]
u_d\big|_{\partial \Omega_{d}}=\frac{1}{2\pi}\ln \frac{1}{|dz-P|}=\frac{1}{2\pi}\Big[
\ln \frac{1}{|z-e_2|} -\ln d\Big],
\end{cases}
\end{equation*}
where $e_2=(0,1)$. We also have
$$\partial\Omega_{d}\cap B(0,R)=\big\{(z_1,z_2), z_2=\phi(z_1)=a_1dz_1^2+O(d^2 |z_1|^3),~z_1^2+z_2^2<R^2\big\}.$$
Let $u_1$ be the solution of
\begin{equation*}
\begin{cases}
\Delta u_1=0,~~~z_2>0, \\
u_1(z_1,0) =\frac{1}{2\pi}
\ln \frac{1}{|z-e_2|}.
\end{cases}
\end{equation*}
Then $u_1(z)=\frac{1}{2\pi}\ln \frac{1}{|z+e_2|}$.

\vskip 0.1cm

Let $\varphi_d(z):=u_d(z)+\frac{1}{2\pi}\ln d-u_1(z)$, then $\Delta \varphi_d(z)=0$ in $\Omega_{d}$. And as $d\to 0$, $\varphi_d\to \varphi$
in $C^2_{loc}(\mathbb R^2_+)$. It is easy to see that $\varphi$ is harmonic, and satisfies
$
\varphi(z_1, 0)=0.$ This gives $\varphi=0$. Then \eqref{sec5-32} holds.

\end{proof}
\begin{lem}\label{sec5-lem5.7}
 For any $x\in \Omega$ with $|x-P|\le Cd$ for some constant $C>0$, it holds
\begin{equation}\label{sec5-33}
\mathcal{R}_\Omega(d x)= -\frac{1}{2\pi}\ln (2x_2)
-\frac{1}{2\pi}\ln d+o(1).
\end{equation}
\end{lem}
\begin{proof}
Let
$v_d(x,z):=H_\Omega(x,d z)$. Then
\begin{equation*}
\begin{cases}
\Delta_z v_d(x,z)=0,~~~\mbox{in}~~~\Omega_{d}:=\big\{z, dz \in \Omega\big\}, \\[2mm]
v_d(x,z)\big|_{z\in \partial \Omega_{d}}=\frac{1}{2\pi}\ln \frac{1}{|dz-x|}=\frac{1}{2\pi}\Big[
\ln \frac{1}{|z-\frac{x}{d}|} -\ln d\Big].
\end{cases}
\end{equation*}
Let $\psi_d(x,z):=v_d(x,z)+\frac{1}{2\pi}\ln d -\frac{1}{2\pi}
\ln \frac{1}{|z+\frac{(-x_1,x_2)}{d }|}$.  Then $ \psi_d(x,z)$ is
harmonic in $\Omega_{d}$. By our assumption, we have $\frac{|x|}{d}\le C$, and then
\begin{equation*}
\begin{split}
\psi_d(x,z)\big|_{z\in \partial \Omega_{d}\cap B(0,R)}=&\frac{1}{4\pi}\ln \frac{|z+\frac{(-x_1,x_2)}{d}|^2}{|z-\frac{x}{d}|^2}\Big|_{z\in \partial \Omega_{d}\cap B(0,R)} =o(1).
\end{split}
\end{equation*}That is
$H_\Omega(x,dz)= \frac{1}{2\pi}
\ln \frac{1}{|z+\frac{(-x_1,x_2)}{d}|}-\frac{1}{2\pi}\ln d+o\big(1\big)$.
Then putting $z=d^{-1}x$, we have \eqref{sec5-33}.
\end{proof}

\begin{rem}\label{sec5-rem5.8}
Using the estimates for the harmonic functions, we can deduce the following estimates:
\begin{equation*}
\begin{cases}\nabla_z\big[G_\O(P,dz)\big] =\frac{1}{2\pi}\nabla_z\big[\ln \frac{|z+e_2|}{|z-e_2|}\big] + o(\frac{1}{d_z}),~~~\mbox{in}~~~\Omega_{d}\cap B(0,R),\\[3mm]
\nabla^2_z\big[G_\O(P,dz)\big] =\frac{1}{2\pi}\nabla_z^2 \big[\ln \frac{|z+e_2|}{|z-e_2|}\big] + o(\frac{1}{d_z^2}),~~~\mbox{in}~~~\Omega_{d}
\cap B(0,R),
\end{cases}\end{equation*}and
\begin{equation*}
\begin{cases}\nabla_z\big[\mathcal{R}_\Omega(dz)\big] =-\frac{1}{2\pi}\nabla_z\left[\ln z_2\right] + o(\frac{1}{d_z^2}),~~~\mbox{in}~~~\Omega_{d}\cap B(0,R),\\[3mm]
\nabla^2_z\big[\mathcal{R}_\Omega(dz)\big] =-\frac{1}{2\pi}\nabla_z^2 \left[\ln z_2\right] + o(\frac{1}{d_z^2}),~~~\mbox{in}~~~\Omega_{d}\cap B(0,R),
\end{cases}\end{equation*}where $d_z:=dist\big\{z,\partial \Omega_d\big\}$,
\end{rem}

Now we give the expansion of $\widetilde{f}_d(z)$.

\begin{lem}\label{sec5-lem5.9}
It holds
\begin{equation*}
\widetilde{f}_d(z)=-\frac{\La_2}{2\pi}\ln (2z_2)
-\frac{\La_1}{\pi}\ln \frac{|z+e_2|}{|z-e_2|}-\frac{(\La_2+2\La_1)}{2\pi}\ln d+o(1), ~~\mbox{in}~~~\Omega_{d}\cap B(0,R),
\end{equation*}
\begin{equation*}
\nabla_z \widetilde{f}_d(z)=-\frac{1}{2\pi}\left[ \nabla\Big( \La_2 \ln z_2
+2 \La_1 \ln \frac{|z+e_2|}{|z-e_2|}\Big)\right]+o\big(\frac{1}{d_z}\big), ~~\mbox{in}~~~\Omega_{d}\cap B(0,R),
\end{equation*}
\begin{equation*}
\nabla^2_z \widetilde{f}_d(z)=-\frac{1}{2\pi}\left[ \nabla^2\Big( \La_2 \ln z_2
+2 \La_1 \ln \frac{|z+e_2|}{|z-e_2|}\Big)\right]+o\big(\frac{1}{d_z^2}\big), ~~\mbox{in}~~~\Omega_{d}\cap B(0,R).
\end{equation*} \end{lem}\begin{proof}
These estimates follow from Lemma \ref{sec5-lem5.6}, Lemma \ref{sec5-lem5.7} and Remark \ref{sec5-rem5.8}.
\end{proof}

Let $F(z):=-\frac{1}{2\pi}\Big( \La_2 \ln z_2
+2 \La_1 \ln \frac{|z+e_2|}{|z-e_2|}\Big)$. Then we have following result.
\begin{lem}
 $F(z)$ has a unique critical point $z_0=(0, \alpha)$, with $\alpha =\frac{2\La_1+\sqrt{4\La^2_1+\La_2^2}}{\La_2}$. Furthermore, $z_0$ is  nondegenerate.
\end{lem}

\begin{proof}
First, we have
\begin{small}\[
\frac{\partial F(z)}{\partial z_1}=\frac{4\La_1z_1z_2}{\pi((z_1^2+z_2^2+1)^2-4z_2^2)}~~~
\mbox{and}~~~
\frac{\partial F(z)}{\partial z_2}=-\frac{1}{2\pi}\left[ \frac{\La_2}{z_2} +4\La_1 \frac{z_1^2+1-z_2^2}{(z_1^2+z_2^2+1)^2-4z_2^2}\right].
\]
\end{small}Hence $F(z)$ has a unique critical point $z_0=(0, \alpha)$, with $\alpha =\frac{2\La_1+\sqrt{4\La^2_1+\La_2^2}}{\La_2}$.

\vskip 0.05cm

Furthermore,
\begin{small}\[
\frac{\partial^2F(z)}{\partial z_1^2}\Big|_{z=(0, \alpha)}=
\frac{4\La_1z_2}{\pi((z_2^2+1)^2-4z_2^2)}\Big|_{z_2=\alpha}\neq 0,\,\,\,~~~~~~
\frac{\partial^2F(z)}{\partial z_1\partial z_2}\Big|_{z=(0, \alpha)}=
 0,
\]
\end{small}and
\begin{small}\[
\frac{\partial^2F(z)}{\partial z_2^2}\Big|_{z=(0, \alpha)}=
-\frac{1}{2\pi}\left[ -\frac{\La_2}{z_2^2} +\frac{8\La_1 z_2}{(1-z_2^2)^2} \right]
\Big|_{z_2=\alpha}< 0.
\]
\end{small}Thus $z_0$ is the nondegenerate critical point of $F(z)$.

\end{proof}

Now, we prove the following result.

\begin{lem}\label{sec5-lem5.11}
The function $
\widetilde{f}_d(z)$
has a unique critical point $z_d=\big(o(1), \alpha+o(1)\big)$
in $B(z_0,\delta)$. Furthermore, $z_d$ is nondegenerate.
\end{lem}
\begin{proof}

 From Lemma~\ref{sec5-lem5.9}, we have
\[
 \begin{split}
\nabla_z \widetilde{f}_d(z)=\nabla_z  F(z)+o(1)~~~\mbox{and}~~~
\nabla^2_z \widetilde{f}_d(z)=\nabla^2_z  F(z)+o(1)
 ~~\mbox{in}~~~B(z_0,\delta).
\end{split}\]
This gives that $\widetilde{f}_d$ has a unique critical point in $B(z_0,\delta)$, which is also  nondegenerate.

\end{proof}

For any $y\in\Omega$, we denote $d_y= dist\{y,\partial\Omega\}$, then the following result holds.

\begin{lem}\label{sec5-lem5.12}
Suppose that $y_P=(y_{1,P},y_{2,P})$ is a critical point of $\mathcal{KR}_\Omega(P,y)$ satisfying $|y_P-P|\to 0$ as $d\to 0$. Then
as $d\to 0$,
\[
d^{-1} y_P\to z_0.
\]
In particular, the critical point $y_P$ of $\mathcal{KR}_\Omega(P,y)$ satisfying $|y_P-P|\to 0$ as $d\to 0$ is unique.
\end{lem}
\begin{proof}
After translation and rotation, we assume that $y_P=(0, d_{y_P})$, and
\begin{small}\[
\partial\Omega\cap B(0,\delta)=\big\{ (y_1,y_2):  y_2=a_1 y_1^2 +O(|y_1|^3),~y_1^2+y_2^2<\delta^2\big\}.
\]
\end{small}Let
$w_P(z):=H_\Omega(P, d_{y_P}z)$. Then
\begin{small}\begin{equation*}
\begin{cases}
\Delta w_P=0,~~~\mbox{in}~~~\Omega_{d_{y_P}}:=\big\{z, d_{y_P}z \in \Omega\big\}, \\[2mm]
u\big|_{\partial \Omega_{d_{y_P}}}=\frac{1}{2\pi}\ln \frac{1}{|d_{y_P}z-P|}.
\end{cases}
\end{equation*}\end{small}We claim that $\frac{|P|}{d_{y_P}}\to +\infty$ is impossible.
Suppose that $\frac{|P|}{d_{y_P}}\to +\infty$. Then for any $R>0$,
\begin{small}\[
\ln \frac{1}{|d_{y_P}z-P|}=\ln \frac{1}{|\frac{d_{y_P}}{|P|}z-\frac{P}{|P|}|} - \ln |P|=- \ln |P|+o(1),\quad z\in B(0,R).
\]
\end{small}This gives that
\begin{small}\[
G_\Omega(P, d_{y_P}z)=\frac1{2\pi}\ln \frac{1}{|d_{y_P}z-P|}+\frac1{2\pi}\ln|P| +o(1)=o(1),\quad \text{in}\; C^1_{loc}(\mathbb R^2_+).
\]
\end{small}Hence from $\nabla \mathcal{KR}_\Omega(P,d_{y_P} y_P)=0$, we obtain
\begin{small}\[
 \nabla \mathcal{R}_\Omega(d_{y_P} y_P)=o(1).
 \]
\end{small}This is a contradiction.

\vskip 0.05cm

Now we assume that $ d_{y_P}^{-1} P \to P_1$.
 Then it holds
\begin{equation*}
w_P(z)-\ln d_{y_P}  \to w_0(z)~~\mbox{in}~~C^2_{loc}(\R^2_+),
\end{equation*}
with
\begin{equation*}
\begin{cases}
\Delta w_0(z)=0 ~~\mbox{in}~~\R^2_+,\\[1mm]
w_0(z_1,0)=\frac1{2\pi}\ln\frac1{|z-P_1|}.
\end{cases}\end{equation*}
Hence $w_0(z)=\frac1{2\pi}\ln\frac1{|z-\bar P_1|}$, where $\bar P_1$
is the reflection point of $P_1$ with respect to $z_2=0$. So we have
\begin{small}\[
G_{\Omega}(P, d_{y_P} z)= \frac1{2\pi}\ln \frac{|z-\bar P_1|   }{|z- P_1|  }+o(1).
\]\end{small}From $\nabla\mathcal{KR}_\Omega(P,d_{y_P}y_P)= 0$,
we find
\begin{small}\begin{equation}\label{sec5-34}
\frac{P_{11}  }{| P_1-(0,1)|^2  }-\frac{P_{11}  }{| P_1-(0,-1)|^2  }=0,
\end{equation}
\end{small}and
\begin{small}\begin{equation}\label{sec5-35}
\Lambda_2+2\Lambda_1\Bigl(\frac{1+P_{12}  }{| (0,1)+P_1|^2  }-\frac{1-P_{12}  }{| (0,1)-P_1|^2  }\Bigl)=0,
\end{equation}
\end{small}where we denote $P_1=(P_{11}, P_{12})$.

\vskip 0.05cm

From \eqref{sec5-35}, we find that $P_{12}\ne 0$. Then \eqref{sec5-34}
gives that $P_{11}=0$, while \eqref{sec5-35} gives
\begin{small}\[
P_{12} =\frac{-2\La_1+\sqrt{4\La^2_1+\La_2^2}}{\La_2}=\frac1\alpha=
\frac1{d_{z_0}}.
\]
\end{small}Thus $z_0= \alpha P_1$.
\end{proof}

\begin{proof}[{\bf Proof of  Proposition \ref{sec5-prop5.5}}]
For $y\in B(Q,\delta)$, where $\delta>0$ is small,  let us consider the function
\begin{small}\begin{equation*}
 f(y)=\La_2\mathcal{R}_\O(y)-2\La_1G_\O(P,y).
\end{equation*}
\end{small}We have that the critical points of $f$ provide solutions to \eqref{sec5-30}. Since $\O$ is convex, $\mathcal{R}_\O$ admits exactly one
 critical point  $Q$, which is a nondegenerate minimum point.

\vskip 0.05cm

We observe that for any $y\in B(Q,\delta)$,
$G_\O(P,y)\to 0$ in $C^2(B(Q,\delta))$ as $d\to 0$.  Hence $f(y)$ is a $C^2$ perturbation of $\La_2\mathcal{R}_\O(y)$ for $y\in B(Q,\delta)$ and $d$ small. So
by the implicit function theorem, $f$  has a unique critical  point $y_2(P)$ in $B(Q,\delta)$ such that $y_2(P)$ converges $Q$
 as $d\to 0$. Moreover, $y_2(P)$ is a local minimum point of $f$ and this
 gives $\det(\nabla^2 f(y_2(P))>0.$

\vskip 0.05cm
Next, Lemma \ref{sec5-lem5.11} and Lemma \ref{sec5-lem5.12} show that  that
 $f(y)$ has a unique critical point $y_P$, satisfying
$|y_P-P|\to 0$ as $dist\{P,\partial\Omega\}\to 0$.
This critical point is also nondegenerate.

\vskip 0.05cm

Now we prove that the critical point $y_P$ of $f(y)$ satisfies
that  as $d\to 0$,
 either $|y_P- P|\to 0$, or $y_P\to Q$.
Indeed, suppose that $y_P\to \widetilde y$  and $|P-y_P|\ge \delta>0$.
 From $\La_2\nabla \mathcal{R}_\O\big( y_P\big)=2\La_1\nabla_y G_\O\big(P,y_P\big)$, while $\nabla_y G_\O\big(P,y_1(P)\big)\to0$ as $P$ approaches the boundary (i.e. $d\to 0$).  This implies that $\nabla \mathcal{R}_\O(\widetilde y)=0$ and thus $y_P\to Q$.

\vskip 0.05cm
 In conclusion, $f(y)$ has exactly two critical points, which are
 nondegenerate. And finally, \eqref{sec5-31} holds by above discussions.
\end{proof}

\begin{proof}[\bf{Proof of Theorem \ref{sec1-teo08}}] Since
$\frac{\partial \mathcal{KR}_{\O}(P,y)}{\partial y_j}=0$ ($j=1,2$) has exactly two zero points,
which are nondegenerate, the proof of the existence part
is the same as that in Theorem~\ref{SEC1-TEO06}.
\vskip 0.1cm

Considering the function $g(x)=\La_1\mathcal R_\O(x)-2\La _2G_\O(x,P)$ as in Proposition \ref{sec5-prop5.5},  we get, when $d$ is small, the existence of $x_4(P)\to Q$  that satisfies \eqref{sec1-09}. This gives the existence of the other critical point that verifies \eqref{sec1-09}.

\vskip 0.1cm

Now we turn to the proof of the non-existence part.
Let us show that $\nabla_y\mathcal{KR}_{\O}(P,y_0)=0$ is not verified, if $\O$ is convex  and $|P-Q|$ is small, where $Q$ is the unique critical point of $\mathcal{R}_\O(x)$. We again use the argument in \cite{GrossiTakahashi}.
We apply formula  \eqref{sec3-04} in Lemma \ref{sec3-lem3.1}
with $a_0=a=P$ and $b=y_0$ and then
\begin{equation}\label{sec5-36}
\int_{\partial\O}x\cdot\nu(x)\frac{\partial G_\O(x,P)}{\partial\nu_x}\frac{\partial G_\O(x,y_0)}{\partial\nu_x}ds_x=-(y_0-P)\cdot\nabla_yG_\O(P,y_0).
\end{equation}
Assume by contradiction that
$\frac{\partial \mathcal{KR}_{\O}(P,y_0)}{\partial y_j}=0$. This implies
\begin{equation}\label{sec5-37}
\frac{\partial G_\O(P,y_0)}{\partial y_j}= \frac{\La_2}{2\La_1}\frac{\partial \mathcal{R}_\O(y_0)}{\partial y_j}\Rightarrow
-(y_0-P)\cdot\nabla_yG_\O(P,y_0)=-\frac{\La_2}{2\La_1} (y_0-P)  \cdot\nabla \mathcal{R}_\O(y_0).
\end{equation}
In \cite{ct} it was proved that if $\O$ is convex then the Robin function is strictly convex. In particular,  its level set are strictly star-shaped with respect to $Q$. Hence for any $x\in\O$, it holds
$$\nabla \mathcal{R}_\O(x)\cdot (x-Q)>0\Rightarrow-(y_0-P)\cdot\nabla \mathcal{R}_\O(y_0)\leq C_0|P-Q|,$$
where $C_0>0$ is independent of the point $Q$.
Using \eqref{sec5-37} we get that \eqref{sec5-36} becomes
\begin{equation}\label{sec5-38}
\int_{\partial\O}x\cdot\nu(x)\frac{\partial G_\O(x,P)}{\partial\nu_x}\frac{\partial G_\O(x,y_0)}{\partial\nu_x}ds_x\leq \frac{\La_2C_0}{2\La_1}|P-Q|.
\end{equation}
 On the other hand, $\frac{\partial G_\O(x,P)}{\partial\nu_x}<0$, $\frac{\partial G_\O(x,y_0)}{\partial\nu_x}<0$. Also by the convexity of $\O$, $x\cdot\nu(x)\ge0$, and then there exists
a nonzero measure set $A$ such that $x\cdot\nu(x)>0$ on $A$. Hence we deduce that there exists a constant $C_1>0$, which is  independent of the point $P$, such that
\begin{equation}\label{sec5-39}
\int_{\partial\O}x\cdot\nu(x)\frac{\partial G_\O(x,P)}{\partial\nu_x}\frac{\partial G_\O(x,y_0)}{\partial\nu_x}ds_x\geq C_1.
\end{equation}
So we have a contradiction by \eqref{sec5-38} and \eqref{sec5-39} when $|P-Q|$ is small.
 This ends the proof.
\end{proof}

Now we turn to the existence of type II critical points such that $\nabla \mathcal{KR}_{\O}(P,y_0)=0$.

\vskip 0.2cm

 \begin{proof}[\bf{Proof of Theorem \ref{SEC1-TEO09}}]
First, we observe that an asymptotic expansion of a type II critical point $(x_\e,y_\e)$ is proved in Proposition \ref{sec5-prop5.2}, see \eqref{sec5-08} and \eqref{sec5-11}.
Let $\eta^{(i)}$ be a unit eigenvector of the matrix $\textbf{M}_0$ related to the positive simple eigenvalue $\lambda_i$.
Let us define, for $i=1,2$,
\[\tilde x_\e^{(i),\pm}=P\pm \eta^{(i)}r_{\e,i} \ , \ \tilde y_\e^{(i),\pm}=y_0- \left(\Big(\frac{\partial^2 \mathcal{KR}_\Omega(P,y_0)}{\partial y_k\partial y_j} \Big)_{1\leq k,j\leq 2} \right)^{-1}
   \left(\frac{\partial^2 \mathcal{KR}_\Omega(P,y_0)}{\partial y_k\partial x_j}   \right)_{1\leq k,j\leq 2} \Big(\tilde x_\e^{(i),\pm}-P\Big),\]
where  $r_{\e,i}$ is the unique solution to $\frac{\ln r}{r^2\ln \e}=\frac{\lambda_i \pi}{\Lambda_1^2}$
and the vector field
\begin{small}\begin{equation*}
\begin{split}
L_\e(x,y):=\left(
  \begin{array}{cc}
   \Big(\frac{\partial^2 \mathcal{KR}_\Omega(P,y_0)}{\partial x_k\partial x_j}\Big)_{1\leq k,j\leq 2}  & \Big(\frac{\partial^2 \mathcal{KR}_\Omega(P,y_0)}{\partial x_k\partial y_j}\Big)_{1\leq k,j\leq 2}  \\[4mm]
\Big( \frac{\partial^2 \mathcal{KR}_\Omega(P,y_0)}{\partial y_k\partial x_j}\Big)_{1\leq k,j\leq 2}    & \Big(\frac{\partial^2 \mathcal{KR}_\Omega(P,y_0)}{\partial y_k\partial y_j}\Big)_{1\leq k,j\leq 2}  \\
  \end{array}
\right) \left(
          \begin{array}{c}
            x-P \\[4mm]
            y-y_0 \\
          \end{array}
          \right)-
    \left(
          \begin{array}{c}
           \frac{\La_1^2  \ln |x-P|   }{ \pi |x-P|^2\ln \e}
(x-P)\\[4mm]
            0 \\
          \end{array}
        \right).
\end{split}
\end{equation*}
\end{small}Now using the homotopy invariance of the degree, it can be proved that
\begin{small}\begin{equation}\label{sec5-40}
deg \Big(\nabla\mathcal{KR}_{\O_\e}(x,y),B_\e^{i,\pm},0\Big)=deg \Big(L_\e(x,y)
,B_\e^{i,\pm},0\Big),
\end{equation}
\end{small}where $B_\e^{i,\pm}$ is as in \eqref{sec5-26} with $(\delta_\e^i)^3<<\frac 1{|\ln \e|}$.  Using that for every $(x,y)\in B_\e^{i,\pm}$, it holds
\begin{small}\begin{equation*}\begin{split}
\nabla\mathcal{KR}_{\O_\e}(x,y)=&\left(
  \begin{array}{cc}
   \Big(\frac{\partial^2 \mathcal{KR}_\Omega(P,y_0)}{\partial x_k\partial x_j}\Big)_{1\leq k,j\leq 2}  & \Big(\frac{\partial^2 \mathcal{KR}_\Omega(P,y_0)}{\partial x_k\partial y_j}\Big)_{1\leq k,j\leq 2}  \\[4mm]
\Big( \frac{\partial^2 \mathcal{KR}_\Omega(P,y_0)}{\partial y_k\partial x_j}\Big)_{1\leq k,j\leq 2}    & \Big(\frac{\partial^2 \mathcal{KR}_\Omega(P,y_0)}{\partial y_k\partial y_j}\Big)_{1\leq k,j\leq 2}  \\
  \end{array}
\right)\left(
          \begin{array}{c}
            x-P \\[4mm]
            y-y_0 \\
          \end{array}
        \right)-
    \left(
          \begin{array}{c}
           \frac{\La_1^2  \ln |x-P|   }{ \pi |x-P|^2\ln \e}
(x-P) \\[4mm]
            0 \\
          \end{array}
        \right)
        \\[1mm]&+  O\left(\delta^2+  \left|\frac{\ln |x-P|}{ |x-P| \ln \e } \right|
        \right).
\end{split}
\end{equation*}
\end{small}Finally, let us compute $deg \Big(L_\e(x,y)
,B_\e^{i,\pm},0\Big)$.  Observing that
\begin{small}\begin{equation*}
\begin{split}
\left.\frac{\partial}{\partial x_j}\left(  \frac{\La_1^2  \ln |x-P|   }{ \pi |x-P|^2\ln \e}(x_k-P_k)\right)\right|_{x=\tilde x_\e^{(i),\pm}}= &\frac{\La_1^2}{ \pi\ln \e}\left(\delta_{jk}\frac{\ln|\tilde x_\e^{(i),\pm}-P|}{|\tilde x_\e^{(i),\pm}-P|^2}+\eta_k^{(i)}\eta_j^{(i)}
\frac{1-2\ln|\tilde x_\e^{(i),\pm}-P|}{|\tilde x_\e^{(i),\pm}-P|^2}\right)\\=&
\frac{\la_i}{\ln|\tilde x_\e^{(i),\pm}-P|}\left(\delta_{jk}\ln|\tilde x_\e^{(i),\pm}-P|+\eta_k^{(i)}\eta_j^{(i)}\big(1-2\ln|\tilde x_\e^{(i),\pm}-P|\big)
\right),
\end{split}
\end{equation*}
\end{small}we have
\begin{small}\begin{equation*}
\begin{split}
&Jac\big(L_\e(\tilde x_\e^{(i),\pm}, \tilde y_\e^{(i),\pm})\big)\\=&
\begin{pmatrix}
 \left(\frac{\partial^2 \mathcal{KR}_\Omega(P,y_0)}{\partial x_k\partial x_j} -\frac{\la_i}{\ln|\tilde x_\e^{(i),\pm}-P|}\Big(\delta_{jk}\ln|\tilde x_\e^{(i),\pm}-P|+\eta_k^{(i)}\eta_j^{(i)}\big(1-2\ln|\tilde x_\e^{(i),\pm}-P|\big)
\Big)\right)_{1\leq k,j\leq 2}
& \left(\frac{\partial^2 \mathcal{KR}_\Omega(P,y_0)}{\partial x_k\partial y_j} \right)_{1\leq k,j\leq 2} \\[4mm]
   \left(\frac{\partial^2 \mathcal{KR}_\Omega(P,y_0)}{\partial y_k\partial x_j}\right)_{1\leq k,j\leq 2}  & \left(\frac{\partial^2 \mathcal{KR}_\Omega(P,y_0)}{\partial y_k\partial y_j}\right)_{1\leq k,j\leq 2}
   \end{pmatrix}.
\end{split}
\end{equation*}
\end{small}And we know
\begin{small}\begin{equation*}
\begin{split}
&\det\big(Jac(L_\e(\tilde x_\e^{(i),\pm}, \tilde y_\e^{(i),\pm})\big)\\=&\underbrace{\det\left(\frac{\partial^2 \mathcal{KR}_\Omega(P,y_0)}{\partial y_k\partial y_j} \right)_{1\leq k,j\leq 2} }_{\ne0}
\det\left(\textbf{M}_0-\frac{\la_i}{\ln|\tilde x_\e^{(i),\pm}-P|}\left(\delta_{jk}\ln|\tilde x_\e^{(i),\pm}-P|+\eta_k^{(i)}
\eta_j^{(i)}\big(1-2\ln|\tilde x_\e^{(i),\pm}-P|\big)\right)_{1\leq k,j\leq 2} \right)\\=
&\lambda_i\frac{2\ln|\tilde x_\e^{(i),\pm}-P|-1}{\ln|\tilde x_\e^{(i),\pm}-P|} \left(\lambda_l-\lambda_i\right)\ne0,~~~\mbox{with $l\in \{1,2\}$ and $l\neq i$},
\end{split}
\end{equation*}\end{small}because $\lambda_i>0$ and $\lambda_l\neq \lambda_i$  by assumptions.
This shows that $deg \Big(L_\e(x,y),B_\e^{i,\pm},0\Big)\neq 0$ and by \eqref{sec5-40} there exists at least one critical point $(x_\e^{(i),\pm},  y_\e^{(i),\pm})$ for $\mathcal{KR}_{\O_\e}(x,y)$ in $B_\e^{i,\pm}$.
\vskip 0.1cm

 Now we get in the same way the nondegeneracy of the critical points $(x_\e^{(i),\pm},  y_\e^{(i),\pm})$ in the balls $B_\e^{i,\pm}$. In fact,
letting
\begin{small}$$\mathcal{D}''_\e:=\Big\{(x,y)\in \Omega_\e\times \Omega_\e, |x-P|=O\big(r_\e\big),~~|y-y_0|=O\big(r_\e\big)\Big\},$$\end{small}where $r_\e=\max\{r_{\e,i}\}\to 0$ as $\e\to 0$.
From \eqref{sec5-08} and \eqref{sec5-11}, we know that all the critical points of $\mathcal{KR}_{\O_\e}(x,y)$ satisfying $(x_\e,y_\e)\to (P,y_0)$ belong to $\mathcal{D}''_\e$.
Hence by \eqref{sec3-03} and \eqref{sec3-12}, for any $(x,y)\in \mathcal{D}''_\e$, we have following estimate
\begin{small}\begin{equation*}
\begin{cases}
\frac{\partial^2 \mathcal{KR}_{\O_\e}(x,y) }{\partial   x_i\partial   x_j} =
-\frac{ \La^2_1 \ln |x-P| }{\pi \ln \e}
\left[\frac{\delta_{ij}} {|x-P|^{2}}-
\frac{2(x_i-P_i)(x_j-P_j)} {|x-P|^{4} } \right]+
\frac{\partial^2 \mathcal{KR}_{\O}(P,y_0) }{\partial   x_i\partial   x_j} +o(1),\\[2mm]
\frac{\partial^2 \mathcal{KR}_{\O_\e}(x,y) }{\partial   x_i\partial   y_j} =
\frac{\partial^2 \mathcal{KR}_{\O}(P,y_0) }{\partial   x_i\partial y_j} +o(1),\\[2mm]
\frac{\partial^2 \mathcal{KR}_{\O_\e}(x,y) }{\partial   y_i\partial   y_j}=
\frac{\partial^2 \mathcal{KR}_{\O}(P,y_0) }{\partial   y_i\partial y_j} +o(1).
\end{cases}\end{equation*}
\end{small}Hence by the definition of $L_\e(x,y)$, we deduce that
$\nabla^2\mathcal{KR}_{\O_\e}(x,y)=\nabla L_\e(x,y)\Big(1+o(1)\Big)$,
which implies
\begin{equation*}
\det\big(\nabla^2\mathcal{KR}_{\O_\e}(x_\e^{(i),\pm},  y_\e^{(i),\pm})\big)=2\lambda_i \left(\lambda_l-\lambda_i\right)+o(1)\neq 0,
\end{equation*}
with $l\in \{1,2\}$ and $l\neq i$ for $\e$ small enough, we get the nondegeneracy of the critical point $(x_\e^{(i),\pm},  y_\e^{(i),\pm})$ of $\mathcal{KR}_{\O_\e}$. This gives the uniqueness in the balls $B_\e^{i,\pm}$.
 \vskip 0.1cm

Moreover we have that
\begin{small}\begin{equation*}
\begin{split}
index \big(\nabla \mathcal{KR}_{\O_\e}, (x_\e^{(i),\pm},  y_\e^{(i),\pm})\big)=&
index \big(L_\e, (\tilde x_\e^{(i),\pm}, \tilde y_\e^{(i),\pm})\big)\\=& sign \left[\det\left(\frac{\partial^2 \mathcal{KR}_\Omega(P,y_0)}{\partial y_k\partial y_j} \right)_{1\leq k,j\leq 2}  \left(\lambda_l-\lambda_i\right)\right],
\end{split}
\end{equation*}\end{small}with $l\in \{1,2\}$ and $l\neq i$.
Hence we have the existence of exactly four critical points, which are nondegenerate.
\end{proof}
\vskip 0.2cm

\section{The existence of critical points of Type III}\label{sec-6}

We now discuss critical points of Type III. For the simplicity of the notations,  we assume that $P=0$. From now on, we assume that
both $x$ and $y$ are close to 0.

 \subsection{The location of critical points}\label{sec-6-1}~~
\vskip 0.2cm

In this subsection, we prove Theorem \ref{sec1-teo12}(1). Using that $P=0$ and $\frac{\partial \mathcal{KR}_{\Omega}(x,y)}{\partial x_j}+2\La_1\La_2\frac{\partial S(x,y)}{\partial x_j}=O(1)$, we rewrite \eqref{sec3-11} as
\begin{small}\begin{equation}\label{sec6-01}
\begin{cases}
 \frac{\partial \mathcal{KR}_{\Omega_\e}(x,y)}{\partial x_j}=\frac{\La_1}\pi\left[-\frac{ \La_1 x_{j}}{  |x|^2-\e^2 }
-\frac{\La_2(|y|^2x_{j}-\e^2y_{j})}{ |x|^2|y|^2-2 \e^2x\cdot y+\e^4 }+
\frac{\La_2(x_{j}-y_{j})}{|x-y|^{2}}-\frac{ x_j}{|x|^2}\frac{
 \La_1\ln \frac{|x|}{\e}+ \La_2 \ln \frac{|y|}{\e} }{\ln \e+2\pi \mathcal{R}_{\Omega}(0)}\right]
+O\left(\frac{1}{|x|\cdot|\ln \e|}+ 1\right),\\[3mm]
 \frac{\partial \mathcal{KR}_{\Omega_\e}(x,y)}{\partial y_j}=\frac{\La_2}\pi\left[
-\frac{\La_2y_{j}}{  |y|^2-\e^2 }
-\frac{\La_1(|x|^2y_{j}-\e^2x_{j})}{ |x|^2|y|^2-2 \e^2x\cdot y+\e^4 }+
\frac{\La_1(y_{j}-x_{j})}{|x-y|^{2}}-\frac{y_j}{|y|^2}\frac{
 \La_1\ln \frac{|x|}{\e}+\La_2 \ln \frac{|y|}{\e} }{\ln \e+2\pi \mathcal{R}_{\Omega}(0)}\right]
+O\left(\frac{1}{|y|\cdot|\ln \e|}+ 1\right).
\end{cases}
\end{equation}
\end{small}

\begin{proof}[\bf{Proof of Theorem \ref{sec1-teo12}(1)}]\
 We  divide the proof into several steps.
\vskip 0.2cm

\noindent\textbf{Step 1.} It holds
\begin{small}\begin{equation*}
\frac{|x_\e|}\e\to \infty~~~~~~~\mbox{and}~~~~~~~
\frac{1}{C}\leq \frac{|x_\e|}{|y_\e|}\leq C,~~\mbox{for some positive constant}~~C.
\end{equation*}
\end{small}

First, from $\nabla\mathcal{KR}_{\Omega_\e}(x_\e,y_\e)=0$ and \eqref{sec6-01}, we have
\begin{small}\begin{equation}\label{sec6-02}
\begin{cases}
\frac{\La_1 x_{\e,j}}{  |x_\e|^2-\e^2 }
+\frac{\La_2(|y_\e|^2x_{\e,j}-\e^2y_{\e,j})}{ |x_\e|^2|y_\e|^2-2 \e^2x_\e\cdot y_\e+\e^4 }
- \frac{\La_2(x_{\e,j}-y_{\e,j})}{|x_\e-y_\e|^{2}}
+\frac{x_{\e,j} (\La_1\ln \frac{|x_{\e}|}{\e}+ \La_2 \ln \frac{|y_{\e}|}{\e})  }{  |x_{\e}|^2\ln \e}
=O\left(\frac{1}{|x_{\e}|\cdot|\ln \e|} +1 \right),\\[3mm]
\frac{\La_2 y_{\e,j}}{  |y_\e|^2-\e^2 }
+\frac{\La_1(|x_\e|^2y_{\e,j}-\e^2x_{\e,j})}{ |x_\e|^2|y_\e|^2-2 \e^2x_\e\cdot y_\e+\e^4 }
+\frac{\La_1(x_{\e,j}- y_{\e,j})}{|x_\e-y_\e|^2}
+\frac{y_{\e,j} (\La_1\ln \frac{|x_{\e}|}{\e}+ \La_2\ln \frac{|y_{\e}|}{\e})  }{   |y_\e|^2\ln \e}
=O\left(\frac{1}{|y_\e|\cdot|\ln \e|} +1\right).
\end{cases}
\end{equation}
\end{small}Then $\sum^2_{j=1}\big(x_{\e,j}\times \mbox{the first identity of \eqref{sec6-02}}\big)$, we get \begin{small}\begin{equation}\label{sec6-03}
\frac{\La_1|x_\e|^2}{  |x_\e|^2-\e^2 }
+\frac{\La_2(|y_\e|^2|x_\e|^2-\e^2 x_\e\cdot y_\e)}{ |x_\e|^2|y_\e|^2-2 \e^2x_\e\cdot y_\e+\e^4 }
- \frac{\La_2 x_\e\cdot(x_{\e}-y_{\e}) }{|x_\e-y_\e|^{2}}
+\frac{\La_1\ln \frac{|x_{\e}|}{\e}+ \La_2\ln \frac{|y_{\e}|}{\e}}{\ln \e}
=O\left(\frac{1}{|\ln \e|} +|x_{\e}| \right).
\end{equation}
\end{small}Also,
$\sum^2_{j=1}\big(y_{\e,j}\times \mbox{the second identity of \eqref{sec6-02}}\big)$ gives us
\begin{small}\begin{equation}\label{sec6-04}
\frac{\La_2|y_\e|^2}{  |y_\e|^2-\e^2 }
+\frac{\La_1(|x_\e|^2|y_\e|^2-\e^2x_\e\cdot y_\e)}{ |x_\e|^2|y_\e|^2-2 \e^2x_\e\cdot y_\e+\e^4 }
+\frac{\La_1(x_{\e}- y_{\e})\cdot y_\e}{|x_\e-y_\e|^2}
+\frac{\La_1\ln \frac{|x_{\e}|}{\e}+ \La_2\ln \frac{|y_{\e}|}{\e}}{ \ln \e}
=O\left(\frac{1}{|\ln \e|} +|y_\e|\right).
\end{equation}
\end{small}Hence, letting $\tau:=\frac{\La_1}{\La_2}$, from $\frac{1}{\La_2^2}\big(\La_1\times
\eqref{sec6-03}+\La_2\times \eqref{sec6-04}\big)$, we have
\begin{small}\begin{equation}\label{sec6-05}
\begin{split}&\frac{\tau^2 |x_{\e}|^2}{  |x_\e|^2-\e^2 }+\frac{|y_{\e}|^2}{|y_\e|^2-\e^2 }
+\frac{2\tau(|y_\e|^2|x_{\e}|^2-\e^2x_\e\cdot y_\e)}{ |x_\e|^2|y_\e|^2-2 \e^2x_\e\cdot y_\e+\e^4 }-\tau
+\frac{(\tau+1)(\tau\ln \frac{|x_{\e}|}{\e}+ \ln \frac{|y_{\e}|}{\e})  }{ \ln \e}
\\&=O\left(\frac{1}{|\ln \e|} +|x_\e|+|y_\e|\right),
\end{split}\end{equation}
\end{small}which gives $\frac{|x_{\e}|^2}{  |x_\e|^2-\e^2 }+\frac{|y_{\e}|^2}{|y_\e|^2-\e^2 }\leq C$. Then there exists a  constant $\delta>0$ independent of $\e$ such that
\begin{small}\begin{equation}\label{sec6-06}
\frac{|x_{\e}|}{\e }\geq 1+\delta~~~~\mbox{and}~~~~\frac{|y_{\e}|}{\e }\geq 1+\delta .
\end{equation}
\end{small}Now we claim
\begin{small}\begin{equation}\label{sec6-07}
\frac{|x_\e|}\e\to \infty~~~~~~~\mbox{and}~~~~~~~\frac{|y_\e|}\e\to \infty.
\end{equation}
\end{small}We first prove that $\frac{|x_\e|}\e\leq M, \frac{|y_\e|}\e\leq M$
  does not occur. Suppose that $\frac{x_\e}\e\to w_0$ and $\frac{y_\e}\e\to z_0$. In view of \eqref{sec6-06},
we see that $|w_0|,\; |z_0|>1$. Now
 \eqref{sec6-02} gives  $w_0\ne z_0$ and
\begin{equation}\label{sec6-08}
\begin{cases}
  \frac{\tau w_{0,j}}{|w_0|^2-1}
-   \frac{w_{0,j}-z_{0,j}}{|z_0-w_0|^2} +   \frac{|z_0|^2w_{0,j}-z_{0,j}}{|z_0|^2|w_0|^2-2 w_0\cdot z_0+1}=0,\\[2mm]
  \frac{z_{0,j}}{|z_0|^2-1}
-   \frac{\tau(z_{0,j}-w_{0,j})}{|z_0-w_0|^2} +  \frac{\tau(|w_0|^2z_{0,j}-w_{0,j})}{|z_0|^2|w_0|^2-2 w_0\cdot z_0+1}=0.
\end{cases}\end{equation}
Let us show that system \eqref{sec6-08}  has no solutions
and hence we obtain a contradiction. In fact,  up to a suitable rotation we can assume that $w_{0,2}=0$. This also implies that $z_{0,2}=0$. Then there exists $\lambda\neq 1$ such that  $z_{0,1}=\lambda w_{0,1}$ and
\begin{equation*}
\begin{cases}
\frac{\tau|w_0|^2}{|w_0|^2-1}
+ \frac{\lambda^2 |w_0|^4-\lambda |w_0|^2}{|\lambda|w_0|^2-1|^2} =  \frac{1-\lambda}{|\lambda-1|^2},\\[4mm]
 \frac{\lambda^2 |w_0|^2}{\lambda^2|w_0|^2-1}
+   \frac{\tau(\lambda^2 |w_0|^4-\lambda |w_0|^2)}{|\lambda|w_0|^2-1|^2} = \frac{\tau(\lambda^2-\lambda)}{|\lambda-1|^2},
\end{cases} \end{equation*}
 which gives us that
\begin{small}\begin{equation*}
0< \frac{\lambda^2 |w_0|^2}{\lambda^2|w_0|^2-1}+
\frac{\tau^2|w_0|^2}{|w_0|^2-1}= \frac{\tau(1- \lambda^2 |w_0|^4)}{|\lambda|w_0|^2-1|^2}= \frac{\tau(1- |w_0|^2 \cdot |z_0|^2)}{|\lambda|w_0|^2-1|^2}<0.
 \end{equation*}
\end{small}Here we use that $|w_0|>1$ and $|z_0|> 1$, this gives a contradiction.

\vskip 0.1cm
Suppose that $\frac{|x_\e|}\e\leq M$ and $\frac{|y_\e|}\e\to \infty$ and assume that $\frac{x_\e}\e\to w_0$. Using
\begin{small}\[
 \begin{split}
 \frac{|y_{\e}|^2}{|y_\e|^2-\e^2 }\to 1,\quad \frac{|y_\e|^2|x_{\e}|^2-\e^2x_\e\cdot y_\e}{ |x_\e|^2|y_\e|^2-2 \e^2x_\e\cdot y_\e+\e^4 }\to 1,
\end{split}
\]
\end{small}we derive
 from \eqref{sec6-05} that
\begin{small}\begin{equation*}
\frac{\tau^2 |w_0|^2}{  |w_0|^2-1}
+\underbrace{\frac{(\tau+1)\ln|y_{\e}|  }{ \ln \e}}_{>0}
=o\big(1\big),
\end{equation*}
\end{small}which gives a contradiction. Similarly, we can
 prove that $\frac{|x_\e|}\e \to \infty$ and $\frac{|y_\e|}\e\leq M$
 do not occur.

\vskip 0.1cm

Now we prove
\begin{equation*}
\frac{1}{C}\leq \frac{|x_\e|}{|y_\e|}\leq C,~~\mbox{for some positive constant}~~C.
\end{equation*}
From \eqref{sec6-07}, we find
\begin{small}\[
 \begin{split}
 \frac{|x_{\e}|^2}{|x_\e|^2-\e^2 }\to 1,\quad \frac{|y_{\e}|^2}{|y_\e|^2-\e^2 }\to 1,\quad \frac{|y_\e|^2|x_{\e}|^2-\e^2x_\e\cdot y_\e}{ |x_\e|^2|y_\e|^2-2 \e^2x_\e\cdot y_\e+\e^4 }\to 1.
\end{split}
\]
\end{small}Let $\frac{|x_\e|}{|y_\e|}\to a_0$, with $a_0\in [0,\infty]$. If $a_0=0$, then
\begin{small}
\[
\frac{ x_\e\cdot(x_{\e}-y_{\e}) }{|x_\e-y_\e|^{2}}\to 0.
\]
\end{small}Thus
\eqref{sec6-03} gives
$\frac{\tau\ln |x_\e|+ \ln |y_\e|}{\ln \e }=o(1)$. While, by \eqref{sec6-04}, it holds $\frac{\tau\ln |x_\e|+ \ln |y_\e|}{\ln \e }=-\tau+o(1)$.
 Hence  a contradiction arises.
 Similarly,  $a_0=\infty$
is impossible and this gives that $\frac{|w_\e|}{|z_\e|}\to a_0 \in (0,\infty)$.

 \vskip 0.2cm
\noindent\textbf{Step 2.} It holds
\begin{equation*}
 |x_\e|, |y_\e| \sim \e^\beta, \mbox{~with $\beta=\frac{\tau}{(\tau+1)^2}$}.
\end{equation*}
 \vskip 0.05cm

First, by \eqref{sec6-05} and \eqref{sec6-07}, we have
 $\frac{ \tau \ln  |x_\e|+  \ln |y_\e|  }{ \ln \e  }=\frac{\tau}{\tau+1}+o\big(1\big)$. Also
using that $\frac1 C |y_\e|\le |x_\e|\le C|y_\e|$, we have
$\frac{ \ln  |x_\e|  }{ \ln \e  }= \beta+o\big(1\big)$, which
 implies $|x_\e|\sim\e^\beta$ and then $|y_\e| \sim \e^\beta$.

 \vskip 0.2cm
\noindent\textbf{Step 3.}
Let us compute the asymptotic of $x_\e$ and $y_\e$.

\vskip 0.1cm

Set $A=(A_1,A_2)$ and $B=(B_1,B_2)$ where
 \begin{equation*}
A=\lim\limits_{\e\to0}\frac{x_\e}{\e^\beta},\,\,\,\,\,
B=\lim\limits_{\e\to0}\frac{y_\e}{\e^\beta}, \mbox{~with $\beta=\frac{\tau}{(\tau+1)^2}$}.
\end{equation*}We will use a refinement of \eqref{sec6-01},  obtained by \eqref{sec3-13}.
Due to some cancellations, it will be necessary to consider an expansions up to second order.
Letting $(x_\e,y_\e)=(\e^\beta w_\e,\e^\beta  z_\e)$  and recalling \eqref{sec3-02},  \eqref{sec3-13} becomes
\begin{small}\begin{equation}\label{sec6-09}
\begin{cases}
\frac{\pi\e^\beta}{\La_1\La_2}\frac{\partial \mathcal{KR}_{\Omega_\e}}{\partial x_{j}}(\e^\beta w_\e,\e^\beta z_\e)=\frac{w_{\e,j}}{|w_\e|^2}\left(-\frac{\tau}{\tau+1}-\frac{\tau\ln|w_\e|+ \ln|z_\e| +2\pi\frac{\tau^2+\tau+1}{\tau+1} \mathcal R _\Omega(0)}{\ln\e+2\pi \mathcal R _\Omega(0)}
\right)+\frac{w_{\e,j}-z_{\e,j}}{|z_\e-w_\e|^2}+O\left(\e^\beta\right),
\\[4mm]
\frac{\pi\e^\beta}{\La_2^2}\frac{\partial \mathcal{KR}_{\Omega_\e}}{\partial y_{j}}(\e^\beta w_\e,\e^\beta z_\e)=\frac{z_{\e,j}}{|z_\e|^2}\left(-\frac{\tau}{\tau+1}-\frac{\tau\ln|w_\e|+ \ln|z_\e|+2\pi\frac{\tau^2+\tau+1}{\tau+1} \mathcal R _\Omega(0)}{\ln\e+2\pi \mathcal R _\Omega(0)}\right)-\frac{\tau(w_{\e,j}-z_{\e,j})}{|z_\e-w_\e|^2}+O\left(\e^\beta\right).
\end{cases}
\end{equation}
\end{small}Passing to the limit we have that $A$ and $B$ satisfy
\begin{small}\begin{equation*}
\begin{cases}
-\frac{\tau A_j}{(\tau+1)|A|^2}+\frac{A_j-B_j}{|A-B|^2}=0,
\\[4mm]
-\frac{B_j}{(\tau+1)|B|^2}-\frac{A_j-B_j}{|A-B|^2}=0.
\end{cases}
\end{equation*}
\end{small}This implies that $A\neq B$ and if $A_j=0$, then $B_j=0$.
Thus, we can assume that $|z_{\e,j}-w_{\e,j}|\ge C>0$ for some $j$.
This also gives $|w_{\e,j}|\ge C'>0$.

\vskip 0.1cm

Next, from  $\frac{\partial \mathcal{KR}_{\Omega_\e}}{\partial x_{j}}(\e^\beta w_\e,\e^\beta z_\e)=\frac{\partial \mathcal{KR}_{\Omega_\e}}{\partial y_{j}}(\e^\beta w_\e,\e^\beta z_\e)=0$,  we deduce from
 \eqref{sec6-09} that
\begin{small}\begin{equation*}
\frac{w_{\e,j}}{|w_\e|^2}=-\frac{1}{\tau}\frac{z_{\e,j}}{|z_\e|^2}
\left(\frac{\frac{w_{\e,j}-z_{\e,j}}{|z_\e-w_\e|^2}+O\left(\e^\beta\right)}{\frac{w_{\e,j}-z_{\e,j}}{|z_\e-w_\e|^2}+O\left(\e^\beta\right)}\right)=
-\frac{1}{\tau}\frac{z_{\e,j}}{|z_\e|^2}\Big(1+O\big(\e^\beta\big)\Big),
\end{equation*}
\end{small}which implies
\begin{small}\begin{equation}\label{sec6-10}
|z_\e|=\frac{|w_\e|}{\tau}\Big(1+O\big(\e^\beta\big)\Big)\hbox{ and }z_{\e,j}=-\frac{w_{\e,j}}{\tau}\Big(1+O\big(\e^\beta\big)\Big).
\end{equation}
\end{small}Inserting \eqref{sec6-10} in the first equation of  \eqref{sec6-09}, we obtain
\begin{small}\begin{equation*}
\begin{split}
0=&\frac{w_{\e,j}}{|w_\e|^2}\left(-\frac{\tau}{\tau+1}-\frac{(\tau+1)\ln|w_\e|
-\ln\tau +2\pi\frac{\tau^2+\tau+1}{\tau+1} \mathcal R _\Omega(0)+O\left(\e^\beta\right)}{\ln\e+2\pi \mathcal R _\Omega(0)}
\right)\\
&+ \frac{w_{\e,j}}{|w_\e|^2}\left(\frac{\tau}{\tau+1}+O\left(\e^\beta\right)
\right)+O\left(\e^\beta\right)\\ \Rightarrow
&0=\frac{w_{\e,j}}{|w_\e|^2}\left(\frac{(\tau+1)\ln|w_\e|-\ln\tau +2\pi\frac{\tau^2+\tau+1}{\tau+1} \mathcal R _\Omega(0)}{\ln\e}+O\left(\e^\beta\right)\right),
\end{split}
\end{equation*}
\end{small}which implies that
$|w_\e|\to|A|=C_\tau$ with $C_\tau=
\tau^{\frac{1}{\tau+1}}
e^{-\frac{2\pi \mathcal{R}_{\Omega}(0)(\tau^2+ \tau+1)}{(\tau+1)^2}}$. And in the same way, we get
$|z_\e|\to|B|=\frac{C_\tau}{\tau}$.
This  proves \eqref{sec1-16}, concluding the proof of this part.
\end{proof}

\subsection{Existence and asymptotics}\

\vskip 0.1cm
To prove existence of the critical points, we will  start by \eqref{sec6-01} and look for the critical points of $\nabla\mathcal{KR}_{\Omega_\e }(x,y)$ that are close to those of the first term of the expansion.

\vskip 0.1cm

Motivated by the necessary condition of the previous subsection, we set
\begin{small}\begin{equation}\label{sec6-11}
x=\e^\beta w\hbox{ and }y=\e^\beta z\quad\hbox{with }\beta=\frac{\tau}{(\tau+1)^2}.
\end{equation}
\end{small}Now we analyze the limit function of $\nabla\mathcal{KR}_{\Omega_\e }(x,y)$.
 In view of  \eqref{sec6-09}, we  consider the following system
\begin{small}\begin{equation}\label{sec6-12}
\begin{cases}
\frac{w_j}{|w|^2}\left(-\frac{\tau}{\tau+1}-
\frac{\tau\ln|w|+\ln|z|+2\pi\frac{\tau^2+\tau+1}{\tau+1} \mathcal R _\Omega(0)}{\ln\e+2\pi \mathcal R_\Omega(0)}
\right)+\frac{w_j-z_j}{|z-w|^2}=0,
\\[4mm]
\frac{z_j}{|z|^2}\left(-\frac{\tau}{\tau+1}-
\frac{\tau\ln|w|+ \ln|z|+2\pi\frac{\tau^2+\tau+1}{\tau+1} \mathcal R _\Omega(0)}{\ln\e+2\pi \mathcal R_\Omega(0)}\right)+\frac{\tau(z_j-w_j)}{|z-w|^2}=0,
\end{cases}
\end{equation}
\end{small}whose solutions are given by
\begin{small}\begin{equation}\label{sec6-13}
\big(w,-\frac{w}{\tau}\big)~~~\mbox{with}~~~
|w|=C_\tau=
\tau^{\frac{1}{\tau+1}}
e^{-\frac{2\pi \mathcal{R}_{\Omega}(0)(\tau^2+ \tau+1)}{(\tau+1)^2}}.
\end{equation}
\end{small}Observe that the solutions $(w,z)$ to \eqref{sec6-12} are the critical points of the following function
\[
F_\e(w,z)=-\frac{\tau}{\tau+1}\big( \tau \ln |w|+ \ln |z|\big)+\tau \ln |w-z|
- \frac{\big(\tau\ln|w|+ \ln|z|+2\pi\frac{\tau^2+\tau+1}{\tau+1} \mathcal R _\Omega(0)\big)^2}{2\big(\ln\e+2\pi \mathcal R_\Omega (0)\big)}.
\]
Next we set
\begin{small}\begin{equation*}
\begin{split}
\tilde F_\e(\tilde w,\tilde z) = & F_\e\big((\tilde w,0),(\tilde z,0)\big)\\=&
-\frac{\tau}{\tau+1}\Big( \tau \ln \tilde w+ \ln (-\tilde z)\Big)+\tau \ln |\tilde w-\tilde z|
- \frac{\big(\tau\ln \tilde w+ \ln(-\tilde z)+2\pi\frac{\tau^2+\tau+1}{\tau+1} \mathcal R _\Omega(0)\big)^2}{2(\ln\e+2\pi \mathcal R_\Omega (0))},
\end{split}
\end{equation*}
\end{small}for $\tilde w, \tilde z\in \R$, $\tilde w>0$ and $\tilde z<0$.
Then, critical points of $F_\e(w,z)$ are given by
\begin{small}\begin{equation*}\left\{\left(T\begin{pmatrix}\tilde w_0\\0
\end{pmatrix},\,T\begin{pmatrix}\tilde z_0\\0
\end{pmatrix}\right),\,\,T\in O(2)\right\},\end{equation*}
\end{small}where $(\tilde w_0,\tilde z_0)$ is a critical point of $\tilde F_\e(\tilde w,\tilde z)$.
Moreover, by \eqref{sec6-13}, it follows that $\tilde F_\e(\tilde w,\tilde z)$ has a unique critical point in the set $\tilde w>0$ and $\tilde z<0$, given by
$\tilde w_0=C_\tau$ and
$\tilde z_0=-\frac{C_\tau}{\tau}$.

\vskip 0.05cm

In the next proposition, we show that $(\tilde w_0,\tilde z_0)$ is a minimum for $\tilde F_\e(\tilde w,\tilde z)$.
\begin{prop}\label{sec6-Prop6.1}
The function $\tilde F_\e(\tilde w,\tilde z)$ admits a unique critical point $(\tilde w_0,\tilde z_0)$. Moreover, it is a nondegenerate and minimum point.
\end{prop}
\begin{proof}
The uniqueness of the critical point follows directly from \eqref{sec6-13}.  By straightforward computations, we have that
\begin{small}\begin{equation*}
\begin{split}
&\frac{\partial ^2 \tilde F_\e }{\partial \tilde w^2}(\tilde w_0,\tilde z_0)=\frac {\tau^2} {\tilde w_0^2}\left( \frac{1}{(\tau+1)^2}-\frac {1}{\ln \e}\right)+o\left(\frac1{|\ln\e|}\right)>0,\\
&\frac{\partial ^2 \tilde F_\e }{\partial \tilde w\partial \tilde z }(\tilde w_0,\tilde z_0)=\frac {\tau^2} {\tilde w_0^2}\left( \frac{\tau}{(\tau+1)^2}+\frac {1}{\ln \e}\right)+o\left(\frac1{|\ln\e|}\right),\\
&\frac{\partial ^2 \tilde F_\e }{\partial \tilde z^2}(\tilde w_0,\tilde z_0)=\frac {\tau^2} {\tilde w_0^2}\left( \frac{\tau^2}{(\tau+1)^2}-\frac {1}{ \ln \e}\right)+o\left(\frac1{|\ln\e|}\right)>0.
\end{split}
\end{equation*}
\end{small}Hence  for $\e$ small enough,
\begin{small}\begin{equation*}
\begin{split}
&\det\nabla^2\tilde F_\e(\tilde w_0,\tilde z_0)=-\frac {\tau^4}{\tilde w_0^4  \ln \e} +o\left(\frac1{|\ln\e|}\right)>0,
\end{split}
\end{equation*}
\end{small}which gives the result.
\end{proof}
\vskip 0.2cm

\begin{proof}[\bf{Proof of Theorem~\ref{sec1-teo12}(2)}]
Let $(\tilde w_0,\tilde z_0)$ be the unique critical point of  $\tilde F_\e(\tilde w,\tilde z)$ in $\tilde w>0$ and $\tilde z<0$.
By Proposition \ref{sec6-Prop6.1}, we know that $(\widetilde w_0,\widetilde z_0)$ is a minimum point of $\tilde F_\e (\widetilde{w},\widetilde{z})$.
Let $\delta_\e=\frac 1{|\ln \e|^2}$. For $\e$ small enough we have that  $\overline{B\big(
(\tilde w_0,\tilde z_0),\delta_\e\big)}\cap \Big(\{\tilde w=0\}\cup \{\tilde z=0\}\Big)=\emptyset$.
Next, we define
\begin{small} \begin{equation*}
 \begin{split}
 B^*_{\delta_\e}=&\Big\{(w,z)\in\R^4,\hbox{ such that }\exists \,~\mbox{a rotation}~\, T\in O(2),
\,\,(Tw,Tz)=\big((\widetilde{w},0),(\widetilde{z},0)\big),\,\, (\widetilde w,\widetilde z)\in B\big(
(\tilde w_0,\tilde z_0),\delta_\e\big)
\Big\}.
 \end{split}
 \end{equation*}
\end{small}We have
 the following alternative:
\begin{itemize}
\item  The function $\mathcal{KR}_{\Omega_\e}(\e^\beta w,\e^\beta z)$ has infinitely many critical points.\vskip 0.05cm
\item The critical points of $\mathcal{KR}_{\Omega_\e}(\e^\beta w,\e^\beta z)$ are isolated.
\end{itemize}
In the first case we get the second case in the existence part of Theorem~\ref{sec1-teo12}(2).
Next we assume that the critical points of $\mathcal{KR}_{\Omega_\e}(\e^\beta w,\e^\beta z)$ are isolated.
We want to show that, for $\beta$ as in \eqref{sec6-11},
\begin{equation}\label{sec6-14}
deg\big(\nabla
  \mathcal{KR}_{\Omega_\e}(\e^\beta w,\e^\beta z), {B}^*_{\delta_\e},0\big)= 0.
\end{equation}
We start by showing that
\begin{small}\begin{equation}\label{sec6-15}
\begin{split}
\Big\langle \nabla \mathcal{KR}_{\Omega_\e}(\e^\beta w,\e^\beta z),\nu \Big\rangle>0,\quad \forall\; (w,z)\in \partial B^*_{\delta_\e},
\end{split}
\end{equation}
\end{small}where $\nu$ is the outward unit normal of $\partial B^*_{\delta_\e}$ at $(w,z)$.

\vskip 0.05cm

Indeed for any $(w,z)\in \partial B^*_{\delta_\e}$,  we have that
\begin{small} \begin{equation}\label{sec6-16}
\Big\langle \nabla F_\e( w,z),\nu \Big\rangle=\Big\langle \nabla \tilde F_\e( \tilde w,\tilde z),\tilde\nu \Big\rangle,
\end{equation}
\end{small}where $(\widetilde w,\widetilde z)\in \partial B\big(
(\tilde w_0,\tilde z_0),\delta_\e\big)$ and $\widetilde{\nu}= \frac{(\widetilde w,\widetilde z)-  (\widetilde w_0,\widetilde z_0)}{\delta_\e}$ is the outward unit normal of $\partial B\big(
(\tilde w_0,\tilde z_0),\delta_\e\big)$ at $(\widetilde w,\widetilde z)$.
Then it holds, using that the $\nabla ^3\tilde F_\e (\tilde w,\tilde z)$ is uniformly bounded in $\e$ in a neighborhood of $(\widetilde w_0,\widetilde z_0)$,
\begin{small}\begin{equation*}
\begin{split}
\Big\langle \nabla \tilde F_\e (\widetilde w,\widetilde z), \widetilde{\nu} \Big\rangle=&
\Big\langle \nabla \tilde F_\e (\widetilde w,\widetilde z)- \nabla \tilde F_\e (\widetilde w_0,\widetilde z_0), \widetilde{\nu} \Big \rangle
\\ =&
\delta_\e \Big\langle \nabla^2 \tilde F_\e (\widetilde w_0,\widetilde z_0)\cdot \widetilde{\nu} , \widetilde{\nu} \Big \rangle+O\big(\delta_\e^2\big) \geq
\delta_\e\left[-\frac {c_0(\tau)}{\ln \e}+o\left( \frac 1{|\ln \e|}\right)\right]>0,
\end{split}\end{equation*}
\end{small}where $c_0(\tau)$ is a positive constant that depends only on $\tau$,
by the choice of $\delta_\e$, for $\e$ small enough. Hence  by \eqref{sec6-16}, it holds
\begin{equation*}
\big\langle \nabla F_\e( w,z),\nu \big\rangle>0.
\end{equation*}
Moreover, by \eqref{sec6-09}, we have that
\begin{small}\begin{equation*}
\nabla
  \mathcal{KR}_{\Omega_\e}(\e^\beta w,\e^\beta z)=\frac {\La_2^2}{\e^\beta \pi}\Big[\nabla F_\e ( w, z)+O(\e^\beta)\Big],
\end{equation*}
\end{small}which implies that, for every $(w,z)\in \partial B^*_{\delta_\e}$, for $\e$ small enough,
\begin{small}\[
\begin{split}
\Big\langle \nabla \mathcal{KR}_{\Omega_\e}(\e^\beta w,\e^\beta z),\nu \Big\rangle=&\frac {\La_2^2}{\e^\beta \pi}\left[\Big\langle \nabla F_\e( w,z),\nu \Big\rangle+O(\e^\beta)\right] \geq
\frac {\La_2^2}{\e^\beta \pi}\left[-\frac {c_0(\tau)\delta_\e}{\ln \e}+o\left( \frac {\delta_\e}{|\ln \e|}\right) +O(\e^\beta)\right]\\=&
\frac {\La_2^2}{\e^\beta \pi(\ln\e)^3} \big(-c_0(\tau)+o(1)
\big)>0.
\end{split}
\]
\end{small}

By  \eqref{sec6-15} and the Poincar\'e-Hopf Theorem, we have
 \begin{equation*}
 \deg\big(\nabla \mathcal{KR}_{\Omega_\e}(\e^\beta w,\e^\beta z), B^*_{\delta_\e},0\big)=\chi ( B^*_{\delta_\e}) =\chi(\mathbb S^1)= 0,
 \end{equation*}
 where $\chi(S)$ is the Euler characteristic of $S$.

\vskip 0.05cm
Next since $ \mathcal{KR}_{\Omega_\e}(\e^\beta w,\e^\beta z)$ is continuous in $\overline B^*_{\delta_\e}$ and  $\langle\nabla \mathcal{KR}_{\Omega_\e}(\e^\beta w,\e^\beta z),\nu \rangle>0$ for $(w,z)\in \partial B^*_{\delta_\e}$ by \eqref{sec6-15},
then  $\mathcal{KR}_{\Omega_\e}(\e^\beta w,\e^\beta z)$ has a minimum in $B^*_{\delta_\e}$. Since the minimum  has index $1$ and by \eqref{sec6-14}, $\mathcal{KR}_{\Omega_\e}(\e^\beta w,\e^\beta z)$  admits at least another critical point with negative index.

\vskip 0.05cm

Note that the above arguments hold for any function which is a $C^1$
perturbation of $\mathcal{KR}_{\Omega_\e}$.
This concludes that $\mathcal{KR}_{\Omega_\e}$ has
at least two stable critical points. Hence we finish the proof of Theorem~\ref{sec1-teo12}$(2)$.
\end{proof}

\section{The exact multiplicity of type III critical points}\label{sec7}

As stated in Section \ref{sec-6},
to prove the existence of type III critical points of $\mathcal{KR}_{\Omega_\e}(x,y)$, \eqref{sec6-09}  is sufficient.
 However, we can only determine the length, not the
 direction of the critical points from
 \eqref{sec6-09}, because in the expansion of \eqref{sec6-09}, the effects from
 the location of the small hole and the geometric properties of $\O$
 are totally ignored. To determine the
 direction of the critical points,
  further
 expansion for   $\mathcal{KR}_{\Omega_\e}(x,y)$ is necessary, so that
 the effects from the location and from the geometry of $\O$ can be captured.

Our strategies in the section consist of the following steps.

\begin{itemize}

\item We expand $\nabla \mathcal{KR}_{\Omega_\e}(x,y)$ until
the effects from the location of hole and from the geometry of $\O$ can be captured, in the sense that $\nabla \mathcal{KR}_{\Omega_\e}(x,y)$ can be written as
\begin{equation}\label{1-30-11}
\nabla \mathcal{KR}_{\Omega_\e}(x,y)= {\bf{K}}_\e (x,y)+h.o.t.,
\end{equation}
and we can find the exact number of the solutions  for  ${\bf{K}}_\e (x,y)=0$
and prove their non-degeneracy.\vskip 0.05cm

\item Using \eqref{1-30-11}, we prove the existence of solutions for
$\nabla \mathcal{KR}_{\Omega_\e}(x,y)=0$ by showing the degree of this
vector field is not zero  in each
    small neighborhood of the solutions  for  ${\bf{K}}_\e (x,y)=0$.

\item We prove that each the solution of $\nabla \mathcal{KR}_{\Omega_\e}(x,y)=0$
is nondegenerate and compute the index of each solution.\vskip 0.05cm

\item We prove the local uniqueness of solution of $\nabla \mathcal{KR}_{\Omega_\e}(x,y)=0$ near  every solution of
    ${\bf{K}}_\e (x,y)=0$ by comparing the local degree of each solution
    $\nabla \mathcal{KR}_{\Omega_\e}(x,y)=0$ with the total degree
    of the vector field $\nabla \mathcal{KR}_{\Omega_\e}(x,y)$ in each
    small neighborhood of the solutions  for  ${\bf{K}}_\e (x,y)=0$.
    This local uniqueness implies that the number of solutions for
    $\nabla \mathcal{KR}_{\Omega_\e}(x,y)=0$ equals that of ${\bf{K}}_\e (x,y)=0$.

\end{itemize}

\subsection{The improved expansion for $\nabla\mathcal{KR}_{\Omega_\e}(x,y)$}\

\vskip 0.1cm

We will follow the strategies mentioned above. The most technical part is the expansion of
$\nabla\mathcal{KR}_{\Omega_\e}(x,y)$.

\vskip 0.1cm

  For any type III critical point $(x_\e,y_\e)$   of $\mathcal{KR}_{\Omega_\e}(x,y)$ with $\Omega_\e=\Omega\backslash B(0,\e)$, it holds $|x_\e|,|y_\e|\sim \e^{\beta}$ with $\beta=\frac{\tau}{(\tau+1)^2}$ and $\tau=\frac{\La_1}{\La_2}$.
Then  we have following results.
\begin{lem}
For $x,y\in \Omega_\e$ and $j=1,2$, if $|x|,|y|\sim \e ^\beta$, then  it holds
\begin{small}\begin{equation}\label{sec7-01}
\begin{cases}
\frac{\partial \mathcal{KR}_{(B(0,\e))^c}(x,y)}{\partial x_j}=-\frac{\La_1}\pi
 \left[\big(\La_1+\La_2\big)\frac{x_j}{|x|^2}-
\frac{\La_2(x_{j}-y_{j})}{|x-y|^{2}}\right]
+O\Big(\e^{2-3\beta}\Big),\\[3mm]
\frac{\partial \mathcal{KR}_{(B(0,\e))^c}(x,y)}{\partial y_j}=-\frac{\La_2}\pi
 \left[\big(\La_1+\La_2\big)\frac{y_j}{|y|^2}-
\frac{\La_1(y_{j}-x_{j})}{|x-y|^{2}}\right]
+O\Big(\e^{2-3\beta}\Big).
\end{cases}
\end{equation}
\end{small}\end{lem}
\begin{proof}
First, by \eqref{sec2-03}, we recall
\begin{small} \begin{equation*}
\begin{split}
\frac{\partial \mathcal{KR}_{(B(0,\e))^c}(x,y)}{\partial x_j}=&
-\frac{ \La_1 }{\pi} \left[ \frac{\La_1   x_{j}}{ |x|^2-\e^2 }
+  \frac{\La_2  (|y|^2x_{j}-\e^2y_{j})}{ |x|^2|y|^2-2 \e^2x\cdot y+\e^4 }-  \frac{\La_2( x_{j}-y_{j})}{|x-y|^{2}}\right].
\end{split}\end{equation*}\end{small}Also from $|x|,|y|\sim \e ^\beta$, we see
\begin{small} \begin{equation*}
 \frac{x_{j}}{ |x|^2-\e^2 }=\frac{x_{j}}{|x|^2}+O\left(\frac{\e^2}{(|x|^2-\e^2)|x|}\right)
 =\frac{x_{j}}{|x|^2}+O\Big(\e^{2-3\beta}\Big),
\end{equation*}\end{small}and
\begin{small} \begin{equation*}
\frac{|y|^2x_{j}-\e^2y_{j}}{ |x|^2|y|^2-2 \e^2x\cdot y+\e^4 } =\frac{x_{j}}{|x|^2}+O\left(\frac{\e^2|x|\cdot|y|+\e^4}{( |x|^2|y|^2-2 \e^2x\cdot y+\e^4 )|x|}\right)
 =\frac{x_{j}}{|x|^2}+O\Big(\e^{2-3\beta}\Big).
\end{equation*}\end{small}Hence from above computations, we get the first estimate of \eqref{sec7-01}. Similarly, the second estimate of \eqref{sec7-01} holds.
\end{proof}

We now expand $\mathcal{KR}_{\Omega_\e}(x,y)$ until
the effect from the location of the small hole can be seen.

\begin{prop}\label{sec7-prop7.4}
For $x,y\in \Omega_\e$ and $j=1,2$, if $|x|,|y|\sim \e ^\beta$, it holds
\begin{small}\begin{equation}\label{sec7-02}
\begin{cases}
 \frac{\partial \mathcal{KR}_{\Omega_\e}(x,y)}{\partial x_j}=
 -
 \frac{\La_1}\pi
 \left\{ \frac{h(x,y)x_j}{|x|^2}
 +\frac{\pi(
 \La_1\ln  \frac{|x|}{\e} + \La_2\ln  \frac{|y|}{\e})}{\ln \e+2\pi \mathcal{R}_{\Omega}(0)}
  \frac{\partial \mathcal{R}_{\Omega}(0)}{\partial x_j}-
\frac{\La_2(x_{j}-y_{j})}{|x-y|^{2}}\right\}
+O\Big( \frac{1}{|\ln \e|}\Big),\\[3mm]
 \frac{\partial \mathcal{KR}_{\Omega_\e}(x,y)}{\partial y_j}=-\frac{\La_2}\pi  \left\{\frac{h(x,y)y_j}{|y|^2}+\frac{\pi(
 \La_1\ln  \frac{|x|}{\e} + \La_2\ln  \frac{|y|}{\e})}{\ln \e+2\pi \mathcal{R}_{\Omega}(0)}
  \frac{\partial \mathcal{R}_{\Omega}(0)}{\partial x_j}-
\frac{\La_1(y_{j}-x_{j})}{|x-y|^{2}}\right\}
+O\Big( \frac{1}{|\ln \e|}\Big),
\end{cases}
\end{equation}
\end{small}where\begin{small}
\begin{equation}\label{sec7-02a}
h(x,y):=\frac{
 \La_1\ln  |x| + \La_2\ln |y|+2\pi \mathcal{R}_\O(0) (\La_1+\La_2) }{\ln \e+2\pi \mathcal{R}_{\Omega}(0)}.
\end{equation}\end{small}
\end{prop}
\begin{proof}
First, from \eqref {sec7-01} and \eqref{App-A.12}, we have
\begin{small}\begin{equation}\label{sec7-03}
 \frac{\partial \mathcal{KR}_{\Omega_\e}(x,y)}{\partial x_j}= -\frac{\La_1}\pi
 \left[\big(\La_1+\La_2\big)\frac{x_j}{|x|^2}-
\frac{\La_2(x_{j}-y_{j})}{|x-y|^{2}}\right] +\Psi_{\e,j}(x,y) +O\left(\frac{\e^{1-\beta}}{|\ln \e|} \right),
\end{equation}
\end{small}where $\Psi_{\e,j}(x,y)$ is the function in \eqref{App-A.13}.
Now we compute the term $\Psi_{\e,j}(x,y)$. We have
\begin{small}\begin{equation*}
\begin{split}
 \frac{\La_1G_\Omega(x,0)+\La_2G_\Omega(0,y)}{
 \ln \e+2\pi \mathcal{R}_\Omega(0)} =
 - \frac{\frac{\La_1\ln|x|}{2\pi} +\frac{\La_2\ln|y|}{2\pi} +(\La_1H_\Omega(x,0)+\La_2H_\Omega(0,y))}{ \ln \e+2\pi \mathcal{R}_\Omega(0)}=-\frac{{h}(x,y)}{2\pi}+O\Big(\frac{\e^{\beta}}{|\ln\e |}\Big),
\end{split}
\end{equation*}
\end{small}and then
\begin{small}\begin{equation*}
\begin{split}
&\Big(\frac{x_j}{|x|^2}+2\pi
 \frac{\partial H_\Omega(x,0)}{\partial x_j} \Big) \frac{\La_1G_\Omega(x,0)+\La_2G_\Omega(0,y)}{
 \ln \e+2\pi \mathcal{R}_\Omega(0)} =-\frac{ {h}(x,y)x_j}{2\pi|x|^2}
 -
 \frac{
 \La_1\ln  |x| + \La_2\ln |y|
  }{\ln \e+2\pi \mathcal{R}_{\Omega}(0)} \frac{\partial \mathcal{R}_{\Omega}(0)}{\partial x_j}
 +O\Big(\frac{1}{|\ln\e |}\Big).
\end{split}
\end{equation*}
\end{small}Also, by Taylor's expansion, we have
\begin{small}\begin{equation*}
\begin{split}
 &\La_1\frac{\partial \mathcal{R}_\Omega(x)}{\partial x_j}
 +
2\La_2\frac{\partial H_\Omega(x,y)}{\partial x_j} =
(\La_1+\La_2) \frac{\partial \mathcal R_{\Omega}(0)}{\partial x_j}+O\Big(\e^{\beta}\Big).
\end{split}
\end{equation*}
\end{small}Hence from above computations, we get
\begin{small}\begin{equation}\label{sec7-04}
\begin{split}
\Psi_{\e,j}(x,y) =&\frac{\La_1}{\pi}\left\{ -\frac{x_j}{|x|^2}\big({h}(x,y)-\La_1-\La_2\big)-
 \frac{\pi (\La_1\ln \frac{|x|}{\e} + \La_2\ln  \frac{|y|}{\e} )
  }{\ln \e+2\pi \mathcal{R}_{\Omega}(0)} \frac{\partial \mathcal{R}_{\Omega}(0)}{\partial x_j}
  \right\}
+O\left(\frac{1}{|\ln \e|} \right).
\end{split}\end{equation}
\end{small}Finally, from \eqref{sec7-03} and \eqref{sec7-04}, we prove the first estimate of \eqref{sec7-02}. Similarly, it is possible to deduce the second estimate of \eqref{sec7-02}.
\end{proof}

\begin{rem}
Let  $(x_\e,y_\e)$ be a type III critical point of
$\mathcal{KR}_{\Omega_\e}(x,y)$. Set
$(w_\e,\gamma_\e):=\big(\frac{x_\e}{\e^{\beta}},\frac{x_\e+\tau y_\e}{\e^{2\beta}}\big)$. Then
 from \eqref{sec6-10}, we have
\[
\gamma_\e= \frac{x_\e+\tau y_\e}{\e^{2\beta}}=\frac{w_\e+\tau z_\e}{\e^\beta}=O\big(1\big).
\]Hence we get that the type III critical points of $\mathcal{KR}_{\Omega_\e}(x,y)$ must belong to $\mathcal{H}_\e$, where
\begin{small}\begin{equation*}
\mathcal{H}_\e:= \left\{(x,y)\in \Omega_\e\times \Omega_\e, ~
|x|,|y|\sim \e^{\beta },
~~\lim_{\e\to 0}\Big|\frac{x+\tau y}{\e^{2\beta}}\Big|<\infty~~~
\mbox{with}~~~\beta=\frac{\La_1\La_2}{(\La_1+\La_2)^2}~~~
\mbox{and}~~~\tau=\frac{\La_1}{\La_2}\right\}.
\end{equation*}\end{small}
\end{rem}
To determine the direction of the critical points, we introduce following transform
\begin{small}
\begin{equation*}
(w,\gamma)=\big(\frac{x}{\e^\beta},\frac{x+\tau y}{\e^{2\beta}}\big),~~\mbox{with}~~\beta=\frac{\tau}{(1+\tau)^2}~~~\mbox{and}~~~\tau=\frac{\La_1}{\La_2}.
\end{equation*}
\end{small}We rewrite the  expansion of $\nabla \mathcal{KR}_{\Omega_\e}(x,y)$
as follows.

\begin{prop}\label{sec7-prop7.5}
Let $\mathcal{H}'_\e:=
\Big\{(w,\gamma);~(x,y):=
\big(\e^{\beta }w,
\frac{-\e^{\beta }w+
\e^{2\beta} \gamma }{\tau}\big)
\in \mathcal{H}_\e\Big\}$, then for any $(w,\gamma)\in \mathcal{H}'_\e$, it holds
 \begin{small}\begin{equation}\label{sec7-05}
\begin{split}
 & \frac{\partial \mathcal{KR}_{\Omega_\e}(x,y)}{\partial x_j}\Big|_{(x,y)=
\big(\e^{\beta}w,
\frac{-\e^{\beta} w+
\e^{ 2\beta } \gamma }{\tau}\big)}
 \\=&-\frac{\La_1\La_2}{\pi} \left\{\left[
\frac{k\big( |w|,\tau\big)}{|w|^2\e^{\beta}(\ln \e+2\pi \mathcal{R}_{\Omega}(0))}-
 2\beta  \frac{(w\cdot \gamma)}{|w|^4}
\right]w_j
-\frac{\pi(1+\tau+\tau^2)}{1+\tau}
\frac{\partial \mathcal{R}_\Omega(0)}{\partial x_j}+\frac{\beta}{|w|^2}\gamma_j  \right\}
+O\left(\frac{1}{|\ln \e|}
\right),
\end{split}
\end{equation}
\end{small}and\begin{small}
\begin{equation}\label{sec7-06}
\begin{split}
 &
 \frac{\partial \mathcal{KR}_{\Omega_\e}(x,y)}{\partial y_j}\Big|_
 {(x,y)=
\big(\e^{\beta }w,
\frac{-\e^{\beta } w+
\e^{2\beta } \gamma }{\tau}\big)}
 \\=&-\frac{\La_2^2}{\pi}
  \left\{\left[-
\frac{\tau k\big( |w|,\tau\big)}{|w|^2\e^{\beta}(\ln \e+2\pi \mathcal{R}_{\Omega}(0))}-
 2\beta \tau^2 \frac{(w\cdot \gamma)}{|w|^4}
\right]w_j
-\frac{\pi(1+\tau+\tau^2)}{1+\tau}
\frac{\partial \mathcal{R}_\Omega(0)}{\partial x_j}+\frac{\tau^2\beta}{|w|^2}\gamma_j  \right\}
+O\left(\frac{1}{|\ln \e|}
\right),
\end{split}
\end{equation}\end{small}where
\begin{small}
\begin{equation}\label{sec7-07}
 k(r,\tau):=(1+\tau)\big(\ln r+2(1-\beta)\pi \mathcal{R}_{\Omega}(0)\big)-\ln \tau.
\end{equation}\end{small}
\end{prop}
\begin{proof}
The first estimate of \eqref{sec7-02} gives
\begin{small}\begin{equation}\label{sec7-08}
\begin{split}
 \frac{\partial \mathcal{KR}_{\Omega_\e}(x,y)}{\partial x_j}=&
 - \frac{\La_1\La_2}{\pi}
 \left\{\left[ \frac{
 \tau\ln  |x| + \ln |y|+2\pi \mathcal{R}_\O(0) (1+\tau) }{\ln \e+2\pi \mathcal{R}_{\Omega}(0)}
 \right] \frac{x_j}{|x|^2}-\frac{\pi(\tau^2+\tau+1)}{\tau+1}
  \frac{\partial \mathcal{R}_{\Omega}(0)}{\partial x_j}\right\}\\&\,\,\,\,\,\,\,\,\,\,\,\,\,\,\,\,\,\,\,\,\,\,\,\,\,\,\,\,\,\,+
\frac{\La_1\La_2(x_{j}-y_{j})}{\pi|x-y|^{2}}
+O\Big( \frac{1}{|\ln \e|}\Big).
\end{split}\end{equation}
\end{small}Also letting $(x,y)=
\big(\e^{\beta }w,
\frac{-\e^{\beta } w+
\e^{2\beta } \gamma }{\tau}\big)$, by Taylor's expansion, we have
\begin{small}
\begin{equation}\label{sec7-09}
\begin{split} \frac{
 \tau \ln  |x| + \ln |y| }{\ln \e+2\pi \mathcal{R}_{\Omega}(0)}
 =
 \frac{ (1+\tau)
(\beta \ln \e + \ln |w|)-\ln \tau }{
\ln \e+2\pi \mathcal{R}_{\Omega}(0)}+O\left(\frac{\e^{\beta}}{|\ln \e|}\right),
\end{split}\end{equation}
\end{small}
\begin{small}\begin{equation}\label{sec7-10}
	\begin{split}
    \frac{x_{j}-y_{j }}{|x-y|^{2}}  =
    &
 \frac{ \tau}{ (1+\tau)^2\e^{\beta}|w|^2 }\left[(1+\tau)w_{j} +
\e^{\beta} \Big(\frac{2(w\cdot\gamma)w_j}{|w|^2}-\gamma_j\Big)
 \right]+
 O\Big( \e^{\beta}\Big).
	\end{split}
\end{equation}
\end{small}Hence inserting \eqref{sec7-09} and \eqref{sec7-10} into \eqref{sec7-08}, we deduce \eqref{sec7-05}.

\vskip 0.1cm

Similarly, from the second estimate of \eqref{sec7-02}, we get
\begin{small}\begin{equation*}
\begin{split}
 \frac{\partial \mathcal{KR}_{\Omega_\e}(x,y)}{\partial x_j}=&
 -
 \frac{\La_2^2}\pi
 \left\{\left[ \frac{
 \tau\ln  |x| + \ln |y|+2\pi \mathcal{R}_\O(0) (1+\tau) }{\ln \e+2\pi \mathcal{R}_{\Omega}(0)}
 \right] \frac{y_j}{|y|^2}-\frac{\pi(\tau^2+\tau+1)}{\tau+1}
  \frac{\partial \mathcal{R}_{\Omega}(0)}{\partial x_j} \right\}\\&\,\,\,\,\,\,\,\,\,\,\,\,\,\,\,\,\,\,\,\,\,\,\,\,+
\frac{\La_1\La_2(y_{j}-x_{j})}{\pi|x-y|^{2}}
+O\Big( \frac{1}{|\ln \e|}\Big).
\end{split}\end{equation*}
\end{small}Also, by Taylor's expansion, we know
\begin{small}\begin{equation}\label{sec7-11}
	\begin{split}
    \frac{y_{j }}{|y|^{2}}  =
    &-
 \frac{ \tau}{ \e^{\beta}|w|^2}\left[ w_{j} +
 \e^{\beta}\Big(\frac{2(w\cdot\gamma)w_{j}}{|w|^2}-\gamma_{j}\Big)
 \right]+
 O\Big( \e^{\beta}\Big).
	\end{split}
\end{equation}\end{small}So from \eqref{sec7-09} and \eqref{sec7-11}, we deduce
\begin{small}\begin{equation*}
\begin{split}
 &\left[ \frac{
 \tau\ln  |x| + \ln |y|+2\pi \mathcal{R}_\O(0) (1+\tau) }{\ln \e+2\pi \mathcal{R}_{\Omega}(0)}
 \right] \frac{y_j}{|y|^2}\\=&-\frac{\tau}{\e^\beta}\left[
 \frac{ (1+\tau)
(\beta \ln \e + \ln |w|+2\pi \mathcal{R}_{\Omega}(0) )-\ln \tau }{
\ln \e+2\pi \mathcal{R}_{\Omega}(0)} \right]\frac{w_j}{|w|^2}
-\frac{\tau^2}{(1+\tau)|w|^2}
\Big[\frac{2(w\cdot\gamma)w_{j}}{|w|^2}-\gamma_{j}\Big]+O\Big( \frac{1}{|\ln \e|}\Big).
\end{split}\end{equation*}
\end{small}Hence \eqref{sec7-06} follows by above computations.
\end{proof}

Now  we define  ${\bf{V}}_\e(w,\gamma)$ on $\mathcal{H}'_\e$ as follows:
\begin{small}
\begin{equation}\label{sec7-12}
{\bf{V}}_\e(w,\gamma)=\Big( \nabla_x\mathcal{KR}_{\Omega_\e}(x,y),\nabla_y\mathcal{KR}_{\Omega_\e}(x,y) \Big)\Big|_{(x,y)=
 \Big(\e^{\beta} w,\frac{-\e^{\beta} w
 +\e^{2\beta} \gamma }{\tau}\Big)}.
\end{equation}
\end{small}Also
 $\widetilde{\bf{V}}_\e(w,\gamma)$ is given by
\begin{equation*}
\widetilde{\bf{V}}_\e(w,\gamma)=\Big(\widetilde{\bf{V}}_{\e,1}(w,\gamma), \widetilde{\bf{V}}_{\e,2}(w,\gamma)  \Big),
\end{equation*}
with
\begin{small}\begin{align*}
\begin{cases}
\widetilde{\bf{V}}_{\e,1}(w,\gamma)=- \left[
\frac{k( |w|,\tau)}{|w|^2\e^{\beta}(\ln \e+2\pi \mathcal{R}_{\Omega}(0))}-
 2\beta  \frac{(w\cdot \gamma)}{|w|^4}
\right]w
+\frac{\pi(1+\tau+\tau^2)}{1+\tau}
\nabla \mathcal{R}_\Omega(0)-\frac{\beta}{|w|^2}\gamma, \\[4mm]
\widetilde{\bf{V}}_{\e,2}(w,\gamma)=\left[
\frac{\pi k( |w|,\tau)}{|w|^2\e^{\beta}(\ln \e+2\pi \mathcal{R}_{\Omega}(0))}+
 2\beta \tau^2 \frac{(w\cdot \gamma)}{|w|^4}
\right]w
+\frac{\pi(1+\tau+\tau^2)}{1+\tau}
\nabla \mathcal{R}_\Omega(0)-\frac{\tau^2\beta}{|w|^2}\gamma,
\end{cases}
\end{align*}\end{small}and $\tau=\frac{\La_1}{\La_2}$, $\beta=\frac{\tau}{(\tau+1)^2}$, $k(r,\tau)$ is the function in \eqref{sec7-07}. Then Proposition \ref{sec7-prop7.5} means that
\begin{small}\begin{equation}\label{sec7-13}
{\bf{V}}_{\e}(w,\gamma)=\widetilde{\bf{V}}_{\e}(w,\gamma)  \left(\begin{array}{cc}
    \frac{\La_1\La_2}{\pi} {\bf{E}}_{2\times 2} &
   {\bf{O}}_{2\times 2}\\[3mm] {\bf{O}}_{2\times 2}
    &  \frac{\La_2^2}{\pi} {\bf{E}}_{2\times 2}
  \end{array}\right)+
O\Big( \frac{1}{|\ln \e|}\Big)~~~\mbox{for any}~~~ (w,\gamma)\in  \mathcal{H}'_\e,
\end{equation}\end{small}where ${\bf{E}}_{2\times 2}=\left(
                                            \begin{array}{cc}
                                              1 & 0 \\[2mm]
                                              0 & 1 \\
                                            \end{array}
                                          \right)
$ and ${\bf{O}}_{2\times 2}=\left(
                              \begin{array}{cc}
                                0 & 0 \\[2mm]
                                0 & 0 \\
                              \end{array}
                            \right)
$. Furthermore we give a $C^1$-estimate of \eqref{sec7-05} and \eqref{sec7-06}.
\begin{prop}\label{sec7-teo7.6}
For any $(w,\gamma)\in  \mathcal{H}'_\e$,
it holds\begin{small}
\begin{equation*}
\nabla_{(w,\gamma)}{\bf{V}}_{\e}(w,\gamma)=\nabla_{(w,\gamma)} \widetilde{\bf{V}}_{\e}(w,\gamma)  \left(\begin{array}{cc}
    \frac{\La_1\La_2}{\pi} {\bf{E}}_{2\times 2} &
   {\bf{O}}_{2\times 2}\\[3mm] {\bf{O}}_{2\times 2}
    &  \frac{\La_2^2}{\pi} {\bf{E}}_{2\times 2}
  \end{array}\right)+
  O\Big( \frac{1}{|\ln \e|}\Big).
\end{equation*}
\end{small}
\end{prop}
\begin{proof} First, we denote by
 \begin{small}\begin{equation*}{\bf{V}}_{\e}(w,\gamma)=\big({\bf{V}}_{\e,1}(w,\gamma),{\bf{V}}_{\e,2}(w,\gamma)\big)\,\,
\mbox{and}\,\,
{\bf{V}}_{\e,m}(w,\gamma)=\big({{V}}_{\e,m,1}(w,\gamma),{{V}}_{\e,m,2}(w,\gamma)\big)\,\,
\mbox{with}\,\, m=1,2.\end{equation*}\end{small}
Next we have by Lemma \ref{app-lem-A.6},
\begin{small}\begin{align*}
\frac{\partial {{V}}_{\e,1,j}(w,\gamma)}{\partial w_i}
=&\frac{\partial}{\partial w_i}\left[\frac{\partial \mathcal{KR}_{\Omega_\e}(x,y)}
{\partial x_j}\Big|_{(x,y)=
 \Big(\e^{\beta }w
 ,\frac{-\e^{\beta}
 w+\e^{2\beta} \gamma }{\tau}\Big)}\right]\\
=&
\left[\e^{\beta }
\frac{\partial^2 \mathcal{KR}_{\Omega_\e}(x,y)}{\partial x_i\partial  x_j}
-\frac{\e^{\beta}}{\tau}
\frac{\partial^2 \mathcal{KR}_{\Omega_\e}(x,y)}{\partial y_i\partial  x_j}
\right]\Big|_{(x,y)=
 \Big(\e^{\beta }w,\frac{-\e^{\beta }w
 +\e^{2\beta }\gamma }{\tau}\Big) }\\
=&
\e^{\beta} \left[
\frac{\partial^2 \mathcal{KR}_{(B(0,\e))^c}(x,y)}{\partial x_j\partial x_i}+
  \frac{\partial\Psi_{\e,j}(x,y)}{\partial x_i} \right]\Big|_{(x,y)=
 \Big(\e^{\beta }w,
 \frac{-\e^{\beta }w+\e^{2\beta }
 \gamma}{\tau}\Big)}\\&-
\frac{\e^{\beta}}{\tau}\left[
\frac{\partial^2 \mathcal{KR}_{(B(0,\e))^c}(x,y)}{\partial x_j\partial y_i}+
 \frac{\partial\Psi_{\e,j}(x,y)}{\partial y_i}  \right]\Big|_{(x,y)=
 \Big(\e^{\beta }w,\frac{-\e^{\beta }w+
 \e^{2\beta } \gamma }{\tau}\Big)}
 +O\left(\frac{\e^{1-\beta }}{|\ln \e|}\right) \\
=&
\frac{\partial}{\partial w_i} \left\{ \Big[
\frac{\partial \mathcal{KR}_{(B(0,\e))^c}(x,y)}{\partial x_j }+\Psi_{\e,j}(x,y) \Big]
 \Big|_{(x,y)=
 \Big(\e^{\beta }w,
 \frac{-\e^{\beta }w+\e^{2\beta }
 \gamma }{\tau}\Big)}\right\}
 +O\left(\frac{\e^{1-\beta }}{|\ln \e|}\right),
\end{align*}
\end{small}and\begin{small}
\begin{align*}
\frac{\partial {V}_{\e,1,j}(w,\gamma)}{\partial \gamma_i}=&\frac{\partial}{\partial \gamma_i} \left[\frac{\partial \mathcal{KR}_{\Omega_\e}(x,y)}{\partial x_j}\Big|_{(x,y)=
 \Big(\e^{\beta }w,\frac{-\e^{\beta } w+
 \e^{2\beta } \gamma }{\tau}\Big)}\right]\\=&
  \frac{\e^{\beta}}{\tau}
\frac{\partial^2 \mathcal{KR}_{\Omega_\e}(x,y)}{\partial y_i\partial  x_j}
 \Big|_{(x,y)=
 \Big(\e^{\beta }w,
 \frac{-\e^{\beta }w+\e^{2\beta } \gamma }{\tau}\Big)}
  \\=&-
\frac{\e^{1-2\beta}}{\tau}\left[
\frac{\partial^2 \mathcal{KR}_{(B(0,\e))^c}(x,y)}{\partial x_j\partial y_i}
 +  \frac{\partial \Psi_{\e,j}(x,y)}{\partial y_i}\right]\Big|_{(x,y)=
 \Big(\e^{\beta }w,\frac{-\e^{\beta }w+
 \e^{2\beta } \gamma }{\tau}\Big)}+
 O\left(\frac{\e^{1-\beta }}{|\ln \e|}\right)\\=&
\frac{\partial}{\partial \gamma_i} \left\{ \left[
\frac{\partial \mathcal{KR}_{(B(0,\e))^c}(x,y)}{\partial x_j }+\Psi_{\e,j}(x,y)
 \right]\Big|_{(x,y)=
 \Big(\e^{\beta }w,\frac{-\e^{\beta }w+
 \e^{2\beta }
 \gamma }{\tau}\Big)}\right\}
 +O\left(\frac{\e^{1-\beta }}{|\ln \e|}\right),
\end{align*}
\end{small}where $\Psi_{\e,j}(x,y)$ is the function in \eqref{App-A.13}.
Hence for any $(w,\gamma)\in \mathcal{H}'_\e$, it holds
\begin{small}\begin{equation}\label{sec7-14}\begin{split}
\nabla_{(w,\gamma)} {V}_{\e,1,j}(w,\gamma)=&
\nabla_{(w,\gamma)}  \left\{ \left[
\frac{\partial \mathcal{KR}_{(B(0,\e))^c}(x,y)}{\partial x_j }+\Psi_{\e,j}(x,y)
\right]\Big|_{(x,y)=\Big(\e^{\beta }w,\frac{-\e^{\beta }
w+\e^{2\beta } \gamma }{\tau}\Big)} \right\}
+O\left(\frac{\e^{1-\beta }}{|\ln \e|}\right).
\end{split}\end{equation}\end{small}From \eqref{sec2-03} and \eqref{App-A.13},  for any $(w,\gamma)\in \mathcal{H}'_\e$ we can compute directly that
\begin{small}\begin{align*}
&\frac{\partial}{\partial w_i} \left[ \frac{\partial \mathcal{KR}_{(B(0,\e))^c}(x,y)}{\partial x_j}
\Big|_{(x,y) =\Big(\e^{\beta }w,\frac{-\e^{\beta }w+
\e^{2\beta } \gamma }{\tau}\Big) } \right] \\[2mm]& = -\frac{\La_1\La_2}{\pi}
  \frac{\partial}{\partial w_i}\left[ \Big(  \frac{(1+\tau)(1-\beta)}{  \e^{\beta} |w|^2}
   -\frac{2\beta (w\cdot\gamma) }{|w|^4}
   \Big) w_j  - \frac{\beta }{  |w|^2} \gamma_j \right]
 + O\left(\e^{\beta}\right),
\end{align*}
\end{small}
\begin{small}\begin{align*}
&\frac{\partial}{\partial \gamma_i} \left[ \frac{\partial \mathcal{KR}_{(B(0,\e))^c}(x,y)}
{\partial x_j}\Big|_{(x,y) =\Big(\e^{\beta }w,\frac{-\e^{\beta }
w+\e^{2\beta } \gamma }{\tau}\Big) } \right]=
\frac{\La_1 \La_2 \beta}{\pi}
 \left[\delta_{ij}-\frac{2w_iw_j}{|w|^2}\right]
 + O\left(\e^{\beta}\right),
 \end{align*} \end{small}
\begin{small}\begin{align*}
&\frac{\partial}{\partial w_i}\left[\Psi_{\e,j}(x,y)
\Big|_{(x,y)=\Big(\e^{\beta }w,\frac{-\e^{\beta } w
+\e^{2\beta } \gamma }{\tau}\Big)}   \right]=-\frac{\La_1\La_2}{\pi\e^{\beta }}
  \frac{\partial}{\partial w_i}\left[
 \frac{((1+\tau)\ln |w|+(\beta-1)\ln \e)w_j}{(\ln \e +2\pi \mathcal{R}_{\Omega}(0)) |w|^2 } \right]
  + O\left(\frac{1}{|\ln \e|}\right),
\end{align*}
\end{small}and
\begin{small}\begin{align*}
&\frac{\partial}{\partial \gamma_i}\left[\Psi_{\e,j}(x,y)
\Big|_{(x,y)=\Big(\e^{\beta }w,\frac{-\e^{\beta }w+
\e^{2\beta } \gamma }{\tau}\Big)}   \right]=
 O\left(\frac{1}{|\ln \e|}\right).
\end{align*}
\end{small}Thus from above computations, we deduce that for any $(w,\gamma)\in \mathcal{H}'_\e$,
\begin{small}\begin{align*}
&\nabla_{(w,\gamma)} {\bf{V}}_{\e,1}(w,\gamma)=  \frac{\La_1\La_2}{\pi}
\nabla_{(w,\gamma)} \widetilde{\bf{V}}_{\e,1}(w,\gamma)+
O\left(\frac{1}{|\ln \e|}\right).
\end{align*}
\end{small}Similarly, for any $(w,\gamma)\in \mathcal{H}'_\e$, we have
\begin{small}\begin{align*}
&\nabla_{(w,\gamma)} {\bf{V}}_{\e,2}(w,\gamma) = \frac{\La_2^2}{\pi} \nabla_{(w,\gamma)} \widetilde{\bf{V}}_{\e,2}(w,\gamma)
  +O\left(\frac{1}{|\ln \e|}\right).
\end{align*}
\end{small}Hence we complete the proofs of Proposition \ref{sec7-teo7.6}.
\end{proof}

\vskip 0.2cm

\subsection{The case $\Lambda_1\ne \Lambda_2$ and $\nabla \mathcal{R}_\Omega(0)\ne 0$ (Proof of Theorem \ref{sec1-teo15})} ~\

 \vskip 0.2cm

In the case $\La_1\neq \La_2$ and $\nabla\mathcal{R}_\O(0)\ne 0$, the expansion in
\eqref{sec7-13} is sufficient.

\vskip 0.1cm

From \eqref{sec7-13} and Proposition \ref{sec7-teo7.6}, it is essential to consider the solution of $\widetilde{\bf{V}}_{\e}(w,\gamma)=0$.
We write
$\widetilde{\bf{V}}_{\e}(w,\gamma)=0$ as for $j=1,2$,
\begin{small}\begin{align}\label{sec7-15}
& \left(
   \begin{array}{cc}
 \displaystyle
  \frac{\pi(1+\tau+\tau^2)}{1+\tau}
 &- \beta  \\[3mm]
\displaystyle
  \frac{\pi(1+\tau+\tau^2)}{1+\tau}
 &- \tau^2\beta \\
   \end{array}
 \right)\left(
          \begin{array}{c}
            \frac{\partial \mathcal{R}_\O(0)}{\partial x_j} \\[3mm]
           \frac{ \gamma_{j}}{|w|^2} \\
          \end{array}
        \right)=
        \left(
          \begin{array}{c}
\left[
\frac{k( |w|,\tau)}{|w|^2\e^{\beta}(\ln \e+2\pi \mathcal{R}_{\Omega}(0))}-
 2\beta  \frac{(w\cdot \gamma)}{|w|^4}
\right] w_j
\\[3mm]\left[-
\frac{\tau k( |w|,\tau)}{|w|^2\e^{\beta}(\ln \e+2\pi \mathcal{R}_{\Omega}(0))}-
 2\beta \tau^2 \frac{(w\cdot \gamma)}{|w|^4}
\right] w_j \\
          \end{array}
        \right).
\end{align}\end{small}We denote the matrix in the left hand side of \eqref{sec7-15} by  $\textbf{Q}$. Then $det~\textbf{Q}\neq 0\Leftrightarrow \tau\neq 1(\La_1\neq \La_2)$. More importantly, if $det~\textbf{Q}\neq 0$
 and $\nabla\mathcal{R}_\O(0)\ne 0$, it holds that
$w\parallel\nabla \mathcal{R}_\O(0)$,  and $\gamma \parallel\nabla \mathcal{R}_\O(0).$
This gives the direction of the solution $(w,\gamma)$ of $\widetilde{\bf{V}}_{\e}(w,\gamma)=0$.
\vskip 0.2cm

\begin{prop}\label{sec7-prop7.7}
If $\La_1\neq \La_2$ and $\nabla \mathcal{R}_{\Omega}(0)\neq 0$, then
 $\widetilde{\bf{V}}_{\e}(w,\gamma)=0$ possesses exactly  two solutions
 $ (\widetilde{w}_\e^{(1)}, \widetilde{\gamma}_\e^{(1)}    )$  and
$ (\widetilde{w}_\e^{(2)}, \widetilde{\gamma}_\e^{(2)}    )$,
 satisfying
 \begin{small}
  \begin{equation*}
  \frac{\widetilde{w}_\e^{(1)}}{|\widetilde{w}_\e^{(1)}|}
  =\frac{\nabla \mathcal{R}_\O(0)}{|\nabla \mathcal{R}_\O(0)|},\quad
  \frac{\widetilde{w}_\e^{(2)}}{|\widetilde{w}_\e^{(2)}|}
  =-\frac{\nabla \mathcal{R}_\O(0)}{|\nabla \mathcal{R}_\O(0)|}
  \end{equation*}
\end{small}and
\begin{small} \begin{equation}\label{sec7-18}
  \begin{split}
det&~Jac~ \widetilde{\bf{V}}_{\e}(\widetilde{w}_\e^{(m)},\widetilde{\gamma}_\e^{(m)})
= \frac{\pi \tau^2 (\tau^3-1)  }{(C_\tau)^5\e^{\beta }  \ln \e}\left(  1+O\Big(\frac{1}{|\ln \e|}\Big) \right) (-1)^{m-1},
\end{split}
 \end{equation}\end{small}where $\tau=\frac{\La_1}{\La_2}$, $\beta=\frac{\tau}{(\tau+1)^2}$ and $C_\tau:=  \tau^{\frac{1}{1+\tau}}
e^{-\frac{2\pi \mathcal{R}_{\Omega}(0)(\tau^2+ \tau+1 )}{(1+\tau)^2}}$.
\end{prop}

\begin{proof}

If $\La_1\neq \La_2$ and $\nabla \mathcal{R}_{\Omega}(0)\neq 0$, then $det~\mathbf{Q} \neq 0$ and \eqref{sec7-15} implies
  \begin{equation}\label{sec7-19}
 w\parallel\nabla \mathcal{R}_\O(0)\quad\hbox{and}\quad\gamma \parallel\nabla \mathcal{R}_\O(0).
 \end{equation}
Next we split the proof in three different steps.
\vskip 0.2cm

\noindent\textbf{Step 1: Computation of $\gamma$ in \eqref{sec7-19}.}

\vskip 0.2cm

 By \eqref{sec7-15}, we write $\gamma =\big(\frac{w}{|w|}\cdot \gamma\big)\frac{w}{|w|}  $  and $\nabla \mathcal{R}_\O(0) =\big(\frac{w}{|w|}\cdot \nabla \mathcal{R}_\O(0)\big)\frac{w}{|w|}  $.  Then \eqref{sec7-15}
 is equivalent to
\begin{small}
\begin{equation}\label{sec7-20}
\begin{split}
\frac{{k}\big( |w|,\tau\big)}{\e^{ \beta }(\ln \e+2\pi \mathcal{R}_{\Omega}(0)) }=&
\frac{\pi(1+\tau+\tau^2)}{1+\tau}
\big(\nabla \mathcal{R}_{\Omega}(0) \cdot w\big)
+  \frac{\beta }{ |w|^2}
\big(w\cdot \gamma \big),
\end{split}
\end{equation}
\end{small}and\begin{small}
\begin{equation}\label{sec7-21}
\begin{split}
\frac{\tau {k}\big( |w|,\tau\big)}{\e^{ \beta }(\ln \e+2\pi \mathcal{R}_{\Omega}(0)) }=&-
\frac{\pi(1+\tau+\tau^2)}{1+\tau}
\big(\nabla \mathcal{R}_{\Omega}(0) \cdot w\big)
- \frac{\beta \tau^2}{ |w|^2}
\big(w\cdot \gamma \big).
\end{split}
\end{equation}
\end{small}Hence from $\tau \times \eqref{sec7-20}-\eqref{sec7-21}$, we get
\begin{small}\begin{equation}\label{sec7-22}
  \begin{split}
\pi (1+\tau+\tau^2)
\big(\nabla \mathcal{R}_{\Omega}(0) \cdot w\big)
=-\frac{\beta(\tau+\tau^2) }{ |w|^2}
\big(w\cdot \gamma \big).
 \end{split}
\end{equation}
\end{small}Inserting $\frac w{|w|}=\pm \frac {\nabla \mathcal{R}_\O(0)}{|\nabla \mathcal{R}_\O(0)|}
 $ into \eqref{sec7-22}, we find

 \[
 \pi (1+\tau+\tau^2)
|\nabla \mathcal{R}_{\Omega}(0) | \cdot|w|
=-\frac{\beta(\tau+\tau^2) }{ |w|}
\frac {\nabla \mathcal{R}_\O(0)}{|\nabla \mathcal{R}_\O(0)|}\cdot \gamma,
\]
 which, together with
  $ \gamma \parallel \,\nabla \mathcal{R}_\O(0)$, gives $\gamma =\big(\frac {\nabla \mathcal{R}_\O(0)}{|\nabla \mathcal{R}_\O(0)|}\cdot \gamma \big) \frac {\nabla \mathcal{R}_\O(0)}{|\nabla \mathcal{R}_\O(0)|} =-\frac{\pi(1+\tau+\tau^2)(1+\tau)|w|^2}{\tau^2}
 \nabla \mathcal{R}_\Omega(0)$.

\vskip 0.2cm

\noindent\textbf{Step 2: Computation of $w$ in \eqref{sec7-19}.}

\vskip 0.2cm
As stated above, we know that $w\parallel  \nabla \mathcal{R}_\O(0)$. Hence $w$
has exact two directions. The crucial point is to solve the length of $w$.
Inserting \eqref{sec7-22} into  \eqref{sec7-20}, we obtain
\begin{small}\begin{equation}\label{sec7-23}
\begin{split}
k\big( |w|,\tau\big)=&\pi d_\tau
\big(\nabla \mathcal{R}_{\Omega}(0) \cdot w\big)\e^{\beta}\big(\ln \e+2\pi \mathcal{R}_{\Omega}(0)\big)~\,\,~~\mbox{with}~~\,\,~d_\tau:=
\frac{(\tau^3-1)}{(1+\tau)\tau}.
\end{split}
\end{equation}
\end{small}We have the  following alternative.

\vskip 0.1cm

\noindent \textbf{Case 1.  $\frac{w}{|w|}=\frac{\nabla \mathcal{R}_\O(0)}{|\nabla \mathcal{R}_\O(0)|}$}

\vskip 0.1cm

In this case, \eqref{sec7-23} becomes
 \begin{small}$$\breve{k}_\e(r):=k(r,\tau)-\pi
d_\tau
\big|\nabla \mathcal{R}_{\Omega}(0) \big|r\e^{\beta}\big(\ln \e+2\pi \mathcal{R}_{\Omega}(0)\big)=0,$$
\end{small}where $k(r,\tau)$ is defined in \eqref{sec7-07}.
Then $\breve{k}_\e(r)=0$ possesses exact one solution $r_\e$. In fact, from
$\frac{\partial {k}(r,\tau)}{\partial r} >0$, and
\begin{small}\begin{equation*}
\begin{split}\breve{k}_\e(C_\tau-|\ln\e|^2\e^{\beta })<0, ~~
\breve{k}_\e(C_\tau+|\ln\e|^2\e^{\beta })>0,
\end{split}
\end{equation*}\end{small}we see that $\breve{k}_\e(r)=0$ possesses exact one solution $r_\e^{(1)}$ satisfying
\begin{small}\begin{equation*}r_\e^{(1)}=C_\tau+\frac{\pi (C_\tau)^2 d_\tau |\nabla \mathcal{R}_\O(0)|(\ln \e +2\pi \mathcal{R}_{\Omega}(0))}{1+\tau}\e^{\beta} +O\big(\e^{2\beta}|\ln\e|^2\big),\end{equation*}\end{small}where $C_\tau$ and $d_\tau$ are the constants in Theorem \ref{sec1-teo15} and \eqref{sec7-23}.
\vskip 0.2cm

\noindent \textbf{Case 2.  $\frac{w}{|w|}=-\frac{\nabla \mathcal{R}_\O(0)}{|\nabla \mathcal{R}_\O(0)|}$}

\vskip 0.2cm

In this case, \eqref{sec7-23} becomes
\begin{small}\begin{equation}\label{sec7-24}
\begin{split}
k(r,\tau)=-
d_\tau
\big|\nabla \mathcal{R}_{\Omega}(0) \big|r\e^{\beta}\big(\ln \e+2\pi \mathcal{R}_{\Omega}(0)\big).
\end{split}
\end{equation}
\end{small}In a similar way we deduce that
\eqref{sec7-24} has a unique solution $r_\e^{(2)}$ satisfying
\begin{small}\begin{equation*}
r_\e^{(2)}=C_\tau-\frac{\pi (C_\tau)^2 d_\tau |\nabla \mathcal{R}_\O(0)|(\ln \e +2\pi \mathcal{R}_{\Omega}(0))}{1+\tau}\e^{\beta} +O\big(\e^{2\beta}|\ln\e|^2\big).
\end{equation*}
\end{small}Hence from above discussions, we know that
 $\widetilde{\bf{V}}_{\e}(w,\gamma)=0$ possesses exactly two solutions $(\widetilde{w}_\e^{(1)},\widetilde{\gamma}_\e^{(1)})$ and  $(\widetilde{w}_\e^{(2)},\widetilde{\gamma}_\e^{(2)})$ satisfying
 \begin{small}\begin{equation}\label{sec7-16}
 \begin{cases}
\widetilde{w}_\e^{(1)} = \left[C_\tau+\frac{\pi (C_\tau)^2 d_\tau |\nabla \mathcal{R}_\O(0)|(\ln \e +2\pi \mathcal{R}_{\Omega}(0))}{1+\tau}\e^{\beta} +O\big(\e^{2\beta}|\ln\e|^2\big) \right]
 \frac{\nabla \mathcal{R}_\O(0)}{|\nabla \mathcal{R}_\O(0)|},\\[3mm]
\widetilde{\gamma}^{(1)}_\e=-\left[
\frac{\pi(1+\tau+\tau^2)(1+\tau)|\nabla \mathcal{R}_\Omega(0)|(C_\tau)^2}{\tau^2}
 + O\big(\e^{\beta}|\ln\e|\big) \right]
\nabla \mathcal{R}_\Omega(0),
\end{cases}\end{equation}
\end{small}and
\begin{small}\begin{equation}\label{sec7-17}
 \begin{cases}
\widetilde{w}_\e^{(2)} = -\left[C_\tau-\frac{\pi (C_\tau)^2 d_\tau |\nabla \mathcal{R}_\O(0)|(\ln \e +2\pi \mathcal{R}_{\Omega}(0))}{1+\tau}\e^{\beta} +O\big(\e^{2\beta}|\ln\e|^2\big) \right]
\frac{\nabla \mathcal{R}_\O(0)}{|\nabla \mathcal{R}_\O(0)|},\\[3mm]
\widetilde{\gamma}^{(2)}_\e=-\left[
\frac{\pi(1+\tau+\tau^2)(1+\tau)|\nabla \mathcal{R}_\Omega(0)|(C_\tau)^2}{\tau^2}
+ O\big(\e^{\beta}|\ln\e|\big) \right]
 \nabla \mathcal{R}_\Omega(0).
\end{cases}\end{equation}
\end{small}
 \vskip 0.2cm

\noindent\textbf{Step 3: Proof of \eqref{sec7-18}.}

\vskip 0.2cm

Let $\widetilde{\bf{V}}_{\e}(w,\gamma)=
\big(\widetilde{\bf{V}}_{\e,1}(w,\gamma),\widetilde{\bf{V}}_{\e,2}(w,\gamma)\big)$ and
$\widetilde{\bf{V}}_{\e,j}(w,\gamma)=
\big(\widetilde{{V}}_{\e,j,1}(w,\gamma),\widetilde{V}_{\e,j,2}(w,\gamma)\big)$ for $j=1,2$.
Then for $i,j=1,2$, we compute
\begin{small}
\begin{equation*}
\begin{split}
\frac{\partial \widetilde{V}_{\e,1,i}(w,\gamma)}{\partial w_j}
=&-\left[
\frac{k( |w|,\tau)}{|w|^2\e^{\beta}(\ln \e +2\pi \mathcal{R}_{\Omega}(0))}-
 2\beta  \frac{(w\cdot \gamma)}{|w|^4}
\right]\delta_{ij}
- \left[\frac{1}{|w|\e^{\beta}(\ln \e +2\pi \mathcal{R}_{\Omega}(0))}
\frac{ \partial {k}(r,\tau)}{\partial r}\big|_{r=
|w|}\right.
\\[2mm]&
\left.
-\frac{2 {k}(|w|,\tau)}{ |w|^2\e^{\beta}(\ln \e +2\pi \mathcal{R}_{\Omega}(0)) }
+\frac{8 (w\cdot \gamma )  }{ |w|^4}\right]\frac{w_iw_j}{|w|^2}
+  \frac{2 \beta(   w_i \gamma_j +w_j\gamma_i )}{ |w|^4},\end{split}\end{equation*}\end{small}and
 \begin{small}
 \begin{equation*}
\begin{split}
 \frac{\partial \widetilde{V}_{\e,1,i}(w,\gamma)}{\partial \gamma_j}
=-\frac{ \beta }{ |w|^2}
 \delta_{ij}+\frac{ 2\beta w_iw_j}{ |w|^4}.
\end{split}
\end{equation*}
\end{small}By
\eqref{sec7-20} and \eqref{sec7-22}, we have
\begin{small}
\begin{equation}\label{sec7-25}
\frac{k( |\widetilde{w}^{(m)}_\e|,\tau)}{|\widetilde{w}^{(m)}_\e|^2\e^{\beta}(\ln \e +2\pi \mathcal{R}_{\Omega}(0))}-
 2\beta  \frac{(\widetilde{w}^{(m)}_\e\cdot \widetilde{\gamma}^{(m)}_\e)}{|\widetilde{w}^{(m)}_\e|^4}=
 \pi(1+\tau+\tau^2)
 \frac{(\nabla \mathcal{R}_{\Omega}(0) \cdot \widetilde{w}^{(m)}_\e )}{|\widetilde{w}^{(m)}_\e|^2}.
 \end{equation}
\end{small}Also, using that $w^{(m)}_\e\parallel   \widetilde{\gamma}^{(m)}_\e$,  \eqref{sec7-07} and \eqref{sec7-25}, we obtain
\begin{small}\begin{equation*}
 \frac{\partial \widetilde{V}_{\e,1,i}(w,\gamma)}{\partial w_j}\Big|_{(w,\gamma)=(\widetilde{w}_\e^{(m)},\widetilde{\gamma}_\e^{(m)})}
=c_1(\widetilde{w}_\e^{(m)},\widetilde{\gamma}_\e^{(m)}) \delta_{ij}+
c_2(\widetilde{w}_\e^{(m)},\widetilde{\gamma}_\e^{(m)})\frac{\widetilde{w}_{\e,i}^{(m)}\widetilde{w}_{\e,j}^{(m)}}{|\widetilde{w}_\e^{(m)}|^2},
\end{equation*}
\end{small}with
\begin{small}
\begin{equation*}
\begin{cases}
c_1(\widetilde{w}_\e^{(m)},\widetilde{\gamma}_\e^{(m)})=  -\frac{ \pi(1+\tau+\tau^2)}{\tau (C_\tau)^2}
 \Big(1+O(\frac{1}{|\ln\e|})\Big)(\nabla \mathcal{R}_{\Omega}(0) \cdot \widetilde{w}^{(m)}_\e ),\\[2mm] c_2(\widetilde{w}_\e^{(m)},\widetilde{\gamma}_\e^{(m)})=- \frac{1+\tau}{(C_\tau)^2\e^{\beta}\ln\e}\Big(1+O(\frac{1}{|\ln\e|})\Big).
\end{cases}\end{equation*}
\end{small}

On the other hand, direct computations give
\begin{small}\begin{equation*}
 \frac{\partial \widetilde{V}_{\e,1,i}(w,\gamma)}{\partial \gamma_j}\Big|_{(w,\gamma)=(\widetilde{w}_\e^{(m)},\widetilde{\gamma}_\e^{(m)})}
=\underbrace{- \frac{ \beta }{ |\widetilde{w}_\e^{(m)}|^2}}_{:=c_3(\widetilde{w}_\e^{(m)},\widetilde{\gamma}_\e^{(m)})} \delta_{ij}+
\underbrace{\frac{ 2\beta }{|\widetilde{w}_\e^{(m)}|^2}}_{:=c_4(\widetilde{w}_\e^{(m)},\widetilde{\gamma}_\e^{(m)})}\frac{\widetilde{w}_{\e,i}^{(m)}
\widetilde{w}_{\e,j}^{(m)}}{|\widetilde{w}_\e^{(m)}|^2}.
\end{equation*}
\end{small}Similarly, we compute
\begin{small}
\begin{equation*}
\begin{split}
\frac{\partial \widetilde{V}_{\e,2,i}(w,\gamma)}{\partial w_j}
=&\left[
\frac{\tau k( |w|,\tau)}{|w|^2\e^{\beta}(\ln \e +2\pi \mathcal{R}_{\Omega}(0))}+
 2\beta  \tau^2\frac{(w\cdot \gamma)}{|w|^4}
\right]\delta_{ij}
+ \left[\frac{\tau}{|w|\e^{\beta}(\ln \e +2\pi \mathcal{R}_{\Omega}(0))}
\frac{ \partial {k}(r,\tau)}{\partial r}\big|_{r=
|w|} \right. \\[2mm]&
\left.-\frac{2 {k}(|w|,\tau)}{ |w|^2\e^{\beta}(\ln \e +2\pi \mathcal{R}_{\Omega}(0)) }
-\frac{8 \tau^2(w\cdot \gamma )  }{ |w|^4}\right]\frac{w_iw_j}{|w|^2}
 + \frac{2\tau^2 \beta(   w_i \gamma_j +w_j\gamma_i )}{ |w|^4},
 \end{split}\end{equation*}\end{small}and
 \begin{small}\begin{equation*}
\begin{split}
 \frac{\partial \widetilde{V}_{\e,2,i}(w,\gamma)}{\partial \gamma_j}
=-\frac{ \tau^2\beta }{ |w|^2}
 \delta_{ij}+\frac{ 2\tau^2\beta w_iw_j}{ |w|^4}.
\end{split}
\end{equation*}
\end{small}So we obtain
\begin{small}\begin{equation*}
 \frac{\partial \widetilde{V}_{\e,2,i}(w,\gamma)}{\partial w_j}\Big|_{(w,\gamma)=(\widetilde{w}_\e^{(m)},\widetilde{\gamma}_\e^{(m)})}
=c_5(\widetilde{w}_\e^{(m)},\widetilde{\gamma}_\e^{(m)}) \delta_{ij}+
c_6(\widetilde{w}_\e^{(m)},\widetilde{\gamma}_\e^{(m)}) \frac{\widetilde{w}_{\e,i}^{(m)}\widetilde{w}_{\e,j}^{(m)}}{|\widetilde{w}_\e^{(m)}|^2},
\end{equation*}
\end{small}with
\begin{small}
\begin{equation*}
\begin{cases}
c_5(\widetilde{w}_\e^{(m)},\widetilde{\gamma}_\e^{(m)})
= -\frac{\pi(1+\tau+\tau^2)}{(C_\tau)^2}
 \Big(1+O(\frac{1}{|\ln\e|})\Big)\big(\nabla \mathcal{R}_{\Omega}(0) \cdot \widetilde{w}^{(m)}_\e \big),\\[2mm] c_6(\widetilde{w}_\e^{(m)},\widetilde{\gamma}_\e^{(m)})=\frac{\tau(1+\tau)}{(C_\tau)^2\e^{\beta}\ln\e}\Big(1+O(\frac{1}{|\ln\e|})\Big).
\end{cases}\end{equation*}
\end{small}We also have
\begin{small}\begin{equation*}
 \frac{\partial \widetilde{V}_{\e,2,i}(w,\gamma)}{\partial \gamma_j}\Big|_{(w,\gamma)=(\widetilde{w}_\e^{(m)},\widetilde{\gamma}_\e^{(m)})}
=c_7(\widetilde{w}_\e^{(m)},\widetilde{\gamma}_\e^{(m)}) \delta_{ij}+c_8(\widetilde{w}_\e^{(m)},\widetilde{\gamma}_\e^{(m)})\frac{\widetilde{w}_{\e,i}^{(m)}
\widetilde{w}_{\e,j}^{(m)}}{|\widetilde{w}_\e^{(m)}|^2},
\end{equation*}
\end{small}with
$c_7(\widetilde{w}_\e^{(m)},\widetilde{\gamma}_\e^{(m)}) =\tau^2c_3(\widetilde{w}_\e^{(m)},\widetilde{\gamma}_\e^{(m)})$ and $ c_8(\widetilde{w}_\e^{(m)},\widetilde{\gamma}_\e^{(m)}) =\tau^2c_4(\widetilde{w}_\e^{(m)},\widetilde{\gamma}_\e^{(m)})$.

\vskip 0.1cm

From the above computations, we have, for $m=1,2$,
\begin{small}\begin{equation*}
\begin{split}
& Jac~ \widetilde{\bf{V}}_{\e}(\widetilde{w}_\e^{(m)},\widetilde{\gamma}_\e^{(m)}) =
   \left(\begin{array}{cc}
   \Big( c_1\delta_{ij}+c_2\frac{w_iw_j}{|w|^2}\Big)_{1\leq i,j\leq 2}&   \Big( c_3\delta_{ij}+c_4\frac{w_iw_j}{|w|^2}\Big)_{1\leq i,j\leq 2} \\[6mm]
  \Big(c_5\delta_{ij}+c_6\frac{w_iw_j}{|w|^2}\Big)_{1\leq i,j\leq 2} & \Big(c_7\delta_{ij}+c_8\frac{w_iw_j}{|w|^2}\Big)_{1\leq i,j\leq 2}
  \end{array}\right)\Big|_{(w,\gamma)=\big(\widetilde{w}_\e^{(m)},\widetilde{\gamma}_\e^{(m)} \big)}.
\end{split}\end{equation*}
\end{small}Then we get
\begin{small}\begin{equation*}
\begin{split}
 det&~Jac~ \widetilde{\bf{V}}_{\e}(\widetilde{w}_\e^{(m)},\widetilde{\gamma}_\e^{(m)}) =
\left[ \Big(c_1+c_2\Big)\times \Big(c_7+c_8\Big)-\Big(c_3+c_4\Big)\times \Big(c_5+c_6\Big)\right]\times \left[ c_1 c_7 -c_3c_5\right]\Big|_{(w,\gamma)=\big(\widetilde{w}_\e^{(m)},\widetilde{\gamma}_\e^{(m)} \big)}.
\end{split}\end{equation*}
\end{small}Next, we compute
\begin{small}\begin{align*}
&\left[\Big(c_3+c_4\Big)\times \Big(c_5+c_6\Big)-\Big(c_1+c_2\Big)\times \Big(c_7+c_8\Big)\right]\Big|_{(w,\gamma)=
\big(\widetilde{w}_\e^{(m)},\widetilde{\gamma}_\e^{(m)} \big)}\\=&
\left[\frac{ \beta}{ |w|^2}
\Big(c_5+c_6-\tau^2\big(c_1+c_2\big)\Big) \right]\Big|_{(w,\gamma)=
\big(\widetilde{w}_\e^{(m)},\widetilde{\gamma}_\e^{(m)} \big)} =
\frac{\tau^2}{(C_\tau)^4 \e^{\beta }  \ln \e}
\Big(1+ O\big(\frac{1}{|\ln \e|}\big)\Big)
< 0.
\end{align*}
\end{small}We also have
\begin{small}\begin{equation*}
\begin{split}
 \left[c_1 c_7-c_3 c_5\right]\Big|_{(w,\gamma)=
\big(\widetilde{w}_\e^{(m)},\widetilde{\gamma}_\e^{(m)} \big)} =&
\left[ \frac{ \beta}{ |w|^2}
\Big(c_5 -\tau^2 c_1 \Big)\right]\Big|_{(w,\gamma)=
\big(\widetilde{w}_\e^{(m)},\widetilde{\gamma}_\e^{(m)} \big)}\\
 =&
\frac{\pi (\tau^3-1)  }{(C_\tau)^2}\left(  1+O\Big(\frac{1}{|\ln \e|}\Big) \right) \big(\nabla \mathcal{R}_{\Omega}(0) \cdot \widetilde{w}_\e^{(m)}\big)\\=&
\frac{\pi (\tau^3-1)  }{(C_\tau)^2}\left(  1+O\Big(\frac{1}{|\ln \e|}\Big) \right)|\widetilde{w}_\e^{(m)}|(-1)^{m-1}.
\end{split}\end{equation*}
\end{small}Hence we obtain \eqref{sec7-18}, which ends the proof.
\end{proof}

\vskip 0.2cm

 \begin{prop}
 For each solution $(\widetilde{w}_\e^{(m)},\widetilde{\gamma}_\e^{(m)})$ of $\widetilde{\bf{V}}_\e(w,\gamma)=0$ with $m=1,2$,
letting $\delta >0$ small and such that
$B\big((\widetilde{w}_\e^{(1)}, \widetilde{\gamma}_\e^{(1)}),\delta\big) \cap  B\big( (\widetilde{w}_\e^{(2)}, \widetilde{\gamma}_\e^{(2)}),\delta\big)=\emptyset$
it holds
\begin{small}
\begin{equation}\label{sec7-26}
\deg\Big( {\bf{V}}_{\e}, 0, B\big((\widetilde{w}_\e^{(m)}, \widetilde{\gamma}_\e^{(m)}    ),\delta\big) \Big) =\deg\Big(\widetilde {\bf{V}}_{\e}, 0, B\big((\widetilde{w}_\e^{(m)}, \widetilde{\gamma}_\e^{(m)}    ),\delta\big)\Big)
=\begin{cases}
(-1)^{m-1},~~~&\mbox{if}~~~\tau<1,\\[1mm]
(-1)^{m},~~~&\mbox{if}~~~\tau>1.
\end{cases}
\end{equation}
\end{small}Hence  ${\bf{V}}_\e(w,\gamma)=0$ has at least
 one solution in $B\big((\widetilde{w}_\e^{(m)}, \widetilde{\gamma}_\e^{(m)}    ),\delta\big)$ for a small
 $\delta>0$ and $m=1,2$.
 \end{prop}

 \begin{proof}First,  we show that
\begin{small}
 \begin{equation}\label{sec7-27}
 |\widetilde {\bf{V}}_{\e}(w,\gamma)|\ge c_0>0,\quad \forall\;(w,\gamma)
 \in\partial  B\big((\widetilde{w}_\e^{(m)}, \widetilde{\gamma}_\e^{(m)}    ),\delta\big).
\end{equation}
\end{small}We argue by contradiction and suppose that there are $(w_\e,\gamma_\e)
 \in\partial B\big((\widetilde{w}_\e^{(m)}, \widetilde{\gamma}_\e^{(m)}    ),\delta\big)$
 such that $|\widetilde {\bf{V}}_{\e}(w_\e,\gamma_\e)|\to 0$. Assume
 that  $(w_\e,\gamma_\e)\to (w,\gamma)$ and similarly to \eqref{sec7-15}, we have
 \begin{small}\begin{align}\label{sec7-28}
& \left(
   \begin{array}{cc}
 \displaystyle
  \frac{\pi(1+\tau+\tau^2)}{1+\tau}
 &-\frac{\beta}{|w_\e|^2} \\[3mm]
\displaystyle
  \frac{\pi(1+\tau+\tau^2)}{1+\tau}
 &-\frac{\tau^2\beta}{|w_\e|^2} \\
   \end{array}
 \right)\left(
          \begin{array}{c}
            \frac{\partial \mathcal{R}_\O(0)}{\partial x_j} \\[3mm]
            \gamma_{\e,j} \\
          \end{array}
        \right)=
        \left(
          \begin{array}{c}
\left[
\frac{k( |w_\e|,\tau)}{|w_\e|^2\e^{\beta}(\ln \e+2\pi \mathcal{R}_{\Omega}(0))}-
 2\beta  \frac{(w_\e\cdot \gamma_\e)}{|w_\e|^4}
\right] w_{\e,j}
\\[3mm]\left[-
\frac{\tau k( |w_\e|,\tau)}{|w_\e|^2\e^{\beta}(\ln \e+2\pi \mathcal{R}_{\Omega}(0))}-
 2\beta \tau^2 \frac{(w_\e\cdot \gamma_\e)}{|w_\e|^4}
\right] w_{\e,j} \\
          \end{array}
        \right)+ o(1).
\end{align}\end{small}This gives that
\begin{small}
\[
\left|\frac{\tau k( |w_\e|,\tau)}{|w_\e|^2\e^{\beta}(\ln \e+2\pi \mathcal{R}_{\Omega}(0))}\right|\le C.
\]
\end{small}Letting $\e\to 0$ in \eqref{sec7-28} leads to $w,\; \gamma \parallel \,\nabla \mathcal{R}_\O(0)$. Moreover,  it holds that $k( |w|,\tau)=0$, which implies $|w|=C_\tau$.
So we find that $(\widetilde{w}_\e^{(m)}, \widetilde{\gamma}_\e^{(m)}    )
\to (w,\gamma)$. This is a contradiction to $(w_\e,\gamma_\e)
 \in\partial  B\big((\widetilde{w}_\e^{(m)}, \widetilde{\gamma}_\e^{(m)}    ),\delta\big)$.

\vskip 0.1cm
 From \eqref{sec7-13} and \eqref{sec7-27} we get \eqref{sec7-26},
which gives  that ${\bf{V}}_{\e}=0$ has a solution in $B\big((\widetilde{w}_\e^{(m)}, \widetilde{\gamma}_\e^{(m)}    ),\delta\big)$, $m=1, 2$.
 \end{proof}

\begin{lem}\label{sec7-lem7.9}
 If $(\widetilde{w}_\e,\widetilde{\gamma}_\e)$ is a solution of ${\bf{V}}_\e(w,\gamma)=0$, then there exists $m\in \{1,2\}$ such that
\begin{small}
\begin{align}\label{sec7-29}
(\widetilde{w}_\e,\widetilde{\gamma}_\e)=\left(\widetilde w_\e^{(m)}+O\Big(\frac{1}{|\ln \e |}\Big),\widetilde\gamma_\e^{(m)}
+O\Big(\frac{1}{|\ln \e |}\Big)\right)~~~~\mbox{and}~~~~|\widetilde{w}_\e|-|\widetilde w_\e^{(m)}|=O\Big(\e^{\beta} \Big),
 \end{align}\end{small}where
$(\widetilde w_\e^{(m)},\widetilde \gamma_\e^{(m)})$ are as in \eqref{sec7-16} and \eqref{sec7-17}, $\beta=\frac{\tau}{(\tau+1)^2}$ with $\tau=\frac{\La_1}{\La_2}$.
\end{lem}
\begin{proof}
 Let
 $(\widetilde{w}_\e,\widetilde{\gamma}_\e)$ be a solution of ${\bf{V}}_\e(w,\gamma)=0$.  Then
 \begin{small}\begin{align*}
& \left(
   \begin{array}{cc}
 \displaystyle
  \frac{\pi(1+\tau+\tau^2)}{1+\tau}
 &-\frac{\beta}{|w_\e|^2} \\[3mm]
\displaystyle
  \frac{\pi(1+\tau+\tau^2)}{1+\tau}
 &-\frac{\tau^2\beta}{|w_\e|^2} \\
   \end{array}
 \right)\left(
          \begin{array}{c}
            \frac{\partial \mathcal{R}_\O(0)}{\partial x_j} \\[3mm]
            \gamma_{\e,j} \\
          \end{array}
        \right)=
        \left(
          \begin{array}{c}
\left[
\frac{k( |w_\e|,\tau)}{|w_\e|^2\e^{\beta}(\ln \e+2\pi \mathcal{R}_{\Omega}(0))}-
 2\beta  \frac{(w_\e\cdot \gamma_\e)}{|w_\e|^4}
\right] w_{\e,j}
\\[3mm]\left[-
\frac{\tau k( |w_\e|,\tau)}{|w_\e|^2\e^{\beta}(\ln \e+2\pi \mathcal{R}_{\Omega}(0))}-
 2\beta \tau^2 \frac{(w_\e\cdot \gamma_\e)}{|w_\e|^4}
\right] w_{\e,j} \\
          \end{array}
        \right)+ O\left(\frac1{|\ln\e|}\right).
\end{align*}\end{small}Thus we have that
\begin{small}
\begin{equation*}
\frac{\widetilde{w}_\e}{|\widetilde{w}_\e|}=\pm\frac{\nabla \mathcal{R}_\O(0)}{|\nabla \mathcal{R}_\O(0)|}+ O\left(\frac{1}{|\ln \e |}\right),
\quad \frac{\widetilde{\gamma}_\e}{|\gamma_\e|}=\frac{\widetilde{w}_\e}{|\widetilde{w}_\e|}
+ O\left(\frac{1}{|\ln \e |}\right).
\end{equation*}
\end{small}As in \eqref{sec7-20} and  \eqref{sec7-21} in Proposition~\ref{sec7-prop7.7}, we can
 derive
  \begin{small}
\begin{equation}\label{sec7-30}
\begin{split}
\frac{{k}\big( |\widetilde{w}_\e|,\tau\big)}{\e^{ \beta }(\ln \e+2\pi \mathcal{R}_{\Omega}(0)) }=&
\frac{\pi(1+\tau+\tau^2)}{1+\tau}
\big(\nabla \mathcal{R}_{\Omega}(0) \cdot \widetilde{w}_\e\big)
+  \frac{\beta }{ |\widetilde{w}_\e|^2}
\big(\widetilde{w}_\e\cdot\widetilde{\gamma}_\e \big)+ O\left(\frac{1}{|\ln \e |}\right),
\end{split}
\end{equation}
\end{small}and\begin{small}
\begin{equation}\label{sec7-31}
\begin{split}
\frac{\tau {k}\big( |\widetilde{w}_\e|,\tau\big)}{\e^{ \beta }(\ln \e+2\pi \mathcal{R}_{\Omega}(0)) }=&-
\frac{\pi(1+\tau+\tau^2)}{1+\tau}
\big(\nabla \mathcal{R}_{\Omega}(0) \cdot \widetilde{w}_\e\big)
- \frac{\beta \tau^2}{ |\widetilde{w}_\e|^2}
\big(\widetilde{w}_\e\cdot \widetilde{\gamma}_\e\big)+ O\left(\frac{1}{|\ln \e |}\right).
\end{split}
\end{equation}
\end{small}Hence from $\tau \times \eqref{sec7-30}-\eqref{sec7-31}$, we get
\begin{small}\begin{equation}\label{sec7-32}
  \begin{split}
\pi (1+\tau+\tau^2)
\big(\nabla \mathcal{R}_{\Omega}(0) \cdot \widetilde{w}_\e\big)
=-\frac{\beta(\tau+\tau^2) }{ |\widetilde{w}_\e|^2}
\big(\widetilde{w}_\e\cdot \widetilde{\gamma}_\e \big)+ O\left(\frac{1}{|\ln \e |}\right).
 \end{split}
\end{equation}
\end{small}Inserting $\frac {\widetilde{w}_\e}{|\widetilde{w}_\e|}=\pm \frac {\nabla \mathcal{R}_\O(0)}{|\nabla \mathcal{R}_\O(0)|}+ O\left(\frac{1}{|\ln \e |}\right)
 $ into \eqref{sec7-32}, we find
\begin{small}
 \[
 \pi (1+\tau+\tau^2)
|\nabla \mathcal{R}_{\Omega}(0) |\cdot |\widetilde{w}_\e|
=-\frac{\beta(\tau+\tau^2) }{ |\widetilde{w}_\e|}
\frac {\nabla \mathcal{R}_\O(0)}{|\nabla \mathcal{R}_\O(0)|}\cdot \widetilde{\gamma}_\e,
\]
\end{small}which, together with
  $\frac{\widetilde{\gamma}_\e}{|\widetilde{\gamma}_\e|}=\pm\frac{\nabla \mathcal{R}_\O(0)}{|\nabla \mathcal{R}_\O(0)|}+ O\left(\frac{1}{|\ln \e |}\right)$, gives
  \begin{small}
\begin{equation*}
\begin{split}
\widetilde{\gamma}_\e=&\Big(
\frac {\nabla \mathcal{R}_\O(0)}{|\nabla \mathcal{R}_\O(0)|}\cdot \gamma \Big)\frac {\nabla \mathcal{R}_\O(0)}{|\nabla \mathcal{R}_\O(0)|} + O\left(\frac{1}{|\ln \e |}\right)
=-\frac{\pi(1+\tau+\tau^2)(1+\tau)|\widetilde{w}_\e|^2}{\tau^2}
 \nabla \mathcal{R}_\Omega(0)+ O\left(\frac{1}{|\ln \e |}\right).
 \end{split}
 \end{equation*}\end{small}Next we compute the expansion of $\widetilde{w}_\e$. From $\tau^2\times \eqref{sec7-30}+\eqref{sec7-31}$,   we get
 \begin{small}
\begin{equation*}
\begin{split}
\frac{k\big( |\widetilde{w}_\e|,\tau\big)}{\e^{\beta}\big(\ln \e+2\pi \mathcal{R}_{\Omega}(0)\big)}=&\pi d_\tau
\big(\nabla \mathcal{R}_{\Omega}(0) \cdot \widetilde{w}_\e\big)+ O\left(\frac{1}{|\ln \e |}\right) ~\,\,~~\mbox{with}~~\,\,~d_\tau:=
\frac{\tau^3-1}{(1+\tau)\tau}.
 \end{split}
 \end{equation*}\end{small}That is
  \begin{small}
\begin{equation}\label{2-27-8y-t}
\begin{split}
k\big( |\widetilde{w}_\e|,\tau\big)=&\pi d_\tau
\big(\nabla \mathcal{R}_{\Omega}(0) \cdot \widetilde{w}_\e\big)\e^{\beta}\big(\ln \e+2\pi \mathcal{R}_{\Omega}(0) \big)+O\left(\e^{\beta}\right)\\=&
\pi d_\tau
\big|\nabla \mathcal{R}_{\Omega}(0)\big| \cdot \big|\widetilde{w}_\e\big|\e^{\beta}\big(\ln \e+2\pi \mathcal{R}_{\Omega}(0) \big)+O\left(\e^{\beta}\right).
 \end{split}
 \end{equation}\end{small}Also we recall
  \begin{small}
\begin{equation}\label{2-27-8y-tm}
\begin{split}
k\big( |\widetilde{w}^{(1)}_\e|,\tau\big)=&\pi d_\tau
\big|\nabla \mathcal{R}_{\Omega}(0)\big| \cdot \big|\widetilde{w}^{(1)}_\e\big|\e^{\beta}\big(\ln \e+2\pi \mathcal{R}_{\Omega}(0)\big).
 \end{split}
 \end{equation}\end{small}Hence
from $\frac{\widetilde{w}_\e}{|\widetilde{w}_\e|}=\frac{\widetilde{w}^{(1)}_\e}{|\widetilde{w}^{(1)}_\e|}+ O\left(\frac{1}{|\ln \e |}\right)$, \eqref{2-27-8y-t} and \eqref{2-27-8y-tm}, we have
\begin{small}\begin{equation*}
\begin{split}
\frac{ k\big( |\widetilde{w}_\e|,\tau\big)}{|\widetilde{w}_\e|}-
\frac{ k\big( |\widetilde{w}^{(1)}_\e|,\tau\big)}{|\widetilde{w}^{(1)}_\e|}
=O\Big( \e^{\beta} \Big),
 \end{split}
 \end{equation*}\end{small}which gives $|\widetilde{w}_\e|-|\widetilde w_\e^{(1)}|=O\Big( \e^{\beta} \Big)$ and then $
(\widetilde{w}_\e,\widetilde{\gamma}_\e)=\left(\widetilde w_\e^{(1)}+O\big(\frac{1}{|\ln \e |}\big),\widetilde\gamma_\e^{(1)}
+O\big(\frac{1}{|\ln \e |}\big)\right)$.
\end{proof}

 We now consider the non-degeneracy of the solutions of ${\bf{V}}_\e(w,\gamma)
 =0$.
\begin{prop}\label{sec7-prop7.10}
 If $(\widetilde{w}_\e,\widetilde{\gamma}_\e)$ is a solution of ${\bf{V}}_\e(w,\gamma)=0$, then it holds
\begin{small}
\begin{align}\label{sec7-35}
det~Jac~ {\bf{V}}_\e(\widetilde{w}_\e,\widetilde{\gamma}_\e)= \frac{\pi \tau^2 (\tau^3-1)  }{(C_\tau)^5\e^{\beta }  \ln \e}\left(  1+O\Big(\frac{1}{|\ln \e|}\Big) \right) (-1)^{m-1},
 \end{align}
\end{small}where $m\in \{1,2\}$ is as in \eqref{sec7-29}, $\tau=\frac{\La_1}{\La_2}$, $\beta=\frac{\tau}{(\tau+1)^2}$ and $C_\tau:=  \tau^{\frac{1}{1+\tau}}
e^{-\frac{2\pi \mathcal{R}_{\Omega}(0)(\tau^2+ \tau+1 )}{(1+\tau)^2}}$.
\end{prop}

\begin{proof} By Lemma \ref{sec7-lem7.9}, we can consider the case \begin{small}
\begin{align*}
(\widetilde{w}_\e,\widetilde{\gamma}_\e)=\left(\widetilde w_\e^{(1)}+O\Big(\frac{1}{|\ln \e |}\Big),\widetilde\gamma_\e^{(1)}
+O\Big(\frac{1}{|\ln \e |}\Big)\right)~~~~\mbox{and}~~~~|\widetilde{w}_\e|-|\widetilde w_\e^{(1)}|=O\Big(\e^{\beta}\Big).
 \end{align*}\end{small}The
computations for the other
case are similar.  First,   we have
\begin{small}
\begin{equation*}
\begin{split}
\frac{\partial  {V}_{\e,1,i}(w,\gamma)}{\partial w_j}=&\frac{\partial \widetilde{V}_{\e,1,i}(w,\gamma)}{\partial w_j}+O(\frac{1}{|\ln\e|})
\\=&-\left[
\frac{k( |w|,\tau)}{|w|^2\e^{\beta}(\ln \e +2\pi \mathcal{R}_{\Omega}(0))}-
 2\beta  \frac{(w\cdot \gamma)}{|w|^4}
\right]\delta_{ij}
- \left[\frac{1}{|w|\e^{\beta}(\ln \e +2\pi \mathcal{R}_{\Omega}(0))}
\left.\frac{ \partial {k}(r,\tau)}{\partial r}\right|_{r=
|w|}\right.
\\[2mm]&
\left.
-\frac{2 {k}(|w|,\tau)}{ |w|^2\e^{\beta}(\ln \e +2\pi \mathcal{R}_{\Omega}(0)) }
+\frac{8 (w\cdot \gamma )  }{ |w|^4}\right]\frac{w_iw_j}{|w|^2}
+  \frac{2 \beta(   w_i \gamma_j +w_j\gamma_i )}{ |w|^4}+
O(\frac{1}{|\ln\e|}).\end{split}\end{equation*}\end{small}By
\eqref{sec7-30} and \eqref{sec7-31}, we have
\begin{small}
\begin{equation}\label{sec7-36}
\frac{k( |\widetilde{w}_\e|,\tau)}{|\widetilde{w}_\e|^2\e^{\beta}(\ln \e +2\pi \mathcal{R}_{\Omega}(0))}-
 2\beta  \frac{(\widetilde{w}_\e\cdot \widetilde{\gamma}_\e)}{|\widetilde{w}_\e|^4}=
 \pi(1+\tau+\tau^2)
 \frac{(\nabla \mathcal{R}_{\Omega}(0) \cdot \widetilde{w}_\e )}{|\widetilde{w}_\e|^2}+O(\frac{1}{|\ln\e|}).
 \end{equation}
\end{small}Also, using $\frac{\widetilde{w}_\e}{|\widetilde{w}_\e|}-  \frac{\widetilde{w}^{(1)}_\e}{|\widetilde{w}^{(1)}_\e|}=O(\frac{1}{|\ln\e|})$ and
$\frac{\widetilde{\gamma}_\e}{|\widetilde{\gamma}_\e|}-  \frac{\widetilde{\gamma}^{(1)}_\e}{|\widetilde{\gamma}^{(1)}_\e|}=O(\frac{1}{|\ln\e|})$
by Lemma \ref{sec7-lem7.9}, \eqref{sec7-07} and \eqref{sec7-36}, we obtain
\begin{small}\begin{equation*}
 \frac{\partial {V}_{\e,1,i}(w,\gamma)}{\partial w_j}\Big|_{(w,\gamma)=(\widetilde{w}_\e,\widetilde{\gamma}_\e)}
=\widetilde{c}_1 \delta_{ij}+
\widetilde{c}_2 \frac{\widetilde{w}_{\e,i} \widetilde{w}_{\e,j} }{|\widetilde{w}_\e|^2}+O(\frac{1}{|\ln\e|}),
\end{equation*}
\end{small}with
\begin{small}
\begin{equation*}
\widetilde{c}_1 =  -\frac{ \pi(1+\tau+\tau^2)|\nabla \mathcal{R}_{\Omega}(0)|}{\tau C_\tau},\,\,~~ \widetilde{c}_2=-\frac{1+\tau}{C_\tau^2\e^{\beta}\ln\e}\Big(1+O(\frac{1}{|\ln\e|})\Big).
\end{equation*}
\end{small}On the other hand we can compute
\begin{small}\begin{equation*}
\begin{split}
 \frac{\partial  {V}_{\e,1,i}(w,\gamma)}{\partial \gamma_j}\Big|_{(w,\gamma)=(\widetilde{w}_\e,\widetilde{\gamma}_\e)}=&\frac{\partial \widetilde{V}_{\e,1,i}(w,\gamma)}{\partial \gamma_j}\Big|_{(w,\gamma)=(\widetilde{w}_\e,\widetilde{\gamma}_\e)}+O(\frac{1}{|\ln\e|})
\\=& \underbrace{- \frac{ \beta }{ C^2_\tau}}_{:=\widetilde{c}_3} \delta_{ij}+
\underbrace{\frac{ 2\beta }{C^2_\tau}}_{:=\widetilde{c}_4 }\frac{\widetilde{w}_{\e,i}
\widetilde{w}_{\e,j}}{|\widetilde{w}_\e |^2}+O(\frac{1}{|\ln\e|}).
\end{split}\end{equation*}
\end{small}Similarly, we obtain
\begin{small}\begin{equation*}
 \frac{\partial {V}_{\e,2,i}(w,\gamma)}{\partial w_j}\Big|_{(w,\gamma)=(\widetilde{w}_\e,\widetilde{\gamma}_\e)}
=\widetilde{c}_5\delta_{ij}+
\widetilde{c}_6 \frac{\widetilde{w}_{\e,i}\widetilde{w}_{\e,j}}{|\widetilde{w}_\e|^2}+O(\frac{1}{|\ln\e|}),
\end{equation*}
\end{small}with
\begin{small}
\begin{equation*}
\widetilde{c}_5
= -\frac{\pi(1+\tau+\tau^2)|\nabla \mathcal{R}_{\Omega}(0)|}{C_\tau},\,\,\,\,\,\,\,\,\,\, \widetilde{c}_6
=\frac{\tau(1+\tau)}{(C_\tau)^2\e^{\beta}\ln\e}\Big(1+O(\frac{1}{|\ln\e|})\Big).
 \end{equation*}
\end{small}We also have
\begin{small}\begin{equation*}
 \frac{\partial  {V}_{\e,2,i}(w,\gamma)}{\partial \gamma_j}\Big|_{(w,\gamma)=(\widetilde{w}_\e,\widetilde{\gamma}_\e)}
=\widetilde{c}_7 \delta_{ij}+\widetilde{c}_8 \frac{\widetilde{w}_{\e,i}
\widetilde{w}_{\e,j}}{|\widetilde{w}_\e|^2}+O(\frac{1}{|\ln\e|}),
\end{equation*}
\end{small}with
$\widetilde{c}_7=\tau^2\widetilde{c}_3$ and $\widetilde{c}_8 =\tau^2\widetilde{c}_4$.
Now denote by ${\bf Q}_{\e,1}$ the $2\times2$ matrix
 \begin{small}
\begin{align*}
{\bf Q}_{\e,1}:=
\left(
  \begin{array}{cc}
    \frac{\widetilde w_{\e,1}^{(1)}}{|\widetilde w_\e^{(1)}|} &- \frac{\widetilde w_{\e,2}^{(1)}}{|\widetilde w_\e^{(1)}|} \\[4mm]
    \frac{\widetilde w_{\e,2}^{(1)}}{|\widetilde w_\e^{(1)}|} & \frac{\widetilde w_{\e,1}^{(1)}}{|\widetilde w_\e^{(1)}|} \\
  \end{array}
\right).
\end{align*}
\end{small}Then it holds
\begin{small}
                      \begin{equation*}
                      \textbf{Q}^T_{\e,1} \left(\frac{\widetilde{w}_{\e,i}^{(1)}
                      \widetilde{w}_{\e,j}^{(1)}}{|\widetilde{w}_\e^{(1)}|^2}\right)_{1\leq i,j\leq 2}\textbf{Q}_{\e,1}=\left(
                                            \begin{array}{c}
                                              1 \\
                                              0 \\
                                            \end{array}
                                          \right)\left(
                                     \begin{array}{cc}
                                     1 &0 \\
                                     \end{array}
                                   \right).
                      \end{equation*}\end{small}Hence we have
                      \begin{small}\begin{equation*}
\begin{split}
& \left(
     \begin{array}{cc}
       \textbf{Q}^T_{\e,1} & \textbf{O}_{2\times 2} \\[4mm]
       \textbf{O}_{2\times 2} & \textbf{Q}^T_{\e,1} \\
     \end{array}
   \right) ~Jac~  {\bf{V}}_{\e}(\widetilde{w}_\e,\widetilde{\gamma}_\e)
~\left(
     \begin{array}{cc}
       \textbf{Q}_{\e,1} & \textbf{O}_{2\times 2} \\[4mm]
       \textbf{O}_{2\times 2} & \textbf{Q}_{\e,1} \\
     \end{array}
   \right)\\[3mm]=&
       \left(
             \begin{array}{cccc}
               \widetilde{c}_1+\widetilde{c}_2+O(\frac{1}{|\ln\e|}) & O(\frac{1}{|\ln\e|}) & \widetilde{c}_3+\widetilde{c}_4+O(\frac{1}{|\ln\e|}) & O(\frac{1}{|\ln\e|}) \\[3mm]
                O(\frac{1}{|\ln\e|}) & \widetilde{c}_1+O(\frac{1}{|\ln\e|}) & O(\frac{1}{|\ln\e|}) & \widetilde{c}_3+O(\frac{1}{|\ln\e|}) \\[3mm]
               \widetilde{c}_5+\widetilde{c}_6 +O(\frac{1}{|\ln\e|}) &  O(\frac{1}{|\ln\e|}) & \widetilde{c}_7+\widetilde{c}_8+O(\frac{1}{|\ln\e|}) & O(\frac{1}{|\ln\e|}) \\[3mm]
                O(\frac{1}{|\ln\e|}) & \widetilde{c}_5+O(\frac{1}{|\ln\e|}) & O(\frac{1}{|\ln\e|}) & \widetilde{c}_7 +O(\frac{1}{|\ln\e|})\\
             \end{array}
           \right)\\[3mm]=&
       \left(
             \begin{array}{cccc}
-\frac{1+\tau}{C_\tau^2\e^{\beta}\ln\e}\Big(1+O(\frac{1}{|\ln\e|})\Big)  & O(\frac{1}{|\ln\e|}) & \frac{ \beta }{ C^2_\tau}+O(\frac{1}{|\ln\e|}) & O(\frac{1}{|\ln\e|}) \\[3mm]
                O(\frac{1}{|\ln\e|}) & -\frac{ \pi(1+\tau+\tau^2)|\nabla \mathcal{R}_{\Omega}(0)|}{\tau C_\tau}+O(\frac{1}{|\ln\e|}) & O(\frac{1}{|\ln\e|}) & - \frac{ \beta }{ C^2_\tau}+O(\frac{1}{|\ln\e|}) \\[3mm]
               \frac{\tau(1+\tau)}{(C_\tau)^2\e^{\beta}\ln\e}\Big(1+O(\frac{1}{|\ln\e|})\Big) &  O(\frac{1}{|\ln\e|}) & \frac{ \beta\tau^2 }{ C^2_\tau}+O(\frac{1}{|\ln\e|}) & O(\frac{1}{|\ln\e|}) \\[3mm]
                O(\frac{1}{|\ln\e|}) & -\frac{\pi(1+\tau+\tau^2)|\nabla \mathcal{R}_{\Omega}(0)|}{C_\tau}+O(\frac{1}{|\ln\e|}) & O(\frac{1}{|\ln\e|}) & -\frac{ \beta\tau^2 }{ C^2_\tau} +O(\frac{1}{|\ln\e|})\\
             \end{array}
           \right).
\end{split}\end{equation*}\end{small}From above computations, we have
\begin{small}\begin{equation*}
\begin{split}
det&~Jac~ \widetilde{\bf{V}}_{\e}(\widetilde{w}_\e,\widetilde{\gamma}_\e)
= \frac{\pi \tau^2 (\tau^3-1)  }{(C_\tau)^5\e^{\beta }  \ln \e}\left(  1+O\Big(\frac{1}{|\ln \e|}\Big) \right).
\end{split}\end{equation*}
\end{small}Similarly, for the case the case \begin{small}
\begin{align*}
(\widetilde{w}_\e,\widetilde{\gamma}_\e)=\left(\widetilde w_\e^{(2)}+O\Big(\frac{1}{|\ln \e |}\Big),\widetilde\gamma_\e^{(2)}
+O\Big(\frac{1}{|\ln \e |}\Big)\right)~~~~\mbox{and}~~~~|\widetilde{w}_\e|-|\widetilde w_\e^{(2)}|=O\Big(\e^{\beta}\Big),
 \end{align*}\end{small}we can compute
\begin{small}\begin{equation*}
\begin{split}
det&~Jac~ \widetilde{\bf{V}}_{\e}(\widetilde{w}_\e,\widetilde{\gamma}_\e)
= - \frac{\pi \tau^2 (\tau^3-1)  \big|\nabla \mathcal{R}_{\Omega}(0)\big|  }{(C_\tau)^5\e^{\beta }  \ln \e}\left(  1+O\Big(\frac{1}{|\ln \e|}\Big) \right).
\end{split}\end{equation*}
\end{small}
\end{proof}
\begin{proof}[\bf{Proof of Theorem \ref{sec1-teo15}}]
The existence of at least two solutions follows from Proposition \ref{sec7-prop7.7}.
Moreover, from \eqref{sec7-35}, the critical points of
$\mathcal{KR}_{\Omega_\e}(x,y)$ are all nondegenerate. Hence, the number
of solutions are finite.
\vskip 0.1cm

Next, we prove that for any fixed $m \in \{1,2\}$, ${\bf{V}}_\e(w,\gamma)=0$ has a unique solution in
$B\big((\widetilde{w}_\e^{(m)}, \widetilde{\gamma}_\e^{(m)}    ),\delta\big)$.
For example, suppose that
there are $l$ solutions in $B\big((\widetilde{w}_\e^{(m)}, \widetilde{\gamma}_\e^{(m)}    ),\delta\big)$. And then
by Proposition \ref{sec7-prop7.10}, we have that
\begin{small}
\begin{align}\label{sec7-37}
\deg\Big( {\bf{V}}_\e(w,\gamma),0, B\big((\widetilde{w}_\e^{(m)}, \widetilde{\gamma}_\e^{(m)}    ),\delta\big)\Big) =
\begin{cases}
l\cdot (-1)^{m-1},~~~&\mbox{if}~~~\tau<1,\\[1mm]
l\cdot (-1)^{m},~~~&\mbox{if}~~~\tau>1.
\end{cases}
\end{align}
\end{small}On the other hand, it follows
from \eqref{sec7-26} that
\begin{small}
\begin{align*}
\deg\Big( {\bf{V}}_\e(w,\gamma),0, B\big((\widetilde{w}_\e^{(m)}, \widetilde{\gamma}_\e^{(m)}    ),\delta\big)\Big) =
\begin{cases}
(-1)^{m-1},~~~&\mbox{if}~~~\tau<1,\\[1mm]
(-1)^{m},~~~&\mbox{if}~~~\tau>1,
\end{cases}
\end{align*}
\end{small}which, together with \eqref{sec7-37}, implies that $l=1$.
\vskip 0.1cm
Moreover, outside of the ball $B\big((\widetilde{w}_\e^{(m)}, \widetilde{\gamma}_\e^{(m)}    ),\delta\big)$ by Proposition  \ref{sec7-prop7.7}. the equation ${\bf{V}}_{\e}(w,\gamma)=0$ has no solution.
Therefore, we have shown that ${\bf{V}}_{\e}(w,\gamma)=0$ has
exactly two solutions.
Moreover, by Proposition~\ref{sec7-prop7.7}, these two critical points of $\mathcal{KR}_{\Omega_\e}(x,y)$
are nondegenerate.
\end{proof}

Above discussions can also been used to handle the case $\O_\e=B(0,1)\backslash B(0,\e)$.

\begin{proof}[\bf{Proof of Theorem \ref{sec1-teo19}}] If $\O_\e=B(0,1)\backslash B(0,\e)$, then $\nabla \mathcal{R}_{B(0,1)}(0)=0$, and  \eqref{sec7-15} becomes

\begin{small}
\begin{equation}\label{sec7-38}
\begin{cases}
-\frac{ \beta \gamma_{j}}{|w|^2} =\left[
\frac{k( |w|,\tau)}{|w|^2\e^{\beta} \ln \e }-
 2\beta  \frac{(w\cdot \gamma)}{|w|^4}
\right] w_j,\\[2mm]
 \frac{ \beta \tau^2 \gamma_{j}}{|w|^2} =\left[
\frac{\tau k( |w|,\tau)}{|w|^2\e^{\beta} \ln \e }+
 2\beta \tau^2 \frac{(w\cdot \gamma)}{|w|^4}
\right] w_j,
\end{cases}
\end{equation}
\end{small}for $j=1,2$. Adding $\tau^2\times $ the first equation of \eqref{sec7-38}  with the second equation of \eqref{sec7-38} yields
$k(|w|,\tau)=0$, which gives $|w|=C_\tau$. Putting this into the first equation of \eqref{sec7-38}, we get $\frac{ \beta \gamma_{j}}{|w|^2} =
 2\beta  \frac{(w\cdot \gamma)}{|w|^4}w_j$, from
 which we can derive $\gamma=0$. This shows that if $(x_\e,y_\e)$ is a critical point of $\mathcal{KR}_\Omega(x,y)$, then letting $(w_\e,\gamma_\e)=(\frac{x_\e}{\e^\beta},\frac{x_\e+\tau y_\e}{\e^{2\beta}})$, it holds
\begin{small}
\begin{equation*}
\lim_{\e\to 0}|w_\e|=C_\tau~~~\mbox{and}~~~\lim_{\e\to o}|\gamma_\e|=0~~~\mbox{with}~~~
\lim_{\e\to 0}\frac{w_\e}{|w_\e|}=\lim_{\e\to 0}\frac{\gamma_\e}{|\gamma_\e|}.
\end{equation*}
\end{small}By a suitable  rotation, we can assume that $x_\e=(|x_\e|,0)$, Denoting $\displaystyle\lim_{\e\to 0}w_\e=w_0$ and $\displaystyle\lim_{\e\to 0}\gamma_\e=\gamma_0$, we have $w_0=(C_\tau,0)$ and $\gamma_0=(0,0)$.

\vskip 0.1cm

We define ${\bf{F}}_\e(r,\gamma_1,\gamma_2)=\big( {F}_{\e,0}(r,\gamma_1,\gamma_2), {F}_{\e,1}(r,\gamma_1,\gamma_2),{F}_{\e,2}(r,\gamma_1,\gamma_2)\big)$ with
\begin{small}
\begin{equation*}
\begin{cases}
{F}_{\e,0}(r,\gamma_1,\gamma_2)=\frac{\partial \mathcal{KR}_{\Omega_\e}(x,y)}{\partial x_1}\Big|_{(x,y)=(\e^\beta r,0, \frac{\e^{2\beta}\gamma_1-r}{\tau}, \frac{\e^{2\beta}\gamma_2}{\tau} )},\\[1mm]
{F}_{\e,j}(r,\gamma_1,\gamma_2) =\frac{\partial \mathcal{KR}_{\Omega_\e}(x,y)}{\partial y_j}\Big|_{(x,y)=(\e^\beta r,0, \frac{\e^{2\beta}\gamma_1-r}{\tau}, \frac{\e^{2\beta}\gamma_2}{\tau} )},~~~~~\mbox{for}~~~~~j=1,2.
\end{cases}\end{equation*}
\end{small}Hence $\nabla  \mathcal{KR}_{\Omega_\e}(x_\e,y_\e)=0$ with $x_\e=(|x_\e|,0)$ is equivalent to ${\bf F}_{\e}(\bar r_\e,\gamma_{\e,1},\gamma_{\e,2})=0$ with
$(x_\e,y_\e)=(\e^\beta \bar r_\e,0, \frac{\e^{2\beta}\gamma_{\e,1}-\bar r_\e}{\tau}, \frac{\e^{2\beta}\gamma_{\e,2}}{\tau})$.

\vskip 0.1cm

Next, we have
\begin{small}
\begin{equation*}
\nabla_{(r,\gamma_1,\gamma_2)}{\bf{F}}_{\e}(r,\gamma_1,\gamma_2)=\nabla_{(r,\gamma_1,\gamma_2)} \widetilde{\bf{F}}_{\e}(r,\gamma_1,\gamma_2)
\left(
  \begin{array}{ccc}
    \frac{\La_1\La_2}{\pi}  & 0 & 0 \\
    0 &  \frac{\La_2^2}{\pi}  & 0 \\
    0 & 0 &  \frac{\La_2^2}{\pi}  \\
  \end{array}
\right) +
  O\Big( \frac{1}{|\ln \e|}\Big),
\end{equation*}
\end{small}where $\widetilde{\bf{F}}_\e(r,\gamma_1,\gamma_2)=\big( \widetilde{F}_{\e,0}(r,\gamma_1,\gamma_2), \widetilde{F}_{\e,1}(r,\gamma_1,\gamma_2),\widetilde{F}_{\e,2}(r,\gamma_1,\gamma_2)\big)$ with
\begin{small}
\begin{equation*}
\widetilde{F}_{\e,0}(r,\gamma_1,\gamma_2)=-
\frac{k(r,\tau)}{r^2\e^{\beta} \ln \e }+ \frac{\beta \gamma_1}{r^2},~~
\widetilde{F}_{\e,1}(r,\gamma_1,\gamma_2)=  \frac{\beta \tau^2\gamma_1}{r^2},~~~
\widetilde{F}_{\e,2}(r,\gamma_1,\gamma_2)=- \frac{\beta\tau^2 \gamma_2}{r^2}.
\end{equation*}
\end{small}We can verify that $\widetilde{\bf{F}}_\e(r,\gamma_1,\gamma_2)$ has a unique solution $(C_\tau,0,0)$. And then we deduce that any critical point  of ${\bf{F}}_\e(r,\gamma_1,\gamma_2)=0$ belongs to $\mathcal{S}$, with
\begin{small}
\begin{equation*}
\mathcal{S}=\left\{(r,\gamma)\in \R^+\times \R^2,~ |r-C_\tau|\leq \e^{\beta-\delta},~ |\gamma|\leq
\frac{1}{|\ln \e|^{1-\delta}}~~~\mbox{for some small fixed}~~~\delta>0
\right\}.
\end{equation*}
\end{small}Now we calculate
\begin{small}
\begin{equation*}
    Jac~{\bf{F}}_\e(r,\gamma_1,\gamma_2)\Big|_{
    (r,\gamma_1,\gamma_2)\in \mathcal{S} } =\left(
             \begin{array}{ccc}
-\frac{1+\tau}{r^2\e^{\beta}\ln\e}\Big(1+O(\frac{1}{|\ln\e|})\Big)  & O(\frac{1}{|\ln\e|^{1-\delta}}) & O(\frac{1}{|\ln\e|})   \\[3mm]
  O(\frac{1}{|\ln\e|}) & \frac{ \beta\tau^2 }{ r^2 }+O(\frac{1}{|\ln\e|}) & O(\frac{1}{|\ln\e|}) \\[3mm]
                O(\frac{1}{|\ln\e|}) &  O(\frac{1}{|\ln\e|}) & -\frac{ \beta\tau^2 }{r^2 } +O(\frac{1}{|\ln\e|})\\
             \end{array}
           \right).
 \end{equation*}\end{small}Hence
 ${\bf{F}}_\e(r,\gamma_1,\gamma_2)$ has a unique solution $(r_{\e,0},\gamma_{\e,1},\gamma_{\e,2})$, which tends to $(C_\tau,0,0)$.

 \vskip 0.1cm

We claim that $\gamma_{\e,2}=0$. In fact, if $\gamma_{\e,2}\neq 0$, by
the symmetry of the domain,
$(r_{\e,0},\gamma_{\e,1},-\gamma_{\e,2})$ is also the solution of ${\bf{F}}_\e(r,\gamma_1,\gamma_2)=0$. This contradicts the uniqueness of the
solution.

 \vskip 0.1cm

 Similarly, we can use a suitable rotation to get $y_\e=(|y_\e|,0)$.
 Then
 $x_\e$ is uniquely determined with $x_\e=(-|x_\e|,0)$.
 So we have proved that up to a rotation, $ \mathcal{KR}_{\Omega_\e}(x,y)$
 has exactly two different critical points
  if $\La_1\neq \La_2$,  while it has exactly one critical
   point if $\La_1=\La_2$.
Finally, these critical points  are nondegenerate in the radial direction.
\end{proof}

\vskip 0.2cm
\subsection{Further expansion of $\nabla\mathcal{KR}_{\Omega_\e}(x,y)$} ~
\vskip 0.1cm

To study this case $\Lambda_1=\Lambda_2$ or $\nabla \mathcal{R}_\Omega(0)=0$, we need to further expand $\mathcal{KR}_{\Omega_\e}(x,y)$ in Proposition \ref{sec7-prop7.4} and Proposition \ref{sec7-prop7.5}.
\begin{prop}\label{sec7-prop7.11}
For $(x,y)\in \mathcal{H}_\e$ and $j=1,2$, it holds
\begin{small}\begin{equation}\label{sec7-39}
\begin{split}
 \frac{\partial \mathcal{KR}_{\Omega_\e}(x,y)}{\partial x_j}=
 -
 \frac{\La_1}\pi
 &\left\{ \frac{\widetilde{h}(x,y)x_j}{|x|^2}+\frac{\pi(\La_1\ln \frac{|x|}{\e}+ \La_2\ln \frac{|y|}{\e} ) }{\ln \e+ 2\pi \mathcal{R}_{\Omega}(0)}
  \frac{\partial \mathcal{R}_{\Omega}(0)}{\partial x_j}-
\frac{\La_2(x_{j}-y_{j})}{|x-y|^{2}}\right. \\[1mm]
& \left.\,\,\,
+2\pi \sum^2_{i=1}\left[
\frac{\La_1\ln \frac{|x|}{\e}+ \La_2\ln \frac{|y|}{\e} }{\ln \e+ 2\pi \mathcal{R}_{\Omega}(0)}
x_i \frac{\partial^2 H_{\Omega}(0,0)}{\partial x_i\partial x_j}
-(\La_1x_i+\La_2y_i) \frac{\partial^2 H_{\Omega}(0,0)}{\partial y_i\partial x_j}
\right]
\right\}
+O\Big( \frac{\e^{\beta}}{|\ln \e|}\Big),
\end{split}
\end{equation}
\end{small}
\begin{small}\begin{equation}\label{sec7-40}
\begin{split}
 \frac{\partial \mathcal{KR}_{\Omega_\e}(x,y)}{\partial y_j}=
 -
 \frac{\La_2}\pi
 &\left\{ \frac{\widetilde{h}(x,y)y_j}{|y|^2}+\frac{\pi(\La_1\ln \frac{|x|}{\e}+ \La_2\ln \frac{|y|}{\e} ) }{\ln \e+ 2\pi \mathcal{R}_{\Omega}(0)}
  \frac{\partial \mathcal{R}_{\Omega}(0)}{\partial x_j}-
\frac{\La_1(y_{j}-x_{j})}{|x-y|^{2}}\right. \\[1mm]
& \left.\,\,\,
+2\pi \sum^2_{i=1}\left[
\frac{\La_1\ln \frac{|x|}{\e}+ \La_2\ln \frac{|y|}{\e} }{\ln \e+ 2\pi \mathcal{R}_{\Omega}(0)}
y_i \frac{\partial^2 H_{\Omega}(0,0)}{\partial y_i\partial y_j}
-(\La_1x_i+\La_2y_i) \frac{\partial^2 H_{\Omega}(0,0)}{\partial x_i\partial y_j}
\right]
\right\}
+O\Big( \frac{\e^{\beta}}{|\ln \e|}\Big),
\end{split}
\end{equation}
\end{small}where $\widetilde{h}(x,y):={h}(x,y)+\frac{ \pi\big( \La_1(\nabla\mathcal{R}_\O(0)\cdot x)+\La_2(\nabla\mathcal{R}_\O(0)\cdot y)\big)
  }{\ln \e+2\pi \mathcal{R}_{\Omega}(0)}$ with ${h}(x,y)$ being the function in \eqref{sec7-02a}.
\end{prop}
\begin{rem}
Now we compare Proposition \ref{sec7-prop7.11} with Proposition \ref{sec7-prop7.4}. The extra terms in the second lines of \eqref{sec7-39} and \eqref{sec7-40} enable us to determine the direction of the critical point in the case $\La_1=\La_2$ or/and $\nabla\mathcal{R}_\Omega(0)=0$.
\end{rem}
\begin{proof}[\bf{Proof of Proposition \ref{sec7-prop7.11}}]
To prove this proposition, it suffices to compute the term $\Psi_{\e,j}(x,y)$ with greater precision than in Proposition \ref{sec7-prop7.4}, as follows:
\begin{small}\begin{equation*}
\begin{split}
 \frac{\La_1G_\Omega(x,0)+\La_2G_\Omega(0,y)}{
 \ln \e+2\pi \mathcal{R}_\Omega(0)} =
 - \frac{\frac{\La_1\ln|x|}{2\pi} +\frac{\La_2\ln|y|}{2\pi} +(\La_1H_\Omega(x,0)+\La_2H_\Omega(0,y))}{ \ln \e+2\pi \mathcal{R}_\Omega(0)}=-\frac{\widetilde{h}(x,y)}{2\pi}+O\Big(\frac{\e^{2\beta}}{|\ln\e |}\Big).
\end{split}
\end{equation*}
\end{small}Then we get
\begin{small}\begin{equation*}
\begin{split}
&\Big(\frac{x_j}{|x|^2}+2\pi
 \frac{\partial H_\Omega(x,0)}{\partial x_j} \Big) \frac{\La_1G_\Omega(x,0)+\La_2G_\Omega(0,y)}{
 \ln \e+2\pi \mathcal{R}_\Omega(0)} \\=&-\frac{\widetilde{h}(x,y)x_j}{2\pi|x|^2}
 -
 \frac{
 \La_1\ln  |x| + \La_2\ln |y|+2\pi \mathcal{R}_\O(0) (\La_1+\La_2)
  }{\ln \e+2\pi \mathcal{R}_{\Omega}(0)} \frac{\partial \mathcal{R}_{\Omega}(0)}{\partial x_j}\\&
 -
\frac{\La_1\ln |x| + \La_2\ln  |y|  }{\ln \e+ 2\pi \mathcal{R}_{\Omega}(0)}\sum^2_{i=1}
x_i \frac{\partial^2 H_{\Omega}(0,0)}{\partial x_i\partial x_j}
 +O\Big(\frac{\e^{\beta}}{|\ln\e |}\Big).
\end{split}
\end{equation*}
\end{small}Also, by Taylor's expansion, we have
\begin{small}\begin{equation*}
\begin{split}
 &\La_1\frac{\partial \mathcal{R}_\Omega(x)}{\partial x_j}
 +
2\La_2\frac{\partial H_\Omega(x,y)}{\partial x_j} \\&=
(\La_1+\La_2) \frac{\partial \mathcal R_{\Omega}(0)}{\partial x_j}+ 2(\La_1+\La_2)
\sum^2_{i=1}
x_i \frac{\partial^2 H_{\Omega}(0,0)}{\partial x_i\partial x_j}
+2\sum^2_{i=1} \frac{\partial^2 H_{\Omega}(0,0)}{\partial y_i\partial x_j}(\La_1x_i+\La_2y_i)+O\Big(\e^{2\beta}\Big).
\end{split}
\end{equation*}
\end{small}Hence from above computations, we get
\begin{small}\begin{equation}\label{sec7-41}
\begin{split}
\Psi_{\e,j}(x,y) =&\frac{\La_1}{\pi}\left\{ -\frac{x_j}{|x|^2}\big(\widetilde{h}(x,y)-\La_1-\La_2\big)-
 \frac{\pi (\La_1\ln \frac{|x|}{\e} + \La_2\ln  \frac{|y|}{\e} )
  }{\ln \e+2\pi \mathcal{R}_{\Omega}(0)} \frac{\partial \mathcal{R}_{\Omega}(0)}{\partial x_j}
 \right. \\[1mm]&
\left. -2\pi
\frac{\La_1\ln \frac{|x|}{\e} + \La_2\ln  \frac{|y|}{\e}  }{\ln \e+ 2\pi \mathcal{R}_{\Omega}(0)}\sum^2_{i=1}
x_i \frac{\partial^2 H_{\Omega}(0,0)}{\partial x_i\partial x_j}
+2\pi\sum^2_{i=1}\frac{\partial^2 H_{\Omega}(0,0)}{\partial y_i\partial x_j}(\La_1x_i+\La_2y_i)
  \right\}
+O\left(\frac{\e^{\beta}}{|\ln \e|} \right).
\end{split}\end{equation}
\end{small}Finally, from \eqref{sec7-03} and \eqref{sec7-41}, we prove \eqref{sec7-39}. In a similar way we  deduce \eqref{sec7-40}.
\end{proof}

\begin{prop}\label{sec7-prop7.12}
For  $(x,y)\in \Omega_\e\times \Omega_\e$ satisfying
$|x|,|y|\sim \e^{\beta }$ with $\beta=\frac{\tau}{(1+\tau)^2}$ and $\tau=\frac{\La_1}{\La_2}$, letting
$(w,\gamma)=\left(\frac{x}{\e^\beta}, \frac{x+\tau y}{\e^{2\beta }}\right)$,
then it holds
\begin{small}\begin{align}\label{sec7-42}
 & \frac{\partial \mathcal{KR}_{\Omega_\e}(x,y)}{\partial x_j}\Big|_{(x,y)=
 \Big(\e^{\beta} w, \frac{-\e^{\beta} w+ \e^{2\beta } \gamma }{\tau} \Big) }
=-\frac{\La_1\La_2}{\pi}\overline{V}_{\e,1,j}(w,\gamma)+
   O\left(\frac{ \e^{\beta}}{|\ln \e|}
\right),
\end{align}\end{small}and
\begin{small}\begin{align}\label{sec7-43}
 & \frac{\partial \mathcal{KR}_{\Omega_\e}(x,y)}{\partial y_j}\Big|_{(x,y)=
 \Big(\e^{\beta} w, \frac{-\e^{\beta} w + \e^{2\beta } \gamma }{\tau} \Big) }
= - \frac{\La_2^2}{\pi}\overline{V}_{\e,2,j}(w,\gamma)+
   O\left(\frac{ \e^{\beta}}{|\ln \e|}
\right),
\end{align}\end{small}where
\begin{small}\begin{align}\label{sec7-42a}
  \overline{V}_{\e,1,j}(w,\gamma)=& \left[
  \frac{k(|w|,\tau)}{\e^{\beta}|w|^2(\ln \e +2\pi \mathcal{R}_{\Omega}(0))}
+\frac{\pi (\tau^2- 1) \big(\nabla \mathcal{R}_{\Omega}(0) \cdot w\big)
   }{ \tau |w|^2 (\ln \e +2\pi \mathcal{R}_{\Omega}(0))}-
\frac{(w\cdot \gamma) }{|w|^4}
  \left(2\beta- \frac{1}{\ln \e +2\pi \mathcal{R}_{\Omega}(0)}\right)
  \right. \notag\\[1mm]&
\left.   \,\,\,\,\,\, -\frac{  \tau \e^{\beta}\big( 4(w\cdot \gamma)^2-|w|^2\cdot|\gamma|^2\big) }{(\tau+1)^3|w|^6 }
\right]
  w_j + \frac{ \beta }{|w|^2}
  \left[1+ \frac{2 (w\cdot \gamma)\e^{\beta} }{ (\tau+1) |w|^2} \right]\gamma_j
  \notag \\[1mm]&  \,\,\,\,\,\,+
2\pi \e^{\beta} \sum^2_{i=1}\left[ (1+\tau)(\beta-1)
 \frac{\partial^2H_{\Omega}(0,0)}{\partial x_i\partial x_j}
-\big(  \tau-\frac{1}{\tau} \big) \frac{\partial^2H_{\Omega}(0,0)}{\partial y_i\partial x_j} \right]
 w_i\notag  \\[1mm]&
  \,\,\,\,\,\,\, +\pi \left[ \frac{(\tau+1) (\ln|w|+(\beta-1)\ln \e) - \ln\tau }{
  \ln \e +2\pi \mathcal{R}_{\Omega}(0)}\right]\frac{\partial
   \mathcal{R}_\Omega(0)}{\partial x_j},
\end{align}\end{small}and
\begin{small}\begin{align}\label{sec7-43a}
  \overline{V}_{\e,2,j}(w,\gamma)=&
 \left[-
  \frac{ \tau k(|w|,\tau)}{|w|^2 \e^{\beta}(\ln \e +2\pi \mathcal{R}_{\Omega}(0))}-
  \frac{\pi (\tau^2- 1) \big(\nabla \mathcal{R}_{\Omega}(0) \cdot w\big)
   }{ |w|^2 (\ln \e +2\pi \mathcal{R}_{\Omega}(0))}
  -
\frac{  \tau(w\cdot \gamma)}{|w|^4}
\Big(\frac{2 k(|w|,\tau)-1}{\ln \e +2\pi \mathcal{R}_{\Omega}(0)}+2\tau\beta\Big)\right. \notag  \\[1mm]&
\left.
 \,\,\,\,\,\, -\frac{ \beta \tau^2(\tau+2)  \e^{\beta}\big( 4(w\cdot \gamma)^2-|w|^2\cdot|\gamma|^2\big) }{(\tau+1) |w|^6}
\right]
  w_j
+  \frac{ \tau}{|w|^2} \left[
 \tau \beta+\frac{k(|w|,\tau)}{\ln \e +2\pi \mathcal{R}_{\Omega}(0)}
  +
\frac{2\beta(\tau+2)  (w\cdot \gamma) \e^{\beta}}{  (1+\tau) |w|^2}    \right]
  \gamma_j
 \notag  \\[1mm]&  \,\,\,\,\,\,   +
 2\pi\e^{\beta}  \sum^2_{i=1}\left[ -\frac{(1+\tau)(\beta-1)}{\tau}
 \frac{\partial^2H_{\Omega}(0,0)}{\partial x_i\partial x_j}
-\big(  \tau-\frac{1}{\tau} \big) \frac{\partial^2H_{\Omega}(0,0)}{\partial y_i\partial x_j} \right]
 w_i\notag  \\[1mm]&
  \,\,\,\,\,\,  +\pi \left[ \frac{(\tau+1) (\ln|w|+(\beta-1)\ln \e) - \ln\tau }{
  \ln \e +2\pi \mathcal{R}_{\Omega}(0)}\right]\frac{\partial
   \mathcal{R}_\Omega(0)}{\partial x_j}
\end{align}\end{small}with $k(r,\tau)$ being the function in \eqref{sec7-07}.
\end{prop}

\begin{proof}
This is similar to the proof of Proposition \ref{sec7-prop7.5}. Here we need more precise estimates in Taylor's expansions.
From \eqref{sec7-39}, we have
\begin{small}\begin{equation}\label{sec7-44}\begin{split}
 \frac{\partial \mathcal{KR}_{\Omega_\e}(x,y)}{\partial x_j}=
 -
 \frac{\La_1\La_2}\pi
 &\left\{ \frac{\widetilde{h}(x,y)x_j}{\La_2|x|^2}+\frac{\pi(\tau\ln \frac{|x|}{\e}+ \ln \frac{|y|}{\e} ) }{\ln \e+ 2\pi \mathcal{R}_{\Omega}(0)}
  \frac{\partial \mathcal{R}_{\Omega}(0)}{\partial x_j}-
\frac{ (x_{j}-y_{j})}{|x-y|^{2}}\right. \\[1mm]
& \left.
+2\pi \sum^2_{i=1}\left[
\frac{\tau\ln \frac{|x|}{\e}+  \ln \frac{|y|}{\e} }{\ln \e+ 2\pi \mathcal{R}_{\Omega}(0)}
x_i \frac{\partial^2 H_{\Omega}(0,0)}{\partial x_i\partial x_j}
-(\tau x_i+ y_i) \frac{\partial^2 H_{\Omega}(0,0)}{\partial y_i\partial x_j}
\right]
\right\}
+O\Big( \frac{\e^{\beta}}{|\ln \e|}\Big).
\end{split}
\end{equation}
\end{small}Taking $(x,y)=
 \Big(\e^{\beta} w, \frac{-\e^{\beta} w + \e^{2\beta } \gamma }{\tau} \Big)$, then
 \begin{small}
\begin{equation*}
\begin{split}
\frac{\widetilde{h}(x,y)}{\La_2}=&\frac{
 \tau\ln  |x| +  \ln |y|+2\pi \mathcal{R}_\O(0) (\tau+1)
  +\pi\big( \tau(\nabla\mathcal{R}_\O(0)\cdot x)+ (\nabla\mathcal{R}_\O(0)\cdot y)\big)
  }{\ln \e+2\pi \mathcal{R}_{\Omega}(0)}\\=&
  \frac{
 \tau\ln  |x| +  \ln |y|+2\pi \mathcal{R}_\O(0) (\tau+1)}{\ln \e+2\pi \mathcal{R}_{\Omega}(0)}
 +
 \frac{ \pi(\tau-\frac{1}{\tau}) (\nabla\mathcal{R}_\O(0)\cdot w) \e^{\beta}
  }{\ln \e+2\pi \mathcal{R}_{\Omega}(0)} +O\Big( \frac{\e^{2\beta}}{|\ln \e|}\Big).
\end{split}\end{equation*}\end{small}Expanding \eqref{sec7-09} and \eqref{sec7-10} to the next order
we get
\begin{small}
\begin{equation*}
\begin{split} \frac{
 \tau \ln  |x| + \ln |y| }{\ln \e+2\pi \mathcal{R}_{\Omega}(0)}
 =
 \frac{ (1+\tau)
(\beta \ln \e + \ln |w|)-\ln \tau -\frac{(w\cdot \gamma)}{|w|^2} \e^{\beta }}{
\ln \e+2\pi \mathcal{R}_{\Omega}(0)}+O\left(\frac{\e^{2\beta}}{|\ln \e|}\right),
\end{split}\end{equation*}
\end{small}
\begin{small}\begin{equation*}
	\begin{split}
    \frac{x_{j}-y_{j }}{|x-y|^{2}}  =
    &
 \frac{ \tau}{ (1+\tau)^2\e^{\beta}|w|^2 }\left[(1+\tau)w_{j} +
\e^{\beta} \Big(\frac{2(w\cdot\gamma)w_j}{|w|^2}-\gamma_j\Big)+\e^{2\beta} \Big( \frac{(
4(w\cdot\gamma)^2 -|w|^2|\gamma|^2)w_j -2|w|^2(w\cdot\gamma)\gamma_j }{(\tau+1)|w|^4}\Big)
 \right]\\&+
 O\Big( \e^{2\beta}\Big).
	\end{split}
\end{equation*}
\end{small}Observing that $\tau x_i+ y_i=(\tau-\frac{1}{\tau})\e^{\beta}w_i+O\big(\e^{2\beta}\big)$ we prove \eqref{sec7-42} by inserting the  above computations into \eqref{sec7-44}.

\vskip 0.1cm
Proceeding in the same way for \eqref{sec7-11} leads to
\begin{small}\begin{equation*}
	\begin{split}
    \frac{y_{j }}{|y|^{2}}  =
    &-
 \frac{ \tau}{ \e^{\beta}|w|^2}\left[ w_{j} +
 \e^{\beta}\Big(\frac{2(w\cdot\gamma)w_{j}}{|w|^2}-\gamma_{j}\Big)
+\e^{2\beta} \Big( \frac{(
4(w\cdot\gamma)^2 -|w|^2|\gamma|^2)w_j -2|w|^2(w\cdot\gamma)\gamma_j }{|w|^4}\Big) \right] +
 O\Big( \e^{2\beta}\Big),
	\end{split}
\end{equation*}\end{small}which proves \eqref{sec7-43} and ends the proof.
\end{proof}
\vskip0.2cm
As in \eqref{sec7-12}, we define
$${\bf{V}}_\e(w,\gamma)=\Big( \nabla_x\mathcal{KR}_{\Omega_\e}(x,y),\nabla_y\mathcal{KR}_{\Omega_\e}(x,y) \Big)\Big|_{(x,y)=
 \Big(\e^{\beta} w,\frac{-\e^{\beta} w
 +\e^{2\beta} \gamma }{\tau}\Big)}.$$
Then Proposition \ref{sec7-prop7.12} gives that
\begin{small}\begin{equation*}
{\bf{V}}_{\e}(w,\gamma)=\overline{\bf{V}}_{\e}(w,\gamma)  \left(\begin{array}{cc}
    -\frac{\La_1\La_2}{\pi} {\bf{E}}_{2\times 2} &
   {\bf{O}}_{2\times 2}\\[3mm] {\bf{O}}_{2\times 2}
    &  -\frac{\La_2^2}{\pi} {\bf{E}}_{2\times 2}
  \end{array}\right)+
O\Big( \frac{\e^{\beta}}{|\ln \e|}\Big)~~~\mbox{for any}~~~ (w,\gamma)\in  \mathcal{H}'_\e,
\end{equation*}
\end{small}where
 $\overline{\bf{V}}_\e(w,\gamma)$ is defined by
 \begin{small}
\begin{equation*}
\overline{\bf{V}}_\e(w,\gamma)=\Big(\overline{\bf{V}}_{\e,1}(w,\gamma), \overline{\bf{V}}_{\e,2}(w,\gamma)  \Big)~~~\mbox{with}~~~
\overline{\bf{V}}_{\e,i}(w,\gamma)=\Big(\overline{V}_{\e,i,1}(w,\gamma), \overline{V}_{\e,i,2}(w,\gamma)  \Big)~~~\mbox{for}~~~i=1,2.
\end{equation*}\end{small}Here $\overline{V}_{\e,i,1}(w,\gamma)$ and $\overline{V}_{\e,i,2}(w,\gamma)$ are the functions in
\eqref{sec7-42a} and \eqref{sec7-43a}.
Next the analogous of Proposition \ref{sec7-teo7.6} holds.
\begin{prop}\label{sec7-teo7.13}
For any $(w,\gamma)\in  \mathcal{H}'_\e$,
it holds\begin{small}
\begin{equation*}
\nabla_{(w,\gamma)}{\bf{V}}_{\e}(w,\gamma)=\nabla_{(w,\gamma)} \overline{\bf{V}}_{\e}(w,\gamma)  \left(\begin{array}{cc}
    -\frac{\La_1\La_2}{\pi} {\bf{E}}_{2\times 2} &
   {\bf{O}}_{2\times 2}\\[3mm] {\bf{O}}_{2\times 2}
    &  -\frac{\La_2^2}{\pi} {\bf{E}}_{2\times 2}
  \end{array}\right)+
  O\Big( \frac{\e^{\beta}}{|\ln \e|}\Big).
\end{equation*}
\end{small}
\end{prop}
\begin{proof}
Recalling \eqref{sec7-14} of Proposition \ref{sec7-teo7.6}, we get
\begin{small}\begin{equation*}\begin{split}
\nabla_{(w,\gamma)} {\bf{V}}_{\e,1,j}(w,\gamma)=&
\nabla_{(w,\gamma)}  \left\{ \left[
\frac{\partial \mathcal{KR}_{(B(0,\e))^c}(x,y)}{\partial x_j }+\Psi_{\e,j}(x,y)
\right]\Big|_{(x,y)=\Big(\e^{\beta }w,\frac{-\e^{\beta }
w+\e^{2\beta } \gamma }{\tau}\Big)} \right\}
+O\left(\frac{\e^{1-\beta }}{|\ln \e|}\right).
\end{split}\end{equation*}\end{small}Arguing exactly as in the proof of Proposition \ref{sec7-teo7.6} the claim follows.
\end{proof}

\vskip 0.2cm

\subsection{The case $\La_1=\La_2$ (Proof of Theorem \ref{sec1-teo16})}~

\vskip 0.1cm
From now we will use Proposition \ref{sec7-prop7.12} and Proposition \ref{sec7-teo7.13} to prove Theorem \ref{sec1-teo16}.
 If $\Lambda_1=\Lambda_2$, then $\tau=1$ and $\beta=\frac14$. In this
 case, the results in Proposition \ref{sec7-prop7.12} can be stated as follows.

\begin{prop}\label{sec7-prop7.14}
For $\La_1=\La_2$, $(x,y)\in \Omega_\e\times \Omega_\e$ satisfying $|x|,|y|\sim \e^{\frac{1}{4} }$, letting
$(w,\gamma)=\left(\frac{x}{\e^{\frac{1}{4} }}, \frac{x+y}{\e^{\frac{1}{2} }}\right)$, then it holds
\begin{small}\begin{align*}
 & \frac{\partial \mathcal{KR}_{\Omega_\e}(x,y)}{\partial x_j}\Big|_{(x,y)=
 \Big(\e ^{\frac{1}{4} } w,  -\e ^{\frac{1}{4} }w+ \e^{\frac{1}{2} } \gamma   \Big) }\notag
 \\=& -\frac{\La_1^2}{\pi}\left\{ \left[
  \frac{k(|w| )}{ |w|^2 \e^{\frac{1}{4}}(\ln \e +2\pi \mathcal{R}_{\Omega}(0))} -
\frac{ (w\cdot \gamma) }{|w|^4}
  \left(\frac{1}{2}- \frac{1}{\ln \e +2\pi \mathcal{R}_{\Omega}(0)}
  \right) -\frac{  \e^{\frac{1}{4}}( 4(w\cdot \gamma)^2-|w|^2\cdot|\gamma|^2) }{8
 |w|^6}\right] w_j  \right.
 \notag \\[1mm]& \left.   + \frac{ 1}{4   |w|^2}
  \left[1+ \frac{ \e^{\frac{1}{4}} (w\cdot \gamma) }{   |w|^2} \right]\gamma_j+\frac{ \pi(4\ln|w|-3\ln \e) }{2( \ln \e +2\pi \mathcal{R}_{\Omega}(0))}\frac{\partial \mathcal{R}_\Omega(0)}{\partial x_j}
  - 3\pi \e^{\frac{1}{4} } \sum^2_{i=1} \frac{\partial^2H_{\Omega}(0,0)}{\partial x_i\partial x_j}w_i\right\}
 +O\left(\frac{ \e^{\frac{1}{4}}}{|\ln \e|}
\right),
\end{align*}\end{small}and
\begin{small}\begin{align*}
 & \frac{\partial \mathcal{KR}_{\Omega_\e}(x,y)}{\partial y_j}\Big|_{(x,y)=
 \Big(\e ^{\frac{1}{4} } w,  -\e ^{\frac{1}{4} }w+ \e^{\frac{1}{2} } \gamma   \Big) }\notag
 \\=&-\frac{\La_1^2}{\pi}\left\{\left[-
  \frac{k(|w| )}{ |w|^2 \e^{\frac{1}{4}}(\ln \e +2\pi \mathcal{R}_{\Omega}(0))} -
\frac{ (w\cdot \gamma) }{|w|^4}
  \left(\frac{1}{2}+ \frac{2k(|w| )-1}{\ln \e +2\pi \mathcal{R}_{\Omega}(0)}
  \right) -\frac{  3\e^{\frac{1}{4}}( 4(w\cdot \gamma)^2-|w|^2\cdot|\gamma|^2) }{8
 |w|^6}\right] w_j  \right.
 \notag \\[1mm]& \left.   + \frac{ 1}{    |w|^2}
  \left[ \frac{1}{4}+\frac{k(|w| )}{\ln \e +2\pi \mathcal{R}_{\Omega}(0)}+ \frac{ 3 \e^{\frac{1}{4}} (w\cdot \gamma) }{  4 |w|^2} \right]\gamma_j+\frac{ \pi(4\ln|w|-3\ln \e) }{2( \ln \e +2\pi \mathcal{R}_{\Omega}(0))}\frac{\partial \mathcal{R}_\Omega(0)}{\partial x_j}
  + 3\pi \e^{\frac{1}{4} } \sum^2_{i=1} \frac{\partial^2H_{\Omega}(0,0)}{\partial x_i\partial x_j}w_i\right\} \notag \\[1mm]&
 +O\left(\frac{ \e^{\frac{1}{4}}}{|\ln \e|}
\right),
\end{align*}\end{small}where $k(r):=k(r,\tau)|_{\tau=1}=2\ln r+3\pi \mathcal{R}_{\Omega}(0)$ with $k(r,\tau)$ being the function in \eqref{sec7-07}.
\end{prop}

Denote
\begin{small}
\begin{equation*}
\begin{cases}
  l_{\e,j}(w,\gamma):=-\frac{\pi}{\Lambda_1^2}\frac{\partial \mathcal{KR}_{\Omega_\e}(x,y)}{\partial x_j}\Big|_{(x,y)=
 \Big(\e ^{\frac{1}{4} } w,  -\e ^{\frac{1}{4} }w+ \e^{\frac{1}{2} } \gamma   \Big) },\\[2mm]
m_{\e,j}(w,\gamma):=-\frac{\pi}{\Lambda_2^2}\frac{\partial \mathcal{KR}_{\Omega_\e}(x,y)}{\partial y_j}\Big|_{(x,y)=
 \Big(\e ^{\frac{1}{4} } w,  -\e ^{\frac{1}{4} }w+ \e^{\frac{1}{2} } \gamma   \Big) }.
 \end{cases}
 \end{equation*}\end{small}First, we give the main part of $l_{\e,j}(w,\gamma)$ and $m_{\e,j}(w,\gamma)$.
For any $f(w,\gamma)$, we denote
\begin{small}\begin{equation}\label{sec7-45b}
\partial^0 f(w,\gamma):=f(w,\gamma), ~~\partial^1 f(w,\gamma):=\nabla_{(w,\gamma)}f(w,\gamma).
\end{equation}\end{small}Then following result holds.
\begin{lem}\label{sec7-lem7.22a}
For any $(w,\gamma)\in \mathcal{\widetilde{H}}_\e=\Big\{(w,\gamma)\in \mathcal{H}'_\e,\tau=1\Big\}$, we have
 \begin{small}\begin{align}\label{sec7-45}
 \begin{cases}
\partial^k l_{\e,j}(w,\gamma)=\partial^k l^*_{\e,j}(w,\gamma) +O\left(\frac{ \e^{\frac{1}{4}}}{|\ln \e|}
\right),\\[3mm]
\partial^k m_{\e,j}(w,\gamma)=\partial^k m_{\e,j}^*(w,\gamma) +O\left(\frac{ \e^{\frac{1}{4}}}{|\ln \e|}
\right),\end{cases}
\end{align}\end{small}for $k=0,1$, $j=1,2$, where $\mathcal{H}'_\e$ is the notation in Proposition \ref{sec7-prop7.5},
\begin{small}\begin{align*}
 l^{*}_{\e,j}(w,\gamma):= &\left[
  \frac{k(|w| )}{ |w|^2 \e^{\frac{1}{4}}(\ln \e +2\pi \mathcal{R}_{\Omega}(0))} -
\frac{ (w\cdot \gamma) }{|w|^4}
  \left(\frac{1}{2}- \frac{1}{\ln \e +2\pi \mathcal{R}_{\Omega}(0)}
  \right) -\frac{  \e^{\frac{1}{4}}( 4(w\cdot \gamma)^2-|w|^2\cdot|\gamma|^2) }{8
 |w|^6}\right] w_j
 \notag \\[1mm]&     + \frac{ 1}{4   |w|^2}
  \left[1+ \frac{ \e^{\frac{1}{4}} (w\cdot \gamma) }{   |w|^2} \right]\gamma_j+\frac{ \pi(4\ln|w|-3\ln \e) }{2( \ln \e +2\pi \mathcal{R}_{\Omega}(0))}\frac{\partial \mathcal{R}_\Omega(0)}{\partial x_j}
  - 3\pi \e^{\frac{1}{4} } \sum^2_{i=1} \frac{\partial^2H_{\Omega}(0,0)}{\partial x_i\partial x_j}w_i,
\end{align*}\end{small}and
\begin{small}\begin{align*}
m^*_{\e,j}(w,\gamma):=& \left[-
  \frac{k(|w| )}{ |w|^2 \e^{\frac{1}{4}}(\ln \e +2\pi \mathcal{R}_{\Omega}(0))} -
\frac{ (w\cdot \gamma) }{|w|^4}
  \left(\frac{1}{2}+ \frac{2k(|w| )-1}{\ln \e +2\pi \mathcal{R}_{\Omega}(0)}
  \right) -\frac{  3\e^{\frac{1}{4}}( 4(w\cdot \gamma)^2-|w|^2\cdot|\gamma|^2) }{8
 |w|^6}\right] w_j
 \notag \\[1mm]&    + \frac{ 1}{    |w|^2}
  \left[ \frac{1}{4}+\frac{k(|w| )}{\ln \e +2\pi \mathcal{R}_{\Omega}(0)}+ \frac{ 3 \e^{\frac{1}{4}} (w\cdot \gamma) }{  4 |w|^2} \right]\gamma_j+\frac{ \pi(4\ln|w|-3\ln \e) }{2( \ln \e +2\pi \mathcal{R}_{\Omega}(0))}\frac{\partial \mathcal{R}_\Omega(0)}{\partial x_j}
  + 3\pi \e^{\frac{1}{4} } \sum^2_{i=1} \frac{\partial^2H_{\Omega}(0,0)}{\partial x_i\partial x_j}w_i.
\end{align*}\end{small}
\end{lem}
 \begin{proof}First, by Proposition \ref{sec7-prop7.12} and Proposition \ref{sec7-prop7.14}, we have \eqref{sec7-45}  with $k=0$.
Also, using Proposition \ref{sec7-teo7.13}, we obtain \eqref{sec7-45} with $k=1$.
\end{proof}
Now we devote to solving $l_{\e,j}(w,\gamma)=0$ and $m_{\e,j}(w,\gamma)=0$. For this purpose,  we introduce following transform firstly.
 Let
\begin{small}
\begin{equation*}
p_{\e,j}(w,\gamma) := \e^{-\frac14} l_{\e,j}(w,\gamma) \left[ 1 + \frac{4 k(|w|)}{\ln \e + 2\pi \mathcal{R}_{\Omega}(0)} + \frac{3 \e^{\frac{1}{4}} (w \cdot \gamma)}{|w|^2} \right]
- \e^{-\frac14} m_{\e,j}(w,\gamma) \left[ 1 + \frac{\e^{\frac{1}{4}} (w \cdot \gamma)}{|w|^2} \right],
\end{equation*}
\end{small}and
\begin{small}
\begin{equation}\label{sec7-46}
q_{\e,j}(w,\gamma) := \frac{1}{2} \Big( l_{\e,j}(w,\gamma) + m_{\e,j}(w,\gamma) \Big).
\end{equation}
\end{small}Then we easily get following result.
\begin{lem}It holds
\begin{small}
\begin{equation}\label{sec7-47}
\begin{cases}
l_{\e,j}(w,\gamma) = 0,\\
m_{\e,j}(w,\gamma) = 0,
\end{cases} \Leftrightarrow
\begin{cases}
p_{\e,j}(w,\gamma) = 0,\\
q_{\e,j}(w,\gamma) = 0.
\end{cases}
\end{equation}
\end{small}Moreover, if $(w,\gamma)$ solves $l_{\e,j}(w,\gamma) = m_{\e,j}(w,\gamma) = 0$ for $j = 1,2$, then it holds
\begin{small}
\begin{equation}\label{sec7-48}
\det
\begin{pmatrix}
\Big( \frac{\partial l_{\e,j}(w,\gamma)}{\partial w_i} \Big)_{1\le i,j \le 2} &
\Big( \frac{\partial l_{\e,j}(w,\gamma)}{\partial \gamma_i} \Big)_{1\le i,j \le 2} \\[4mm]
\Big( \frac{\partial m_{\e,j}(w,\gamma)}{\partial w_i} \Big)_{1\le i,j \le 2} &
\Big( \frac{\partial m_{\e,j}(w,\gamma)}{\partial \gamma_i} \Big)_{1\le i,j \le 2}
\end{pmatrix} \neq 0
\Leftrightarrow
\det
\begin{pmatrix}
\Big( \frac{\partial p_{\e,j}(w,\gamma)}{\partial w_i} \Big)_{1\le i,j \le 2} &
\Big( \frac{\partial p_{\e,j}(w,\gamma)}{\partial \gamma_i} \Big)_{1\le i,j \le 2} \\[4mm]
\Big( \frac{\partial q_{\e,j}(w,\gamma)}{\partial w_i} \Big)_{1\le i,j \le 2} &
\Big( \frac{\partial q_{\e,j}(w,\gamma)}{\partial \gamma_i} \Big)_{1\le i,j \le 2}
\end{pmatrix} \neq 0.
\end{equation}
\end{small}

\end{lem}

Now we give the expansion of $p_{\e,j}(w,\gamma)$ and $q_{\e,j}(w,\gamma)$.
\begin{prop}
For any $(w,\gamma)\in \mathcal{\widetilde{H}}_\e=\Big\{(w,\gamma)\in \mathcal{H}'_\e,\tau=1\Big\}$, we have
\begin{small}
\begin{equation}\label{sec7-84a}
\begin{cases}
\partial^k p_{\e,j}(w,\gamma)= \partial^k  \bar p_{\e,j}(w,\gamma) + \partial^k \widetilde{p}_{\e,j}(w,\gamma) + O\Big(\frac{1}{|\ln \e|}\Big), \\[2mm]
\partial^k  q_{\e,j}(w,\gamma) =\partial^k  \bar q_{\e,j}(w,\gamma) + O\Big(\frac{1}{|\ln \e|}\Big),
 \end{cases}
 \end{equation}\end{small}with $k=0,1$, $j=1,2$, $\partial^k$ and
$\mathcal{H}'_\e$ being the notations in \eqref{sec7-45b} and Proposition \ref{sec7-prop7.5},
\begin{small}
\begin{equation}\label{sec7-49}
\bar p_{\e,j}(w,\gamma) := \left[ \frac{2 k(|w|)}{|w|^2 \e^{1/2} (\ln \e + 2\pi \mathcal{R}_{\Omega}(0))} - \frac{|\gamma|^2}{4 |w|^4} \right] w_j
- \frac{3 \pi (w \cdot \gamma)}{|w|^2} \frac{\partial \mathcal{R}_\Omega(0)}{\partial x_j}
- 6 \pi \sum_{i=1}^2 \frac{\partial^2 H_\Omega(0,0)}{\partial x_i \partial x_j} w_i,
\end{equation}
\end{small}\begin{small}
\begin{equation}\label{sec7-50}
\bar q_{\e,j}(w,\gamma) := \frac{(w \cdot \gamma) w_j}{|w|^4} - \frac{\gamma_j}{2 |w|^2} + 3 \pi \frac{\partial \mathcal{R}_\Omega(0)}{\partial x_j},
\end{equation}
\end{small}and
\begin{small}
\begin{equation*}
\begin{split}
\widetilde{p}_{\e,j}(w,\gamma) :=&
\Bigg[ \Big( \frac{2 k(|w|)}{|w|^2 \e^{1/2} (\ln \e + 2\pi \mathcal{R}_{\Omega}(0))}
- \frac{(w \cdot \gamma)}{|w|^4 \e^{1/4}} \big( 1 - \frac{2}{\ln \e + 2\pi \mathcal{R}_{\Omega}(0)} \big) \Big) w_j
+ \frac{\pi (4 \ln|w| - 3 \ln \e)}{\e^{1/4} (\ln \e + 2\pi \mathcal{R}_{\Omega}(0))} \frac{\partial \mathcal{R}_\Omega(0)}{\partial x_j} \Bigg]\\&
\times
\frac{2 k(|w|)}{\ln \e + 2\pi \mathcal{R}_{\Omega}(0)}.
\end{split}\end{equation*}
\end{small}
\end{prop}
\begin{proof}
First, by \eqref{sec7-45}, we have
\begin{small}
\begin{equation*}
\begin{split}
p_{\e,j}(w,\gamma) &= \e^{-\frac14} l^*_{\e,j}(w,\gamma) \left[ 1 + \frac{4 k(|w|)}{\ln \e + 2\pi \mathcal{R}_{\Omega}(0)} + \frac{3 \e^{\frac{1}{4}} (w \cdot \gamma)}{|w|^2} \right]
- \e^{-\frac14} m^*_{\e,j}(w,\gamma) \left[ 1 + \frac{\e^{\frac{1}{4}} (w \cdot \gamma)}{|w|^2} \right] + O\Big(\frac{1}{|\ln \e|}\Big) \\
&= \bar p_{\e,j}(w,\gamma) + \widetilde{p}_{\e,j}(w,\gamma) + O\Big(\frac{1}{|\ln \e|}\Big).
\end{split}
\end{equation*}
\end{small}Also using \eqref{sec7-45}, we compute
\begin{small}
\begin{equation*}
q_{\e,j}(w,\gamma) = \frac{1}{2} \Big( l^*_{\e,j}(w,\gamma) + m^*_{\e,j}(w,\gamma) \Big) + O\Big(\frac{\e^{\frac{1}{4}}}{|\ln \e|}\Big)
= \bar q_{\e,j}(w,\gamma) + O\Big(\frac{1}{|\ln \e|}\Big).
\end{equation*}
\end{small}These give \eqref{sec7-84a} with $k=0$. Finally, using Proposition \ref{sec7-teo7.13}, we obtain \eqref{sec7-84a} with $k=1$.
\end{proof}
\begin{rem}
Here we point out that the term $\widetilde{p}_{\e,j}(w,\gamma)$ is crucial. To estimate  ${p}_{\e,j}(w,\gamma)$ and ${q}_{\e,j}(w,\gamma)$,
\eqref{sec7-84a} with $k=0$ can be written as follows:
\begin{small}
\begin{equation*}
\begin{cases}
p_{\e,j}(w,\gamma)= \bar p_{\e,j}(w,\gamma)+O\Big(\frac{1}{|\ln \e|}\Big), \\[2mm]
q_{\e,j}(w,\gamma)= \bar q_{\e,j}(w,\gamma)
  +O\Big(\frac{1}{|\ln \e|} \Big).
 \end{cases}
 \end{equation*}\end{small}But to estimate  the derivative of $p_{\e,j}(w,\gamma)$, the term $\widetilde{p}_{\e,j}(w,\gamma)$ in \eqref{sec7-84a} cannot be ignored.
\end{rem}

\begin{prop}
 Set
\begin{small}$$\widetilde{\bf{M}}:=\left[\frac{\partial^2 H_\O(0,0)}{\partial y_i\partial y_k}
 -{3\pi}\frac{\partial \mathcal{R}_\O(0)}{\partial y_i} \frac{\partial \mathcal{R}_\O(0)}{\partial y_k} \right]_{1\leq i,k\leq 2}.$$
\end{small}If $\widetilde{\bf{M}}$ has two different eigenvalues $\lambda_1,\lambda_2$, then system
  \begin{small}
 \begin{equation}\label{sec7-52}
 \begin{cases}
 \bar p_{\e,j}(w,\gamma)=0,\\[1mm]
 \bar q_{\e,j}(w,\gamma)=0,
 \end{cases}
 \end{equation}
 \end{small}has exactly four solutions $(w_\e^{(m),\pm},\gamma_\e^{(m),\pm})$ for $m=1,2$, with
 \begin{small}
 \begin{equation}\label{sec7-53}
\begin{cases}
 w_\e^{(m),\pm}= \pm \left[e^{-\frac{3\mathcal{R}_{\Omega}(0)}{2}}+\frac{1}{2}e^{-\frac{3\mathcal{R}_{\Omega}(0)}{2}}\big( 9\pi^2  |\nabla \mathcal{R}_{\Omega}(0)|^2
   + 6\pi  \lambda_m+o(1)\big)\e^{\frac{1}{2}} \ln \e\right] \nu^{(m)},\\[2mm]
 \gamma^{(m),\pm}_\e=6\pi| w_\e^{(m),\pm}|^2\nabla \mathcal{R}_\O(0)-12\pi
\big(\nabla\mathcal{R}_\O(0)\cdot  w_\e^{(m),\pm}\big)  w_\e^{(m),\pm},
\end{cases}
 \end{equation}
 \end{small}where
$\nu^{(m)}$ is the unit eigenfunction corresponding to the eigenvalue $\lambda_m$ of $\widetilde{\bf{M}}$.
 \end{prop}
 \begin{proof}
 From \eqref{sec7-50} we find
\begin{small}\begin{equation}\label{sec7-54}
\begin{split}
 w\cdot \gamma=- 6\pi |w|^2 \nabla \mathcal{R}_\Omega(0)\cdot  w.
 \end{split}\end{equation}\end{small}and then
 \begin{small}\begin{equation}\label{Sec7-55}
\begin{split}
\gamma=  6\pi |w|^2 \nabla \mathcal{R}_\O(0)- 12\pi
\big(\nabla\mathcal{R}_\O(0)\cdot w\big) w.
 \end{split}
\end{equation}
\end{small}Inserting \eqref{sec7-54} and \eqref{Sec7-55} into \eqref{sec7-49}, for $j=1,2$, we obtain
\begin{small}
\begin{equation}\label{sec7-56}
\begin{split}
\left[
  \frac{2 k(|w| )}{ |w|^2 \e^{\frac{1}{2}}(\ln \e +2\pi \mathcal{R}_{\Omega}(0))} -
   9\pi^2  |\nabla \mathcal{R}_{\Omega}(0)|^2
\right] w_j
   = 6\pi  \big(\widetilde{\mathbf{M}}w\big)_j.
\end{split}
\end{equation}
\end{small}Thus $w$ must be an eigenvector of $\widetilde{\mathbf{M}}$.

\vskip 0.05cm

Suppose that $\widetilde{\mathbf{M}}$ has two distinct eigenvalues, $\lambda_1$ and $\lambda_2$, and let
$\nu^{(m)}$ denote the unit eigenvector corresponding to the eigenvalue $\lambda_m$.
Then   the eigenvector $w_\e^{(m),\pm}$ to
\eqref{sec7-56} must be proportional to either $\pm\nu^{(1)}$ or $\pm\nu^{(2)}$.
Hence by \eqref{sec7-56} we get
\begin{small}
\begin{equation}\label{sec7-57}
\begin{split}
  \frac{2 k(|w_\e^{(m),\pm}| )}{ |w_\e^{(m),\pm}|^2 \e^{\frac{1}{2}}(\ln \e +2\pi \mathcal{R}_{\Omega}(0))} - 9\pi^2  |\nabla \mathcal{R}_{\Omega}(0)|^2
   = 6\pi  \lambda_m,
\end{split}
\end{equation}
\end{small}which has a unique solution satisfying
\begin{small}
\begin{equation*}
\begin{split}
 |w_\e^{(m),\pm}| = e^{-\frac{3\mathcal{R}_{\Omega}(0)}{2}}
 + \frac{1}{2}e^{-\frac{3\mathcal{R}_{\Omega}(0)}{2}}\Big( 9\pi^2  |\nabla \mathcal{R}_{\Omega}(0)|^2
   + 6\pi  \lambda_m \Big)\e^{\frac{1}{2}} \ln \e.
\end{split}
\end{equation*}
\end{small}Therefore, system \eqref{sec7-52} admits exactly
four solutions given by \eqref{sec7-53}.
\end{proof}
Let $\check {\bf{V}}_\e(w,\gamma)=\big(\bar p_{\e,1}(w,\gamma),\bar p_{\e,2}(w,\gamma),\bar q_{\e,1}(w,\gamma),\bar q_{\e,2}
(w,\gamma)\big)$.
We have the following results.
\begin{prop}\label{sec7-prop7.16}
If $\widetilde{\bf{M}}$ has two different eigenvalues $\lambda_1,\lambda_2$, then it holds
\begin{small}
\begin{equation}\label{sec7-58}
\begin{split}
& det~Jac~{\bf \check{V}}_\e(w_\e^{(m),\pm},\gamma^{(m),\pm}_\e)
=  \frac{6\pi e^{9\pi \mathcal{R}_{\Omega}(0)} (\lambda_j-\lambda_m)}{   \e^{\frac{1}{2}} \ln \e }\big(1+o(1)\big)~~~\mbox{for}~~~m,j=1,2~~~\mbox{and}~~~j\neq m.
 \end{split}
 \end{equation}\end{small}
\end{prop}
 \begin{proof}
 By direct computations, we have
 \begin{small}
\begin{equation}\label{sec7-59}
\begin{split}
\frac{\partial \bar p_{\e,i}(w,\gamma)}{\partial w_j}=&\left[
  \frac{2 k(|w| )}{ |w|^2 \e^{\frac{1}{2}}(\ln \e +2\pi \mathcal{R}_{\Omega}(0))} -
  \frac{ |\gamma|^2     }{
 4|w|^4}\right] \delta_{ij}+\left[
  \frac{2 (|w|\cdot k'(|w| )-2k(|w|))}{ |w|^2 \e^{\frac{1}{2}}(\ln \e +2\pi \mathcal{R}_{\Omega}(0))}+
  \frac{  |\gamma|^2     }{
 |w|^4}\right] \frac{w_iw_j}{|w|^2} \\[1mm]&
   - \frac{3\pi\gamma_j }{|w|^2}  \frac{\partial \mathcal{R}_\Omega(0)}{\partial x_i}
   +\frac{6\pi (w\cdot\gamma)w_j}{|w|^4}\frac{\partial \mathcal{R}_\Omega(0)}{\partial x_i}
  -6\pi   \frac{\partial^2H_{\Omega}(0,0)}{\partial x_i\partial x_j},
 \end{split}
 \end{equation}\end{small}
 \begin{small}
\begin{equation*}
\begin{split}
\frac{\partial \bar p_{\e,i}(w,\gamma)}{\partial \gamma_j}=& -
  \frac{\gamma_jw_i}{2|w|^4}
   - \frac{3\pi  w_j }{|w|^2}\frac{\partial \mathcal{R}_\Omega(0)}{\partial x_i},
 \end{split}
 \end{equation*}\end{small}
 \begin{small}
\begin{equation*}
\begin{split}
\frac{\partial \bar q_{\e,i}(w,\gamma)}{\partial w_j}= \frac{(w\cdot \gamma)\delta_{ij} }{|w|^4}+ \frac{\gamma_jw_i+\gamma_iw_j}{ |w|^4}-\frac{4(w\cdot \gamma)w_iw_j}{|w|^6},\,\,\,\,
\frac{\partial \bar q_{\e,i}(w,\gamma)}{\partial \gamma_j}=  \frac{ w_iw_j }{|w|^4} - \frac{1}{  2|w|^2}\delta_{ij}.
 \end{split}
 \end{equation*}\end{small}Now $\bar q_{\e,j}(w,\gamma)=0$ (see \eqref{sec7-46} for the definition of $\bar q_{\e,j}$) gives
$\frac{ (w\cdot \gamma)w_j }{|w|^4} - \frac{\gamma_j }{ 2 |w|^2}=-3\pi \frac{\partial \mathcal{R}_\Omega(0)}{\partial x_j}$, and then
  \begin{small}
\begin{equation*}
\begin{split}- \frac{3\pi\gamma_j }{|w|^2}  \frac{\partial \mathcal{R}_\Omega(0)}{\partial x_i}
   +\frac{6\pi (w\cdot\gamma)w_j}{|w|^4}\frac{\partial \mathcal{R}_\Omega(0)}{\partial x_i}=-18\pi^2 \frac{\partial \mathcal{R}_\Omega(0)}{\partial x_i}  \frac{\partial \mathcal{R}_\Omega(0)}{\partial x_j}.
   \end{split}
 \end{equation*}\end{small}Inserting this into \eqref{sec7-59}, we have
  \begin{small}
\begin{equation*}
\begin{split}
&\Big( \frac{\partial \bar p_{\e,i}(w_\e^{(m),\pm},\gamma^{(m),\pm}_\e)}{\partial w_j}\Big)_{1\leq i,j\leq 2}\\=&\left[
  \frac{2 k(|w_\e^{(m),\pm}| )}{ |w_\e^{(m),\pm}|^2 \e^{\frac{1}{2}}(\ln \e +2\pi \mathcal{R}_{\Omega}(0))} -
  \frac{ |\gamma^{(m),\pm}_\e|^2     }{
 |w_\e^{(m),\pm}|^4}\right] {\bf E_{2\times 2}}+\left[
  \frac{4}{ |w_\e^{(m),\pm}|^2 \e^{\frac{1}{2}} \ln \e }(1+o(1))\right] \Big( \frac{w_{\e,i}^{(m),\pm}w_{\e,j}^{(m),\pm}}{|w_\e^{(m),\pm}|^2} \Big)_{1\leq i,j\leq 2}\\[1mm]&
  -6\pi {\bf \widetilde{M}}  -36\pi^2\left(\frac{\partial \mathcal{R}_\Omega(0)}{\partial x_i}\frac{\partial \mathcal{R}_\Omega(0)}{\partial x_j}\right)_{1\leq i,j\leq 2}.
 \end{split}
 \end{equation*}\end{small}Also inserting \eqref{sec7-54} and \eqref{Sec7-55} into above computations,
 we obtain
  \begin{small}
\begin{equation*}
\begin{split}
\frac{\partial \bar p_{\e,i}(w_\e^{(m),\pm},\gamma^{(m),\pm}_\e)}{\partial \gamma_j}=&
   - \frac{3\pi   }{|w_\e^{(m),\pm}|^2}\Big( \frac{\partial \mathcal{R}_\Omega(0)}{\partial x_i }w_{\e,j}^{(m),\pm}+
   \frac{\partial \mathcal{R}_\Omega(0)}{\partial x_j}w_{\e,i}^{(m),\pm}\Big)+
   \frac{6\pi\big( \nabla\mathcal{R}_\Omega(0)\cdot w_{\e}^{(m),\pm}\big)}{|w_\e^{(m),\pm}|^4}w_{\e,i}^{(m),\pm}w_{\e,j}^{(m),\pm},
 \end{split}
 \end{equation*}\end{small}and
  \begin{small}
\begin{equation*}
\begin{split}
\frac{\partial \bar q_{\e,i}(w_\e^{(m),\pm},\gamma^{(m),\pm}_\e)}{\partial w_j}=&
   \frac{6\pi   }{|w_\e^{(m),\pm}|^2}\Big( \frac{\partial \mathcal{R}_\Omega(0)}{\partial x_i }w_{\e,j}^{(m),\pm}+
   \frac{\partial \mathcal{R}_\Omega(0)}{\partial x_j}w_{\e,i}^{(m),\pm}\Big)-
   \frac{6\pi\big( \nabla\mathcal{R}_\Omega(0)\cdot w_{\e}^{(m),\pm}\big)}{|w_\e^{(m),\pm}|^2}\delta_{ij}.
 \end{split}
 \end{equation*}\end{small}Now we consider the Jacobian matrix of the vector $\check{\bf{V}}_\e$ at $(w_\e^{(1),+},\gamma_\e^{(1),+})$. Let
 \begin{small}
\begin{align*}
{\bf Q}_{\e,2}:=\left(\frac{w_\e^{(1),+}}{|w_\e^{(1),+}|},\frac{w_\e^{(2),+}}{|w_\e^{(2),+}|}\right)=
\left(
  \begin{array}{cc}
    \frac{w_{\e,1}^{(1),+}}{|w_\e^{(1),+}|} & \frac{w_{\e,1}^{(2),+}}{|w_\e^{(2),+}|} \\[4mm]
    \frac{w_{\e,2}^{(1),+}}{|w_\e^{(1),+}|} & \frac{w_{\e,2}^{(2),+}}{|w_\e^{(2),+}|} \\
  \end{array}
\right).
\end{align*}
\end{small}Then  ${\bf Q}_{\e,2}$ is an orthogonal matrix and satisfies  ${\bf Q}^T_{\e,2}\frac{w_\e^{(1),+}}{|w_\e^{(1),+}|}=\left(
                                                                       \begin{array}{c}
                                                                         1 \\
                                                                         0 \\
                                                                       \end{array}
                                                                     \right)
$ and
\begin{small}
\begin{align*}
 &{\bf Q}_{\e,2}^{T}\left( \frac{w_{\e,i}^{(1),+}w_{\e,j}^{(1),+} }{|w_\e^{(1),+}|^2} \right)_{1\leq i,j\leq 2} {\bf Q}_{\e,2}=
 \left(
   \begin{array}{c}
     1 \\
     0 \\
   \end{array}
 \right)\left(
          \begin{array}{cc}
            1 & 0\\
          \end{array}
        \right)
=\left(
   \begin{array}{cc}
     1 & 0 \\[1mm]
     0 & 0 \\
   \end{array}
 \right),
\end{align*}
\end{small}
\begin{small}
\begin{align*}
{\bf Q}^{T}_{\e,2}{\bf \widetilde{M}} {\bf Q}_{\e,2}=diag\Big(\lambda_1,\lambda_2\Big),\,\,\,\,\,~~~~~
 \left[
  \frac{2 k(|w_\e^{(1),+}| )}{ |w_\e^{(1),+}|^2 \e^{\frac{1}{2}}(\ln \e +2\pi \mathcal{R}_{\Omega}(0))} -
  \frac{ |\gamma_\e^{(1),+}|^2 }{4
 |w_\e^{(1),+}|^4}\right]  =6\pi\lambda_1,
 \end{align*}\end{small}from which we get
  \begin{small}
\begin{equation*}
\begin{split}
&{\bf Q}^{T}_{\e,2} \left( \frac{\partial \bar p_{\e,i}(w_\e^{(1),+},\gamma_\e^{(1),+})}{\partial w_j}\right)_{1\leq i,j\leq 2}{\bf Q}_{\e,2}\\ = &\left(
   \begin{array}{cc}
  \frac{4e^{{3\pi \mathcal{R}_{\Omega}(0)}} }{  \e^{\frac{1}{2}} \ln \e }\big(1+o(1)\big)  & 0 \\[3mm]
     0 & 6\pi(\lambda_1-\lambda_2)   \\
   \end{array}
 \right)- 36\pi^2 {\bf Q}_{\e,2}^{T}
 \left( {\frac{\partial \mathcal{R}_\Omega(0)}{\partial x_i}  \frac{\partial \mathcal{R}_\Omega(0)}{\partial x_j} }  \right)_{1\leq i,j\leq 2} {\bf Q}_{\e,2}.
 \end{split}
 \end{equation*}\end{small}Also, we have
   \begin{small}
\begin{equation*}
\begin{split}
& {\bf Q}_{\e,2}^{T}
 \left(  {\frac{\partial \mathcal{R}_\Omega(0)}{\partial x_i}  \frac{\partial \mathcal{R}_\Omega(0)}{\partial x_j} } \right)_{1\leq i,j\leq 2} {\bf Q}_{\e,2} \\=&
 \left(\nabla \mathcal{R}_\Omega(0)\cdot \frac{w_\e^{(1),+}}{|w_\e^{(1),+}|},\nabla \mathcal{R}_\Omega(0)\cdot \frac{w_\e^{(2),+}}{|w_\e^{(2),+}|} \right)^T \left(\nabla \mathcal{R}_\Omega(0)\cdot \frac{w_\e^{(1),+}}{|w_\e^{(1),+}|},\nabla \mathcal{R}_\Omega(0)\cdot \frac{w_\e^{(2),+}}{|w_\e^{(2),+}|} \right)\\[2mm]=&
 \left(
  \begin{array}{cc}
  (\nabla \mathcal{R}_\Omega(0)\cdot \frac{w_\e^{(1),+}}{|w_\e^{(1),+}|})^2  & (\nabla \mathcal{R}_\Omega(0)\cdot \frac{w_\e^{(1),+}}{|w_\e^{(1),+}|})(\nabla \mathcal{R}_\Omega(0)\cdot \frac{w_\e^{(2),+}}{|w_\e^{(2),+}|}) \\[3mm]
    (\nabla \mathcal{R}_\Omega(0)\cdot \frac{w_\e^{(1),+}}{|w_\e^{(1),+}|})(\nabla \mathcal{R}_\Omega(0)\cdot \frac{w_\e^{(2),+}}{|w_\e^{(2),+}|})&  (\nabla \mathcal{R}_\Omega(0)\cdot \frac{w_\e^{(2),+}}{|w_\e^{(2),+}|})^2  \\
   \end{array}
 \right).
 \end{split}
 \end{equation*}\end{small}Hence it holds
   \begin{small}
\begin{equation*}
\begin{split}
&{\bf Q}^{T}_{\e,2} \left( \frac{\partial \bar p_{\e,i}(w_\e^{(1),+},\gamma_\e^{(1),+})}{\partial w_j}\right)_{1\leq i,j\leq 2}{\bf Q}_{\e,2}=\left(
  \begin{array}{cc}
  \frac{4e^{{3\pi \mathcal{R}_{\Omega}(0)}} }{  \e^{\frac{1}{2}} \ln \e }\big(1+o(1)\big)  & O(1) \\[3mm]
     O(1) & 6\pi(\lambda_1-\lambda_2)  -36 \pi^2(\nabla \mathcal{R}_\Omega(0)\cdot \frac{w_\e^{(2),+}}{|w_\e^{(2),+}|})^2   \\
   \end{array}
 \right).
 \end{split}
 \end{equation*}\end{small}On the other hand, we have
  \begin{small}
\begin{align*}
&{\bf Q}_{\e,2}^{T}  \left( \frac{\partial \mathcal{R}_\Omega(0)}{\partial x_i }w_{\e,j}^{(1),\pm}+
   \frac{\partial \mathcal{R}_\Omega(0)}{\partial x_j}w_{\e,i}^{(1),\pm} \right)_{1\leq i,j\leq 2}{\bf Q}_{\e,2} \\[2mm]=&
   {\bf Q}_{\e,2}^{T}  \left[ \left(
                         \begin{array}{c}
                           \frac{\partial \mathcal{R}_\Omega(0)}{\partial x_1} \\[2mm]
                          \frac{ \partial \mathcal{R}_\Omega(0)}{\partial x_2}\\
                         \end{array}
                       \right) \left(
                                \begin{array}{cc}
                                   w_{\e,1}^{(1),\pm} & w_{\e,2}^{(1),\pm}  \\
                                \end{array}
                              \right) +
                              \left(
                         \begin{array}{c}
                            w_{\e,1}^{(1),\pm} \\[2mm] w_{\e,2}^{(1),\pm} \\
                         \end{array}
                       \right) \left(
                                \begin{array}{cc}
                                  \frac{\partial \mathcal{R}_\Omega(0)}{\partial x_1}  &
                          \frac{ \partial \mathcal{R}_\Omega(0)}{\partial x_2}\\
                                \end{array}
                              \right) \right]  {\bf Q}_{\e,2} \\[2mm]=&
                              \left(
                         \begin{array}{c}
                            \nabla \mathcal{R}_\Omega(0)\cdot \frac{w_\e^{(1),+}}{|w_\e^{(1),+}|}  \\[2mm]  \nabla \mathcal{R}_\Omega(0)\cdot \frac{w_\e^{(2),+}}{|w_\e^{(2),+}|}
                           \\
                         \end{array}
                       \right) \left(
                                \begin{array}{cc}
                                 |w_\e^{(1),+}| & 0  \\
                                \end{array}
                              \right) +
                              \left(
                         \begin{array}{c}
                           |w_\e^{(1),+}| \\[2mm] 0 \\
                         \end{array}
                       \right) \left(
                                \begin{array}{cc}
                                 \nabla \mathcal{R}_\Omega(0)\cdot \frac{w_\e^{(1),+}}{|w_\e^{(1),+}|}   &   \nabla \mathcal{R}_\Omega(0)\cdot \frac{w_\e^{(2),+}}{|w_\e^{(2),+}|} \\
                                \end{array}
                              \right)
                              \\[2mm]=&
    \left(
   \begin{array}{cc}
  2 \nabla \mathcal{R}_\Omega(0)\cdot  w_\e^{(1),+}    & |w_\e^{(1),+}|  (\nabla \mathcal{R}_\Omega(0)\cdot \frac{w_\e^{(2),+}}{|w_\e^{(2),+}|})  \\[4mm]
 |w_\e^{(1),+}|(\nabla \mathcal{R}_\Omega(0)\cdot \frac{w_\e^{(2),+}}{|w_\e^{(2),+}|} )  & 0 \\
   \end{array}
 \right).
 \end{align*}\end{small}This gives
   \begin{small}
\begin{equation}\label{sec7-60}
\begin{split}
&{\bf Q}^{T}_{\e,2} \left( \frac{\partial \bar p_{\e,i}(w_\e^{(1),+},\gamma_\e^{(1),+})}{\partial \gamma_j}\right)_{1\leq i,j\leq 2}{\bf Q}_{\e,2}=
    \left(
   \begin{array}{cc}0  & - \frac{3\pi (\nabla \mathcal{R}_\Omega(0)\cdot  w_\e^{(2),+} )}{|w_\e^{(1),+}|\cdot|w_\e^{(2),+}| }    \\[4mm]
- \frac{3\pi (\nabla \mathcal{R}_\Omega(0)\cdot  w_\e^{(2),+} )}{|w_\e^{(1),+}|\cdot|w_\e^{(2),+}| } & 0 \\
   \end{array}
 \right).
 \end{split}
 \end{equation}\end{small}Moreover, we have
 \begin{small}
\begin{equation}\label{sec7-61}
\begin{split}
{\bf Q}_{\e,2}^{T} \left( \frac{\partial \bar q_{\e,i}(w_\e^{(1),+},\gamma_\e^{(1),+})}{\partial w_j}\right)_{1\leq i,j\leq 2}{\bf Q}_{\e,2} =  \left(
   \begin{array}{cc}\frac{6\pi ( \nabla\mathcal{R}_\Omega(0)\cdot w_{\e}^{(1),+} )}{|w_\e^{(1),+}|^2}  & \frac{6\pi (\nabla \mathcal{R}_\Omega(0)\cdot  w_\e^{(2),+} )}{|w_\e^{(1),+}|\cdot|w_\e^{(2),+}| }    \\[4mm]
  \frac{6\pi (\nabla \mathcal{R}_\Omega(0)\cdot  w_\e^{(2),+} )}{|w_\e^{(1),+}|\cdot|w_\e^{(2),+}| } & -\frac{6\pi ( \nabla\mathcal{R}_\Omega(0)\cdot w_{\e}^{(1),+} )}{|w_\e^{(1),+}|^2} \\
   \end{array}
 \right),
 \end{split}
 \end{equation}\end{small}
 \begin{small}
\begin{equation}\label{sec7-62}
\begin{split}
{\bf Q}_{\e,2}^{T}  \left( \frac{\partial \bar q_{\e,i}(w_\e^{(1),+},\gamma_\e^{(1),+})}{\partial \gamma_j}\right)_{1\leq i,j\leq 2}{\bf Q}_{\e,2} = \left(
   \begin{array}{cc}
  \frac{1}{ 2|w_\e^{(1),+}|^2}  & 0 \\[2mm]
     0 & -\frac{1}{ 2|w_\e^{(1),+}|^2} \\
   \end{array}
 \right).
 \end{split}
 \end{equation}\end{small}Thus it holds
  \begin{small}
\begin{align*}
& \left(
  \begin{array}{cc}
   \bf Q_{\e,2}^{T} & \bf O_{2\times 2} \\[4mm]
    \bf O_{2\times 2} & \bf Q^{T}_{\e,2} \\
  \end{array}
\right)
 \left(
  \begin{array}{cc}
   \Big( \frac{\partial \bar p_{\e,j}(w_\e^{(1),+},\gamma_\e^{(1),+})}{\partial w_i}\Big)_{1\leq i,j\leq 2} &
   \Big( \frac{\partial \bar p_{\e,j}(w_\e^{(1),+},\gamma_\e^{(1),+})}{\partial \gamma_i}\Big)_{1\leq i,j\leq 2} \\[4mm]
   \Big( \frac{\partial \bar q_{\e,j}(w_\e^{(1),+},\gamma_\e^{(1),+})}{\partial w_i}\Big)_{1\leq i,j\leq 2}  &
   \Big( \frac{\partial \bar q_{\e,j}(w_\e^{(1),+},\gamma_\e^{(1),+})}{\partial \gamma_i}\Big)_{1\leq i,j\leq 2}  \\
  \end{array}
\right) \left(
  \begin{array}{cc}
   \bf Q_{\e,2} & \bf O_{2\times 2} \\[4mm]
    \bf O_{2\times 2} & \bf Q_{\e,2} \\
  \end{array}
\right)\\[3mm]&
=\left(
          \begin{array}{cccc}
            \frac{4e^{{3\pi \mathcal{R}_{\Omega}(0)}}}{ \e^{\frac{1}{2}} \ln \e }\big(1+o(1)\big)  & O(1)  & 0 & O(1) \\[3mm]
            O(1) & 6\pi(\lambda_1-\lambda_2) -36 \pi^2(\nabla \mathcal{R}_\Omega(0)\cdot \frac{w_\e^{(2),+}}{|w_\e^{(2),+}|})^2   &      - \frac{3\pi (\nabla \mathcal{R}_\Omega(0)\cdot w_\e^{(2),+})  }{|w_\e^{(1),\pm}|\cdot|w_\e^{(2),+}|}   & 0 \\[3mm]
           O(1) &  \frac{6\pi  (\nabla \mathcal{R}_\Omega(0)\cdot w_\e^{(2),+}) }{|w_\e^{(1),\pm}|\cdot|w_\e^{(2),+}|}  & \frac{1}{ 2|w_\e^{(1),+}|^2} & 0 \\[3mm]
            O(1) & O(1) & 0 & -\frac{1}{ 2|w_\e^{(1),+}|^2} \\[3mm]
          \end{array}
        \right).
 \end{align*}\end{small}And then
  \begin{small}
\begin{align*}
& det~\left(
  \begin{array}{cc}
   \Big( \frac{\partial \bar p_{\e,j}(w_\e^{(1),+},\gamma_\e^{(1),+})}{\partial w_i}\Big)_{1\leq i,j\leq 2} &
   \Big( \frac{\partial \bar p_{\e,j}(w_\e^{(1),+},\gamma_\e^{(1),+})}{\partial \gamma_i}\Big)_{1\leq i,j\leq 2} \\[4mm]
   \Big( \frac{\partial \bar q_{\e,j}(w_\e^{(1),+},\gamma_\e^{(1),+})}{\partial w_i}\Big)_{1\leq i,j\leq 2}  &
   \Big( \frac{\partial \bar q_{\e,j}(w_\e^{(1),+},\gamma_\e^{(1),+})}{\partial \gamma_i}\Big)_{1\leq i,j\leq 2}  \\
  \end{array}
\right) \\[2mm]&
=det \left(
          \begin{array}{cccc}
            \frac{4e^{{3\pi \mathcal{R}_{\Omega}(0)}}}{ \e^{\frac{1}{2}} \ln \e }\big(1+o(1)\big)  & O(1)  & 0 & O(1) \\[3mm]
            O(1) & 6\pi(\lambda_1-\lambda_2) -36 \pi^2(\nabla \mathcal{R}_\Omega(0)\cdot \frac{w_\e^{(2),+}}{|w_\e^{(2),+}|})^2   &      - \frac{3\pi (\nabla \mathcal{R}_\Omega(0)\cdot w_\e^{(2),+})  }{|w_\e^{(1),\pm}|\cdot|w_\e^{(2),+}|}   & 0 \\[3mm]
           O(1) &  \frac{6\pi  (\nabla \mathcal{R}_\Omega(0)\cdot w_\e^{(2),+}) }{|w_\e^{(1),\pm}|\cdot|w_\e^{(2),+}|}  & \frac{1}{ 2|w_\e^{(1),+}|^2} & 0 \\[3mm]
            O(1) & O(1) & 0 & -\frac{1}{ 2|w_\e^{(1),+}|^2} \\[3mm]
          \end{array}
        \right),
 \end{align*}\end{small}which gives \eqref{sec7-58} for $m=1$. Similarly it is possible to prove \eqref{sec7-58} for $m=2$.
 \end{proof}

Let $\hat {\bf{V}}_\e(w,\gamma)=\big(p_{\e,1}(w,\gamma),
p_{\e,2}(w,\gamma),q_{\e,1}(w,\gamma),q_{\e,2}(w,\gamma)\big)$.
We have the following result.

\begin{prop}\label{sec7-prop7.17}
For each solution $(w_\e^{(m),\pm},\gamma_\e^{(m),\pm})$ of $\check{\mathbf{V}}_\e(w,\gamma)=0$ with $m=1,2$, it holds
\begin{small}
\begin{equation}\label{sec7-63}
\deg\Big(\hat{\mathbf{V}}_\e, 0, B\big((w_\e^{(m),\pm},\gamma_\e^{(m),\pm}),\delta\big) \Big) =
\deg\Big(\check{\mathbf{V}}_\e, 0, B\big((w_\e^{(m),\pm},\gamma_\e^{(m),\pm}),\delta\big)\Big) \ne 0,
\end{equation}
\end{small}and problem  $\hat{\mathbf{V}}_\e(w,\gamma)=0$ has at least
one solution in $B\big((w_\e^{(m),\pm},\gamma_\e^{(m),\pm}),\delta\big)$ for a small $\delta>0$.
\end{prop}

\begin{proof}
First we have
\begin{small}
\begin{equation}\label{sec7-64}
p_{\e,j}(w,\gamma)
=\bar p_{\e,j}(w,\gamma)
+ O\left(\frac{1}{|\ln \e|}
+ \frac{|k(w)|}{\e^{\frac{1}{4}}|\ln \e|}\right)~~~
\mbox{and}~~~q_{\e,j}(w,\gamma)
=\bar q_{\e,j}(w,\gamma)
+ O\left(\frac{1}{|\ln \e|} \right).
\end{equation}
\end{small}Now for any $(w,\gamma)
\in \partial B\big((w_\e^{(1),+},\gamma_\e^{(1),+}),\delta\big)$ for example,
 \eqref{sec7-57} and the first identity of \eqref{sec7-64} gives
\begin{small}
\begin{equation}\label{sec7-65}
\frac{k(|w|)}{\e^{\frac{1}{2}}\ln \e}|w-w_\e^{(1),+}|
=  O\left(1
+ \frac{|k(w)|}{\e^{\frac{1}{4}}|\ln \e|}\right).
\end{equation}
\end{small}We claim that
\begin{small}
\begin{equation}\label{sec7-66}
\frac{k(w)}{\e^{\frac{1}{2}}|\ln \e|} = O(1).
\end{equation}
\end{small}Otherwise, $\frac{k(w)}{\e^{\frac{1}{2}}|\ln \e|} \to \infty$ and
then \eqref{sec7-65} implies $|w-w_\e^{(1),+}|\to 0$.

\vskip 0.1cm

On the other hand, by Taylor's expansion, \eqref{sec7-57} and the second identity of \eqref{sec7-64}, we know
\begin{small}\[|\gamma-\gamma_\e^{(1),+}|=O\Big(|w-w_\e^{(1),+}|\Big)+ O\left(\frac{1}{|\ln \e|} \right)=o(1).
\]
\end{small}This is a contradiction with $(w,\gamma)
\in \partial B\big((w_\e^{(1),+},\gamma_\e^{(1),+}),\delta\big)$. Hence \eqref{sec7-66} holds.

\vskip 0.1cm

Now for any $t \in [0,1]$, it holds
\begin{small}
\[
t \check{\mathbf{V}}_\e(w,\gamma) + (1-t)\hat{\mathbf{V}}_\e(w,\gamma)
= \check{\mathbf{V}}_\e(w,\gamma)
+ O\left(\frac{1}{|\ln \e|}
+ \frac{|k(w)|}{\e^{\frac{1}{4}}|\ln \e|}\right)
\ne 0, \quad \forall (w,\gamma)
\in \partial B\big((w_\e^{(m),\pm},\gamma_\e^{(m),\pm}),\delta\big).
\]
\end{small}Then  as in the proof of Theorem \ref{sec1-teo15}, we obtain a contradiction.
Therefore, \eqref{sec7-63} follows, which implies that the problem $\hat{\mathbf{V}}_\e(w,\gamma)=0$ admits at least
one solution in $B\big((w_\e^{(m),\pm},\gamma_\e^{(m),\pm}),\delta\big)$ for some small $\delta>0$.
\end{proof}
Next result is the analogous of Lemma \ref{sec7-lem7.9}.

\begin{lem}\label{sec7-lem7.18}
 If $(\widetilde{w}_\e,\widetilde{\gamma}_\e)$ is a solution of $\hat{\bf{V}}_\e(w,\gamma)=0$, then there exists $m\in \{1,2\}$ such that
\begin{small}
\begin{align*}
(\widetilde{w}_\e,\widetilde{\gamma}_\e)=\Big(w_\e^{(m),+}+o\big(1\big),\gamma_\e^{(m),+}
+o\big(1\big)\Big)~~~~\mbox{and}~~~~|\widetilde{w}_\e|-|w_\e^{(m),-}|=o\big(\e^{\frac{1}{2}}{|\ln \e |}\big)
 \end{align*}\end{small}or\begin{small}
\begin{align*}(\widetilde{w}_\e,\widetilde{\gamma}_\e)= \Big(w_\e^{(m),-}+o\big(1\big),\gamma_\e^{(m),-}+o\big(1\big)\Big)~~~~\mbox{and}~~~~
|\widetilde{w}_\e|-|w_\e^{(m),-}|=o\big(\e^{\frac{1}{2}}{|\ln \e |}\big),
 \end{align*}\end{small}where
$(w_\e^{(m),\pm},\gamma_\e^{(m),\pm})$ are as in \eqref{sec7-53}.
\end{lem}
\begin{proof}
 Let
 $(\widetilde{w}_\e,\widetilde{\gamma}_\e)$ be a solution of $\hat{\bf{V}}_\e(w,\gamma)=0$.  Then
  \begin{small}
\begin{equation}\label{sec7-67}
\begin{split}
  \left[
  \frac{2 k(|\widetilde{w}_\e| )}{ |\widetilde{w}_\e|^2 \e^{\frac{1}{2}}(\ln \e +2\pi \mathcal{R}_{\Omega}(0))} -
  \frac{ |\gamma|^2     }{
 |\widetilde{w}_\e|^4}\right] \widetilde{w}_{\e,j}  = 3\pi  (\widetilde{w}_\e\cdot \widetilde{\gamma}_\e) \frac{\partial \mathcal{R}_\Omega(0)}{\partial x_j}
  +6\pi  \sum^2_{i=1} \frac{\partial^2H_{\Omega}(0,0)}{\partial x_i\partial x_j}\widetilde{w}_{\e,i}+ O\left(\frac{1}{|\ln \e |}\right)
 \end{split}
 \end{equation}\end{small}and
  \begin{small}
\begin{equation}\label{sec7-68}
\begin{split}
 &\frac{ (\widetilde{w}_\e\cdot \widetilde{\gamma}_\e)\widetilde{w}_{\e,j} }{|\widetilde{w}_\e|^4} - \frac{\widetilde{\gamma}_{\e,j}}{2 |\widetilde{w}_\e|^2} +3 \frac{\partial \mathcal{R}_\Omega(0)}{\partial x_j} =O\left(\frac{1}{|\ln \e |}\right).
 \end{split}
 \end{equation}\end{small}From $\displaystyle\sum^2_{j=1} \widetilde{w}_{\e,j}  \times \eqref{sec7-68}$, we have
 \begin{small}\begin{equation}\label{sec7-69}
\begin{split}
\widetilde{w}_\e\cdot \widetilde{\gamma}_\e=- 6\pi |\widetilde{w}_\e|^2 \nabla \mathcal{R}_\Omega(0)\cdot  \widetilde{w}_\e+O\left(\frac{1}{|\ln \e |}\right),
 \end{split}\end{equation}\end{small}and then
 \begin{small}\begin{equation}\label{Sec7-70}
\begin{split}
\widetilde{\gamma}_\e=  6\pi |\widetilde{w}_\e|^2 \nabla \mathcal{R}_\O(0)- 12\pi
\big(\nabla\mathcal{R}_\O(0)\cdot \widetilde{w}_\e\big) \widetilde{w}_\e+O\left(\frac{1}{|\ln \e |}\right).
 \end{split}
\end{equation}
\end{small}Inserting \eqref{sec7-69} and \eqref{Sec7-70} into \eqref{sec7-67}, we get
\begin{small}
\begin{equation*}
\begin{split}
\left[
  \frac{2 k(|\widetilde{w}_\e| )}{ |\widetilde{w}_\e|^2 \e^{\frac{1}{2}}(\ln \e +2\pi \mathcal{R}_{\Omega}(0))}-
   9\pi^2  |\nabla \mathcal{R}_{\Omega}(0)|^2    \right] \widetilde{w}_{\e,j}
   = 6\pi  \Big(\widetilde{\bf{M}}\widetilde{w}_\e\Big)_j+O\left(\frac{1}{|\ln \e |}\right).
 \end{split}
 \end{equation*}\end{small}Let $\frac{\widetilde{w}_{\e}}{|\widetilde{w}_{\e}|}\to \eta$. Then there exists $m\in \{1,2\}$ such that $\eta=\nu^{(m)}$ or $\eta=-\nu^{(m)}$. Thus,
\begin{small} \begin{equation*}
  \lim_{\e\to 0}\left[\frac{2 k(|\widetilde{w}_\e| )}{ |\widetilde{w}_\e|^2 \e^{\frac{1}{2}}(\ln \e +2\pi \mathcal{R}_{\Omega}(0))}-
   9\pi^2  |\nabla \mathcal{R}_{\Omega}(0)|^2\right]=6\pi \lambda_m.
 \end{equation*}\end{small}Without loss of generality, we suppose
 $\eta=\nu^{(1)}=\frac{w_\e^{(1),+}}{|w_\e^{(1),+}|}$,  then it holds
 $\frac{\widetilde{w}_{\e}}{|\widetilde{w}_{\e}|}-\frac{w_\e^{(1),+}}{|w_\e^{(1),+}|}\to 0$
  and
  \begin{small} \begin{equation}\label{sec7-71}
 \frac{2 k(|\widetilde{w}_\e| )}{ |\widetilde{w}_\e|^2 \e^{\frac{1}{2}}(\ln \e +2\pi \mathcal{R}_{\Omega}(0))}=
   9\pi^2  |\nabla \mathcal{R}_{\Omega}(0)|^2+6\pi \lambda_1+o(1).
 \end{equation}\end{small}Also we recall \begin{small}
 \begin{equation}\label{sec7-72}
\begin{split}
  \frac{2 k(|w_\e^{(1),+}| )}{ |w_\e^{(1),+}|^2 \e^{\frac{1}{2}}(\ln \e +2\pi \mathcal{R}_{\Omega}(0))}- 9\pi^2  |\nabla \mathcal{R}_{\Omega}(0)|^2
   = 6\pi  \lambda_1.
 \end{split}
 \end{equation}\end{small}Hence from
 \eqref{sec7-71} and \eqref{sec7-72}, we get
\begin{small}
 \begin{equation*}
\begin{split}
   \frac{ k(|\widetilde{w}_\e| )}{ |\widetilde{w}_\e|^2 \e^{\frac{1}{2}}(\ln \e +2\pi \mathcal{R}_{\Omega}(0))}-\frac{ k(|w_\e^{(1),+}| )}{ |w_\e^{(1),+}|^2 \e^{\frac{1}{2}}(\ln \e +2\pi \mathcal{R}_{\Omega}(0))}=o(1).
 \end{split}
 \end{equation*}\end{small}This gives $|\widetilde{w}_\e|-|w_\e^{(1),+}|=o\big(\e^{\frac{1}{2}}{|\ln \e |}\big)$ and then $|\widetilde{w}_{\e} - w_\e^{(1),+}|=o(1)$ by
 $\frac{\widetilde{w}_{\e}}{|\widetilde{w}_{\e}|}-\frac{w_\e^{(1),+}}{|w_\e^{(1),+}|}\to 0$.
\end{proof}

 We now consider the non-degeneracy of the solutions of $\hat{\bf{V}}_\e(w,\gamma)
 =0$.
\begin{prop}\label{sec7-prop7.19}
If $\widetilde{\bf{M}}$ has two different eigenvalues $\lambda_1,\lambda_2$ and $(\widetilde{w}_\e,\widetilde{\gamma}_\e)$ is a solution of $\hat{\bf{V}}_\e(w,\gamma)=0$, then it holds
\begin{small}
\begin{align}\label{sec7-73}
det~Jac~ \hat{\bf{V}}_\e(\widetilde{w}_\e,\widetilde{\gamma}_\e)=\frac{6\pi e^{9\pi \mathcal{R}_{\Omega}(0)} (\lambda_j-\lambda_m)}{   \e^{\frac{1}{2}} \ln \e }\big(1+o(1)\big)\neq 0~~~\mbox{with}~~~j=1,2~~~\mbox{and}~~~j\neq m.
 \end{align}
\end{small}
\end{prop}

\begin{proof}
By Lemma \ref{sec7-lem7.18}, we just consider the case
\begin{small}
\begin{align}\label{sec7-74}
(\widetilde{w}_\e,\widetilde{\gamma}_\e)=\Big(w_\e^{(1),+}+o\big(1\big),\gamma_\e^{(1),+}
+o\big(1\big)\Big)~~~~\mbox{and}~~~~|\widetilde{w}_\e|-|w_\e^{(1),+}|=o\big(\e^{\frac{1}{2}}{|\ln \e |}\big).
 \end{align}\end{small}The computations for the other one are similar.  First, by \eqref{sec7-84a}, we have
 \begin{small}
\begin{equation*}
\frac{\partial p_{\e,i}(\widetilde{w}_\e,\widetilde{\gamma}_\e)}{\partial w_j}=
\frac{\partial\bar p_{\e,i}(\widetilde{w}_\e,\widetilde{\gamma}_\e)}{\partial w_j} +  \frac{\partial \widetilde{p}_{\e,i}(\widetilde{w}_\e,\widetilde{\gamma}_\e)}{\partial w_j}
  +O\left(\frac{1}{|\ln \e |}\right).
 \end{equation*}\end{small}By \eqref{sec7-59} and \eqref{sec7-74}, we have
 \begin{small}
\begin{equation*}
\begin{split}
&\Big( \frac{\partial  \bar p_{\e,i}(\widetilde{w}_\e,\widetilde{\gamma}_\e)}{\partial w_j}\Big)_{1\leq i,j\leq 2}\\=&\left[
  \frac{2 k(|w_\e^{(1),+}| )}{ |w_\e^{(1),+}|^2 \e^{\frac{1}{2}}(\ln \e +2\pi \mathcal{R}_{\Omega}(0))} -
  \frac{ |\gamma^{(1),+}_\e|^2     }{
 |w_\e^{(1),+}|^4}\right] {\bf E_{2\times 2}}+\left[
  \frac{4}{ |w_\e^{(1),+}|^2 \e^{\frac{1}{2}} \ln \e }(1+o(1))\right] \Big( \frac{\widetilde{w}_{\e,i}\widetilde{w}_{\e,j}}{|\widetilde{w}_{\e}|^2} \Big)_{1\leq i,j\leq 2}\\[1mm]&
  -6\pi {\bf \widetilde{M}}  -36\pi^2\left(\frac{\partial \mathcal{R}_\Omega(0)}{\partial x_i}\frac{\partial \mathcal{R}_\Omega(0)}{\partial x_j}\right)_{1\leq i,j\leq 2}+\left(
                  \begin{array}{cc}
                    o(1) & o(1) \\[2mm]
                    o(1) & o(1) \\
                  \end{array}
                \right)\\[2mm]=&
  6\pi \lambda_1\left(
                  \begin{array}{cc}
                    1 +o(1) & o(1) \\[2mm]
                    o(1) & 1+o(1) \\
                  \end{array}
                \right)+\left[
  \frac{4}{ |w_\e^{(1),+}|^2 \e^{\frac{1}{2}} \ln \e }(1+o(1))\right] \Big( \frac{\widetilde{w}_{\e,i}\widetilde{w}_{\e,j}}{|\widetilde{w}_{\e}|^2} \Big)_{1\leq i,j\leq 2}\\[1mm]&
  -6\pi {\bf \widetilde{M}}  -36\pi^2\left(\frac{\partial \mathcal{R}_\Omega(0)}{\partial x_i}\frac{\partial \mathcal{R}_\Omega(0)}{\partial x_j}\right)_{1\leq i,j\leq 2}+\left(
                  \begin{array}{cc}
                    o(1) & o(1) \\[2mm]
                    o(1) & o(1) \\
                  \end{array}
                \right).
 \end{split}
 \end{equation*}\end{small}Next compute
 \begin{small}
\begin{equation*}
\begin{split}
  \frac{\partial \widetilde{p} _{\e,i}(\widetilde{w}_\e,\widetilde{\gamma}_\e)}{\partial w_j}
  =& \frac{2 k(|\widetilde{w}_\e| )}{\ln \e +2\pi \mathcal{R}_{\Omega}(0)}\left[
  \frac{2 k(|\widetilde{w}_\e| )}{ |\widetilde{w}_\e|^2 \e^{\frac{1}{2}}(\ln \e +2\pi \mathcal{R}_{\Omega}(0))}  -
  \frac{ (\widetilde{w}_\e\cdot\widetilde{\gamma}_\e)}{ |\widetilde{w}_\e|^4 \e^{\frac{1}{4}} }\big(1-\frac{2}{\ln \e +2\pi \mathcal{R}_{\Omega}(0)}\big) \right]\delta_{ij}\\&
  -\left[\frac{2 (\widetilde{w}_\e\cdot\widetilde{\gamma}_\e ) k'(|\widetilde{w}_\e| )}{|\widetilde{w}_\e|^3 \e^{\frac{1}{4}}(\ln \e +2\pi \mathcal{R}_{\Omega}(0))} (1+o(1))\right] \Big( \frac{\widetilde{w}_{\e,i}\widetilde w_{\e,j}}{|\widetilde{w}_\e|^2} \Big)_{1\leq i,j\leq 2}
  \\&+  \frac{ \pi(4\ln|\widetilde{w}_\e|-3\ln \e)k'(|\widetilde{w}_\e|) }{ \e^{\frac{1}{4}} ( \ln \e +2\pi \mathcal{R}_{\Omega}(0))}\frac{\partial \mathcal{R}_\Omega(0)}{\partial x_i} \frac{\widetilde{w}_{\e,j}}{|\widetilde{w}_\e|}.
 \end{split}
 \end{equation*}\end{small}We take $\bf\widetilde{Q}_{\e,2}$ being an orthogonal matrix such that
  \begin{small}
\begin{equation*}
{\bf \widetilde{Q}}_{\e,2}^T\frac{\widetilde{w}_\e}{|\widetilde{w}_\e|}= \left(
                                                                        \begin{array}{c}
                                                                          1 \\
                                                                          0 \\
                                                                        \end{array}
                                                                      \right)~~~~\mbox{and}~~~~
                          \Big({\bf \widetilde{Q}}_{\e,2}-{\bf{Q}}_{\e,2}\Big)_{ij}=o\big(1\big),
 \end{equation*}\end{small}where $\bf Q_{\e,2}$ is the orthogonal matrix in the proof of Proposition \ref{sec7-prop7.16}. Then it holds
 \begin{small}
\begin{equation*}
\begin{split}
   \bf \widetilde Q_{\e,2}^{T}
 \left(
  \begin{array}{cc}
   \Big( \frac{\partial  p_{\e,i}(\widetilde{w}_\e,\widetilde{\gamma}_\e)}{\partial w_j}\Big)_{1\leq i,j\leq 2} \end{array}\right)\bf \widetilde Q_{\e,2}
   =\left(
      \begin{array}{cc}
        \frac{4e^{{3\pi \mathcal{R}_{\Omega}(0)}}}{ \e^{\frac{1}{2}} \ln \e }\big(1+o(1)\big)   & O(1) \\[2mm]
        O\left(\frac{1}{\e^{\frac{1}{4}}}\right) & 6\pi(\lambda_1-\lambda_2) -36 \pi^2(\nabla \mathcal{R}_\Omega(0)\cdot w_\e^{(2),+})^2+o(1) \\
      \end{array}
    \right).
 \end{split}\end{equation*}\end{small}Also we recall
  \begin{small}
\begin{equation*}
\frac{\partial p_{\e,i}(\widetilde{w}_\e,\widetilde{\gamma}_\e)}{\partial \gamma_j}=
\frac{\partial\bar p_{\e,i}(\widetilde{w}_\e,\widetilde{\gamma}_\e)}{\partial \gamma_j} +  \frac{\partial \widetilde{p}_{\e,i}(\widetilde{w}_\e,\widetilde{\gamma}_\e)}{\partial \gamma_j}
  +O\left(\frac{1}{|\ln \e |}\right).
 \end{equation*}\end{small}By direct computations, \eqref{sec7-74} and the fact that $k(|w_\e^{(1),+}|)=O\big(\e^{\frac{1}{2}}|\ln \e|\big)$, we have
\begin{small}
\begin{equation*}
\begin{split}
  \frac{\partial \widetilde{p} _{\e,i}(\widetilde{w}_\e,\widetilde{\gamma}_\e)}{\partial \gamma_j}:
  =&   -\Big(
  \frac{  k(|\widetilde{w}_\e| )}{ |\widetilde{w}_\e|^2 \e^{\frac{1}{4}}(\ln \e +2\pi \mathcal{R}_{\Omega}(0)) }\big(1-\frac{2}{\ln \e +2\pi \mathcal{R}_{\Omega}(0)}\big) \Big)\frac{\widetilde{w}_{\e,i}\widetilde w_{\e,j}}{|\widetilde{w}_\e|^2}=O\Big(\e^{\frac{1}{4}}\Big).
 \end{split}
 \end{equation*}\end{small}Then using \eqref{sec7-74}, it holds
  \begin{small}
\begin{equation}\label{sec7-75}
\frac{\partial p_{\e,i}(\widetilde{w}_\e,\widetilde{\gamma}_\e)}{\partial \gamma_j}=
\frac{\partial\bar p_{\e,i}(\widetilde{w}_\e,\widetilde{\gamma}_\e)}{\partial \gamma_j}
  +O\left(\frac{1}{|\ln \e |}\right)=
\frac{\partial\bar p_{\e,i}(w_\e^{(1),+},\gamma_\e^{(1),+})}{\partial \gamma_j}
  +o\big(1\big).
 \end{equation}\end{small}Hence from \eqref{sec7-60} and \eqref{sec7-75}, we get  \begin{small}
\begin{equation*}
\begin{split}
   \bf \widetilde Q_{\e,2}^{T}
 \left(
  \begin{array}{cc}
   \Big( \frac{\partial  p_{\e,i}(\widetilde{w}_\e,\widetilde{\gamma}_\e)}{\partial \gamma_j}\Big)_{1\leq i,j\leq 2} \end{array}\right)\bf \widetilde Q_{\e,2}
   =\left(
   \begin{array}{cc}o(1)  & - \frac{3\pi (\nabla \mathcal{R}_\Omega(0)\cdot  w_\e^{(2),+} )}{|w_\e^{(1),+}|\cdot|w_\e^{(2),+}| }  +o(1)  \\[4mm]
- \frac{3\pi (\nabla \mathcal{R}_\Omega(0)\cdot  w_\e^{(2),+} )}{|w_\e^{(1),+}|\cdot|w_\e^{(2),+}| } +o(1) & o(1) \\
      \end{array}
    \right)
 \end{split}\end{equation*}\end{small}and
  \begin{small}
\begin{equation*}
\frac{\partial q_{\e,i}(\widetilde{w}_\e,\widetilde{\gamma}_\e)}{\partial w_j}=
\frac{\partial\bar q_{\e,i}(w_\e^{(1),+},\gamma_\e^{(1),+})}{\partial w_j}
  +o\big(1\big),\,\,\,\,~~~~~\frac{\partial q_{\e,i}(\widetilde{w}_\e,\widetilde{\gamma}_\e)}{\partial \gamma_j}=
\frac{\partial\bar q_{\e,i}(w_\e^{(1),+},\gamma_\e^{(1),+})}{\partial \gamma_j}
  +o\big(1\big).
 \end{equation*}\end{small}Then using the above estimate, \eqref{sec7-61} and \eqref{sec7-62}, we have
 \begin{small}
\begin{equation*}
\begin{split}
\bf \widetilde Q_{\e,2}^{T}
 \left(
  \begin{array}{cc}
   \Big( \frac{\partial  q_{\e,i}(\widetilde{w}_\e,\widetilde{\gamma}_\e)}{\partial w_j}\Big)_{1\leq i,j\leq 2} \end{array}\right)\bf \widetilde Q_{\e,2}
   =  \left(
   \begin{array}{cc}\frac{6\pi ( \nabla\mathcal{R}_\Omega(0)\cdot w_{\e}^{(1),+} )}{|w_\e^{(1),+}|^2} +o\big(1\big) & \frac{6\pi (\nabla \mathcal{R}_\Omega(0)\cdot  w_\e^{(2),+} )}{|w_\e^{(1),+}|\cdot|w_\e^{(2),+}| } +o\big(1\big)   \\[4mm]
  \frac{6\pi (\nabla \mathcal{R}_\Omega(0)\cdot  w_\e^{(2),+} )}{|w_\e^{(1),+}|\cdot|w_\e^{(2),+}| }+o\big(1\big) & -\frac{6\pi ( \nabla\mathcal{R}_\Omega(0)\cdot w_{\e}^{(1),+} )}{|w_\e^{(1),+}|^2} +o\big(1\big)\\
   \end{array}
 \right),
 \end{split}
 \end{equation*}\end{small}
 \begin{small}
\begin{equation*}
\begin{split}
\bf \widetilde Q_{\e,2}^{T}
 \left(
  \begin{array}{cc}
   \Big( \frac{\partial q_{\e,i}(\widetilde{w}_\e,\widetilde{\gamma}_\e)}{\partial \gamma_j}\Big)_{1\leq i,j\leq 2} \end{array}\right)\bf \widetilde Q_{\e,2}
    = \left(
   \begin{array}{cc}
  \frac{1}{ 2|w_\e^{(1),+}|^2}+o\big(1\big)  & o\big(1\big)\\[2mm]
     o\big(1\big) & -\frac{1}{ 2|w_\e^{(1),+}|^2}+o\big(1\big) \\
   \end{array}
 \right).
 \end{split}
 \end{equation*}\end{small}So we have proved that
  \begin{small}
\begin{equation*}
\begin{split}
& \left(
  \begin{array}{cc}
   \bf \widetilde Q_{\e,2}^{T} & \bf O_{2\times 2} \\[4mm]
    \bf O_{2\times 2} & \bf \widetilde Q^{T}_{\e,2} \\
  \end{array}
\right)
 \left(
  \begin{array}{cc}
   \Big( \frac{\partial  p_{\e,j}(\widetilde{w}_\e,\widetilde{\gamma}_\e)}{\partial w_i}\Big)_{1\leq i,j\leq 2} &
   \Big( \frac{\partial   p_{\e,j}(\widetilde{w}_\e,\widetilde{\gamma}_\e)}{\partial \gamma_i}\Big)_{1\leq i,j\leq 2} \\[4mm]
   \Big( \frac{\partial   q_{\e,j}(\widetilde{w}_\e,\widetilde{\gamma}_\e)}{\partial w_i}\Big)_{1\leq i,j\leq 2}  &
   \Big( \frac{\partial   q_{\e,j}(\widetilde{w}_\e,\widetilde{\gamma}_\e)}{\partial \gamma_i}\Big)_{1\leq i,j\leq 2}  \\
  \end{array}
\right) \left(
  \begin{array}{cc}
   \bf \widetilde Q_{\e,2} & \bf O_{2\times 2} \\[4mm]
    \bf O_{2\times 2} & \bf \widetilde Q_{\e,2} \\
  \end{array}
\right)\\[3mm]&
=\left(
          \begin{array}{cccc}
            \frac{4e^{{3\pi \mathcal{R}_{\Omega}(0)}}}{ \e^{\frac{1}{2}} \ln \e }\big(1+o(1)\big)  & O(1)  & o(1) & O(1) \\[3mm]
            O\left(\frac{1}{\e^{\frac{1}{4}}}\right) & a_\e +o(1)  &      - \frac{3\pi (\nabla \mathcal{R}_\Omega(0)\cdot w_\e^{(2),+})  }{|w_\e^{(1),\pm}|\cdot|w_\e^{(2),+}|}   & o(1) \\[3mm]
           O(1) &  \frac{6\pi  (\nabla \mathcal{R}_\Omega(0)\cdot w_\e^{(2),+}) }{|w_\e^{(1),\pm}|\cdot|w_\e^{(2),+}|}  & \frac{1}{ 2|w_\e^{(1),+}|^2}+o(1) & o(1) \\[3mm]
            O(1) & O(1) & o(1) & -\frac{1}{ 2|w_\e^{(1),+}|^2} +o(1)\\[3mm]
          \end{array}
        \right),
 \end{split}
 \end{equation*}\end{small}with $a_\e:=6\pi(\lambda_1-\lambda_2) -36 \pi^2\Big(\nabla \mathcal{R}_\Omega(0)\cdot \frac{w_\e^{(2),+}}{|w_\e^{(2),+}|}\Big)^2$. This shows
 \begin{small}
\begin{align*}
det~Jac~ \hat{\bf{V}}_\e(\widetilde{w}_\e,\widetilde{\gamma}_\e)=\frac{6\pi e^{9\pi \mathcal{R}_{\Omega}(0)} (\lambda_j-\lambda_m)}{   \e^{\frac{1}{2}} \ln \e }\big(1+o(1)\big)\neq 0,
 \end{align*}
\end{small}for the case
$(\widetilde{w}_\e,\widetilde{\gamma}_\e)= \Big(w_\e^{(1),+}+o\big(1\big),\gamma_\e^{(1),+}+o\big(1\big)\Big)$
which proves \eqref{sec7-73}.
\end{proof}

\begin{proof}[\bf{Proof of Theorem \ref{sec1-teo16}}]  The existence of at least four solutions follows from Proposition \ref{sec7-prop7.17} and \eqref{sec7-47}.
Also  from \eqref{sec7-47}, \eqref{sec7-48} and \eqref{sec7-73}, the critical points of
$\mathcal{KR}_{\Omega_\e}(x,y)$ are all nondegenerate. Therefore, the number
of solutions is finite.

\vskip 0.1cm

Next, we prove that for any fixed $m \in \{1,2\}$, $\hat{\mathbf{V}}_\e(w,\gamma)=0$ has a unique solution in
$B\big((w_\e^{(m),+},\gamma_\e^{(m),+}),\delta\big)$ or
$B\big((w_\e^{(m),-},\gamma_\e^{(m),-}),\delta\big)$.
For instance, suppose that there are $l$ solutions in $B\big((w_\e^{(m),+},\gamma_\e^{(m),+}),\delta\big)$.
Then, by Proposition \ref{sec7-prop7.19}, we have
\begin{small}
\begin{align}\label{sec7-76}
\deg\Big(\hat{\mathbf{V}}_\e(w,\gamma),0, B\big((w_\e^{(m),+},\gamma_\e^{(m),+}),\delta\big) \Big)
= l\, \text{sign} (\lambda_m - \lambda_j), \quad j \ne m.
\end{align}
\end{small}On the other hand, it follows from \eqref{sec7-58} and \eqref{sec7-63} that
\begin{small}
\begin{align*}
\deg\Big(\hat{\mathbf{V}}_\e(w,\gamma),0, B\big((w_\e^{(m),+},\gamma_\e^{(m),+}),\delta\big)\Big)
= \text{sign} (\lambda_m - \lambda_j), \quad j \ne m,
\end{align*}
\end{small}which, together with \eqref{sec7-76}, implies that $l=1$.

\vskip 0.1cm

Hence, we have proved that $\mathcal{KR}_{\Omega_\e}(x,y)$ possesses exactly four type III critical points. Since $\La_1=\La_2$ in the expression of $\mathcal{KR}_{\Omega_\e}(x,y)$, then
 $\mathcal{KR}_{\Omega_\e}(x,y)=\mathcal{KR}_{\Omega_\e}(y,x)$. Hence if $(x_\e,y_\e)$ is a
 critical point of $\nabla \mathcal{KR}_{\Omega_\e}(x_\e,y_\e)=0$ is equivalent to $\nabla \mathcal{KR}_{\Omega_\e}(y_\e,x_\e)=0$. This means that only two of them are nontrivially distinct.
\end{proof}
\begin{rem}\label{sec7-rem7.20}
Now we give a domain on which the assumptions of Theorem \ref{sec1-teo16} hold.
For example, let $\Omega = B(Q,1)$ with $0 < |Q| < 1$. Then $\nabla \mathcal{R}_\Omega(0) \neq 0$.
Moreover, by direct computation we obtain
\begin{small}
\begin{equation*}
\frac{\partial^2 H_{\Omega}(x,y)}{\partial y_i \partial y_k}\Big|_{x=y}
= \frac{|y-Q|^2}{2\pi(1-|y-Q|^2)^{2}}
\left( \delta_{ik} - \frac{2(y_i-Q_i)(y_k-Q_k)}{|y-Q|^2} \right),
~~~\text{for } i,k = 1,2,
\end{equation*}
\end{small}and
\begin{small}
\begin{equation*}
\frac{\partial \mathcal{R}_{\Omega}(y)}{\partial y_i}
= -\frac{2(y_i-Q_i)}{2\pi(1-|y-Q|^2)},
~~~\text{for } i = 1,2.
\end{equation*}
\end{small}Hence, we have
\begin{small}
\begin{equation*}
\frac{\partial^2 H_{\Omega}(0,0)}{\partial y_i \partial y_k}
- 3\pi \frac{\partial \mathcal{R}_{B(Q,1)}(0)}{\partial y_i}
  \frac{\partial \mathcal{R}_{B(Q,1)}(0)}{\partial y_k}
= \frac{|Q|^2}{2\pi(1-|Q|^2)^{2}}
  \left( \delta_{ik} - \frac{8Q_i Q_k}{|Q|^2} \right),
~~~\text{for } i,k= 1,2.
\end{equation*}
\end{small}Recalling that
\begin{small}\[
\widetilde{\mathbf{M}}
:= \left[
\frac{\partial^2 H_\Omega(0,0)}{\partial y_i \partial y_k}
- 3\pi \frac{\partial \mathcal{R}_\Omega(0)}{\partial y_i}
  \frac{\partial \mathcal{R}_\Omega(0)}{\partial y_k}
\right]_{1 \le i,k \le 2},
\]
\end{small}we find that the two eigenvalues of $\widetilde{\mathbf{M}}$ are
$\frac{|Q|^2}{2\pi(1-|Q|^2)^{2}}$ and $-\frac{7|Q|^2}{2\pi(1-|Q|^2)^{2}}$.

\vskip 0.1cm

 Furthermore, by a perturbation argument, in the ellipse
\begin{small}$$\Omega_\delta=\Big\{(x_1,x_2)\in \mathbb{R}^2, (x_1-Q_1)^2\big(1+\alpha_1\delta\big)^2+(x_2-Q_2)^2\big(1+\alpha_2\delta\big)^2 <1\Big\},$$
\end{small}where  $Q=(Q_1,Q_2)\in B(0,1)$, $\alpha_1,\alpha_2\geq 0, \alpha_1\neq \alpha_2$,
and $\delta>0$ is small, we can show that the corresponding matrix $\widetilde{\mathbf{M}}$
has  two different eigenvalues
$\frac{|Q|^2}{2\pi(1-|Q|^2)^{2}}+o_\delta(1)$ and $-\frac{7|Q|^2}{2\pi(1-|Q|^2)^{2}}+o_\delta(1)$.

\end{rem}
\vskip 0.2cm

\subsection{The case $\nabla\mathcal{R}_\O(0)= 0$ (Proof of Theorem \ref{sec1-teo17}) }~

\vskip 0.2cm

We now use Proposition~\ref{sec7-prop7.12} to study the case
 $\nabla\mathcal{R}_\O(0)= 0$. In this case
\eqref{sec7-42} and \eqref{sec7-43} become
\begin{small}\begin{align}\label{sec7-77}
 & \frac{\partial \mathcal{KR}_{\Omega_\e}(x,y)}{\partial x_j}\Big|_{(x,y)=
 \Big(\e^{\beta} w, \frac{-\e^{\beta} w+ \e^{2\beta } \gamma }{\tau} \Big) }\notag
 \\=&-\frac{\La_1\La_2}{\pi}\left\{ \left[
  \frac{k(|w|,\tau)}{\e^{\beta}|w|^2(\ln \e +2\pi \mathcal{R}_{\Omega}(0))} -
\frac{(w\cdot \gamma) }{|w|^4}
  \left(2\beta- \frac{1}{\ln \e +2\pi \mathcal{R}_{\Omega}(0)}\right)
-\frac{  \tau \e^{\beta}\big( 4(w\cdot \gamma)^2-|w|^2\cdot|\gamma|^2\big) }{(\tau+1)^3|w|^6 }
\right]
  w_j   \right.\notag\\[1mm]&
 \left.   + \frac{ \beta }{|w|^2}
  \left[1+ \frac{2 (w\cdot \gamma)\e^{\beta} }{ (\tau+1) |w|^2} \right]\gamma_j+
2\pi \e^{\beta} \sum^2_{i=1}\left[ (1+\tau)(\beta-1)
 \frac{\partial^2H_{\Omega}(0,0)}{\partial x_i\partial x_j}
-\big(  \tau-\frac{1}{\tau} \big) \frac{\partial^2H_{\Omega}(0,0)}{\partial y_i\partial x_j} \right]
 w_i \right\} +
   O\left(\frac{ \e^{\beta}}{|\ln \e|}
\right),
\end{align}\end{small}and
\begin{small}\begin{align}\label{sec7-78}
 & \frac{\partial \mathcal{KR}_{\Omega_\e}(x,y)}{\partial y_j}\Big|_{(x,y)=
 \Big(\e^{\beta} w, \frac{-\e^{\beta} w + \e^{2\beta } \gamma }{\tau} \Big) }\notag
 \\[1mm]=&- \frac{\La_2^2}{\pi}\left\{
 \left[-
  \frac{ \tau k(|w|,\tau)}{|w|^2 \e^{\beta}(\ln \e +2\pi \mathcal{R}_{\Omega}(0))}
  -
\frac{  \tau(w\cdot \gamma)}{|w|^4}
\Big(\frac{2 k(|w|,\tau)-1}{\ln \e +2\pi \mathcal{R}_{\Omega}(0)}+2\tau\beta\Big)
\right.\right.
 \notag  \\[1mm]& \left.\left. -\frac{ \beta \tau^2(\tau+2)  \e^{\beta}\big( 4(w\cdot \gamma)^2-|w|^2\cdot|\gamma|^2\big) }{(\tau+1) |w|^6}
\right]
  w_j
+  \frac{ \tau}{|w|^2} \left[
 \tau \beta+\frac{k(|w|,\tau)}{\ln \e +2\pi \mathcal{R}_{\Omega}(0)}
  +
\frac{2\beta(\tau+2)  (w\cdot \gamma) \e^{\beta}}{  (1+\tau) |w|^2}    \right]
  \gamma_j  \right.
 \notag  \\[1mm]& \left.
 - 2\pi\e^{\beta}  \sum^2_{i=1}\left[ \frac{(1+\tau)(\beta-1)}{\tau}
 \frac{\partial^2H_{\Omega}(0,0)}{\partial x_i\partial x_j}
+\big(  \tau-\frac{1}{\tau} \big) \frac{\partial^2H_{\Omega}(0,0)}{\partial y_i\partial x_j} \right]
 w_i \right\} +
   O\left(\frac{ \e^{\beta}}{|\ln \e|}
\right).
\end{align}\end{small}Let $(x_\e,y_\e)$ be a Type~III critical point of $\mathcal{KR}_{\Omega_\e}(x,y)$, and define
$(w_\e,\gamma_\e):=
 \Big(\frac{x_\e}{\e^{\beta}}, \frac{x_\e+\tau y_\e}{\e^{2\beta}}\Big)$, then we have that $\displaystyle\lim_{\e\to0}|\gamma_\e|<\infty$.
However, if $\nabla\mathcal{R}_\O(0)=0$, it is possible to deduce a more precise estimate.
\begin{prop}\label{sec7-prop7.21}
If $\nabla\mathcal{R}_\O(0)= 0$, then it holds
 \begin{equation}\label{sec7-79}
\lim_{\e\to 0} \frac{|\gamma_\e|}{\e^{\beta}}<\infty,
 \end{equation}where $\gamma_\e:= \frac{x_\e+\tau y_\e}{\e^{2\beta}}$ with
 $(x_\e,y_\e)$ being the type III critical point of $\mathcal{KR}_{\Omega_\e}(x,y)$.
\end{prop}
\begin{proof}
From \eqref{sec7-77}$\times \tau-$\eqref{sec7-78}, we have
$\frac{k(|w_\e|,\tau)}{\e^{\beta} (\ln \e +2\pi \mathcal{R}_{\Omega}(0))} =O\Big(\e^{\beta}\Big)$.
Putting this into \eqref{sec7-77}, we get
\begin{small}
\begin{equation*} -
\frac{(w_\e\cdot \gamma_\e) }{|w_\e|^4}
  \left(2\beta- \frac{1}{\ln \e +2\pi \mathcal{R}_{\Omega}(0)}\right) w_{\e,j}  + \frac{ \beta }{|w_\e|^2}\gamma_{\e,j}
 =O\Big(\e^{\beta}\Big),
\end{equation*}
\end{small}which gives \eqref{sec7-79}.
\end{proof}

Now using Proposition \ref{sec7-prop7.21}, we have
\begin{small}\begin{align}
 & \frac{\partial \mathcal{KR}_{\Omega_\e}(x,y)}{\partial x_j}\Big|_{(x,y)=
 \Big(\e^{\beta} w, \frac{-\e^{\beta} w+ \e^{2\beta } \gamma }{\tau} \Big) }\notag
 \\=&-\frac{\La_1\La_2}{\pi}\left\{ \left[
  \frac{k(|w|,\tau)}{\e^{\beta}|w|^2(\ln \e +2\pi \mathcal{R}_{\Omega}(0))} -
\frac{2 \beta (w\cdot\gamma) }{|w|^4}
\right]
  w_j   \right.\notag\\[1mm]&
 \left.   + \frac{ \beta }{|w|^2} \gamma_j+
2\pi \e^{\beta} \sum^2_{i=1}\left[ (1+\tau)(\beta-1)
 \frac{\partial^2H_{\Omega}(0,0)}{\partial x_i\partial x_j}
-\big(  \tau-\frac{1}{\tau} \big) \frac{\partial^2H_{\Omega}(0,0)}{\partial y_i\partial x_j} \right]
 w_i \right\} +
   O\left(\frac{ \e^{\beta}}{|\ln \e|}
\right),\notag
\end{align}\end{small}and
\begin{small}\begin{align}
 & \frac{\partial \mathcal{KR}_{\Omega_\e}(x,y)}{\partial y_j}\Big|_{(x,y)=
 \Big(\e^{\beta} w, \frac{-\e^{\beta} w + \e^{2\beta } \gamma }{\tau} \Big) }\notag
 \\[1mm]=&- \frac{\La_2^2}{\pi}\left\{
 \left[-
  \frac{ \tau k(|w|,\tau)}{|w|^2 \e^{\beta}(\ln \e +2\pi \mathcal{R}_{\Omega}(0))}
  -
\frac{  2\tau^2\beta(w\cdot \gamma)}{|w|^4}
\right]
  w_j\right.
 \notag  \\[1mm]& \left.
+  \frac{ \tau^2 \beta }{|w|^2}
  \gamma_j -2\pi\e^{\beta}   \sum^2_{i=1}\left[\frac{(1+\tau)(\beta-1)}{\tau}
 \frac{\partial^2H_{\Omega}(0,0)}{\partial x_i\partial x_j}
+\big(  \tau-\frac{1}{\tau} \big) \frac{\partial^2H_{\Omega}(0,0)}{\partial y_i\partial x_j} \right]
 w_i \right\} +
   O\left(\frac{ \e^{\beta}}{|\ln \e|}
\right).\notag
\end{align}\end{small}Define
\begin{small}
\begin{equation}\label{sec7-80}
\begin{cases}
  f_{\e,j}(w,\gamma):=-\frac{\pi}{\Lambda_1\Lambda_2}\frac{\partial \mathcal{KR}_{\Omega_\e}(x,y)}{\partial x_j}\Big|_{(x,y)=
 \Big(\e^{\beta} w, \frac{-\e^{\beta} w + \e^{2\beta } \gamma }{\tau} \Big)  },\\[2mm]
g_{\e,j}(w,\gamma):=-\frac{\pi}{\Lambda_2^2}\frac{\partial \mathcal{KR}_{\Omega_\e}(x,y)}{\partial y_j}\Big|_{(x,y)=\Big(\e^{\beta} w, \frac{-\e^{\beta} w + \e^{2\beta } \gamma }{\tau} \Big)  }.
 \end{cases}
 \end{equation}\end{small}First, we give the main part of $f_{\e,j}$ and $g_{\e,j}$.

\begin{lem}\label{sec7-lem7.22}
For any $(w,\gamma)\in \mathcal{H}_\e^*=\Big\{(w,\gamma)\in \mathcal{H}'_\e,\displaystyle\lim_{\e\to 0}\frac{|\gamma|}{\e^\beta}<\infty\Big\}$, we have
 \begin{small}\begin{align}\label{sec7-81}
 \begin{cases}
\partial^k f_{\e,j}(w,\gamma)=\partial^k f^*_{\e,j}(w,\gamma) +O\left(\frac{ \e^{\beta}}{|\ln \e|}
\right),\\[3mm]
\partial^k g_{\e,j}(w,\gamma)=\partial^k g_{\e,j}^*(w,\gamma) +O\left(\frac{ \e^{\beta}}{|\ln \e|}
\right),\end{cases}
\end{align}\end{small}with $k=0,1$, $j=1,2$, $\partial^k$ and
$\mathcal{H}'_\e$ being the notations in \eqref{sec7-45b} and Proposition \ref{sec7-prop7.5},
 \begin{small}\begin{align}
f^*_{\e,j}(w,\gamma):=&\left[
  \frac{k(|w|,\tau)}{\e^{\beta}|w|^2(\ln \e +2\pi \mathcal{R}_{\Omega}(0))} -
\frac{2 \beta (w\cdot\gamma) }{|w|^4}
\right]
  w_j    \notag\\[1mm]&
  + \frac{ \beta }{|w|^2} \gamma_j+
2\pi \e^{\beta} \sum^2_{i=1}\left[ (1+\tau)(\beta-1)
 \frac{\partial^2H_{\Omega}(0,0)}{\partial x_i\partial x_j}
-\big(  \tau-\frac{1}{\tau} \big) \frac{\partial^2H_{\Omega}(0,0)}{\partial y_i\partial x_j} \right]
 w_i,\notag
\end{align}\end{small}and
\begin{small}\begin{align}
g^*_{\e,j}(w,\gamma)
:= &\left[-
  \frac{ \tau k(|w|,\tau)}{|w|^2 \e^{\beta}(\ln \e +2\pi \mathcal{R}_{\Omega}(0))}
  -
\frac{  2\tau^2\beta(w\cdot \gamma)}{|w|^4}
\right]
  w_j
 \notag  \\[1mm]&
+  \frac{ \tau^2 \beta }{|w|^2}
  \gamma_j -2\pi\e^{\beta}   \sum^2_{i=1}\left[\frac{(1+\tau)(\beta-1)}{\tau}
 \frac{\partial^2H_{\Omega}(0,0)}{\partial x_i\partial x_j}
+\big(  \tau-\frac{1}{\tau} \big) \frac{\partial^2H_{\Omega}(0,0)}{\partial y_i\partial x_j} \right]
 w_i.\notag
\end{align}\end{small}
\end{lem}
\begin{proof}First, \eqref{sec7-81} with $k=0$ holds by Proposition \ref{sec7-prop7.12} and Proposition \ref{sec7-prop7.21}.
Also using Proposition \ref{sec7-teo7.13}, we can deduce \eqref{sec7-81} with $k=1$.
\end{proof}

Now we devote to solve $f_{\e,j}(w,\gamma)=0$ and $g_{\e,j}(w,\gamma)=0$. This is tedious and we introduce following transform firstly.
 Let
\begin{small}\begin{equation*}
\begin{cases}
  h_{\e,j}(w,\gamma):=\frac{1}{\tau(\tau+1)\e^{\beta}}\Big( \tau^2 f_{\e,j}(w,\gamma) -  g_{\e,j}(w,\gamma)\Big),\\[2mm]
  n_{\e,j}(w,\gamma):= \frac{1}{2\beta \tau(\tau+1)\e^\beta}\Big( \tau f_{\e,j}(w,\gamma)+ g_{\e,j}(w,\gamma)\Big).
  \end{cases}
 \end{equation*}
 \end{small}Then we have following results directly.

 \begin{lem}\label{sec7-lem7.23}It holds
 \begin{small}
 \begin{equation}\label{sec7-83}
 \begin{cases}
 f_{\e,j}(w,\gamma)=0,\\[1mm]
 g_{\e,j}(w,\gamma)=0,
 \end{cases} \Leftrightarrow ~~~
 \begin{cases}
 h_{\e,j}(w,\gamma)=0,\\[1mm]
 n_{\e,j}(w,\gamma)=0.
 \end{cases}
 \end{equation}
 \end{small}Moreover, if $(w,\gamma)$ solves $f_{\e,j}(w,\gamma)=
g_{\e,j}(w,\gamma)=0$ for $j=1,2$, then it holds
\begin{small} \begin{equation*}
 det~\left(
  \begin{array}{cc}
   \Big( \frac{\partial f_{\e,j}(w,\gamma)}{\partial w_i}\Big)_{1\leq i,j\leq 2} &
   \Big( \frac{\partial f_{\e,j}(w,\gamma)}{\partial \gamma_i}\Big)_{1\leq i,j\leq 2} \\[4mm]
   \Big( \frac{\partial g_{\e,j}(w,\gamma)}{\partial w_i}\Big)_{1\leq i,j\leq 2}  &
   \Big( \frac{\partial g_{\e,j}(w,\gamma)}{\partial \gamma_i}\Big)_{1\leq i,j\leq 2}  \\
  \end{array}
\right)\neq 0 \Leftrightarrow ~~~
 det~\left(
  \begin{array}{cc}
   \Big( \frac{\partial h_{\e,j}(w,\gamma)}{\partial w_i}\Big)_{1\leq i,j\leq 2} &
   \Big( \frac{\partial h_{\e,j}(w,\gamma)}{\partial \gamma_i}\Big)_{1\leq i,j\leq 2} \\[4mm]
   \Big( \frac{\partial n_{\e,j}(w,\gamma)}{\partial w_i}\Big)_{1\leq i,j\leq 2}  &
   \Big( \frac{\partial n_{\e,j}(w,\gamma)}{\partial \gamma_i}\Big)_{1\leq i,j\leq 2}  \\
  \end{array}
\right)\neq 0.
 \end{equation*}
 \end{small}
 \end{lem}

Now we give the expansion of $h_{\e,j}(w,\gamma)$ and $n_{\e,j}(w,\gamma)$.

\begin{prop}
For any $(w,\gamma)\in \mathcal{H}_\e^*=\Big\{(w,\gamma)\in \mathcal{H}'_\e,\displaystyle\lim_{\e\to 0}\frac{|\gamma|}{\e^\beta}<\infty\Big\}$, we have
\begin{small}
\begin{equation}\label{sec7-84}
\begin{cases}
\partial^k h_{\e,j}(w,\gamma)= \partial^k \bar h_{\e,j}(w,\gamma) +  \partial^k h^*_{\e,j}(w,\gamma)+O\Big(\frac{1}{|\ln \e|}\Big), \\[2mm]
\partial^k n_{\e,j}(w,\gamma)=\partial^k \bar n_{\e,j}(w,\gamma)
 +\partial^k n^*_{\e,j}(w,\gamma)+O\Big(\frac{1}{|\ln \e|} \Big),
 \end{cases}
 \end{equation}\end{small}with $k=0,1$, $j=1,2$, $\partial^k$ and
$\mathcal{H}'_\e$ being the notations in \eqref{sec7-45b} and Proposition \ref{sec7-prop7.5},
 \begin{small}
\begin{equation*}
\begin{cases}
 \bar h_{\e,j}(w,\gamma):=\left[
  \frac{k(|w|,\tau)}{\e^{2\beta}|w|^2(\ln \e +2\pi \mathcal{R}_{\Omega}(0))}
\right]
  w_j   -\frac{
2\pi}{\tau^2(\tau+1)}  \Big(\overline{\bf{M}}w \Big)_j,\\[2mm]
 h^*_{\e,j}(w,\gamma):=
\frac{(\tau-1)(w\cdot\gamma)}{\tau|w|^4\e^\beta(\ln \e +2\pi \mathcal{R}_{\Omega}(0))}
  w_j,
 \end{cases}
 \end{equation*}\end{small}and
  \begin{small}
\begin{equation*}
\begin{cases}
 \bar n_{\e,j}(w,\gamma):=\frac{ (w\cdot \gamma)w_j }{\e^\beta|w|^4} - \frac{\gamma_j }{ 2\e^\beta |w|^2}
-\frac{ \pi(\tau^2-1)}{\tau^3}\Big({\bf{M}}_1 w \Big)_j,\\
n^*_{\e,j}(w,\gamma):=
-\frac{(\tau+1)(w\cdot \gamma)w_j}{\tau\e^\beta|w|^4(\ln \e +2\pi \mathcal{R}_{\Omega}(0))},
 \end{cases}
 \end{equation*}\end{small}with
 \begin{small}$$\overline{\bf{M}}:=\left[ (\tau^4+\tau^2+ 1)
 \frac{\partial^2H_{\Omega}(0,0)}{\partial x_i\partial x_j}
+(\tau^2-1)^2 \frac{\partial^2H_{\Omega}(0,0)}{\partial y_i\partial x_j} \right]_{1\leq i,j\leq 2}$$\end{small}and
  ${\bf{M}}_1:=\left[   (\tau^2+\tau+1)
 \frac{\partial^2H_{\Omega}(0,0)}{\partial x_i\partial x_j}
+(\tau+1)^2 \frac{\partial^2H_{\Omega}(0,0)}{\partial y_i\partial x_j}\right]_{1\leq i,j\leq 2}$.
\end{prop}
\begin{proof}
First, \eqref{sec7-84} with $k=0$ holds by Lemma \ref{sec7-lem7.22} and Lemma \ref{sec7-lem7.23}. Next, using Proposition \ref{sec7-teo7.13}, we can get \eqref{sec7-84} with $k=1$.
\end{proof}

\begin{rem}
Here we point out that the terms ${h}^*_{\e,j}(w,\gamma)$ and ${n}^*_{\e,j}(w,\gamma)$ are crucial. To estimate  ${h}_{\e,j}(w,\gamma)$ and ${n}_{\e,j}(w,\gamma)$,
\eqref{sec7-84} with $k=0$ can be written as follows:
\begin{small}
\begin{equation*}
\begin{cases}
h_{\e,j}(w,\gamma)= \bar h_{\e,j}(w,\gamma)+O\Big(\frac{1}{|\ln \e|}\Big), \\[2mm]
n_{\e,j}(w,\gamma)= \bar n_{\e,j}(w,\gamma)
  +O\Big(\frac{1}{|\ln \e|} \Big).
 \end{cases}
 \end{equation*}\end{small}But to estimate  the derivative of ${h}_{\e,j}(w,\gamma)$ and ${n}_{\e,j}(w,\gamma)$, the terms ${h}^*_{\e,j}(w,\gamma)$ and $ {n}^*_{\e,j}(w,\gamma)$ in \eqref{sec7-84} cannot be ignored.

\end{rem}

\begin{prop}\label{sec7-prop7.26}
If $\overline{\bf{M}}$ has two different eigenvalues $\mu_1$ and $\mu_2$
with unit eigenvectors $v^{(1)}$ and $v^{(2)}$, then system
  \begin{small}
 \begin{equation}\label{sec7-87}
 \begin{cases}
 \bar h_{\e,j}(w,\gamma)=0,\\[1mm]
 \bar n_{\e,j}(w,\gamma)=0,
 \end{cases}
 \end{equation}
 \end{small}has exactly four
solutions $(\overline w_\e^{(m),\pm},\overline \gamma_\e^{(m),\pm})$ with
\begin{small}
\begin{equation}\label{sec7-88}
\begin{cases}
\overline w_\e^{(m),\pm}=\pm\left[C_\tau+ \frac{ \pi C_\tau^3 \mu_m}{(\tau^2+\tau) } \e^{2\beta} \ln \e\big(1+o(1)\big)\right] v^{(m)},\\[2mm]
\overline \gamma_\e^{(m),\pm}=
-\frac{2\pi(\tau^2-1)|\overline w_\e^{(m),\pm}|^2\e^\beta}{\tau^3}{\bf{M}}_1\overline w_\e^{(m),\pm}+ \frac{ 4\pi(\tau^2-1)\e^\beta}{\tau^3}\Big(\overline w_\e^{(m),\pm}{\bf{M}}_1 \overline w_\e^{(m),\pm} \Big)\overline w_\e^{(m),\pm},
\end{cases}
\end{equation}
\end{small}where $C_\tau= \tau^{\frac{1}{1+\tau}}
e^{-\frac{2\pi \mathcal{R}_{\Omega}(0)(\tau^2+ \tau+1 )}{(1+\tau)^2}}$ is the   constant in Theorem \ref{sec1-teo12}, $\tau=\frac{\La_1}{\La_2}$ and $\beta=\frac{\tau}{(\tau+1)^2}$.
\end{prop}
\begin{proof}Since
$\overline{\bf{M}}$ has two different eigenvalues $\mu_1$ and $\mu_2$
with unit eigenvectors $v^{(1)}$ and $v^{(2)}$, then
\begin{small}
\begin{equation*}
  \frac{k(|w|,\tau)}{\e^{2\beta}|w|^2(\ln \e +2\pi \mathcal{R}_{\Omega}(0))} = \frac{
2\pi \mu_i}{\tau^2(\tau+1)},\;\; \frac{w}{|w|}=\pm v^{(i)},\quad i=1,2,
\end{equation*}
\end{small}which allows to compute $|w|$ as follows,
\begin{small}
\begin{equation*}
 |w| = C_\tau+   \frac{
 \pi \mu_i \e^{2\beta} \ln \e\big(1+o(1)\big)}{\tau^2(\tau+1)}.
\end{equation*}
\end{small}Since  $\bar h_{\e,j}(w,\gamma)$ is independent of $\gamma$, then
$\widetilde h_{\e,j}(w):=\bar h_{\e,j}(w,\gamma)=0$ has four solutions
\begin{small}
\begin{equation*}
\overline w_\e^{(m),\pm}=\pm\left[C_\tau+ \frac{ \pi C_\tau^3 \mu_m}{(\tau^2+\tau) } \e^{2\beta} \ln \e\big(1+o(1)\big)\right] v^{(m)}.
\end{equation*}
\end{small}On the other hand, $\sum^2_{j=1}w_j \bar n_{\e,j}(w,\gamma)=0$ gives
\begin{small}\begin{align*}
 w\cdot\gamma  = \frac{ 2\pi(\tau^2-1)|w|^2\e^\beta}{\tau^3}\Big(w{\bf{M}}_1 w \Big).
\end{align*}\end{small}Inserting this
into $\bar n_{\e,j}(w,\gamma)=0$, we can uniquely determine $\gamma$ by $w$,
\begin{small}\begin{align*}
 \gamma= -\frac{2\pi(\tau^2-1)|w|^2\e^\beta}{\tau^3}{\bf{M}}_1w+ \frac{ 4\pi(\tau^2-1)\e^\beta}{\tau^3}\Big(w{\bf{M}}_1 w \Big) w.
\end{align*}\end{small}This
shows that \eqref{sec7-87} has exactly four
solutions as in \eqref{sec7-88}.
\end{proof}

Let
${\bf{W}}_\e(w,\gamma)=\big(\bar h_{\e,1}(w,\gamma),\bar h_{\e,2}(w,\gamma),\bar n_{\e,1}(w,\gamma),\bar n_{\e,2}
(w,\gamma)\big)$.
We have the following results.
\begin{prop}\label{sec7-prop7.27}
If $\overline{\bf{M}}$ has two different eigenvalues $\mu_1$ and $\mu_2$, then  it holds
\begin{small}
\begin{align}\label{sec7-89}
det~Jac~{\bf{W}}_\e(\overline w_\e^{(m),\pm},\overline \gamma_\e^{(m),\pm})=  \frac{\pi (\mu_j-\mu_m)}{2\tau^2 C_\tau^6\e^{4\beta}\ln \e }\big(1+o(1)\big)\neq 0~~~\mbox{with}~~~j=1,2~~~\mbox{and}~~~j\neq m,
 \end{align}
\end{small}where $(\overline{w}_\e^{(m),\pm},\overline{\gamma}_\e^{(m),\pm})$ with $m=1,2$ are all solutions of ${\bf{W}}_\e(w,\gamma)=0$ and $C_\tau= \tau^{\frac{1}{1+\tau}}
e^{-\frac{2\pi \mathcal{R}_{\Omega}(0)(\tau^2+ \tau+1 )}{(1+\tau)^2}}$ is the constant in Theorem \ref{sec1-teo12}.
\end{prop}
\begin{proof}
By direct computations,
we have
\begin{small}
\begin{align}\label{bar-h}
\frac{\partial \bar{h}_{\e,i}(w,\gamma)}{\partial {w_j}}=&
 \left[
  \frac{k(|w|,\tau)}{\e^{2\beta}|w|^2(\ln \e +2\pi \mathcal{R}_{\Omega}(0))}
\right]\delta_{ij}   -\frac{
2\pi}{\tau^2(\tau+1)} \overline{\bf{M}}_{ij}
 \notag \\&
 +\left[
  \frac{1}{\e^{2\beta}|w|^2(\ln \e +2\pi \mathcal{R}_{\Omega}(0))} \Big(|w| \cdot \frac{\partial k(|w|,\tau)}{\partial r}- 2k(|w|,\tau)\Big)
\right] \frac{w_iw_j}{|w|^2},
\end{align}
\end{small}
\begin{small}
\begin{align*}
\frac{\partial \bar h_{\e,i}(w,\gamma)}{\partial {\gamma_j}}= 0,~~~~~
\frac{\partial \bar n_{\e,i}(w,\gamma)}{\partial {\gamma_j}}=&
\frac{ \delta_{ij}}{2|w|^2\e^\beta} -
\frac{  w_iw_j }{|w|^4\e^\beta}.
\end{align*}
\end{small}Since $\frac{\partial \bar h_{\e,i}(w,\gamma)}{\partial {\gamma_j}}= 0$, then
\begin{small}
\begin{align*}
&det~\left(
  \begin{array}{cc}
 \Big(\frac{\partial \bar h_{\e,i}(w,\gamma)}{\partial {w_j}}\Big)_{1\leq i,j\leq 2} & \Big(\frac{\partial \bar h_{\e,i}(w,\gamma)}{\partial {\gamma_j}}\Big)_{1\leq i,j\leq 2} \notag \\[4mm]
  \Big( \frac{\partial \bar n_{\e,i}(w,\gamma)}{\partial {w_j}}\Big)_{1\leq i,j\leq 2} &  \Big(\frac{\partial \bar n_{\e,i}(w,\gamma)}{\partial {\gamma_j}}\Big)_{1\leq i,j\leq 2} \\
  \end{array}
\right)
=det \left( \frac{\partial \bar h_{\e,i}(w,\gamma)}{\partial {w_j}}\right)_{1\leq i,j\leq 2} \cdot
det \left( \frac{\partial \bar n_{\e,i}(w,\gamma)}{\partial {w_j}}\right)_{1\leq i,j\leq 2}.
\end{align*}
\end{small}Now we consider the Jacobian matrix of ${\bf{W}}_\e$ at $(\overline{w}_\e^{(1),+},\overline{\gamma}_\e^{(1),+})$ and denote by
 \begin{small}
\begin{align*}
{\bf Q}_{\e,3}:=\left(\frac{\overline{w}_\e^{(1),+}}{|\overline{w}_\e^{(1),+}|},
\frac{\overline{w}_\e^{(2),+}}{|\overline{w}_\e^{(2),+}|}\right)=
\left(
  \begin{array}{cc}
    \frac{\overline{w}_{\e,1}^{(1),+}}{|\overline{w}_\e^{(1),+}|} & \frac{\overline{w}_{\e,1}^{(2),+}}{|\overline{w}_\e^{(2),+}|} \\[4mm]
    \frac{\overline{w}_{\e,2}^{(1),+}}{|\overline{w}_\e^{(1),+}|} & \frac{\overline{w}_{\e,2}^{(2),+}}{|\overline{w}_\e^{(2),+}|} \\
  \end{array}
\right).
\end{align*}
\end{small}Then  ${\bf Q}_{\e,3}$ is an orthogonal matrix and satisfies  ${\bf Q}^T_{\e,3}\frac{\overline{w}_\e^{(1),+}}{|\overline{w}_\e^{(1),+}|}=\left(
                                                                       \begin{array}{c}
                                                                         1 \\
                                                                         0 \\
                                                                       \end{array}
                                                                     \right)
$.
Hence we have
\begin{small}
\begin{align*}
 &{\bf Q}_{\e,3}^{T}\left( \frac{\overline w_{\e,i}^{(1),+}\overline w_{\e,j}^{(1),+} }{|\overline w_\e^{(1),+}|^2} \right)_{1\leq i,j\leq 2} {\bf Q}_{\e,3}=
 \left(
   \begin{array}{c}
     1 \\
     0 \\
   \end{array}
 \right)\left(
          \begin{array}{cc}
            1 & 0\\
          \end{array}
        \right)
=\left(
   \begin{array}{cc}
     1 & 0 \\[1mm]
     0 & 0 \\
   \end{array}
 \right).
\end{align*}
\end{small}
\begin{small}
\begin{align*}
{\bf Q}^{T}_{\e,3}{\bf \overline{M}} {\bf Q}_{\e,3}=diag\Big(\mu_1,\mu_2\Big),\,\,\,\,\,~~~~~
  \frac{ k(|\overline{w}_\e^{(1),+}|,\tau )}{ |\overline{w}_\e^{(1),+}|^2 \e^{2\beta}(\ln \e +2\pi \mathcal{R}_{\Omega}(0))} =\frac{
2\pi\mu_1}{\tau^2(\tau+1)}.
 \end{align*}\end{small}So we can get
\begin{small}\begin{align*}
{\bf Q}^T_{\e,3}\left( \frac{\partial \bar h_{\e,i}(\overline{w}_\e^{(1),+},\overline{\gamma}_\e^{(1),+})}{\partial {w_j}}\right)_{1\leq i,j\leq 2}{\bf Q}_{\e,3}
=\left(
  \begin{array}{cc}
    \frac{(1+\tau)(1+o(1))}{\e^{2\beta}\ln \e|\overline{w}_\e^{(1),+}|^2  } & 0 \\[4mm]
    0&  \frac{
2\pi(\mu_1-\mu_2)}{\tau^2(\tau+1)} \\
  \end{array}
\right)
\end{align*}\end{small}and
\begin{small}\begin{align*}
{\bf Q}^T_{\e,3}\left( \frac{\partial \bar n_{\e,i}(\overline{w}_\e^{(1),+},\overline{\gamma}_\e^{(1),+})}{\partial {w_j}}\right)_{1\leq i,j\leq 2}{\bf Q}_{\e,3}
=\left(
  \begin{array}{cc}
   -\frac{1}{2|\overline{w}_\e^{(1),+}|^2 \e^\beta } & 0 \\[4mm]
    0&
\frac{1}{2|\overline{w}_\e^{(1),+}|^2 \e^\beta }\\
  \end{array}
\right).
\end{align*}\end{small}Hence we deduce \eqref{sec7-89} for $m=1$. In a similar way we get \eqref{sec7-89} for $m=2$.
\end{proof}
Let
$\hat {\bf{W}}_\e(w,\gamma)=\big(h_{\e,1}(w,\gamma),
h_{\e,2}(w,\gamma),n_{\e,1}(w,\gamma),n_{\e,2}(w,\gamma)\big)$.
We have following result.
 \begin{prop}\label{sec7-prop7.28}
 For each solution $(\overline{w}_\e^{(m),\pm},\overline{\gamma}_\e^{(m),\pm})$ of $ {\bf{W}}_\e(w,\gamma)=0$ with $m=1,2$, it holds
\begin{small}
 \begin{equation}\label{sec7-90}
 \deg\Big(\hat {\bf{W}}_\e, 0, B\big((\overline{w}_\e^{(m),\pm},\overline{\gamma}_\e^{(m),\pm}),\delta\big)\Big)=
 \deg\Big( {\bf{W}}_\e, 0,  B\big((\overline{w}_\e^{(m),\pm},\overline{\gamma}_\e^{(m),\pm}),\delta\big)\Big)\ne 0,
 \end{equation}
\end{small}and problem  $\hat {\bf{W}}_\e(w,\gamma)=0$ has at least
 one solution in $B\big((\overline{w}_\e^{(m),\pm},\overline{\gamma}_\e^{(m),\pm}),\delta\big)$ for a small
 $\delta>0$.
 \end{prop}

 \begin{proof}
First, we have
\begin{small}\begin{equation*}
 \hat {\bf{W}}_\e(w,\gamma)= {\bf{W}}_\e(w,\gamma)+O\left(\frac1{|\ln\e|} \right).
 \end{equation*}
\end{small}Then for any $t\in [0,1]$, it holds
\begin{small} \[
 t {\bf{W}}_{\e}(w,\gamma)+(1-t)\hat {\bf{W}}_\e(w,\gamma)
 = {\bf{W}}_{\e}(w,\gamma)+O\left(\frac1{|\ln\e|} \right)
 \ne 0,\quad \forall\;(w,\gamma)
 \in\partial   B\big((\overline{w}_\e^{(m),\pm},\overline{\gamma}_\e^{(m),\pm}),\delta\big).
 \]
\end{small}This, together with \eqref{sec7-89}, gives \eqref{sec7-90}. So the equation  $\hat {\bf{W}}_\e(w,\gamma)=0$ has at least
 one solution in $B\big((\overline{w}_\e^{(m),\pm},\overline{\gamma}_\e^{(m),\pm}),\delta\big)$ for a small
 $\delta>0$.
 \end{proof}
\begin{lem}\label{sec7-lem7.29}
 If $(\overline{w}_\e,\overline{\gamma}_\e)$
 is a solution of $\hat{\bf{W}}_\e(w,\gamma)=0$, then there exists $m\in \{1,2\}$ such that
\begin{small}
\begin{align}\label{sec7-91}
(\overline{w}_\e,\overline{\gamma}_\e)=\Big(\overline{w}_\e^{(m),+}+o\big(  1\big),\overline{\gamma}_\e^{(m),+}+o\big(\e^\beta\big)\Big)~~~
\mbox{and}~~~|\overline{w}_\e|-|\overline{w}_\e^{(m),+}|=o\big(\e^{2\beta}{|\ln \e |}\big),
 \end{align}\end{small}or\begin{small}
\begin{align}\label{sec7-92}(\overline{w}_\e,\overline{\gamma}_\e)= \Big(\overline{w}_\e^{(m),-}+o\big(1\big),\overline{\gamma}_\e^{(m),-}+o\big(\e^\beta\big)\Big)~~~
\mbox{and}~~~|\overline{w}_\e|-|\overline{w}_\e^{(m),-}|=o\big(\e^{2\beta}{|\ln \e |}\big),
 \end{align}\end{small}where
$(\overline{w}_\e^{(m),\pm},\overline{\gamma}_\e^{(m),\pm})$ are as in \eqref{sec7-88}.
\end{lem}
\begin{proof}
Let
 $(\overline{w}_\e,\overline{\gamma}_\e)$ be a solution of $\hat{\bf{W}}_\e(w,\gamma)=0$.  Then
  \begin{small}
\begin{equation}\label{sec7-93}
\begin{split}
\left[
  \frac{k(|\overline{w}_\e|,\tau)}{\e^{2\beta}|\overline{w}_\e|^2(\ln \e +2\pi \mathcal{R}_{\Omega}(0))}
\right]
\overline{w}_{\e,j}  -\frac{
2\pi}{\tau^2(\tau+1)}  \Big(\overline{\bf{M}}\overline{w}_\e \Big)_j= O\left(\frac{1}{|\ln \e |}\right)
 \end{split}
 \end{equation}\end{small}and
  \begin{small}
\begin{equation}\label{sec7-94}
\begin{split}
 &\frac{ (\overline{w}_\e\cdot \overline{\gamma}_\e)\overline{w}_{\e,j} }{\e^\beta|\overline{w}_\e|^4} - \frac{\overline{\gamma}_{\e,j} }{ 2\e^\beta |\overline{\gamma}_\e|^2}
-\frac{ \pi(\tau^2-1)}{\tau^3}\Big({\bf{M}}_1 \overline{w}_\e \Big)_j =O\left(\frac{1}{|\ln \e |}\right).
 \end{split}
 \end{equation}\end{small}Let $\frac{\overline{w}_\e}{|\overline{w}_\e|}\to \eta$. Then from \eqref{sec7-93}, there exists $m\in \{1,2\}$ such that either $\eta=v^{(m)}$ or $\eta=-v^{(m)}$. Thus,
\begin{small} \begin{equation*}
  \lim_{\e\to 0}\left[\frac{ k(|\overline{w}_\e|,\tau )}{ |\overline{w}_\e|^2 \e^{2\beta}(\ln \e +2\pi \mathcal{R}_{\Omega}(0))} \right]=\frac{
2\pi\mu_m}{\tau^2(\tau+1)}.
 \end{equation*}\end{small}Without loss of generality, we suppose that
 $\eta=v^{(1)}=\frac{\overline{w}_\e^{(1),+}}{|\overline{w}_\e^{(1),+}|}$, then it holds
 $\frac{\overline{w}_\e}{|\overline{w}_\e|}-\frac{\overline{w}_\e^{(1),+}}{|\overline{w}_\e^{(1),+}|}\to 0$
  and
  \begin{small} \begin{equation}\label{sec7-95}
 \frac{k(|\overline{w}_\e| )}{ |\overline{w}_\e|^2 \e^{2\beta}(\ln \e +2\pi \mathcal{R}_{\Omega}(0))}= \frac{
2\pi\mu_m}{\tau^2(\tau+1)}+o(1).
 \end{equation}\end{small}Also we recall \begin{small}
 \begin{equation}\label{sec7-96}
\begin{split}
  \frac{k(|\overline{w}_\e^{(1),+}|,\tau )}{ |\overline{w}_\e|^2 \e^{2\beta}(\ln \e +2\pi \mathcal{R}_{\Omega}(0))}=\frac{
2\pi\mu_m}{\tau^2(\tau+1)}.
 \end{split}
 \end{equation}\end{small}Hence from
 \eqref{sec7-95} and \eqref{sec7-96}, we get
\begin{small}
 \begin{equation*}
\begin{split}
   \frac{ k(|\overline{w}_\e|,\tau )}{ |\overline{w}_\e|^2 \e^{2\beta}(\ln \e +2\pi \mathcal{R}_{\Omega}(0))}-\frac{ k(|\overline{w}_\e^{(1),+}|,\tau )}{ |\overline{w}_\e^{(1),+}|^2 \e^{2\beta}(\ln \e +2\pi \mathcal{R}_{\Omega}(0))}=o(1).
 \end{split}
 \end{equation*}\end{small}This gives that $|\overline{w}_\e|-|\overline{w}_\e^{(1),+}|=o\big(\e^{2\beta}{|\ln \e |}\big)$ and then $|\overline{w}_{\e} - \overline{w}_\e^{(1),+}|=o(1)$ by
 $\frac{\overline{w}_{\e}}{|\overline{w}_\e|}-\frac{\overline{w}_\e^{(1),+}}{|\overline{w}_\e^{(1),+}|}\to 0$.

\vskip 0.1cm

On the other hand, from $\sum^2_{j=1} \overline{w}_{\e,j}  \times \eqref{sec7-94}$, we have
\begin{small}\begin{align*}
\overline{w}_{\e}\cdot\overline{\gamma}_{\e} = \frac{ 2\pi(\tau^2-1)|\overline{w}_{\e}|^2\e^\beta}{\tau^3}\Big(\overline{w}_{\e}{\bf{M}}_1 \overline{w}_{\e} \Big)+O\big(\frac{\e^\beta}{|\ln \e |}\big).
\end{align*}\end{small}Inserting this into $\bar n_{\e,j}(w,\gamma)=0$ we get
\begin{small}\begin{align*}
\overline{\gamma}_{\e}= -\frac{2\pi(\tau^2-1)|\overline{w}_{\e}|^2\e^\beta}{\tau^3}{\bf{M}}_1\overline{w}_{\e}+ \frac{ 4\pi(\tau^2-1)\e^\beta}{\tau^3}\Big(\overline{w}_{\e}{\bf{M}}_1 \overline{w}_{\e} \Big) \overline{w}_{\e}+O\big(\frac{\e^\beta}{|\ln \e |}\big),
\end{align*}\end{small}which, together with $|\overline{w}_{\e} - \overline{w}_\e^{(1),+}|=o(1)$,  gives $|\overline{\gamma}_\e -\overline{\gamma}_\e^{(1),+}|=o\big(\e^\beta\big)$.
\end{proof}

 We now consider the non-degeneracy of the solutions of $\hat{\bf{W}}_\e(w,\gamma)
 =0$.
\begin{prop}\label{sec7-prop7.30}
If $\overline{\bf{M}}$ has two different eigenvalues $\mu_1$ and $\mu_2$, $(\overline{w}_\e,\overline{\gamma}_\e)$ is a solution of $\hat{\bf{W}}_\e(w,\gamma)=0$, then  it holds
\begin{small}
\begin{align}\label{sec7-97}
det~Jac~ \hat{\bf{W}}_\e(\overline{w}_\e,\overline{\gamma}_\e)=\frac{\pi (\mu_j-\mu_m)}{2\tau^2 C_\tau^6\e^{4\beta}\ln \e }\big(1+o(1)\big)\neq 0~~~\mbox{with}~~~j=1,2~~~\mbox{and}~~~j\neq m,
 \end{align}
\end{small}where $m\in \{1,2\}$ is such that \eqref{sec7-91} or \eqref{sec7-92} holds.
\end{prop}
\begin{proof}
Recalling  \eqref{sec7-84}  we have
\begin{small}
\begin{equation*}
\frac{\partial h_{\e,i}(w,\gamma)}{\partial w_j}=
\frac{\partial\bar h_{\e,i}(w,\gamma)}{\partial w_j} +  \frac{\partial {h}^*_{\e,i}(w,\gamma)}{\partial w_j}
  +O\left(\frac{1}{|\ln \e |}\right).
 \end{equation*}\end{small}Since $|\overline{\gamma}_\e|=O(\e^\beta)$ by \eqref{sec7-88} and Lemma \ref{sec7-lem7.29}, we get
 \begin{small}
\begin{equation}\label{h*}  \frac{\partial {h}^*_{\e,i}(\overline{w}_\e,\overline{\gamma}_\e)}{\partial w_j}
 =O\left(\frac{1}{|\ln \e |}\right).
 \end{equation}\end{small}Now we will carry out the computations only at
$(\overline{w}_\e,\overline{\gamma}_\e)= \Big(\overline{w}_\e^{(1),+}+o\big(1\big),\overline{\gamma}_\e^{(1),+}+o\big(\e^\beta\big)\Big)$. The computations for the  other
cases are similar. We take $\bf \overline{Q}_{\e,3}$ being an orthogonal matrix such that
  \begin{small}
\begin{equation*}
{\bf \overline{Q}}_{\e,3}^T\frac{\overline{w}_\e}{|\overline{w}_\e|}=
\left(
                                                                        \begin{array}{c}
                                                                          1 \\
                                                                          0 \\
                                                                        \end{array}
                                                                      \right)~~~\mbox{and}~~~\Big({\bf \overline{Q}}_{\e,3}-{\bf Q}_{\e,3}\Big)_{ij}=
                                                                     o\big(1\big),
 \end{equation*}\end{small}where $\bf Q_{\e,3}$ is the orthogonal matrix in the proof of Proposition \ref{sec7-prop7.27}.
 Now using Lemma \ref{sec7-lem7.29}, we have  $|\overline{w}_\e|-|\overline{w}_\e^{(1),+}|=o\big(\e^{2\beta}{|\ln \e |}\big)$. And then by \eqref{h*}, we get
we get
 \begin{small}
\begin{equation*}
\begin{split}
& {\bf \overline{Q}}_{\e,3}^T\Big(\frac{\partial  h_{\e,i}(\overline{w}_\e,\overline{\gamma}_\e)}{\partial w_j}\Big)_{1\leq i,j\leq 2}  {\bf \overline{Q}}_{\e,3}  ={\bf \overline{Q}}_{\e,3}^T\Big(\frac{\partial  \overline{h}_{\e,i}(\overline{w}_\e,\overline{\gamma}_\e)}{\partial w_j}\Big)_{1\leq i,j\leq 2}  {\bf \overline{Q}}_{\e,3}+\left(
                                         \begin{array}{cc}
                                           O\left(\frac{1}{|\ln \e|}
\right) & O\left(\frac{1}{|\ln \e|}
\right) \\[2mm]
                                           O\left(\frac{1}{|\ln \e|}
\right) & O\left(\frac{1}{|\ln \e|}
\right) \\
                                         \end{array}
                                       \right).
 \end{split}
 \end{equation*}\end{small}Also, by \eqref{bar-h}, we have
 \begin{small}
\begin{align*}
\frac{\partial \bar{h}_{\e,i}(\overline{w}_\e,\overline{\gamma}_\e)}{\partial {w_j}}=&
 \left[
  \frac{k(|\overline{w}_\e|,\tau)}{\e^{2\beta}|\overline{w}_\e|^2(\ln \e +2\pi \mathcal{R}_{\Omega}(0))}
\right]\delta_{ij}   -\frac{
2\pi}{\tau^2(\tau+1)} \overline{\bf{M}}_{ij}
 \notag \\&
 +\left[
  \frac{1}{\e^{2\beta}|\overline{w}_\e|^2(\ln \e +2\pi \mathcal{R}_{\Omega}(0))} \Big(|\overline{w}_\e| \cdot \frac{\partial k(|\overline{w}_\e|,\tau)}{\partial r}- 2k(|\overline{w}_\e|,\tau)\Big)
\right] \frac{\overline{w}_{\e,i}\overline{w}_{\e,j}}{|\overline{w}_\e|^2},
\end{align*}
\end{small}and
\begin{small}
\begin{align*}
{\bf \overline Q}^{T}_{\e,3}~{\bf \overline{M}}~ {\bf \overline Q}_{\e,3}=\left(
                                                                           \begin{array}{cc}
                                                                             \mu_1+o(1) & o(1) \\[2mm]
                                                                             o(1) & \mu_2+o(1) \\
                                                                           \end{array}
                                                                         \right),\,\,\,\,\,~~~~~
  \frac{ k(|\overline{w}_\e|,\tau )}{ |\overline{w}_\e|^2 \e^{2\beta}(\ln \e +2\pi \mathcal{R}_{\Omega}(0))} =\frac{
2\pi\mu_1}{\tau^2(\tau+1)}+o(1).
 \end{align*}\end{small}Hence it holds
  \begin{small}
\begin{equation*}
\begin{split}
& {\bf \overline{Q}}_{\e,3}^T\Big(\frac{\partial  h_{\e,i}(\overline{w}_\e,\overline{\gamma}_\e)}{\partial w_j}\Big)_{1\leq i,j\leq 2}  {\bf \overline{Q}}_{\e,3}  =
\left(\begin{array}{cc}
 \frac{(\tau+1)}{C_\tau^2\e^{2\beta} \ln \e } (1+o(1)) & o(1) \\[2mm]o(1) & \frac{
2\pi(\mu_1-\mu_2)}{\tau^2(\tau+1)}+o(1) \\
\end{array}\right).
 \end{split}
 \end{equation*}\end{small}Also, we can compute
 \begin{small}
\begin{equation*}
\begin{split}
& {\bf \overline{Q}}_{\e,3}^T\Big(\frac{\partial  h_{\e,i}(\overline{w}_\e,\overline{\gamma}_\e)}{\partial \gamma_j}\Big)_{1\leq i,j\leq 2}  {\bf \overline{Q}}_{\e,3} \\ =&
\underbrace{{\bf \overline{Q}}_{\e,3}^T\Big(\frac{\partial  \overline{h}_{\e,i}(\overline{w}_\e,\overline{\gamma}_\e)}{\partial \gamma_j}\Big)_{1\leq i,j\leq 2}  {\bf \overline{Q}}_{\e,3}}_{={\bf{O}}_{2\times 2}}+\underbrace{{\bf \overline{Q}}_{\e,3}^T\Big(\frac{\partial  {h}^*_{\e,i}(\overline{w}_\e,\overline{\gamma}_\e)}{\partial \gamma_j}\Big)_{1\leq i,j\leq 2}  {\bf \overline{Q}}_{\e,3}}_{=\left(\begin{array}{cc}
O\left( \frac{1}{ \e^{\beta}|\ln \e|  }   \right)& 0 \\[2mm]
0&   0\\
\end{array}\right)}+ \left(
                                         \begin{array}{cc}
                                           O\left(\frac{1}{|\ln \e|}
\right) & O\left(\frac{1}{|\ln \e|}
\right) \\[2mm]
                                           O\left(\frac{1}{|\ln \e|}
\right) & O\left(\frac{1}{|\ln \e|}
\right) \\
                                         \end{array}
                                       \right) \\=&
\left(\begin{array}{cc}
o\left( \frac{1}{ \e^{\beta}  }   \right)& o(1) \\[2mm]
o(1)&   o(1) \\
\end{array}\right).
 \end{split}
 \end{equation*}\end{small}Similarly, by direct computations, we have
  \begin{small}
\begin{equation*}
\begin{split}
{\bf \overline{Q}}_{\e,3}^T\Big(\frac{\partial  n_{\e,i}(\overline{w}_\e,\overline{\gamma}_\e)}{\partial w_j}\Big)_{1\leq i,j\leq 2}  {\bf \overline{Q}}_{\e,3}  =&
\left(\begin{array}{cc}
 O\big(1 \big)& O\big(1 \big)\\[2mm]
 O\big(1 \big) &  O\big(1 \big)  \\
\end{array}\right)
 \end{split}
 \end{equation*}\end{small}and
  \begin{small}
\begin{equation*}
\begin{split}
{\bf \overline{Q}}_{\e,3}^T\Big(\frac{\partial  n_{\e,i}(\overline{w}_\e,\overline{\gamma}_\e)}{\partial \gamma_j}\Big)_{1\leq i,j\leq 2}  {\bf \overline{Q}}_{\e,3} =&
\left(\begin{array}{cc}
-\frac{1}{ 2C_\tau^2\e^\beta}\big(1+o(1)\big)
& o(1) \\[2mm] o(1) & \frac{1}{ 2C_\tau^2\e^\beta}+o(1)\\
\end{array}\right).
 \end{split}
 \end{equation*}\end{small}That is
  \begin{small}
\begin{equation*}
\begin{split}
& \left(
  \begin{array}{cc}
   {\bf \overline{Q}}_{\e,3}^T & {\bf O}_{2\times 2} \\[4mm]
    {\bf O}_{2\times 2} & {\bf \overline{Q}}_{\e,3}^T \\
  \end{array}
\right)
 \left(
  \begin{array}{cc}
   \Big(\frac{\partial  h_{\e,i}(\overline{w}_\e,\overline{\gamma}_\e)}{\partial w_j}\Big)_{1\leq i,j\leq 2} &
   \Big( \frac{\partial  h_{\e,i}(\overline{w}_\e,\overline{\gamma}_\e)}{\partial \gamma_j}\Big)_{1\leq i,j\leq 2} \\[4mm]
   \Big( \frac{\partial  n_{\e,i}(\overline{w}_\e,\overline{\gamma}_\e)}{\partial w_j} \Big)_{1\leq i,j\leq 2}  &
   \Big( \frac{\partial  n_{\e,i}(\overline{w}_\e,\overline{\gamma}_\e)}{\partial \gamma_j}\Big)_{1\leq i,j\leq 2}  \\
  \end{array}
\right) \left(
  \begin{array}{cc}
  {\bf \overline{Q}}_{\e,3} & {\bf O}_{2\times 2} \\[4mm]
    {\bf O}_{2\times 2} &{\bf \overline{Q}}_{\e,3}\\
  \end{array}
\right)\\[3mm]&
=\left(
          \begin{array}{cccc}
          \frac{(\tau+1)}{C_\tau^2\e^{2\beta} \ln \e } (1+ o(1)) &  o(1)  &    o\left(\frac{1}{\e^\beta}\right) &   o(1)\\[2mm]
             o(1) &\frac{
2\pi(\mu_1-\mu_2)}{\tau^2(\tau+1)}+ o(1)  &   o(1) & o(1) \\[2mm]
           O(1) & O(1)  & \frac{1}{ 2C_\tau^2\e^\beta}\big(1+o(1)\big)  &  o(1) \\[2mm]
            O(1) &O(1)   & o(1) & -\frac{1}{ 2C_\tau^2\e^\beta}+ o(1)\\[2mm]
          \end{array}
        \right),
 \end{split}
 \end{equation*}\end{small}which gives
 \begin{small}
\begin{align*}
det~Jac~ \hat{\bf{W}}_\e(\overline{w}_\e,\overline{\gamma}_\e)=\frac{\pi (\mu_2-\mu_1)}{2\tau^2 C_\tau^6\e^{2\beta}\ln \e }\big(1+o(1)\big)\neq 0,
 \end{align*}
\end{small}for the case
$(\overline{w}_\e,\overline{\gamma}_\e)= \Big(\overline w_\e^{(1),+}+ o(1),\overline \gamma_\e^{(1),+}+ o(1)\Big)$.
And then
 \eqref{sec7-97}  holds for $m=1,2$.
\end{proof}

\begin{proof}[\bf{Proof of Theorem \ref{sec1-teo17}}] The proof is very similar to that of Theorem \ref{sec1-teo16}. \vskip 0.1cm

 First, \eqref{sec7-80}, \eqref{sec7-83}, Proposition \ref{sec7-prop7.26} and Proposition \ref{sec7-prop7.27} give us that
$\mathcal{KR}_{\Omega_\e}(x,y)$ has  at least four critical points, which are all nondegenerate.
\vskip 0.1cm

Also combining Proposition \ref{sec7-prop7.27}, Proposition \ref{sec7-prop7.28} and Proposition \ref{sec7-prop7.30},  we deduce that
 for any fixed $m\in \{1,2\}$, $$\nabla \mathcal{KR}_{\Omega_\e}(x,y)\Big|_{(x,y)=(\e^\beta w,\frac{-\e^\beta w+\e^{2\beta} \gamma}{\tau})}=0$$ has a unique solution on
$B\big((\overline{w}_\e^{(m),+},\overline{\gamma}_\e^{(m),+}),\delta\big)$ or
$B\big((\overline{w}_\e^{(m),-},\overline{\gamma}_\e^{(m),-}),\delta\big)$. And then   $\mathcal{KR}_{\Omega_\e}(x,y)$ has exactly four type III critical points. Moreover, if $\La_1=\La_2$, only two of them are nontrivially different.
\end{proof}

\appendix

\section{Basic estimates for Kirchhoff-Routh function}\label{sec-app1}

In  this section, we establish  some estimates for the Kirchhoff-Routh function $\mathcal{KR}_{\O_\e}(x,y)$.
For this purpose, it is crucial to estimate the regular part of the
Green's function and its derivatives.

\subsection{Estimates for  regular part of the
Green function}\label{sec-app1-1}
\begin{lem}\label{app-lem-A.1}
Let $x,y \in \Omega_\e$, it holds
\begin{small}\begin{equation}\label{App-A.1}
\begin{split}
H_{\O_\e}(x,y)=H_{(B(P,\e))^c}(x, y)+H_{\O }(x,y) +\frac{1}{2\pi} \ln \frac{|x-P|\cdot|y-P|}{\e} -
 \frac{2\pi G_{\O }(x,P)  G_\Omega(P,y)}{\ln \e+
2\pi \mathcal{R}_\Omega(P)}
+ O(\e).
\end{split}\end{equation}\end{small}In particular, we have
\begin{small}\begin{equation}\label{App-A.2}
\begin{split}
 \mathcal{R}_{\O_\e}(x)   =& \mathcal{R}_{(B(P,\e))^c}(x)+  \mathcal{R}_{\O }(x)  +\frac{1}{2\pi}\ln \frac{|x-P|^2}{\e}-  \frac{2\pi (G_\Omega(x,P))^2}{\ln \e+
2\pi \mathcal{R}_\Omega(0)}
+O(\e).
\end{split}
 \end{equation}\end{small}

\end{lem}
\begin{proof}
First, we define \begin{small}$$b_{\e}(x,y):= H_{\O_\e}(x,y)-H_{(B(P,\e))^c}(x, y)-H_{\O }(x,y) -\frac{1}{2\pi} \ln \frac{|x-P|\cdot|y-P|}{\e}+
 \frac{2\pi G_{\O }(x,P)  G_\Omega(P,y)}{\ln \e+
2\pi \mathcal{R}_\Omega(P)}.$$
\end{small}Then $\Delta_xb_{\e}(x,y)=0$. For $x\in \partial \Omega$, we have
$H_{\O }(x,y)=-\frac{1}{2\pi}\ln|x-y|~~\mbox{and}~~G_{\Omega}(x,P)=0$. Hence for $x\in \partial \Omega$, it holds
\begin{small}\begin{equation*}
\begin{split}
b_{\e}(x,y) =&G_{(B(P,\e))^c}(x,y) +\frac{1}{2\pi}\ln |x-y|-\frac{1}{2\pi}\ln \frac{|x-P|\cdot|y-P|}{\e}\\=&
\frac{1}{2\pi}\left(\ln \sqrt{\frac{|x-P|^2\cdot|y-P|2}{\e^2}-2(x-P)\cdot(y-P)+\e^2}-  \ln \frac{|x-P|\cdot|y-P|}{\e}\right)\\=&
O\left(\frac{\e^2}{|x-P|\cdot|y-P|} \right)=
O\left(\frac{\e^2}{ |y-P|} \right).\end{split}\end{equation*}
\end{small}Also for $x\in \partial B(P,\e)$,
\begin{small}\begin{equation*}
\frac{2\pi G_{\Omega}(x,P)}{ \ln \e + 2\pi H_{\Omega}(P,P)}
=\frac{ -\ln \e -2\pi H_{\Omega}(x,P)}{ \ln \e +2\pi H_{\Omega}(P,P)}
=-1+O\Big(\frac{|x-P|}{|\ln \e|}\Big)=-1+O\Big(\frac{\e}{|\ln \e|}\Big),
\end{equation*}
\end{small}and then
\begin{small}\begin{equation*}
\begin{split}b_{\e}(x,y) =&   -H_{\O }(x,y)-  \frac{1}{2\pi} \ln |y-P|-
G_{\O }(P,y)+ O(\e)=  -H_{\O }(x,y)+
H_{\O }(P,y)+ O(\e)=O(\e).
\end{split}\end{equation*}
\end{small}Hence by the maximum principle  we deduce that $b_{\e}(x,y) =O(\e)$ for  $x,y\in \Omega_\e$. Thus  \eqref{App-A.1} holds. Finally,
letting $y=x$ in \eqref{App-A.1}, we get \eqref{App-A.2}.
\end{proof}

\begin{lem}\label{app-lem-A.2}
For $x,y \in \Omega_\e$,  it holds
\begin{small}\begin{equation}\label{App-A.3}
\begin{split}
\begin{cases}
\frac{\partial H_{\O_\e}(x,y) }{\partial   x_j} =
\frac{\partial H_{(B(P,\e))^c}(x, y) }{\partial   x_j}+\frac{\partial H_{\O }(x,y) }{\partial   x_j}+\frac{x_j-P_j}{2\pi |x-P|^2} -
\frac{\partial G_{\O }(x,P) }{\partial   x_j} \frac{2\pi G_\Omega(P,y)}{\ln \e+
2\pi \mathcal{R}_\Omega(P)}\\[2mm]
\,\,\,\,\,\,\,\,\,\,\,\,\,\,\,\,\,\,\,\,\,\,\,\,\,\,\,\,\,\,\,\,\,+O\left(\frac{\e}{|\ln \e|\cdot |x-P|} + \frac{\e^2}{|x-P|^2} +\e\right),\\[3mm]
\frac{\partial H_{\O_\e}(x,y) }{\partial   y_j} =
\frac{\partial H_{(B(P,\e))^c}(x, y) }{\partial   y_j}+\frac{\partial H_{\O }(x,y) }{\partial   y_j}+\frac{y_j-P_j}{2\pi |y-P|^2} -
\frac{\partial G_{\O }(P,y) }{\partial   y_j} \frac{2\pi G_\Omega(x,P)}{\ln \e+
2\pi \mathcal{R}_\Omega(P)}
\\[2mm]\,\,\,\,\,\,\,\,\,\,\,\,\,\,\,\,\,\,\,\,\,\,\,\,\,\,\,\,\,\,\,\,\,+O\left(\frac{\e}{|\ln \e|\cdot |y-P|} + \frac{\e^2}{|y-P|^2} +\e\right).
\end{cases}\end{split}\end{equation}\end{small}In particular,
\begin{small}\begin{equation}\label{App-A.4}
\begin{split}
\frac{\partial \mathcal{R}_{\O_\e}(y) }{\partial   y_j} =&
\frac{\partial \mathcal{R}_{(B(P,\e))^c}(y) }{\partial   y_j}+\frac{\partial \mathcal{R}_{\O }(y) }{\partial   y_j}+\frac{y_j-P_j}{\pi |y-P|^2} -
\frac{\partial G_{\O }(P,y) }{\partial   y_j} \frac{4\pi G_\Omega(y,P)}{\ln \e+
2\pi \mathcal{R}_\Omega(P)}\\&
+O\left(\frac{\e}{|\ln \e|\cdot |y-P|} + \frac{\e^2}{|y-P|^2} +\e\right).
\end{split}
 \end{equation}
\end{small}
\end{lem}
\begin{proof}
Similar to the proof of Lemma \ref{app-lem-A.1},
for $j=1,2$, we define \begin{small}
\begin{equation}\label{App-A.5}
\begin{split}
b_{\e,j}(x,y):= &\frac{\partial \big(H_{\O_\e}(x,y)-H_{(B(P,\e))^c}(x, y)\big)}{\partial   y_j} -
\frac{\partial H_{\O }(x,y) }{\partial   y_j} -
 \frac{y_j-P_j}{2\pi |y-P|^2} +
\frac{\partial G_{\O }(P,y) }{\partial   y_j} \frac{2\pi G_\Omega(x,P)}{\ln \e+
2\pi \mathcal{R}_\Omega(P)}.
\end{split}\end{equation}
\end{small}Then $\Delta_xb_{\e,j}(x,y)=0$ and for $x\in \partial \Omega$, it holds
$b_{\e,j}(x,y) =
O\left(\frac{\e^2}{|y-P|^2} \right)$. Also for $x\in \partial B(P,\e)$,
we know that $b_{\e,j}(x,y) = O\Big(\e+\frac{\e}{|\ln\e|\cdot |y-P|}\Big)$. Hence by the maximum principle, for $x,y\in \Omega_\e$, we deduce that $b_{\e,j}(x,y) =O\Big(\e+\frac{\e}{|\ln\e|\cdot |y-P|}+\frac{\e^2}{ |y-P|^2}\Big)$. Thus, the second estimate of \eqref{App-A.3} holds. By similar computations, we can derive the  first estimate in  \eqref{App-A.3}.
\end{proof}

\begin{lem}\label{app-lem-A.3}
Let $x,y\in \Omega_\e$, $i,j=1,2$,  it holds
\begin{small}\begin{align}\label{App-A.6}
\frac{\partial^2 H_{\O_\e}(x,y) }{\partial   x_i\partial   x_j} =&
\frac{\partial^2 H_{(B(P,\e))^c}(x, y) }{\partial   x_i\partial   x_j}+\frac{\partial^2 H_{\O }(x,y) }{\partial   x_i\partial   x_j}+\frac{1}{2\pi |x-P|^2}\Big(\delta_{ij}-\frac{2(x_i-P_i)(x_j-P_j)}{|x-P|^2}\Big) \Big(1-\frac{\ln|y-P|}{\ln \e}\Big) \notag \\&-
\frac{\partial^2 H_{\O }(x,P) }{\partial   x_i\partial   x_j} \frac{\ln|y-P|}{\ln \e+2\pi
 \mathcal{R}_\Omega(P)}
+O\left(\frac{\e}{|\ln \e|\cdot |x-P|^2} + \e\right),
 \end{align}\end{small}
 \begin{small}\begin{equation}\label{App-A.7}
\begin{split}
\frac{\partial^2 H_{\O_\e}(x,y) }{\partial   x_i\partial   y_j} =&
\frac{\partial^2 H_{(B(P,\e))^c}(x, y) }{\partial   x_i\partial   y_j}+\frac{\partial^2 H_{\O }(x,y) }{\partial   x_i\partial   y_j}
 -\frac{2\pi}{\ln \e+
2\pi \mathcal{R}_\Omega(P)}\frac{\partial G_\Omega(x,P)}{\partial x_i}\frac{\partial G_\Omega(P,y)}{\partial y_j} \\&
+O\left( \frac{1}{dist\big\{x,\partial B(P,\e)\big\}} \Big(\frac{\e}{|\ln \e|\cdot |y-P|} +\frac{\e^2}{|y-P|^2}+ \e \Big) \right),
 \end{split}\end{equation}\end{small}and\begin{small}
 \begin{align}\label{App-A.8}
\frac{\partial^2 H_{\O_\e}(x,y) }{\partial   y_i\partial   y_j} =&
\frac{\partial^2 H_{(B(P,\e))^c}(x, y) }{\partial   y_i\partial   y_j}+\frac{\partial^2 H_{\O }(x,y) }{\partial   y_i \partial   y_j}+\frac{1}{2\pi |y-P|^2}\Big(\delta_{ij}-\frac{2(y_i-P_i)(y_j-P_j)}{|y-P|^2}\Big) \Big(1-\frac{\ln|x-P|}{\ln \e}\Big)\notag \\&-
\frac{\partial^2 H_{\O }(P,y) }{\partial   y_i \partial   y_j} \frac{\ln|x-P|}{\ln \e+
2\pi \mathcal{R}_\Omega(0)}
+O\left(\frac{\e}{|\ln \e|\cdot |y-P|^2}  +\e\right).
 \end{align}\end{small}
\end{lem}

\begin{proof}
First, for $i,j=1,2$, we define
\begin{small}\begin{equation*}
\begin{split}
\bar b_{\e,j,i}(x,y):= &\frac{\partial^2 \big(H_{\O_\e}(x,y)-H_{(B(P,\e))^c}(x, y)\big)}{\partial   y_i\partial   y_j} -
\frac{\partial^2   H_{\O }(x,y) }{\partial  y_i\partial   y_j} -
 \frac{1}{2\pi |y-P|^2}\Big(\delta_{ij}-\frac{2(y_i-P_i)(y_j-P_j)}{|y-P|^2} \Big)\\&
 +
\frac{\partial^2   G_{\O }(P,y) }{\partial  y_i\partial   y_j} \frac{2\pi G_\Omega(x,P)}{\ln \e+
2\pi \mathcal{R}_\Omega(P)}.
\end{split}
\end{equation*}
\end{small}Then $\Delta_x\bar b_{\e,j,i}(x,y)=0$ and for $x\in \partial \Omega$, it holds
$\bar b_{\e,j,i}(x,y) =
O\left(\frac{\e^2}{|y-P|} \right)$. Moreover for $x\in \partial B(P,\e)$,
we know that $\bar b_{\e,j,i}(x,y)
=  O\Big(\e+\frac{\e}{|\ln\e|\cdot |y-P|^2}\Big)$. Hence for $x,y\in \Omega_\e$, we deduce that $\bar b_{\e,j,i}(x,y) =O\Big(\e+\frac{\e}{|\ln\e|\cdot |y-P|^2}\Big)$. Also, it holds
\begin{small}\begin{equation*}
\begin{split} \frac{\partial^2   G_{\O }(P,y) }{\partial  y_i\partial   y_j} \frac{2\pi G_\Omega(x,P)}{\ln \e+
2\pi \mathcal{R}_\Omega(P)}=&
\frac{1}{2\pi |y-P|^2}\Big(\delta_{ij}-\frac{2(y_i-P_i)(y_j-P_j)}{|y-P|^2}\Big)\frac{\ln|x-P|}{\ln\e+2\pi \mathcal{R}_\Omega(P)}\\&
+\frac{\partial^2 H_{\O }(P,y) }{\partial  y_i\partial   y_j} \frac{\ln|x-P|}{\ln \e+
2\pi \mathcal{R}_\Omega(P)} + O\Big(\frac{1}{|y-P|^2\cdot|\ln \e|}\Big),
 \end{split}\end{equation*}
\end{small}and then \eqref{App-A.8} follows. Similarly we derive \eqref{App-A.6}.

\vskip 0.1cm

To prove \eqref{App-A.7}, we first note that the function $b_{\e,j}(x,y)$
  defined in
\eqref{App-A.5} is harmonic in $x$. So we have (page 22 in \cite{gt})
$|\nabla_x b_{\e,j}(x,y)|\le \frac{C}{dist\big\{x,\partial B(P,\e)\big\}} |b_{\e,j}(x,y)|$,
which, together with  \eqref{App-A.3}, gives \eqref{App-A.7}.
\end{proof}

At the end of this subsection, we give the proof of a result on $\mathcal{R}_\Omega(x)$(see Lemma \ref{sec3-lem3.2}).
\begin{proof}[{\bf Proof of Lemma \ref{sec3-lem3.2}}]
Since $\Omega$ is smooth, there exists $d_0>0$ such that for any $x\in \Omega$ with  $dist\{x,\partial \Omega\}<d_0$, there exists a unique $x'\in\partial\Omega$, satisfying $dist\{x,\partial \Omega\}=|x-x'|$. By translation and rotation, we assume that $x=(0, d_x)$, $x'=0$, and there is a $C^1$ function $\phi(y_1)$ such that $\phi(0)=0$,
$\nabla \phi(0)=0$ and
\begin{small}\[
\partial\Omega\cap B(0,\delta)=\bigl\{y:\;  y_2=\phi(y_1)\bigr\}\cap B(0,\delta),~~
\Omega\cap B(0,\delta)=\bigl\{y:\;  y_2>\phi(y_1)\bigr\}\cap B(0,\delta),
\]
\end{small}where $\delta>0$ is a small constant.
 Let $x''=(0, -d_x)$ be the reflection  of $x$ with respect to the boundary of $\Omega$. For $d_0$ small enough, $x''\not\in \Omega$. The function
$\ln \frac{1}{|y-x''|}$
is harmonic in $\Omega$.
Since $\frac{\partial H_{\Omega}(x, y)}{\partial x_i}$ is a harmonic function in $\Omega$ and on the boundary $\partial\Omega$, we have, for $i=1,2$,
\begin{small}$$
\frac{\partial H_\Omega(x, y)}{\partial x_i} =-\frac1{2\pi}\frac{x_i-y_i}{|x-y|^{2}}.
$$
\end{small}We consider two functions, defined on $\Omega$, in the following way
\begin{small}$$ f_1(y):=
 \frac1{2\pi}\frac{y_1}{|x''-y|^{2}},~~~~~~
f_2(y):=- \frac1{2\pi}\frac{d_x+y_2}{|x''-y|^{2}}.
$$
\end{small}We can verify that
\begin{small}$$
\Delta_y \Big(\frac{\partial H_\Omega(x, y)}{\partial x_i}-f_i(y)\Big) =0,~~\mbox{for}~~  y\in \Omega~~~\mbox{and}~~~i=1,2.
$$\end{small}Also
for any $y\in \partial\Omega$, in view of  $|y_2|= |\phi(y_1)|= O(|y_1|^2)$, it holds
\begin{small}\begin{align*}
\frac{\partial H_\Omega(x, y)}{\partial x_1} -f_1(y)=&\frac{d_x}{2\pi} \left(\frac{1}{|x-y|^{2}}-\frac{1}{|x''-y|^{2}}\right)
=\frac{d_x}{2\pi}\left(\frac{1}{ |y|^2 + d_x^2-2 d_x y_2}-\frac{1}{|y|^2 + d_x^2+2 d_x y_2}\right)\\=&O\left(\frac{4d_x^2y_2}{(|y|^2 + d_x^2)^2}\right)=
O\left(\frac{4d_x^2|y_1|^2}{(|y|^2 + d_x^2)^2}\right)=O\big(1\big),
\end{align*}
\end{small}and
\begin{small}\begin{align*}
\frac{\partial H_\Omega(x, y)}{\partial x_2} -f_2(y)=&\frac{1}{2\pi}
\left(\frac{y_2-d_x}{y_1^2+ ( y_2-d_x)^2}-\frac{y_2+d_x}{y_1^2+ ( y_2+d_x)^2 }\right) =O\big(d_x\big),
\end{align*}
\end{small}here we use the fact that letting $f(x)=\frac{x}{1+x^2}$, then $|f'(x)|\leq 3$.
Hence by the maximum principle, we get
\begin{small}\[
\frac{\partial H_\Omega(x,y)}{\partial x_i} =f_i(y)+O\big(1\big),~~~\mbox{for}~~~i=1,2,
\]
\end{small}uniformly in $y\in\Omega$ as $d_x\to 0$, and we get that as $d_x\to 0$,
\begin{small}$$
\frac{\partial\mathcal{R}_{\O}(x)}{\partial x_1}=2\frac{\partial H_\Omega(x,x)}{\partial x_1}=O\big(1\big)~~~~\,\,\mbox{and}~~~~\,\,
\frac{\partial\mathcal{R}_{\O}(x)}{\partial x_2}=2\frac{\partial H_\Omega(x,x)}{\partial x_2}= - \frac{1}{ 2\pi d_x }+O\big(1\big).
$$
\end{small}This completes the proof of \eqref{sec3-05}.
\end{proof}

\subsection{Estimates for Kirchhoff-Routh function}

\begin{proof}[\bf{Proof of Proposition~\ref{p1-7-8}}]

From \eqref{sec2-01}, \eqref{App-A.1} and \eqref{App-A.2}, we get
\begin{small}\begin{equation*}
\begin{split}
\mathcal{KR}_{\Omega_\e}(x,y)=&  \mathcal{KR}_{(B(P,\e))^c}(x,y) +\frac{1}{2\pi}\left( \La_1^2\ln \frac{|x-P|^2}{\e} +\La_2^2\ln \frac{|y-P|^2}{\e}+2\La_1\La_2\ln \frac{|x-P|\cdot|y-P|}{\e} \right)
\\
&+
\mathcal{KR}_{\Omega}(x,y)-\frac{2\pi (\La_1 G_\Omega(x,P)+\La_2 G_\Omega(y,P))^2}{\ln \e+
2\pi \mathcal{R}_\Omega(P)}-\frac{\La_1\La_2}{\pi}\ln|x-y|
+ O\left(\frac{1}{|\ln \e|}\right).
\end{split}
 \end{equation*}
\end{small}Also we can compute
\begin{small}\begin{equation*}
\begin{split}
\big(\La_1 G_\Omega(x,P)+\La_2 G_\Omega(y,P)\big)^2=\frac{1}{4\pi^2}\Big[
\La_1\ln|x-P|+ \La_2\ln |y-P|+2\pi\big(\La_1H_\Omega(x,P)+\La_2 H_\Omega(y,P)\big)\Big]^2.
\end{split}
 \end{equation*}
\end{small}Hence collecting the above computations, we deduce \eqref{sec2-02}.
\end{proof}

Now we give the fundamental estimate of $\nabla \mathcal{KR}_{\Omega_\e}(x,y)$.

\begin{proof}[\bf{Proof of Proposition \ref{sec3-prop3.3}}]
 From  \eqref{App-A.3} and \eqref{App-A.4}, we find that
\begin{small}\begin{equation} \label{App-A.9}
\begin{split}
 \frac{\partial \mathcal{KR}_{\Omega_\e}(x,y)}{\partial y_j}=&
 \La^2_2\Bigg[
\frac{\partial \mathcal{R}_{(B(P,\e))^c}(y) }{\partial   y_j}+\frac{\partial \mathcal{R}_{\O }(y) }{\partial   y_j}+\frac{y_j-P_j}{\pi |y-P|^2} -
\frac{\partial G_{\O }(P,y) }{\partial   y_j} \frac{4\pi G_\Omega(y,P)}{\ln \e+
2\pi \mathcal{R}_\Omega(P)}\Bigg] -
 2\La_1\La_2 \frac{\partial S(x,y)}{\partial y_j}  \\&
 +2\La_1\La_2\Bigg[ \frac{\partial H_{(B(P,\e))^c}(x, y) }{\partial   y_j}+\frac{\partial H_{\O }(x,y) }{\partial   y_j}+\frac{y_j-P_j}{2\pi |y-P|^2} -
\frac{\partial G_{\O }(P,y) }{\partial   y_j} \frac{2\pi G_\Omega(x,P)}{\ln \e+
2\pi \mathcal{R}_\Omega(P)}\Bigg]\\&
+O\left(\frac{\e}{|\ln \e|\cdot |y-P|} + \frac{\e^2}{|y-P|^2} +\e\right)
   \\=&
 \frac{\partial \mathcal{KR}_{\Omega}(x,y)}{\partial y_j}+ \frac{\partial \mathcal{KR}_{(B(P,\e))^c}(x,y)}{\partial y_j}+2\La_1\La_2\frac{\partial S(x,y)}{\partial y_j}
+\frac{\La_2(\La_1+\La_2)(y_j-P_j)}{ \pi |y-P|^2}\\&
-
\frac{\partial G_{\O }(P,y) }{\partial   y_j} \times \frac{4\pi\La_2\big(\La_1 G_\Omega(x,P)
+
\La_2 G_\Omega(y,P)\big)
}{\ln \e+
2\pi \mathcal{R}_\Omega(P)}
+O\left(\frac{\e}{|\ln \e|\cdot |y-P|} + \frac{\e^2}{|y-P|^2} +\e\right).
\end{split}\end{equation}
\end{small}Next, we know
\begin{small}\begin{equation*}
G_\Omega(x,P)=-\frac{1}{2\pi}\ln |x-P|-H_{\Omega}(P,x)
\,\,\mbox{and}\,\,\frac{\partial G_\Omega(P,y)}{\partial y_j}
=-\frac{y_j-P_j}{2\pi|y-P|^2}-\frac{\partial H_{\Omega}(P,y)}{\partial y_j}.
\end{equation*}
\end{small}Then it holds
\begin{small}\begin{equation}\label{App-A.10}\begin{split}  &\frac{4\pi\La_2\big(\La_1 G_\Omega(x,P)
+
\La_2 G_\Omega(y,P)\big)
}{\ln \e+
2\pi \mathcal{R}_\Omega(P)}
\frac{\partial G_{\O }(P,y) }{\partial   y_j}\\&
=
\frac{\La_2(\La_1\ln|x-P|+\La_2\ln|y-P|)y_j}{\pi(\ln \e+
2\pi \mathcal{R}_\Omega(P))|y-P|^2}
+O\Big(\frac{1}{|y-P|\cdot|\ln\e|}+\frac{|\ln |x-P||}{ |\ln\e|}\Big).
\end{split}\end{equation}
\end{small}Combining the above computations, we obtain
\begin{small}\begin{equation*}
\begin{split}
 \frac{\partial \mathcal{KR}_{\Omega_\e}(x,y)}{\partial y_j}=&
 \frac{\partial \mathcal{KR}_{\Omega}(x,y)}{\partial y_j}+ \frac{\partial \mathcal{KR}_{(B(P,\e))^c}(x,y)}{\partial y_j}+2\La_1\La_2\frac{\partial S(x,y)}{\partial y_j}\\&
-\frac{\La_2y_j ( \La_1\ln \frac{|x-P|}{\e}+\La_2\ln \frac{|y-P|}{\e} )  }{\pi  |y-P|^2(\ln \e+2\pi \mathcal{R}_{\Omega}(P))}
  +O\left(\frac{1}{|y-P|\cdot|\ln \e|}+\big| \frac{\ln |x-P|}{\ln \e} \big| +\frac{\e^2}{|y-P|^2} \right).
\end{split}\end{equation*}
\end{small}This proves the second identity in \eqref{sec3-11}.  The first identity in \eqref{sec3-11} can be proved in a similar manner.  Finally  \eqref{sec3-12} can be deduced by \eqref{App-A.6}, \eqref{App-A.7} and \eqref{App-A.8} as in the previous case.

\vskip 0.1cm

It remains to prove \eqref{sec3-13}.
In fact, for $P=0$ and $|x|, |y|\sim \e^{\beta}$, we can compute \eqref{App-A.10} more precisely as follows
\begin{small}\begin{equation} \label{App-A.11}
\begin{split}
&\frac{4\pi\La_2\big(\La_1 G_\Omega(x,0)
+
\La_2 G_\Omega(y,0)\big)
}{\ln \e+
2\pi \mathcal{R}_\Omega(0)}
\frac{\partial G_{\O }(0,y) }{\partial   y_j}
\\=&
\frac{\La_2\big(\La_1\ln|x|+\La_2\ln|y|+2\pi(\La_1+\La_2)\mathcal{R}_\Omega(0)\big)y_j}{\pi(\ln \e+
2\pi \mathcal{R}_\Omega(0))|y|^2}
+\underbrace{O\left(\frac{|x|+|y|}{|y|\cdot|\ln\e|}+\frac{|\ln x|}{ |\ln\e|}\right)}_{=O(1)}.
\end{split}\end{equation}
\end{small}Inserting \eqref{App-A.11} into \eqref{App-A.9} with $P=0$, we get the second equation in \eqref{sec3-13}. As before the first equation can de deduced in a very similar way. This completes  the proof.
\end{proof}

\vskip 0.2cm

Before we end this section, we discuss the expansions for $\mathcal{KR}_{\Omega_\e}(x,y)$ if $x$ and $y$ are close to $P$.
From now we assume that $P=0$ and, from Theorem \ref{sec1-teo12}(2), if $(x_\e,y_\e)$ is a type III critical point of $\mathcal{KR}_{\Omega_\e}(x,y)$, we have $|x_\e|,|y_\e|\sim \e^{\beta}$.
  Now we expand  $\nabla\mathcal{KR}_{\Omega_\e}(x,y)$ on $\big\{(x_,y)\in \Omega_\e\times \Omega_\e; |x|,|y|\sim \e^{\beta}\big\}$.
\begin{lem}\label{app-lem-A.4}
For any $x,y\in \Omega_\e:=\Omega\backslash B(0,\e)$ with $|x|,|y|\sim \e^{\beta}$, we have, for $j=1,2$,
\begin{small}\begin{equation*}
 \frac{\partial H_{\Omega_\e}(x,y) }{\partial y_j}=  \frac{\partial  H_{(B(0,\e))^c}(x, y) }{\partial y_j}+
\frac{\partial H_\O(x,y)}{\partial y_j}
+\frac{\partial H_\O(0,y)}{\partial y_j}
\frac{G_{\Omega}(x,0)}{\frac{\ln \e}{2\pi}+\mathcal{R}_{\Omega}(0)}
+
\frac{y_j}{2\pi |y|^2} \left[ \frac{G_{\Omega}(x,0)}{\frac{\ln \e}{2\pi}
+\mathcal{R}_{\Omega}(0)}+1
\right]+ O\left(\frac{\e^{1-\beta}}{|\ln \e|} \right),\end{equation*}
\end{small}and
\begin{small}\begin{equation*}
 \frac{\partial H_{\Omega_\e}(x,y) }{\partial x_j}=  \frac{\partial  H_{(B(0,\e))^c}(x, y) }{\partial x_j}+\frac{\partial H_\O(x,y)}{\partial x_j}
+\frac{\partial H_\O(x,0)}{\partial x_j}
\frac{G_{\Omega}(0,y)}{\frac{\ln \e}{2\pi}+\mathcal{R}_{\Omega}(0)}
+
\frac{x_j}{2\pi |x|^2} \left[ \frac{G_{\Omega}(0,y)}{\frac{\ln \e}{2\pi}
+\mathcal{R}_{\Omega}(0)}+1
\right]+ O\left(\frac{\e^{1-\beta}}{|\ln \e|} \right).\end{equation*}
\end{small}
\end{lem}
\begin{proof}
Since $P=0$ and $|x|,|y|\sim \e^\beta$, the results follow
  directly from \eqref{App-A.3}.
\end{proof}

Using the above expansions we derive the  following estimates on $\nabla\mathcal{KR}_{\Omega_\e}(x,y)$.
\begin{prop}\label{app-prop-A.5}
For any $x,y\in \Omega_\e:=\Omega\backslash B(0,\e)$ with $|x|,|y|\sim \e^{\beta}$, it holds for $j=1,2$,
\begin{small}\begin{equation}\label{App-A.12}
\begin{cases}
 \frac{\partial \mathcal{KR}_{\Omega_\e}(x,y)}{\partial x_j}=  \frac{\partial \mathcal{KR}_{B_\e^c}(x,y)}{\partial x_j}+\Psi_{\e,j}(x,y) +O\left(\frac{\e^{1-\beta}}{|\ln \e|} \right),\\[5mm]
\frac{\partial \mathcal{KR}_{\Omega_\e}(x,y)}{\partial y_j}= \frac{\partial \mathcal{KR}_{B_\e^c}(x,y)}{\partial y_j}+\Phi_{\e,j}(x,y) +O\left(\frac{\e^{1-\beta}}{|\ln \e|} \right),
\end{cases}
\end{equation}
\end{small}where
\begin{small}\begin{equation}\label{App-A.13}
\begin{split}
 \Psi_{\e,j}(x,y):= &\La_1\left[\Big(\frac{x_j}{|x|^2}+2\pi
 \frac{\partial H_\Omega(x,0)}{\partial x_j} \Big)\frac{\La_1G_\Omega(x,0)+\La_2G_\Omega(0,y)}{
 \ln \e+2\pi \mathcal{R}_\Omega(0)}\right.\\&\left.
 \,\,\,\,\,\,\,\,\,\,+\Big(\La_1\frac{\partial \mathcal{R}_\Omega(x)}{\partial x_j}
 +
2\La_2\frac{\partial H_\Omega(x,y)}{\partial x_j}\Big) +\frac{(\La_1+\La_2)x_j}{\pi |x|^2}
\right],
\end{split}
\end{equation}
\end{small}\begin{small}\begin{equation}\label{App-A.14}
\begin{split}
\Phi_{\e,j}(x,y):=   &\La_2\left[\Big(\frac{y_j}{|y|^2}+2\pi
 \frac{\partial H_\Omega(0,y)}{\partial y_j} \Big)\frac{\La_1G_\Omega(x,0)+\La_2G_\Omega(0,y)}{
 \ln \e+2\pi \mathcal{R}_\Omega(0)}\right.\\&\left.
 \,\,\,\,\,\,\,\,\,\,+\Big(
2\La_1\frac{\partial H_\Omega(x,y)}{\partial y_j}+ \La_2\frac{\partial \mathcal{R}_\Omega(y)}{\partial y_j}\Big) +\frac{(\La_1+\La_2)y_j}{\pi |y|^2}
\right].
\end{split}
\end{equation}
\end{small}
\end{prop}
\begin{proof}
For  any $x,y\in \Omega_\e$ with $|x|,|y|\sim \e^{\beta}$ and $j=1,2$,  the second estimate of \eqref{App-A.12} holds from \eqref{App-A.9} directly.
Similarly we derive the first estimate of \eqref{App-A.12}.
\end{proof}

\begin{lem}\label{app-lem-A.6}
For any $x,y\in \Omega_\e:=\Omega\backslash B(0,\e)$ with $|x|,|y|\sim\e^{\beta}$,  we have, for $j=1,2$,
\begin{small}\begin{equation*}
\begin{cases}
 \frac{\partial^2 \mathcal{KR}_{\Omega_\e}(x,y) }{\partial y_i\partial x_j}=  \frac{\partial^2  \mathcal{KR}_{(B(0,\e))^c}(x, y) }{\partial y_i\partial x_j}+  \frac{\partial \Psi_{\e,j}(x,y) }{\partial y_i} +O\left(\frac{\e^{1-2\beta}}{|\ln \e|} \right),\\[3mm]
 \frac{\partial^2 \mathcal{KR}_{\Omega_\e}(x,y) }{\partial y_j\partial y_i}=  \frac{\partial^2  \mathcal{KR}_{(B(0,\e))^c}(x, y) }{\partial y_j\partial y_i}+   \frac{\partial \Phi_{\e,i}(x,y)}{\partial y_j}
 +O\left(\frac{\e^{1-2\beta}}{|\ln \e|} \right),\\[3mm]
 \frac{\partial^2 \mathcal{KR}_{\Omega_\e}(x,y) }{\partial x_j\partial x_i}=  \frac{\partial^2  \mathcal{KR}_{(B(0,\e))^c}(x, y) }{\partial x_j\partial x_i}+   \frac{\partial \Psi_{\e,i}(x,y)}{\partial x_j} +O\left(\frac{\e^{1-2\beta}}{|\ln \e|} \right),\end{cases}
\end{equation*}
\end{small}where $\Psi_{\e,j}(x,y)$  and $\Phi_{\e,j}(x,y)$
are the functions in  \eqref{App-A.13} and \eqref{App-A.14}.
\end{lem}
\begin{proof}
 Using Lemma \ref{app-lem-A.3}, we can prove this lemma in a similar way
  as in Proposition \ref{app-prop-A.5}.
\end{proof}

\vskip 0.1cm

\section{Examples}\label{app-B}

In this section, we provide some examples of domains that satisfy the assumptions of our main results.

\vskip 0.1cm
\noindent\textbf{1. A disk with punctured holes}.

\vskip 0.1cm

For any fixed $y_0\in  B(0,1)$ with $|y_0|$ closing to $1$, let $\Omega=B(0,1)\backslash B(y_0,\delta)$, where $\delta$ is small. Then, from Theorem \ref{SEC1-TEO06}, $\mathcal{KR}_{\Omega}(x,y)$
has a type II critical point
$(x_\delta,y_\delta)$,
 satisfying $x_\delta\to x_0$($x_0\neq 0$) and $y_\delta\to y_0$ as $\delta\to 0$. Hence
$\frac{\partial \mathcal{KR}_{B(0,1)}(x_0,y_0)}{\partial x_i}=0$. We also have
that
 $|x_0|$ closes to $0$ since $|y_0|$ closes to $1$. We have following result.

\begin{prop}\label{lem-B-1} Let $\Omega=B(0,1)\backslash B(y_0,\delta)$
and  $(x_\delta,y_\delta)$ be as above.  Then following results hold.

\vskip 0.1cm

\begin{itemize}
  \item [(1)] The matrix $\left(\frac{\partial^2 \mathcal{KR}_\Omega(x_\delta,y_\delta)}{\partial y_i\partial y_j}   \right)_{1\leq i,j\leq 2}$ is invertible and the matrix
\begin{small}\begin{equation*}
\textbf{M}_0 =
   \left(\frac{\partial^2 \mathcal{KR}_\Omega(x_\delta,y_\delta)}{\partial x_i\partial x_j}   \right)_{1\leq i,j\leq 2}-\left(\frac{\partial^2 \mathcal{KR}_\Omega(x_\delta,y_\delta)}{\partial x_i\partial y_j}   \right)_{1\leq i,j\leq 2}
   \left( \left(\frac{\partial^2 \mathcal{KR}_\Omega(x_\delta,y_\delta)}{\partial y_i\partial y_j}   \right)_{1\leq i,j\leq 2}\right)^{-1}
   \left(\frac{\partial^2 \mathcal{KR}_\Omega(x_\delta,y_\delta)}{\partial y_i\partial x_j}   \right)_{1\leq i,j\leq 2}
\end{equation*}
\end{small}has two different positive eigenvalues.\vskip 0.1cm
  \item [(2)] The matrix
\begin{small}\[
\widetilde{\mathbf{M}}
:= \left[\frac{\partial^2 H_{\Omega}(x_\delta,x_\delta)}{\partial y_i \partial y_j}
- 3\pi
\frac{\partial \mathcal{R}_{\Omega}(x_\delta)}{\partial y_i}
\frac{\partial \mathcal{R}_{\Omega}(x_\delta)}{\partial y_j}
\right]_{1 \le i,j\le 2}
\]\end{small}has two different  eigenvalues.
\end{itemize}
\end{prop}
\begin{proof}(1) First, by \eqref{5-30}(Swapping the order of $x$ and $y$) we have
\begin{equation}\label{luo-p}
\begin{cases}
\frac{\partial^2 \mathcal{KR}_{\O}(x_\delta,y_\delta) }{\partial   x_i\partial   x_j} =
\frac{\partial^2 \mathcal{KR}_{B(0,1)}(x_0,y_0) }{\partial   x_i\partial x_j}
+o_\delta\left(1 \right),\\[2mm]
\frac{\partial^2 \mathcal{KR}_{\O}(x_\delta,y_\delta) }{\partial   x_i\partial   y_j} =
\frac{\partial^2 \mathcal{KR}_{B(0,1)}(x_0,y_0) }{\partial   x_i\partial y_j} +o\left(\frac{1}{|y_\delta-y_0|} \right),\\[2mm]
\frac{\partial^2 \mathcal{KR}_{\O}(x,y) }{\partial   y_i\partial   y_j}=
-\frac{\La^2_2}{\pi}
\left[\frac{\delta_{ij}} {|y_\delta-y_0|^{2}}-
\frac{2(y_{\delta,i}-y_{0,i})(y_{\delta,j}-y_{0,j})} {|y_\delta-y_0|^{4} } \right]+
\frac{\partial^2 \mathcal{KR}_{B(0,1)}(x_0,y_0) }{\partial   y_i\partial   y_j}
+o_\delta\left(1 \right).
\end{cases}\end{equation}
Let us compute the terms involving $\nabla^2 \mathcal{KR}_{B(0,1)}$ in the right hand side of \eqref{luo-p}.
 From  \eqref{sec5-16a} and
$\frac{\partial \mathcal{KR}_{B(0,1)}(x_0,y_0)}{\partial x_i}=0$, we find
 $x_0 \parallel y_0$. Hence by direct computations, we have
\begin{small}\begin{equation*}
\begin{split}
\frac{\partial ^2\mathcal{KR}_{B(0,1)}(x_0,y_0)}{\partial x_i\partial x_j}=&
 \frac{\La_1\lambda_1}{\pi} \delta_{ij} +
\frac{2\La_1\lambda_2}{\pi}
\frac{x_{0,i}x_{0,j}}{|x_0|^2},
\end{split}
\end{equation*}
\end{small}with
\begin{small}\begin{equation*}
\begin{split}
\lambda_1:=\left(\frac{\La_1}{1-|x_0|^2}+\frac{\La_2}{(|y_0|-|x_0|)^2}
 -  \frac{\La_2  |y_0|^2}{( |x_0|\cdot|y_0|-1)^2 } \right),~~\lambda_2:=
 \left(\frac{\La_1|x_0|^2 }{(1-|x_0|^2)^2}-\frac{\La_2}{(|y_0|-|x_0|)^2}
+ \frac{\La_2  |y_0|^2}{( |x_0|\cdot|y_0|-1)^2 } \right).
\end{split}
\end{equation*}
\end{small}Similarly, by direct calculations, we get
\begin{small}\begin{equation*}
\begin{split}
\frac{\partial ^2\mathcal{KR}_{B(0,1)}(x_0,y_0)}{\partial x_i\partial y_j} =O(1)
~~~\mbox{and}~~~\frac{\partial ^2\mathcal{KR}_{B(0,1)}(x,y)}{\partial y_i\partial y_j}\Big|_{(x,y)=(x_\delta,y_\delta)}= O(1).
\end{split}
\end{equation*}
\end{small}This shows that
$\left(\frac{\partial^2 \mathcal{KR}_\Omega(x_\delta,y_\delta)}{\partial y_i\partial y_j}   \right)_{1\leq i,j\leq 2}$ is invertible, since its  two eigenvalues are
\begin{small}\begin{equation*}
\mu_{\delta,1}=\frac{\La_2^2}{\pi|y_\delta-y_0|^2}\Big(1+o(1)\Big), ~~~~
\mu_{\delta,2}=-\frac{\La_2^2}{\pi|y_\delta-y_0|^2}\Big(1+o(1)\Big).
\end{equation*}
\end{small}The  above computations also yield
\begin{small}\begin{equation*}
 \left(\frac{\partial^2 \mathcal{KR}_\Omega(x_\delta,y_\delta)}{\partial x_i\partial y_j}   \right)_{1\leq i,j\leq 2}
   \left( \left(\frac{\partial^2 \mathcal{KR}_\Omega(x_\delta,y_\delta)}{\partial y_i\partial y_j}   \right)_{1\leq i,j\leq 2}\right)^{-1}
   \left(\frac{\partial^2 \mathcal{KR}_\Omega(x_\delta,y_\delta)}{\partial y_i\partial x_j}   \right)_{1\leq i,j\leq 2}=\left(
                               \begin{array}{cc}
                                 o_\delta(1) &  o_\delta(1) \\[2mm]
                                  o_\delta(1)&  o_\delta(1)\\
                               \end{array}
                             \right).
\end{equation*}
\end{small}Hence the eigenvalues of
$\bf{M}_0$  are given by
\begin{small}\begin{equation*}
\lambda_{\delta,1}=\frac{\La_1\lambda_1}{\pi} +o_{\delta}(1),~~~
\lambda_{\delta,2}=\frac{\La_1(\lambda_1+2\lambda_2)}{\pi}
+o_{\delta}(1).
\end{equation*}\end{small}Here we point out that if $y_0$ closes to $\partial B(0,1)$,  by Theorem \ref{SEC1-TEO06}, we know that $|x_0|$ closes to $0$ and
\begin{small}\begin{equation*}
 \frac{\La_2}{(|y_0|-|x_0|)^2}
-\frac{\La_2  |y_0|^2}{( |x_0|\cdot|y_0|-1)^2 }
 = \frac{(1-|y_0|^2)(1+|y_0|^2-2|x_0|\cdot |y_0|)}{(|y_0|-|x_0|)^2 (|x_0|\cdot|y_0|-1)^2 } ~~\mbox{closes to }~0.
\end{equation*}\end{small}Note that
 $|x_0|$ closes to $0$ since $|y_0|$ closes to $1$. Thus $\lambda_1$ closes to $\La_1$
and
$\lambda_2$ closes to $0$.
Hence $\lambda_{\delta,1}>0~~~\mbox{and}~~~
\lambda_{\delta,2}>0$ if $|y_0|$ is close to 1.

\vskip 0.1cm

To prove that $\lambda_{\delta,1}\neq
\lambda_{\delta,2}$. we just need to prove $\lambda_2\neq 0$.  Now from
$\frac{\partial \mathcal{KR}_{B(0,1)}(x_0,y_0)}{\partial x_i}=0$ for $i=1,2$,  we have
\begin{small}\begin{equation*}
\begin{split}
 \frac{\La_1}{1-|x_0|^2}= \frac{\La_2}{(|y_0|-|x_0|)^2}
+
\frac{\La_2 |y_0|   }{(|x_0|\cdot |y_0| -1)|x_0|}.
\end{split}
\end{equation*}
\end{small}Putting this into the definition of $\lambda_2$, we have
\begin{small}\begin{equation*}
\lambda_2=\frac{\La_2}{( |x_0|\cdot|y_0|-1)^2(1-|x_0|^2)(|y_0|-|x_0|)^2}
\Big(( |x_0|\cdot|y_0|-1)^3+ (|y_0|-|x_0|)^3|y_0| \Big).
\end{equation*}\end{small}Now we can compute
\begin{small}\begin{equation*}
( |x_0|\cdot|y_0|-1)^3+ (|y_0|-|x_0|)^3|y_0|=\underbrace{\Big((|y_0|-|x_0|)^3(|y_0|-1)\Big)}_{<0}
+ \underbrace{\Big((|y_0|-|x_0|)^3 -( 1-|x_0|\cdot|y_0|)^3\Big)}_{<0}<0.
\end{equation*}\end{small}Hence $\textbf{M}_0$ has two different positive eigenvalues.

\vskip 0.2cm

(2) Since $|x_\delta-y_0|\geq C_0>0$ with $C_0$ independent of $\delta$, combining the computations in Remark \ref{sec7-rem7.20}, for $i,j = 1,2$, we obtain
\begin{small}
\begin{equation*}
\frac{\partial^2 H_{\Omega}(x,y)}{\partial y_i \partial y_j}\Big|_{x=y=x_\delta}
=\frac{\partial^2 H_{B(0,1)}(x,y)}{\partial y_i \partial y_j}\Big|_{x=y=x_\delta}+o_\delta(1) = \frac{|x_0|^2}{2\pi(1-|x_0|^2)^{2}}
\left( \delta_{ij} - \frac{2 x_{0,i}x_{0,j}}{|x_0|^2} \right)+o_\delta(1),
\end{equation*}
\end{small}and
\begin{small}
\begin{equation*}
\frac{\partial \mathcal{R}_{\Omega}(y)}{\partial y_i}\Big|_{y=x_\delta}=\frac{\partial \mathcal{R}_{B(0,1)}(y)}{\partial y_i}\Big|_{y=x_\delta}+o_\delta(1)
= -\frac{x_{0,i}}{\pi(1-|x_0|^2)}+o_\delta(1).
\end{equation*}
\end{small}Now we have
\begin{small}
\begin{equation*}
\frac{\partial^2 H_{\Omega}(x,y)}{\partial y_i \partial y_j}\Big|_{x=y=x_\delta}
- 3\pi\Big[
\frac{\partial \mathcal{R}_{\Omega}(y)}{\partial y_i}
\frac{\partial \mathcal{R}_{\Omega}(y)}{\partial y_j}\Big]\Big|_{y=x_\delta}
= \frac{|x_0|^2}{2\pi(1-|x_0|^2)^{2}}
  \left( \delta_{ij} - \frac{8x_{0,i} x_{0,j}}{|x_0|^2} \right)+o_\delta(1).
\end{equation*}
\end{small}Hence we find that the two eigenvalues of $\widetilde{\mathbf{M}}$ are
$\frac{|x_0|^2}{2\pi(1-|x_0|^2)^{2}}+o_\delta(1)$ and $-\frac{7|x_0|^2}{2\pi(1-|x_0|^2)^{2}}+o_\delta(1)$.

\end{proof}

\noindent\textbf{2. A result in a general domain}.
\begin{prop}\label{lem-B-2}
Let $\Omega$ be a bounded  domain. If $\La_1=\La_2$ and $dist\{P,\partial\O\}$ is small, then
the matrix \begin{small}$${\widetilde{\bf{M}}}:=\left( \frac{\partial^2H_{\Omega}(P,P)}{\partial x_i\partial x_j}-3 \pi \frac{\partial\mathcal{R}_\Omega(P)}{\partial x_i}\frac{\partial\mathcal{R}_\Omega(P)}{\partial x_j}  \right)_{1\leq i,j\leq 2}$$\end{small}has two different eigenvalues.
\end{prop}
\begin{proof}

Using the computations  in Lemma \ref{sec5-lem5.6}, we have that $H_{\Omega}(dz,P)=\frac{1}{2\pi}\ln \frac{1}{|z+e_2|}-\frac{\ln d}{2\pi}+o(1)$, where $d:=dist\{P,\partial\O\}$. Hence
\begin{small}\begin{equation*}
\begin{cases}
\frac{\partial H_{\Omega}(x,P)}{\partial x_1}=\frac{1}{d} \frac{\partial H_{\Omega}(dz,P)}{\partial z_1}
=\frac{1}{d} \left(-\frac{z_1}{\pi|z+e_2|^2} +o(1) \right),\\[2mm]
\frac{\partial H_{\Omega}(x,P)}{\partial x_2}=\frac{1}{d} \frac{\partial H_{\Omega}(dz,P)}{\partial z_2}
=\frac{1}{d} \left(-\frac{z_2+1}{\pi|z+e_2|^2} +o(1) \right).
\end{cases}\end{equation*}
\end{small}And then it holds
\begin{small}\begin{equation*}
\frac{\partial\mathcal{R}_{\Omega}(P)}{\partial x_1}=o\Big(\frac{1}{d}\Big),~~~~~~
\frac{\partial\mathcal{R}_{\Omega}(P)}{\partial x_2}
=\frac{1}{d} \left(-\frac{1}{\pi} +o(1) \right).
 \end{equation*}
\end{small}Furthermore,
\begin{small}\begin{equation*}
\frac{\partial^2 H_{\Omega}(x,P)}{\partial x_1^2}\Big|_{x=P}
=\frac{1}{d^2} \left(-\frac{1}{\pi|z+e_2|^2}+ \frac{2z_1^2}{\pi|z+e_2|^4} +o(1) \right)\Big|_{z=(z_1,z_2)=(0,1)}=\frac{1}{d^2} \left(-\frac{1}{4\pi} +o(1) \right),
 \end{equation*}
\end{small}
\begin{small}\begin{equation*}
\frac{\partial^2 H_{\Omega}(x,P)}{\partial x_1\partial x_2}\Big|_{x=P}
=\frac{1}{d^2} \left(\frac{2z_1(z_2+1)}{\pi|z+e_2|^4} +o(1) \right)\Big|_{z=(z_1,z_2)=(0,1)}=o\left(\frac{1}{d^2} \right),
 \end{equation*}
\end{small}
\begin{small}\begin{equation*}
\frac{\partial^2 H_{\Omega}(x,P)}{\partial x_2^2}\Big|_{x=P}
=\frac{1}{d^2} \left(-\frac{1}{\pi|z+e_2|^2}+ \frac{2(z_2+1)^2}{\pi|z+e_2|^4} +o(1) \right)\Big|_{z=(z_1,z_2)=(0,1)}=\frac{1}{d^2} \left(\frac{1}{4\pi} +o(1) \right).
 \end{equation*}
\end{small}Hence we obtain
\begin{small}$${\widetilde{\bf{M}}}=
\left(
 \begin{array}{cc}
  \frac{1}{d^2} \left(- \frac{1}{4\pi} +o(1) \right) & o\left(\frac{1}{d^2} \right) \\[2mm]
   o\left(\frac{1}{d^2} \right) & \frac{1}{d^2} \left(\frac{1}{4\pi} -\frac{3}{\pi}+o(1) \right) \\
\end{array}\right).
$$\end{small}And then the two eigenvalues of $\widetilde{\mathbf{M}}$ are
$\frac{1}{d^2} \left(- \frac{1}{4\pi} +o(1) \right)$ and $ \frac{1}{d^2} \left(-\frac{11}{4\pi} +o(1) \right)$.
\end{proof}
\vskip 0.1cm
\noindent\textbf{3. A result in an ellipse}.

\begin{lem}
Let
\[
\Omega_\delta=\Big\{(x_1,x_2)\in \mathbb{R}^2:\; x_1^2\big(1+\alpha_1\delta\big)^2 + x_2^2\big(1+\alpha_2\delta\big)^2 < 1,\ \delta>0,\ \alpha_1,\alpha_2\ge 0 \Big\}.
\]
Then Robin function $\mathcal{R}_{\Omega_\delta}(x)$ has a unique critical point $P=0$ and is even with respect to $x_1$ and $x_2$. Moreover, for $\delta$ small, it holds
    \begin{small}
    \begin{equation}\label{App-B.3}
    \frac{\partial^2 H_{\Omega_\delta}(0,0)}{\partial y_i \partial y_j}
    =
    -\frac{\delta}{2\pi}
    \left[
        \frac{3}{2}
        - 2\alpha_i
        - \frac{1}{2}\Big( \sum_{m=1}^2 \alpha_m \Big)
    \right]\delta_{ij}
    + O \left( \delta^2\right),
    \end{equation}
    \end{small}and
    \begin{small}
    \begin{equation}\label{App-B.4}
    \frac{\partial^2 H_{\Omega_\delta}(0,0)}{\partial y_i \partial x_j}
    =
    \left[
        -\frac{1}{2\pi}
        + \frac{\delta}{2\pi}\Big(\tfrac12 + \alpha_i\Big)
    \right]\delta_{ij}
    + O\!\left(\delta^2\right).
    \end{equation}
    \end{small}
\end{lem}

\begin{proof}
First, since $\Omega_\delta$ is symmetric with respect to both $x_1$ and $x_2$, the Robin function $\mathcal{R}_{\Omega_\delta}$ is even in $x_1$ and $x_2$. Moreover, since $\Omega_\delta$ is convex, classical results (see \cite{cf2,ct}) imply that $\mathcal{R}_{\Omega_\delta}$ has a unique critical point, namely $P=0$.

\vskip 0.05cm

Now from the computations of the Robin function in Theorem~6.1 of \cite{ggly1}, we obtain that for any $x,y\in \Omega_\delta$,
\begin{small}
\begin{equation*}
\begin{split}
\frac{\partial^2 H_\delta(x,y)}{\partial y_i\partial x_j}= \frac{\partial^2{H}_{B(0,1)}(x,y)}{\partial y_i\partial x_j}
- \frac{\delta}{2\pi}\frac{\partial^2 \big((|y|^2-1)v(x,y)\big) }{\partial y_i\partial x_j}+O\big(\delta^2+|x|^2+|y|^2\big)
\end{split}
\end{equation*}
\end{small}and \begin{small}
\begin{equation*}
\begin{split}
\frac{\partial^2 H_\delta(x,y)}{\partial y_i\partial y_j}= \frac{\partial^2{H}_{B(0,1)}(x,y)}{\partial y_i\partial y_j}
- \frac{\delta}{2\pi}\frac{\partial^2 \big((|y|^2-1)v (x,y)\big) }{\partial y_i\partial y_j}+O\big(\delta^2+|x|^2+|y|^2\big),
\end{split}
\end{equation*}
\end{small}where
\begin{small}\begin{equation*}
v(x,y)= -\frac{1}{2}(|x|^2-1)+\sum^2_{i=1}\alpha_ix_i^2+\frac{1}{2} \langle x,y\rangle + \sum^2_{i=1}\alpha_ix_iy_i+\frac{|y|^2}{2}\Big(1+\frac{1}{8}\sum^2_{i=1}\alpha_i\Big) +
\frac{1}{16} \sum^2_{i=1}
\Big(1+4\alpha_i\Big) y_i^2. \end{equation*}
\end{small}On the other hand, direct computations yield
\begin{small}
\begin{equation*}
\begin{split}
 \frac{\partial^2 \big((|y|^2-1)v(x,y)\big) }{\partial y_i\partial x_j}\Big|_{x=y=0}= -\Big(\tfrac12 + \alpha_i\Big)\delta_{ij},~~~\,\,
 \frac{\partial^2 \big((|y|^2-1)v (x,y)\big) }{\partial y_i\partial y_j}\Big|_{x=y=0}=  \left[
        \frac{3}{2}
        - 2\alpha_i
        - \frac{1}{2}\Big( \sum_{m=1}^2 \alpha_m \Big)
    \right]\delta_{ij}.
\end{split}
\end{equation*}
\end{small}Also we compute
\begin{small}\begin{equation*}
\frac{\partial^2 H_{B(0,1)}(x,y)}{\partial y_i\partial x_j}\Big|_{x=y=0}
=\left( \frac{2x_jy_i-\delta_{ij}}{2\pi \big||x|y-\frac{x}{|x|}\big|^{2}}
-\frac{1}{\pi}  \frac{(|x|^2y_i-x_i)(|y|^2x_j-y_j)}{\big||x|y-\frac{x}{|x|}\big|^{4}}\right)\Big|_{x=y=0}
 = -\frac{ 1}{2\pi }\delta_{ij},
\end{equation*}\end{small}and
\begin{small}\begin{equation*}
\frac{\partial^2 H_{B(0,1)}(x,y)}{\partial y_i\partial y_j}\Big|_{x=y=0}
=\left[\frac{1}{2\pi}  \frac{|x|^2}{\big||x|y-\frac{x}{|x|}\big|^{2}}  \delta_{ij}
-\frac{1}{\pi}  \frac{(|x|^2y_i-x_i)(|x|^2y_j-x_j)}{\big||x|y-\frac{x}{|x|}\big|^{4}}\right]\Big|_{x=y=0}=0.
\end{equation*}
\end{small}Hence  \eqref{App-B.3} and \eqref{App-B.4} hold by above computations.
\end{proof}
\begin{prop}\label{app-teo-B.2}
Let $\Omega_\delta=\Big\{(x_1,x_2)\in \mathbb{R}^2, x_1^2\big(1+\alpha_1\delta\big)^2+x_2^2\big(1+\alpha_2\delta\big)^2 <1,~~\delta>0,~~\alpha_1,\alpha_2\geq 0, \alpha_1\neq \alpha_2\Big\}$, then  ${\overline{\bf{M}}}:=\left[ (\tau^4+\tau^2+ 1)
 \frac{\partial^2H_{\Omega_\delta}(0,0)}{\partial x_i\partial x_j}
+(\tau^2-1)^2 \frac{\partial^2H_{\Omega_\delta}(0,0)}{\partial y_i\partial x_j} \right]_{1\leq i,j\leq 2}$ has two different eigenvalues  for $\delta\in (0,\delta_0]$.
\end{prop}

\begin{proof}Using \eqref{App-B.3} and \eqref{App-B.4}, we have
\begin{small}\begin{equation*}
\begin{split}
 &(\tau^4+\tau^2+ 1)
 \frac{\partial^2H_{\Omega_\delta}(0,0)}{\partial x_i\partial x_j}
+(\tau^2-1)^2 \frac{\partial^2H_{\Omega_\delta}(0,0)}{\partial y_i\partial x_j}  \\=&
\left\{-\frac{(\tau^4+\tau^2+ 1)\delta}{2\pi}  \left[\frac{3}{2}-2\alpha_i-\frac{1}{2}
\Big( \displaystyle\sum^2_{m=1}\alpha_m \Big)\right]+(\tau^2-1)^2 \left[-\frac{ 1}{2\pi }
+\frac{\delta}{2\pi} \Big(\frac{1}{2}+ \alpha_i \Big) \right]\right\}\delta_{ij} +O\Big( \delta^2 \Big) \\=&
\left[ -\frac{  (\tau^2-1)^2 }{2\pi }+\frac{\delta}{4\pi}\Big(
(\tau^4+\tau^2+ 1)\displaystyle\sum^2_{m=1}\alpha_m+6(\tau^4+ 1)\alpha_i
-(2\tau^4+5\tau^2+2)
\Big) \right]\delta_{ij} +O\Big( \delta^2 \Big).
\end{split}\end{equation*}\end{small}Hence, letting $\mu_i$ for $i=1,2$ be the eigenvalues of $\overline{\bf{M}}$, we find
\begin{small}\begin{equation*}
\mu_i=\left[ -\frac{  (\tau^2-1)^2 }{2\pi }+\frac{\delta}{4\pi}\Big(
(\tau^4+\tau^2+ 1)\displaystyle\sum^2_{m=1}\alpha_m+6(\tau^4+ 1)\alpha_i
-(2\tau^4+5\tau^2+2)
\Big) \right]  +O\Big( \delta^2 \Big),~~\mbox{for}~~i=1,2.
\end{equation*}
\end{small}Thus, if  $\alpha_1\neq \alpha_2$,  the two eigenvalues of $\overline{\bf{M}}$ are different for $\delta\in (0,\delta_0]$.
\end{proof}

\noindent\textbf{Acknowledgments} ~ Peng Luo and Shusen Yan were supported by National Key R\&D Program (No. 2023YFA1010002). Massimo Grossi and Francesca Gladiali were supported by INDAM-GNAMPA project.
Francesca Gladiali was funded by Next Generation EU-CUP-J55F21004240001, DM 737--2021, risorse 2022--2023.  Peng Luo was supported by NSFC grants (No. 12422106). Shusen Yan was supported by NSFC grants (No. 12571118).
This work has also been developed within the framework of the project e. INS-Ecosystem of Innovation for Next Generation Sardinia (cod. ECS 00000038) funded by the Italian Ministry for Research and Education (MUR) under the National Recovery and Resilience Plan (NRRP)-MISSION 4 COMPONENT 2, "From research to business" INVESTMENT 1.5, "Creation and strengthening of Ecosystems of innovation" and construction of "Territorial R\&D Leaders".

\begin{small}
\bibliographystyle{abbrv}
\bibliography{GGLYfin1} 
\end{small}
\end{document}